\documentclass[11pt,letterpaper,oneside]{amsart}

\usepackage[utf8]{inputenc}
\usepackage{geometry}

\usepackage{amsfonts,amsmath,amssymb,amsthm,amsxtra,amsbsy,mathtools,mathrsfs,dsfont,lmodern,microtype,longtable,verbatim}
\usepackage[dvipsnames]{xcolor}
\usepackage[colorlinks=true, allcolors=blue]{hyperref}

\DeclarePairedDelimiter{\abs}{\lvert}{\rvert}
\DeclarePairedDelimiter{\norm}{\lVert}{\rVert}
\DeclarePairedDelimiter{\bra}{(}{)}
\DeclarePairedDelimiter{\pra}{[}{]}
\DeclarePairedDelimiter{\set}{\{}{\}}
\DeclarePairedDelimiter{\skp}{\langle}{\rangle}
\DeclareMathAlphabet{\mathup}{OT1}{\familydefault}{m}{n}
\newcommand{\dx}[1]{\mathop{}\!\mathup{d} #1}
\newcommand{\dd}{\dx} %
\newcommand{\pderiv}[3][]{\frac{\mathop{}\!\mathup{d}^{#1} #2}{\mathop{}\!\mathup{d} #3^{#1}}}

\newcommand{\qtq}[1]{\quad\text{#1}\quad}

\newtheorem{theorem}{Theorem}[section]
\newtheorem{ftheorem}[theorem]{Formal Theorem}
\newtheorem{proposition}[theorem]{Proposition}

\newtheorem{corollary}[theorem]{Corollary}

\newtheorem{lemma}[theorem]{Lemma}
\newtheorem{remark}[theorem]{Remark}
\newtheorem{assumption}[theorem]{Assumption}

\theoremstyle{definition}
\newtheorem{definition}[theorem]{Definition}
\newtheorem{example}[theorem]{Example}

\numberwithin{equation}{section}

\newcommand{\ip}[2]{\langle #1,#2\rangle}

\DeclareMathOperator{\xpt}{\mathbb{E}}

\DeclareMathOperator{\tr}{tr}
\DeclareMathOperator{\rank}{rank}
\DeclareMathOperator{\diag}{diag}
\DeclareMathOperator{\law}{law}

\DeclareMathOperator{\C}{C}
\DeclareMathOperator{\Cov}{Cov}
\DeclareMathOperator{\cov}{cov}

\DeclareMathOperator{\im}{Im}
\DeclareMathOperator{\Id}{Id}
\DeclareMathOperator{\I}{I}
\DeclareMathOperator{\II}{II}
\DeclareMathOperator{\mean}{M}

\DeclareMathOperator{\var}{var}

\DeclareMathOperator{\AC}{AC}
\DeclareMathOperator{\CE}{CE}
\DeclareMathOperator{\MV}{MV}
\DeclareMathOperator{\MC}{MC}

\newcommand{\cA}{\mathcal A}
\newcommand{\cE}{\mathcal E}

\newcommand{\cF}{\mathcal F}
\newcommand{\cU}{\mathcal U}

\newcommand{\cH}{\mathcal H}
\newcommand{\cI}{\mathcal I}
\newcommand{\covI}{{\mathcal I}_{\cov}}
\newcommand{\cP}{\mathcal P}
\newcommand{\cW}{\mathcal W}
\newcommand{\cWv}{{\mathcal W}^{{\var}}}
\newcommand{\NN}{\mathcal{N}}

\newcommand{\MOP}{\mathcal{D}}
\newcommand{\vMOP}{\mathcal{D}^{{\var}}}

\newcommand{\eps}{\varepsilon}

\newcommand{\R}{\mathbb{R}}
\renewcommand{\S}{\mathbb{S}}
\newcommand{\SO}{\mathup{SO}}
\newcommand{\Orth}{\mathup{O}}
\newcommand{\HS}{\mathup{HS}}
\newcommand{\GL}{\mathup{GL}}
\newcommand{\Span}{\mathup{span}}

\newcommand{\sym}{\mathsf{sym}}

\newcommand{\subMfd}{\mathscr{M}}

\newcommand{\normal}{\mathsf{N}}

\newcommand{\tran}{{\mkern-1.0mu\mathsf{T}}}

\makeatletter
\let\save@mathaccent\mathaccent
\newcommand*\if@single[3]{%
	\setbox0\hbox{${\mathaccent"0362{#1}}^H$}%
	\setbox2\hbox{${\mathaccent"0362{\kern0pt#1}}^H$}%
	\ifdim\ht0=\ht2 #3\else #2\fi
}
\newcommand*\rel@kern[1]{\kern#1\dimexpr\macc@kerna}
\newcommand*\widebar[1]{\@ifnextchar^{{\wide@bar{#1}{0}}}{\wide@bar{#1}{1}}}
\newcommand*\wide@bar[2]{\if@single{#1}{\wide@bar@{#1}{#2}{1}}{\wide@bar@{#1}{#2}{2}}}
\newcommand*\wide@bar@[3]{%
	\begingroup
	\def\mathaccent##1##2{%
		\let\mathaccent\save@mathaccent
		\if#32 \let\macc@nucleus\first@char \fi
		\setbox\z@\hbox{$\macc@style{\macc@nucleus}_{}$}%
		\setbox\tw@\hbox{$\macc@style{\macc@nucleus}{}_{}$}%
		\dimen@\wd\tw@
		\advance\dimen@-\wd\z@
		\divide\dimen@ 3
		\@tempdima\wd\tw@
		\advance\@tempdima-\scriptspace
		\divide\@tempdima 10
		\advance\dimen@-\@tempdima
		\ifdim\dimen@>\z@ \dimen@0pt\fi
		\rel@kern{0.6}\kern-\dimen@
		\if#31
		\overline{\rel@kern{-0.6}\kern\dimen@\macc@nucleus\rel@kern{0.4}\kern\dimen@}%
		\advance\dimen@0.4\dimexpr\macc@kerna
		\let\final@kern#2%
		\ifdim\dimen@<\z@ \let\final@kern1\fi
		\if\final@kern1 \kern-\dimen@\fi
		\else
		\overline{\rel@kern{-0.6}\kern\dimen@#1}%
		\fi
	}%
	\macc@depth\@ne
	\let\math@bgroup\@empty \let\math@egroup\macc@set@skewchar
	\mathsurround\z@ \frozen@everymath{\mathgroup\macc@group\relax}%
	\macc@set@skewchar\relax
	\let\mathaccentV\macc@nested@a
	\if#31
	\macc@nested@a\relax111{#1}%
	\else
	\def\gobble@till@marker##1\endmarker{}%
	\futurelet\first@char\gobble@till@marker#1\endmarker
	\ifcat\noexpand\first@char A\else
	\def\first@char{}%
	\fi
	\macc@nested@a\relax111{\first@char}%
	\fi
	\endgroup
}
\makeatother

\title[Covariance-modulated optimal transport and gradient flows]{Covariance-modulated optimal transport\\ and gradient flows}
\author[M. Burger]{Martin Burger}
\address{
 Helmholtz Imaging, Deutsches Elektronen-Synchrotron DESY, Notkestr. 85, 22607
Hamburg, Germany. \and Fachbereich Mathematik, Universit\"at Hamburg, Bundesstrasse
55, Hamburg, 20146, Germany \and}
\email{martin.burger@desy.de}

\author[M. Erbar]{Matthias Erbar}
\address{Fakultät für Mathematik, Universität Bielefeld}
\email{erbar@math.uni-bielefeld.de}

\author[F. Hoffmann]{Franca Hoffmann}
\address{Department of Computing and Mathematical Sciences, Caltech}
\email{franca.hoffmann@caltech.edu}

\author[D. Matthes]{Daniel Matthes}
\address{Zentrum Mathematik, Technische Universität München}
\email{matthes@ma.tum.de}

\author[A. Schlichting]{Andr\'e Schlichting}
\address{Institute for Applied Analysis, Ulm University}
\email{andre.schlichting@uni-ulm.de}

\date{\today}



\allowdisplaybreaks

\begin{document}

\begin{abstract}
	We study a variant of the dynamical optimal transport problem in which the energy to be minimised is modulated by the covariance matrix of the distribution. 
	Such transport metrics arise naturally in mean-field limits of certain ensemble Kalman methods for solving inverse problems.
	We show that the transport problem splits into two coupled minimization problems:
    one for the evolution of mean and covariance of the interpolating curve and one for its shape.
	The latter consists in minimising the usual Wasserstein length under the constraint of maintaining fixed mean and covariance along the interpolation. 
	We analyse the geometry induced by this modulated transport distance on the space of probabilities as well as the dynamics of the associated gradient flows. 
	Those show better convergence properties in comparison to the classical Wasserstein metric in terms of exponential convergence rates independent of the Gaussian target. 
	On the level of the gradient flows a similar splitting into the evolution of moments and shapes of the distribution can be observed.
\end{abstract}

\maketitle

\tableofcontents

\subsubsection*{Acknowledgments}

The authors thank Christoph Böhm, Masha Gordina and Karen Habermann for explanations on sub-Riemannian geometry and related topics and thank François-Xavier Vialard to provide remarks and references about Riemannian metrics on groups. The authors gratefully acknowledge the comments of the anonymous referees.

\subsubsection*{Funding}

MB acknowledges support from DESY (Hamburg, Germany),
a member of the Helmholtz Association HGF. MB acknowledges partial financial support by European Union’s Horizon 2020 research
and innovation programme under the Marie Sklodowska-Curie grant agreement No. 777826
(NoMADS) and the German Science Foundation (DFG) through CRC TR 154 ”Mathematical
Modelling, Simulation and Optimization Using the Example of Gas Networks”, Subproject~C06.

ME acknowledges funding by the Deutsche Forschungsgemeinschaft (DFG, German Research Foundation) – SFB 1283/2 2021 – 317210226.

FH is supported by start-up funds at the California Institute of Technology.
FH was also supported by the Deutsche Forschungsgemeinschaft (DFG, German Research Foundation) via project 390685813 - GZ 2047/1 - HCM.

DM's research is supported by the DFG Collaborative Research Center TRR 109, ``Discretization in Geometry and Dynamics.''

AS's research is supported by the Deutsche Forschungsgemeinschaft (DFG, German Research Foundation) under Germany's Excellence Strategy EXC 2044 --390685587, Mathematics M\"unster: Dynamics--Geometry--Structure.  

\subsubsection*{Data availability}

The datasets generated during the current study are available from the corresponding author on reasonable request.

\subsubsection*{Conflict of interest}

The authors have no competing interests to declare that are relevant to the content of this article.

\newpage

\section{Introduction}

In this work, we are concerned with the following dynamical optimal transport problem: for two probability measures $\mu_0,\mu_1$ with a finite second moment, we consider their \emph{covariance-modulated transport distance} $\cW(\mu_0,\mu_1)$ given by \begin{equation}\label{eq:OT-cov-intro}
	\cW(\mu_0,\mu_1)^2 := \inf \left\{
	\int_0^1\int \frac12 |V_t|^2_{\C(\mu_t)}\dd \mu_t\dd t~:~\partial_t\mu_t+\nabla\cdot(\mu_t V_t)=0\right\}\;.
\end{equation}
Here the infimum is over curves of measures interpolating $\mu_0,\mu_1$ subject to the continuity equation and $|V|^2_{\C(\mu)}:=\langle V,\C(\mu)^{-1} V\rangle$ with $\C(\mu)$ denoting the covariance matrix of $\mu$. 

The problem \eqref{eq:OT-cov-intro} bears close resemblance with the dynamic formulation of the classical Wasserstein distance $W_2$ due to Benamou-Brenier~\cite{BenamouBrenier2000}. The new feature here is that the instantaneous cost of moving mass depends in a non-local way on the current distribution through its covariance matrix, i.e. $|V|^2$ is replaced here by the inner product $|V|^2_{\C(\mu)}$.

Whilst with the classical Wasserstein distance $W_2$ the optimal way to transport mass is along shortest paths, the same is not necessarily true for $\cW$. Instead, it can be more economic to invest energy in spreading out the distribution in order to take advantage of the smaller cost of moving when the covariance is larger. The competition between these two effects makes the analysis of the covariance-modulated transport problem both challenging and interesting.

Our first key observation is that the problem \eqref{eq:OT-cov-intro} can be equivalently written as the sum of two coupled minimization problems: one for the evolution of mean and covariance; and one constrained transport problem where mean and covariance matrix are fixed to zero and identity, respectively.
Both minimization problems are coupled through an overall optimization over orthogonal transformations of the marginals (Sections~\ref{sec:intro:split} and~\ref{sec:splitting}).
This splitting can be interpreted as a decomposition of the distance into its Gaussian part, measuring the deviation of the mean and covariance only; and its non-Gaussian part, measuring the difference in shapes after normalization.
The necessity to carry out an outer optimization over orthogonal transformations is related to the non-uniqueness of such normalizations: the class of affine transformations that normalize a given probability distribution to one with zero mean and unit covariance matrix bears the degree of freedom of an orthogonal transformation about the center of mass. 
Equivalently, the square root of a symmetric positive definite matrix is only determined up to an orthogonal transformation.

With this splitting at hand, we study the individual optimization problems for moments and shape in detail. We show the existence of and characterize optimizers for the moment problem revealing an interesting geometry on the space of mean vectors and (roots of) covariance matrices. We further show the existence of optimal curves, i.e.~geodesics, for the covariance-constrained transport problem for sufficiently close or symmetric marginals. A challenging feature here is that the covariance-constraint is critical in the sense that the energy to be minimized is of the same order as the constraint. Through the splitting, this also implies the existence of geodesics for the covariance-modulated problem.

The covariance-constrained optimal transport problem can be seen as a generalization of the variance-constrained optimal transport problem studied by Carlen and Gangbo in ~\cite{CarlenGangbo2003}. We also revisit this problem and recover the variance-constrained optimal transport distance in a splitting result for the analogous \emph{variance-modulated} optimal transport, where one replaces $|V_t|^2_{\C(\mu_t)}$ by $|V_t|^2/\var(\mu_t)$ in \eqref{eq:OT-cov-intro}. As shown in \cite{CarlenGangbo2003}, geodesics for the variance-constrained transport problem are obtained by a simple rescaling (both in time and space) of the usual Wasserstein geodesics (see Section~\ref{sec:var-intro}). This is in stark contrast to the geodesics for the covariance-constrained problem, which feature more complicated interaction between the trajectories and in general cannot be obtained from Wasserstein geodesics in this way, as we show both on analytic and numeric examples (see Section~\ref{sec:intro:geodesics} and Section~\ref{sec:existence}).

\medskip

In the second major part of this work we analyze gradient flows in the covariance-modulated transport geometry on the space of probability measures. In particular, we focus on the non-linear Fokker-Planck equation  
\begin{equation}\label{eq:CovFP}
	\partial_t\mu = \nabla\cdot\left(\C(\mu_t)\left(\nabla \mu_t + \mu_t\nabla H\right)\right)\;,
\end{equation}
which is the gradient flow of the relative entropy $\int \log(\dd\mu/\dd\pi)\dd\mu$ (relative to the equilibrium distribution $\pi(x)=e^{-H(x)}/\int e^{-H(y)}\dd y$) with respect to the distance $\cW$ as initially observed in \cite{GHLS}. Such PDEs arise naturally in the mean-field limit of particle systems preconditioned by their empirical covariance matrix as proposed in \cite{GHLS} for sampling the distribution $\pi$ in the context of Bayesian inverse problems. 

It is observed \cite{leimkuhler-matthews-weare} that pre-conditioning can be used as a tool to accelerate convergence to equilibrium. One of the motivations for our work is to give a theoretical underpinning for this observation by analyzing the longtime behavior of the non-linear Fokker Planck equation \eqref{eq:CovFP} arising in the mean field limit mainly in the case of Gaussian target measures $\pi$.

In the spirit of the splitting into shape and moments for the distance, we obtain a decomposition of the gradient flow evolution~\eqref{eq:CovFP} via a carefully chosen normalization map into (i) a simple Ornstein-Uhlenbeck dynamic for the shape, and (ii) a closed ODE for the first two moments. Based on this representation, we obtain exponential convergence towards the Gaussian target $\pi$ measured in relative entropy, Fisher information, and Wasserstein distance, respectively at a uniform rate independent of the characteristics of $\pi$ (see Section~\ref{sec:GFintro}).
Moreover, our preceding analysis of the covariance-modulated transport distance allows us to exhibit the underlying geometric reason for the uniform trend to equilibrium rooted. Namely, we show that the covariance-constraint has a striking effect on the behavior of free energy functionals along optimal curves. In particular, the Boltzmann-Shannon entropy becomes \emph{$1$-convex} along the geodesics of the covariance-constrained optimal transport distance (Section~\ref{sss:intro:convexity}). In the spirit of the seminal work~\cite{Otto-PME}, this dictates the uniform exponential convergence.
Furthermore, we establish an evolution variational inequality for the gradient flow in the constrained geometry implying an intrinsic stability result for the shape dynamic.

\subsubsection*{Connection to the literature}

Generalizing the flux in the dynamical formulation of the Wasserstein distance to expressions with more general mobilities and nonlinear dependence on the probability distribution is an interesting question in its own right that has been studied in a variety of settings \cite{Dolbeault-Nazaret-Savare,CLSS2010,MR2672546,MR3640545,MR4015181,MR4028472,MR4032229,MR4385602,MR4410035}. To the best of our knowledge, these previous works have considered local scalar mobilities. For systems of PDEs, corresponding mobilities have been defined for example in \cite{zinsl2015transport}. The only appearance of a matrix-valued mobility function for a scalar density has so far been mentioned in \cite{Lisini2009} (\cite{reich2013,reich2015probabilistic} for constant matrices). 
An optimal control perspective on modulation of the volatility matrix is recently provided in~\cite{tschiderer2023diffusion}.
A non-local formulation for a metric on probability measures appears in Stein-Variational Gradient Descent \cite{LuLuNolen2019,duncan2019geometry,NuskenRenger2023} and recently for the aggregation equation~\cite{EspositoGvalaniSchlichtingSchmidtchen2024}. In contrast to the above, the problem \eqref{eq:OT-cov-intro} studied here is concerned with a matrix-valued non-local mobility function, resulting in an anisotropic reweighting of the inner product. And indeed, the properties of the covariance-modulated optimal transport problem differ in some ways significantly from the scalar analog for the variance as described above (also see Section~\ref{sec:var-intro}).

In~\cite{LambertChewiBachBonnabelRigollet2022}, variational inference via Gaussian (mixture) approximations is connected to gradient flows resulting in an effective ODE evolution of the moments. This metric on the space of Gaussians induced from the Wasserstein space $(\cP_2(\R^d), W_2)$ is the Bures-Wasserstein metric providing a very related metric on the space of covariance matrices~\cite{Bures1969,malago2015information,Malag2018,Masarotto2018,Bhatia2019}. 
However, the metric obtained here for the mean and covariance is a different well-studied distance on the space of symmetric positive matrices emerging from a Riemannian metric $g_C(A, B)=\tr\bra*{A C^{-1} B C^{-1}}$, see~\cite{Skovgaard1984,OharaSudaAmari1996,Moakher2005,Bhatia2006,bhatia2009positive}, which appears as part of the action functional in the moment optimization problem obtained as a result of splitting problem~\eqref{eq:OT-cov-intro} into shape and moment parts as described above. 
In addition, the freedom of the choice of square-roots for symmetric positive definite matrices $\S_{\succ 0}^d$ upto an element of $\Orth(d)$ introduces by the polar decomposition of the group of matrices with positive determinant $\GL_+(d) \simeq \S_{\succ 0}^d \times \Orth(d)$, also a Riemannian metric on $\GL_+(d)$. The choice of the square-root is encoded by an additional symmetry constraint, which to the best of our knowledge, gives rise to a new and intriguing sub-Riemannian structure (see Section~\ref{sec:optimal-moments}).
It was brought to our attention, that this metric is related to the \emph{affine-invariant} metric in recent publications, see the review~\cite{ThanwerdasPennec2019} and~\cite{AmariMatsuda2024}.
For completeness, we also mention the Mahanalobis distance~\cite{mahalanobis1936generalized} used in statics as a whitening transformation for the normalization of random data. See also~\cite{KessyLewinStrimmer2018}, where the freedom of the choice of square-roots for symmetric positive definite matrices is discussed in this context.
We expect that all the above described problems have many more links to explore.

The work~\cite{CarrilloVaes} proves uniform exponential convergence in Wasserstein distance for the evolution~\eqref{eq:CovFP} towards Gaussian targets, with multiplicative constants depending on the covariance of the initial and target measure.
In our work, we can improve this result thanks to leveraging the intrinsic covariance-modulated geometry of the equation. Namely, we obtain estimates with the optimal exponential rate and considerably improved pre-exponential factors depending on a joint relative condition number of the covariances of initial and target measure.
For completeness, we note, that for Gaussian targets it is possible to also find non-reversible pre-conditioners improving the convergence rate, see~\cite{LelievreNierPavliotis2013,GuillinMonmarche2016,ArnoldSignorello2022}.

As already mentioned, the work~\cite{CarlenGangbo2003} studies second-moment constraints for Fokker-Planck equations as models for kinetic equations, which is generalized to porous media type equations in~\cite{TudorascuWunsch2011}. Similar constraints are studied  in~\cite{CagliotiPulvirentiRousset2009} for the 2d Navier-Stokes equation.
Dynamic constraints for the mean and the resulting gradient flows for the Boltzmann entropy are studied in~\cite{EberleNiethammerSchlichting2017}.

The idea of constraining moments to improve certain behavior for solutions of partial differential equations and their corresponding functional inequalities has been also noticed and employed in \cite{CarrilloDiFrancescoToscani} for the porous medium equation. 
Similarly in the context of Newtonian gravitation~\cite{Loeper2006} and general relativity~\cite{McCann2020}, the authors observe the interaction of geodesics through gravitational forces.
Here, we show that working on the constrained manifold improves the convexity properties of certain energy functionals, giving rise to improved convergence rates for the solutions of the corresponding PDEs. These results suggest that there is probably more to be understood about the general structure related to projecting PDEs on moment-constrained sub-manifolds.

\subsubsection*{Notation}

\begin{small}
	\begin{longtable}{lll}
		$e^i$ & $i$th standard unit basis vector in $\R^d$ & \\
		$x^{\otimes 2}$ & tensor square of $x\in \R^d: x\otimes x$ & \\
		$\cP(\R^d), \cP_2(\R^d)$ & probability measures on $\R^d$, with finite second moment & \\
		$\mean(\mu)$ & mean of $\mu\in\cP(\R^d)$: $\mean(\mu) = \int x \dx{\mu}$ & Equ.~\eqref{e:def:C} \\
		$\C(\mu)$ & covariance of $\mu\in\cP(\R^d)$: $\C(\mu)=\int \bigl(x-\mean(\mu)\bigr)^{\otimes 2}\dd\mu(x)$ & Equ.~\eqref{e:def:C} \\
		$x^\tran, A^\tran, A^{-\tran}$ & transpose of $x\in \R^d$, $A\in \R^{d\times d}$, and $A^{-1}$ & \\
		$A^\dagger$ & pseudo-inverse of $A\in \R^{d\times d}$ & Lem.~\ref{lem:apriori-cov}\\
		$\var(\mu)$ & variance of $\mu$: $\var \mu= \tr \C(\mu) = \int \abs{x-\mean(\mu)}^2 \dx\mu(x)$ & Equ.~\eqref{e:def:var}\\
		$\S^d, \S_{\succ 0}^d, \S_{\succcurlyeq0}^d$ & symmetric, symmetric positive, symmetric non-negative matrices & \\
		$C^{\frac{1}{2}}$ & symmetric square root of $C\in \S^d_{\succcurlyeq 0}$ \\
		$\GL_+(d)$ & invertible matrices with pos.~determinant $A\in \R^{d\times d}$: $\det A> 0$ & \\
		$\Orth(d), \SO(d)$ & orthogonal, special orthogonal matrices in $\R^{d\times d}$ & Sec.~\ref{sec:optimal-moments}  \\
		$A \succcurlyeq B$, $A \preccurlyeq B$ &  the matrix $A-B$ is positive semidefinite, negative semidefinite & \\
		$A \succ B$, $A \prec B$ &  the matrix $A-B$ is positive definite, negative definite & \\
		$[A,B]$ & commutator of $A,B\in\R^{d\times d}$: $AB-BA$ & \\
		$\abs{\xi}_C^2$ &  for $C\succ 0$, $\xi\in\R^d$: $\skp{\xi, C^{-1}\xi}$; for $C\in \S_{\succcurlyeq0}^d$ induced pseudo-norm & Equ.~\eqref{e:def:Cnorm} \\
		$\norm{A}_\HS$ & Hilbert-Schmidt or Frobenius norm of $A\in \R^{d\times d}$: $\abs[\big]{\sum_{i,j} A_{ij}^2}^{1/2}$ & Equ.~\eqref{eq:def:HS} \\
		$\lambda_{\max}(C)$ & largest eigenvalue of $C\in\S_{\succcurlyeq0}^d$, likewise $\lambda_{\min}(C)$ & \\
		$\norm{A}_2$ & spectral norm of $A\in \R^{d\times d}$: $\abs[\big]{\lambda_{\max}(AA^\tran)}^{1/2}$ & \\
		$\cP_{2,+}(\R^d)$ &  $\mu \in \cP_2(\R^d)$ such that $\C(\mu) \succ 0$ & Equ.~\eqref{eq:def:P+}\\
		$\cP_{0,\Id}(\R^d)$ & normalized probability measures: $\mean(\mu)=0$ and $\C(\mu)=\Id$ \\
		$T_{m,A}$ & normalization $T_{m,A}(x)= A^{-1}(x-m)$ for $m\in \R^d$, $A\in \GL_+(d)$ & Def.~\ref{def:normalization} \\
		$\widebar\mu\in \cP_{0,\Id}(\R^d)$ & normalization of $\mu\in \cP_{2,+}(\R^d)$ w.r.t.~sym. square root $\C(\mu)^{1/2}$ & Def.~\ref{def:normalization} \\
		$\normal_{m,C}$ & normal distribution with mean $m\in\R^d$ and covariance $C\in \S_{\succcurlyeq0}^d$ & Equ.~\eqref{eq:defnormal} \\
		$\AC([0,1],X)$ & absolutely continuous curves from $[0,1]$ into $X$ & \\
		$\CE(\mu_0,\mu_1)$  & pair $(\mu,V)$ solving the continuity equation with marginals $\mu_0,\mu_1$ & Equ.~\eqref{eq:CE-cov} \\
		$\cW(\mu_0,\mu_1)$ & covariance-modulated optimal transport distance & Equ.~\eqref{eq:OT-cov-main}\\
		$\CE_{m,C}(\mu_0,\mu_1)$ & pair $(\mu,V)\in \CE(\mu_0,\mu_1)$ with $(\mean(\mu_t),\C(\mu_t))=(m_t,C_t)$ %
		given  & Def.~\ref{def:constrainedOT} \\
		$\cW_{0,\Id}(\mu_0,\mu_1)$ & covariance-constrained optimal transport distance & Def.~\ref{def:constrainedOT} \\
		$\MC(\mu_0,\mu_1)$ & $(m,C) \in \AC\left([0,1],\R^d\times \S_{\succcurlyeq0}^d\right)$: $(m_i,C_i)=(\mean(\mu_i),\C(\mu_i))$, $i=0,1$ & Equ.~\eqref{e:def:MC} \\
		$\MC_R(\mu_0,\mu_1)$ & $(m,C) \in \MC(\mu_0,\mu_1):$ $C_1^{-1/2}A_1=R$ fixed for $A_t$ solving \eqref{eq:Adot} & Equ.~\eqref{e:def:MCR} \\
		$\MOP_R(\mu_0,\mu_1)$ & rotation-constrained moment optimization problem & Equ.~\eqref{eq:mean-cov-min:R} \\
		$\MOP(\mu_0,\mu_1)$ & moment optimization problem: $\inf_{R\in \SO(d)} \MOP_R(\mu_0,\mu_1)$ & Equ.~\eqref{eq:mean-cov-min}  \\
		$W_2(\mu_0,\mu_1)$ & Wasserstein distance with respect Euclidean norm $\abs{\cdot}^2$ & Equ.~\eqref{def:W2:BB} \\
		$W_{2,C}(\mu_0,\mu_1)$ & Wasserstein distance with respect weighted norm $\abs{\cdot}_C^2$ & Equ.~\eqref{eq:def:W2:C}
	\end{longtable}
\end{small}

\subsection{Results on covariance-modulated optimal transport}\label{sec:ResultsCovOT}

\subsubsection{Definition and first properties}

The goal of this work is to study a dynamic optimal transport distance on the space of probability densities on $\R^d$ with finite second moments $\cP_2(\R^d)$, for which the kinetic energy to be minimized depends on the local covariance of the  distribution. 
For this, we denote for a measure $\mu\in\cP_2(\R^d)$ its mean and covariance matrix by
\begin{equation}\label{e:def:C}
  \mean(\mu)=\int x\dd\mu(x)\;,
  \qquad\text{and}\qquad  
  \C(\mu)=\int \big(x-\mean(\mu)\big) \otimes\big(x-\mean(\mu)\big)\,\dd\mu(x)\;.
\end{equation}
In this way, we obtain the usual (scalar) variance as trace of the covariance matrix
\begin{equation}\label{e:def:var}
	\var(\mu) = \tr \C(\mu)= \int \abs*{x-\mean(\mu)}^2 \dx{\mu(x)} .
\end{equation}
We denote with $\CE(\mu_0,\mu_1)$ the set of pairs $(\mu,V)$, where $(\mu_t)_{t\in[0,1]}$ is a weakly continuous curve of probability measures in $\cP_2(\R^d)$ connecting $\mu_0$ and $\mu_1$ and $(V_t)_{t\in[0,1]}$ is a Borel family of vector fields such that the continuity equation
\begin{equation}
  \label{eq:CE-cov}
  \partial_t\mu_t+\nabla\cdot(\mu_t V_t)=0 
\end{equation}
holds in the distributional sense.
For $\xi\in \R^d$ and $C\in \S_{\succcurlyeq0}^d$, where $\S_{\succcurlyeq0}^d$ denotes the set of symmetric positive semi-definite matrices, we set
\begin{equation}\label{e:def:Cnorm}
    |\xi|^2_{C}:=\begin{cases}
        \ip{\xi}{C^{-1}\xi}\;,& \xi\in \im C\;,\\
        \infty\;, &\text{else}\;,
    \end{cases}
\end{equation}
where for $x,y\in \R^d:$ $\skp{x,y}$ is the standard Euclidean scalar product on $\R^d$. Also note that given $C\in \S_{\succcurlyeq0}^d$, we have the orthogonal decomposition $\R^d=\ker C\oplus \im C$. Therefore, the inverse $C^{-1}$ is well-defined on $\im C$. 

For symmetric matrices $X$ and $Y$, the notation $X \succcurlyeq Y$ (resp. $X \preccurlyeq Y$) means that $X - Y$ is positive semidefinite (resp. negative semidefinite), and similarly, $X \succ Y$ (resp. $X \prec Y$) means that
$X - Y$ is positive definite (resp. negative definite).

The first main object of study is the following modified optimal transport problem.
\begin{definition}[Covariance-modulated Optimal Transport]
\label{def:OT-cov-main}
Given  $\mu_0,\mu_1\in \cP_2(\R^d)$, set
\begin{equation}\label{eq:OT-cov-main}
\cW(\mu_0,\mu_1)^2 := \inf \left\{
\int_0^1\int \frac12 |V_t|^2_{\C(\mu_t)}\dd \mu_t\dd t~:~(\mu,V)\in\CE(\mu_0,\mu_1)\right\}\;.
\end{equation}
\end{definition}
An important first question is whether the covariance $\C(\mu_t)$ could become degenerate (singular) along curves in $\CE(\mu_0,\mu_1)$. Lemma~\ref{lem:apriori-cov} shows that if $C_0=\C(\mu_0)\succ 0$, then the same holds uniformly for any $t\in[0,1]$ provided the action \eqref{eq:def:action} of the curve is finite. 
Further, investigating cases where some directions of the initial or finite covariance are degenerate, the result shows that the evolution along curves of finite action always remains within subspaces where both $\C(\mu_0)$ and $\C(\mu_1)$ are non-degenerate. We formulate this condition for later reference in the following assumption.
\begin{assumption}\label{ass:Wfinite}
	The measures $\mu_0,\mu_1 \in\mathcal{P}_2(\R^d)$ are such that
	\[
	\im \C(\mu_0)=\im \C(\mu_1) \qquad \text{ and } \qquad  \mean(\mu_0)- \mean(\mu_1) \in \im \C(\mu_0) \,.
	\]
\end{assumption}
This assumption guarantees that $\cW(\mu_0,\mu_1)<\infty$, see Theorem~\ref{thm:main-cov}.
To keep the presentation brief, we sometimes work under the more practical assumption, that the measures $\mu_0,\mu_1 \in\mathcal{P}_2(\R^d)$ satisfy $\rank(\C(\mu_0))=\rank(\C(\mu_1))=d$, which is equivalent to $ \C(\mu_i) \succ 0$ for $i=0,1$. 
In other words, we sometimes choose to only deal with covariance matrices that are non-degenerate as otherwise, we can always reduce the problem to a potentially lower-dimensional subspace where non-degeneracy holds, as long as we are in the case $\cW(\mu_0,\mu_1)<\infty$.
	
By the direct method of calculus of variations, it is then easy to conclude, as in the classical case of the Wasserstein distance, that $\mathcal{W}$ actually defines a metric on 
\begin{equation}\label{eq:def:P+}
  \cP_{2,+}(\R^d):= \set*{\mu \in \cP_2(\R^d) : \C(\mu) \succ 0 }\,. 
\end{equation}
Alternatively, the same argument holds for measures that are non-degenerate on affine subspaces of $\cP_2(\R^d)$, thanks to Assumption~\ref{ass:Wfinite} (see Theorem~\ref{thm:W:metric} for the exact statement). 

Despite this complete characterization of the connected components of $\cP_2(\R^d)$ with respect to~$\mathcal{W}$, the existence of geodesics is a very challenging problem, since tightness of second moments is a priori not clear.
This becomes clearer and is tackled after first proving a decomposition of the covariance-modulated optimal transport problem. 

\subsubsection{Splitting in shape and moments up to rotation}\label{sec:intro:split}
To explain, how  the problem~\eqref{eq:OT-cov-main} splits into two minimization for mean and covariance, and one for a constrained transport problem, some preparations are needed.
For brevity, we introduce the notion of \emph{left square root} of a symmetric positive definite matrix $C\in\S_{\succ0}^d$: this is any (possibly non-symmetric) $A\in\R^{d\times d}$ with the property
\begin{align}
	\label{eq:Achoice}
	AA^\tran = C.
\end{align}
There is a high degree of non-uniqueness in the choice of $A$: multiplying \eqref{eq:Achoice} from left and right by the inverse of the (unique) symmetric positive definite square root $C^{\frac12}$, one sees that $C^{-\frac12}A\in\Orth(d)$, and conversely, for any $Q\in\Orth(d)$, the matrix $C^{\frac12}Q$ is a left square root of $C$. 
\begin{definition}[Normalization]\label{def:normalization}
	Given $\mu\in\cP_{2,+}(\R^d)$ with mean $m=\mean(\mu)\in\R^d$ and positive definite covariance matrix $\C(\mu)$, 
	for any left square root $A$ of $\C(\mu)$,
	define $T_{m,A}:\R^d\to\R^d$ by
	\begin{align}\label{eq:def:normalization}
		T_{m,A}(x)=A^{-1}(x-m) \qquad\text{and consequently}\qquad T_{m,A}^{-1}(x)=A x + m\;.
	\end{align}
	Then $(T_{m,A})_\#\mu$ is called a \emph{normalization of $\mu$}. The normalization with respect to the symmetric square root $A=\C(\mu)^{\frac{1}{2}}$ is denoted with $\widebar \mu$.
\end{definition}
The term \emph{normalization} reflects the fact that any such $\tilde \mu:=(T_{m,A})_\#\mu$ satisfies 
\[ \mean(\tilde\mu)=0\qquad\text{and}\qquad \C(\tilde\mu)=\Id. \]
Between two normalized measures, we introduce the \emph{constrained optimization problem}.
\begin{definition}[Covariance-constrained Optimal Transport]\label{def:constrainedOT}
	Given $\mu_0,\mu_1\in \cP_{0,\Id}(\R^d)$, that is $\mean(\mu_i)=0$ and $\C(\mu_i)=\Id$ for $i=0,1$, set
	\begin{equation}\label{eq:OT-cov-constraint-main}
		\cW_{0,\Id}(\mu_0,\mu_1)^2 := \inf\set*{\int_0^1\int \frac12|V_t|^2\dd \mu_t\dd t~:~(\mu,V)\in \CE_{0,\Id}(\mu_0,\mu_1)}\;,  
	\end{equation}
	where $\CE_{0,\Id}(\mu_0,\mu_1)$ is the set of pairs $(\mu,V)\in\CE(\mu_0,\mu_1)$ such that
	$\mean(\mu_t)=0$ and $\C(\mu_t)=\Id$ for all $t\in[0,1]$.
\end{definition}
It remains to specify the optimization problem for mean and covariance. For given $\mu_0,\mu_1\in \cP_{2,+}(\R^d)$, the respective minimization is carried out over a suitable subset of
\begin{equation}\label{e:def:MC}
	\MC(\mu_0,\mu_1) = \set*{(m,C) \in \AC([0,1],\R^d \times \S_{\succcurlyeq0}^d): m_i=\mean(\mu_i), C_i=\C(\mu_i), i=0,1 }.
\end{equation}
To single out that subset, auxiliary quantities are needed: take a curve $(m,C)\in\MC(\mu_0,\mu_1)$, and introduce a left square root $A_t$ for each $C_t$ via the solution to the initial value problem
\begin{equation}\label{eq:Adot}
	\dot A_t = \frac12 \dot C_tA_t^{-\tran} \qquad\text{with}\qquad A_0=C_0^{\frac{1}{2}} \;.
\end{equation}
For each given curve $C_t$ and choice of initial value $A_0$, the solution $A_t$ to \eqref{eq:Adot} is unique.
It is readily checked that $\dx{}/\dx{t}(A_tA_t^\tran )= \dot C_t$, so $A_t$ is indeed a left square root of $C_t$. This special choice of the left square root has been made to ensure symmetry of $A_t^{-1} \dot A_t$, which is crucial for the proof of the splitting theorem below. 
For given $C\in \AC([0,1],\S_{\succcurlyeq0}^d)$ and the curve $(A_t)_{t\in [0,1]}$ solving \eqref{eq:Adot}, we have an induced \emph{curve of co-rotations} defined by
\begin{align}
    \label{eq:def:RC}
    R:\AC([0,1],\S_{\succcurlyeq0}^d) \to \AC([0,1],\SO(d)) \qquad\text{with}\qquad 
    R[C]_t := C_t^{-\frac12}A_t.
\end{align}
Indeed, from $R[C]_0=\Id$, we deduce that $t\mapsto R[C]_t$ is absolutely continuous and $R[C]_t\in\SO(d)$ (see Remark~\ref{rem:rotation} for details). Further comments on the role of the rotation matrix $R[C]_t$ are postponed to Remark~\ref{rem:choiceR} (choice of left square root), Remark~\ref{rem:normalizations} (choice of normalization), Remark~\ref{rem:rotation} (evolution of rotation) and Remark~\ref{rem:Gaussian} (Gaussian targets).

With these preliminary definitions, we can formulate the \emph{moment optimization problem}.
\begin{definition}[Moment Optimization Problem]\label{def:MomentOT}
	For a fixed rotation $R\in \SO(d)$, set
	\begin{equation}\label{e:def:MCR}
		\MC_R(\mu_0,\mu_1)
		:=
		\set*{ (m,C)\in \MC(\mu_0,\mu_1): R[C]_1 = R \text{ in~\eqref{eq:def:RC}}} .
	\end{equation}
	The \emph{rotation-constrained moment optimization problem} is given by
	\begin{equation}
		\label{eq:mean-cov-min:R}
		\MOP_R(\mu_0,\mu_1)^2 = \inf\Big\{I(m,C)~:~(m,C)\in\MC_R(\mu_0,\mu_1)\Big\}\;,
	\end{equation}
	and the \emph{unconstrained moment optimization problem} is 
	\begin{equation}
		\label{eq:mean-cov-min}
		\MOP(\mu_0,\mu_1)^2 = \inf\Big\{I(m,C)~:~(m,C)\in\MC(\mu_0,\mu_1)\Big\} = \inf_{R\in \SO(d)} \MOP_R(\mu_0,\mu_1)^2 \;,
	\end{equation}
	where in both cases
	\begin{equation}\label{eq:def:MOP:action}
		I(m,C):=\int_0^1\frac12\skp*{\dot m_t,C_t^{-1}\dot m_t} +\frac18\tr\bra*{\dot C_tC^{-1}_t\dot C_tC^{-1}_t} \dd t\;.
	\end{equation}
\end{definition}

\begin{remark}\label{rem:equivalent:MOPactions}
The quantity to be minimized in \eqref{eq:mean-cov-min} can be equivalently rewritten in terms of $(A_t)_{t\in [0,1]}$ solving \eqref{eq:Adot} or in terms of the symmetric square root $\varSigma_t =C_t^{\frac{1}{2}}$ as
\begin{equation}\label{eq:equivalent:MOPactions}
\frac14\tr\big(\dot C_t C_t^{-1}\dot C_t  C_t^{-1}\big)= \tr\bigl(\dot A_t^\tran A_t^{-\tran}A_t^{-1}\dot A_t \bigr) = \Vert A_t^{-1}\dot A_t\Vert_\HS^2 =  \frac{1}{4}\norm*{\dot\varSigma_t\varSigma_t^{-1}+\varSigma_t^{-1}\dot\varSigma_t}_\HS^2 \;,
\end{equation}
where the last identity follows from~\eqref{eq:AinvdotA:varSigma} in Remark~\ref{rem:rotation} and $\norm{A}$ for any $A\in \R^{d\times d}$ denotes the \emph{Hilbert-Schmidt norm} defined by 
\begin{equation}\label{eq:def:HS}
	\norm{A}_{\HS}:= \tr\bigl(A^\tran\! A \bigr)^{1/2} = \abs[\Big]{\sum\nolimits_{i,j} A_{ij}^2}^{1/2} \;.
\end{equation}
Hence, the (rotation-constrained) moment optimization problem in Definition~\ref{def:MomentOT} can be equivalently expressed in terms of curves of square root matrices or the symmetric square root of $(C_t)_{t\in [0,1]}$ (see Section~\ref{sec:optimal-moments}).
\end{remark}

Since, $\MOP(\mu_0,\mu_1)$ only depends on the means and covariances of $\mu_0,\mu_1$, we will also use by slight abuse of notation $\MOP\bra[\big]{(m_0,C_0),(m_1,C_1)}$ and similarly, for $\MOP_R$, $\MC$, and $\MC_R$.

Our first key result is the following equivalent description of the covariance-modulated optimal transport problem. Recall that $\widebar\mu\in\cP_{0,\Id}(\R^d)$ denotes the normalization of $\mu\in\cP_{2,+}(\R^d)$ with respect to the symmetric and positive square root $\C(\mu)^{\frac12}$.
\begin{theorem}[Splitting the distance up to rotations]\label{thm:main-cov}
  Every $\mu_0,\mu_1\in \mathcal P_{2,+}(\R^d)$ satisfies $\cW(\mu_0,\mu_1)<\infty$ and
  \begin{align}\label{eq:rewrite-cov}
	\cW(\mu_0,\mu_1)^2 = \inf_{R\in\SO(d)} \set*{ \cW_{0,\Id}(R_\#\widebar\mu_0 , \widebar\mu_1)^2 + \MOP_R(\mu_0,\mu_1)^2 }.
  \end{align}
 \end{theorem}
It follows from Remark~\ref{rem:normalizations} that $\cW_{0,\Id}(R_\# \widebar\mu_0 , \widebar\mu_1)=\cW_{0,\Id}(\widebar\mu_0 , R^\tran_\#\widebar\mu_1)$ and therefore the expression on the right-hand side above is indeed symmetric in $\mu_0,\mu_1$. 
\begin{remark}[Relation of optimizers]\label{rem:rel-opt}
 Consider $\mu_t$ an optimizer for $\cW(\mu_0,\mu_1)$. It follows directly from the splitting result in Theorem~\ref{thm:main-cov} that  $(m_t,C_t)=(\mean(\mu_t),\C(\mu_t))$ is an optimizer for $\MOP_R(\mu_0,\mu_1)$ with $R$ given by $R[C]_1$ in \eqref{eq:def:RC} via the curve $C_t$. This $R$ is precisely the optimizer for the outer minimization problem on the right-hand side of \eqref{eq:rewrite-cov}. And defining $\hat\mu_t=\bra*{T_{m_t,A_t}}_\# \mu_t$ with $A_t$ solving \eqref{eq:Adot}, then $\hat\mu_t$ is the optimizer for the constrained optimization problem $\cW_{0,\Id}(\hat\mu_0 , \hat\mu_1)=\cW_{0,\Id}(R_\#\widebar\mu_0 , \widebar\mu_1)$.
\end{remark}
A splitting independent of the rotation can be obtained if one of the marginals has a \emph{spherical normalizations}, that is $\mu\in \cP_{2}(\R^d)$ satisfies 
\begin{equation}\label{eq:def:spherical-normalization}
	 R_\# \widebar\mu = \widebar\mu \qquad \text{ for all } R\in \SO(d) \,.
\end{equation}
\begin{corollary}[Splitting of normalized symmetric marginals]\label{cor:split-Wcov}
	Let $\mu_0,\mu_1\in \cP_{2,+}(\R^d)$ and assume one marginal has a spherical normalizations in the sense of~\eqref{eq:def:spherical-normalization}, then 
    \begin{align}\label{eq:Cov-Split:symmetric}
		\cW(\mu_0,\mu_1)^2 =  \cW_{0,\Id}(\widebar\mu_0 , \widebar\mu_1)^2 + \MOP(\mu_0,\mu_1)^2 .
	\end{align}
	In particular, the splitting holds if any of the two measures is a Gaussian.
\end{corollary}
A consequence of the splitting Theorem~\ref{thm:main-cov} is a two-sided comparison of the covariance-modulated, covariance-constrained and classical Wasserstein distances for measures with the same mean and covariance.
\begin{proposition}[Comparison for same mean and covariance]\label{prop:comp:W-W2}
	Let $\mu_0,\mu_1\in \cP_{2,+}(\R^d)$ with $\mean(\mu_0)=\mean(\mu_1)$ and $\C(\mu_0)=C=\C(\mu_1)$, then
	\begin{equation}\label{eq:comp:W-W2}
		\frac{W_2(\mu_0,\mu_1)^2}{2\lambda_{\max}(C)}
		\leq \inf_{R\in\SO(d)}  \cW_{0,\Id}(R_\#\widebar\mu_0 , \widebar\mu_1)^2 
		\leq \cW(\mu_0,\mu_1)^2
		\leq 
		\frac{W_2(\mu_0,\mu_1)^2}{\lambda_{\min}(C)}  .
	\end{equation}
	In particular, for $\C(\mu_0)=\Id=\C(\mu_1)$, all distances only differ by a factor of at most $\sqrt{2}$. In this case, any $\mu_0,\mu_1\in \cP_{0,\Id}(\R^d)$ also satisfy
	\begin{equation}
		\label{eq:W_2comp}
		\frac12 W_2(\mu_0,\mu_1)^2\leq \cW_{0,\Id}(\mu_0,\mu_1)^2\leq \frac12 W_2(\mu_0,\mu_1)^2 + o\big(W_2(\mu_0,\mu_1)^2\big) \;.
	\end{equation}
\end{proposition}
\begin{remark}[Rotation dependency]
	In the setting of Proposition~\ref{prop:comp:W-W2}, we also obtain the comparison
	\[
	\cW(\mu_0,\mu_1)^2
	\leq
	\cW_{0,\Id}(\widebar\mu_0,\widebar\mu_1)^2
	\]
	Indeed, this estimate follows by choosing $R=\Id$ in the splitting formula and observing that $(m_t,C_t)\equiv \bra*{\mean(\mu_0),C}$ for $t\in[0,1]$ is an admissible curve for $\MOP_{\Id}(\mu_0,\mu_1)$ in this case, and so $\MOP_{\Id}(\mu_0,\mu_1)=0$.
	We leave the study of the exact dependency of $\cW_{0,\Id}(R_\#\cdot,\cdot)$ and $\MOP_R(\cdot,\cdot)$ on the rotation $R\in \SO(d)$ for later works.
\end{remark}

\subsubsection{Existence of geodesics}\label{sec:intro:geodesics}

The existence of geodesics for the covariance-modulated optimal transport problem turns out to be a non-trivial problem.
The main difficulty in comparison to classical optimal transport and its recent variants is a lack of joint convexity and hence lower semicontinuity of the mapping 
\[
(\mu, J) := (\mu, \mu V) \mapsto  \frac{1}{2}\int \abs*{V}_{\C(\mu)}^2 \dx{\mu} . 
\]
In particular, it is not straightforwardly possible to pass to the limit in a minimizing sequence $(\mu^n, V^n)$ for the problem~\eqref{eq:OT-cov-intro}. Indeed, one can prove easily using classical arguments from optimal transport that even if $\mu^n\rightharpoonup \mu$ and $\mu^n V^n\rightharpoonup \mu V$, it might still happen that $\C(\mu^n) \not\to \C(\mu)$. The question about the convergence of the covariance matrix is critical in the sense that the action functional for the classical optimal transport as well as for covariance-modulated optimal transport only provides boundedness of second moments, but in general, does not imply tightness of the second moment.
We are able to prove tightness by a contradiction argument provided that the distance of the marginals is small enough.
\begin{theorem}[Existence of modulated and shape geodesics I]\label{thm:existence-smalldist}
\hfill  
\begin{itemize}
\item[(1)] Any $\mu_0,\mu_1\in \cP_{0,\Id}(\R^d)$ with $\cW_{0,\Id}(\mu_0,\mu_1)^2<\frac18$ are connected by a $\cW_{0,\Id}$-geodesic.
\item[(2)] Any $\mu_0,\mu_1\in \mathcal P_{2,+}(\R^d)$ with 
$\cW(\mu_0,\mu_1)^2<\frac18 + \MOP(\mu_0,\mu_1)$
are connected by a $\cW$-geodesic.
\end{itemize}
\end{theorem}
We also present a second approach to existence of geodesics for the constrained transport problem assuming reflection symmetry of the marginals in the following sense but no restriction on the distance.
\begin{definition}[$d$-fold reflection symmetry]\label{def:reflection-symmetry}
	The reflection operators $\sigma_1,\ldots,\sigma_d:\R^d\to\R^d$ along the $d$ canonical hyperplanes are defined by 
	\[
	\sigma_k((x_1,...,x_{k-1},x_k,x_{k+1},...,x_d)) := (x_1,...,x_{k-1},-x_k,x_{k+1},...,x_d)\,.
	\]   
	A measure $\mu\in\cP(\R^d)$ has a \emph{$d$-fold reflection symmetry} if $(\sigma_k)_\#\mu=\mu$ for all $k=1,\ldots,d$.
\end{definition}
\begin{theorem}[Existence of shape geodesics II]\label{thm:ex-opt-shape}
Let $\mu_0,\mu_1\in \mathcal P_{0,\Id}(\R^d)$ be absolutely continuous w.r.t.~Lebesgue measure and with $d$-fold reflection
symmetry (see Defintion~\ref{def:reflection-symmetry}). Then $\mu_0,\mu_1$ are connected by a $\cW_{0,\Id}$-geodesic.
\end{theorem}

The proof of Theorem~\ref{thm:ex-opt-shape} follows from Theorem~\ref{thm:existence-symmetry}, where we use a weak dual formulation for the covariance-constrained optimal transport problem (see Theorem~\ref{thm:infsup} and Section~\ref{sec:optimal-constraint} below for the formal duality representation).
The general existence of geodesics, without axis symmetry, is open. Hereby, we expect that the understanding of rotations is crucial, which in Theorem~\ref{thm:existence-symmetry} are ruled out due to the assumed axis symmetry.

The proof of Theorem~\ref{thm:ex-opt-shape} is based on a fixed-point argument, which we numerically implemented for empirical measures. 
In the two examples displayed in Figure~\ref{fig:ex:normalized:Wasserstein}, we compare the geodesics obtained for covariance-constrained optimal transport with normalized Wasserstein geodesics. More precisely, for $\mu_0,\mu_1$, let $(\mu_t)_{t\in [0,1]}$ be the Wasserstein geodesic and we compare with its normalization $(\widebar\mu_t)_{t\in[0,1]}$ according to Definition~\ref{def:normalization}. Our main observation is that both the plans and the trajectories are subtly different and a direct relationship is not apparent. Let us emphasize that this is in stark contrast to the situation for variance-constrained optimal transport, where the constrained geodesics are the normalization of Wasserstein ones, up to re-parametrization (see Remark~\ref{rem:varGeos:NormalW2} explaining the result of \cite{CarlenGangbo2003}).
\begin{figure}[htbp]
	\centering
	\hspace{0.05\textwidth}
	\includegraphics[width=0.4\textwidth]{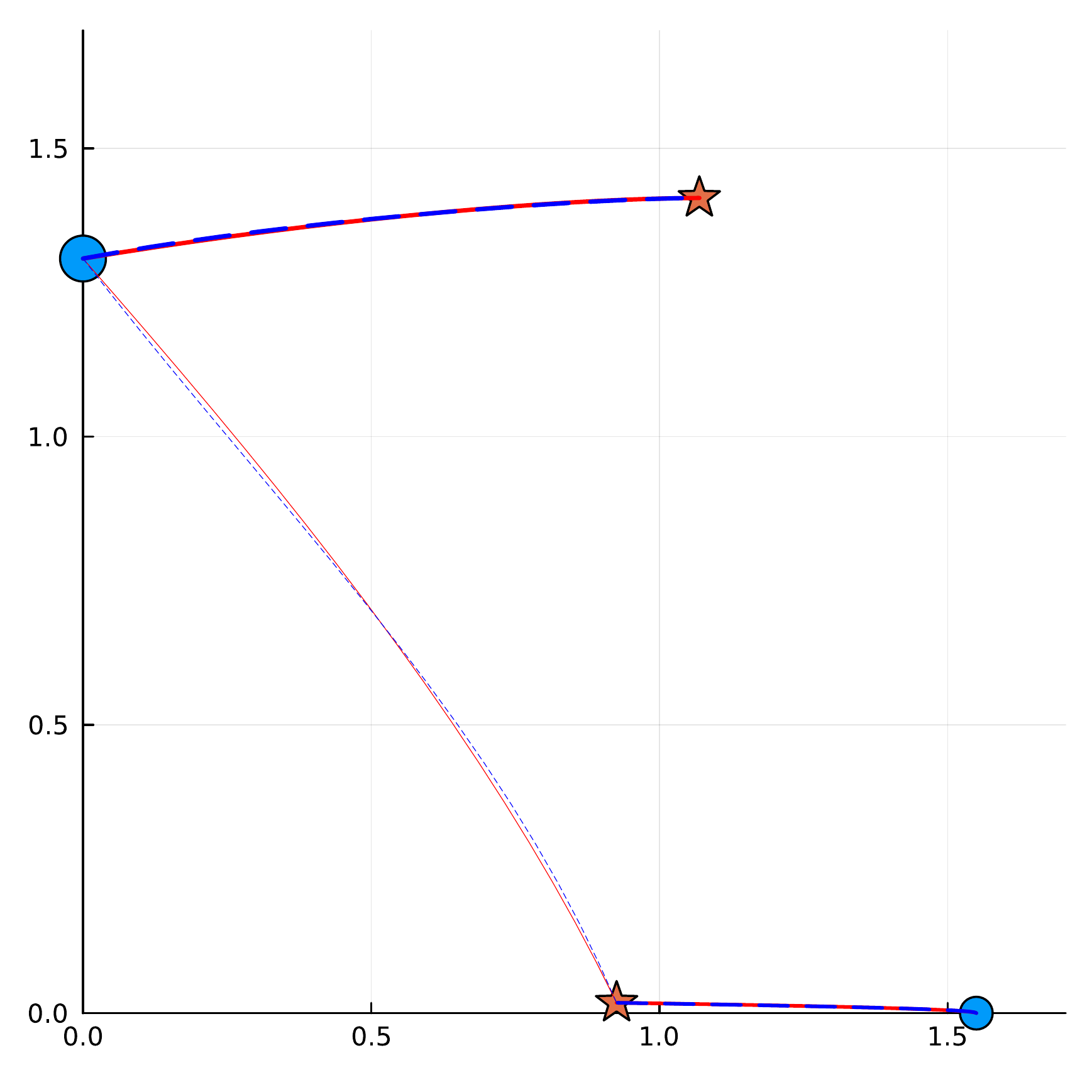}\hfill
	\includegraphics[width=0.4\textwidth]{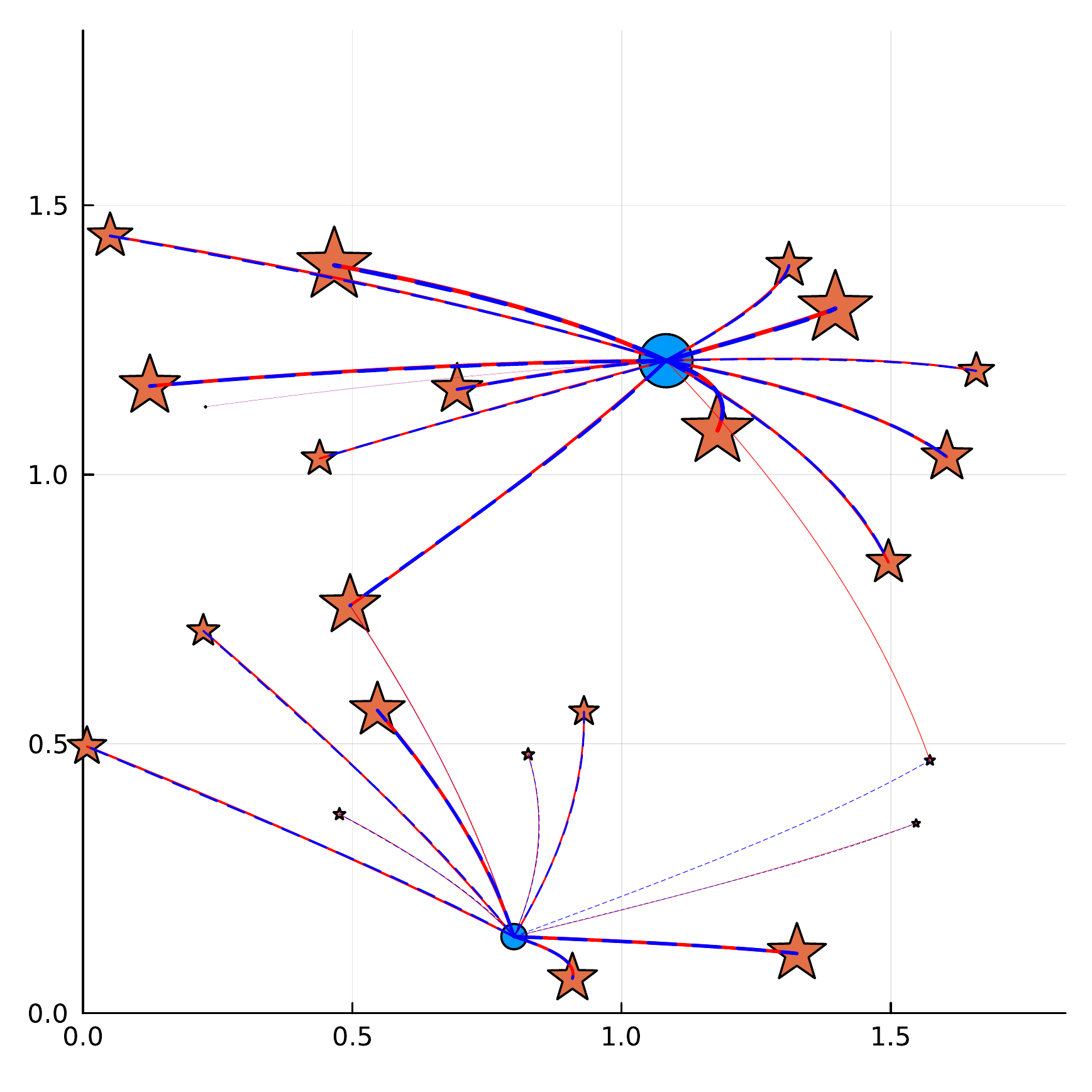}
	\hspace{0.05\textwidth}
	\caption{Comparison of the geodesics for covariance-constrained optimal transport (red) and the normalized geodesics for the classical Wasserstein distance (blue dashed). The first and second marginal are depicted with blue circles and red stars, respectively, with size representing the mass. \\
		The left picture highlights the fact, that the trajectory of mass transport are subtly different, as best observed in the transport from $(0,1.35)$ to $(0.05,0.9)$.\\
		The right picture exemplifies that the plans itself might differ, which is seen that the atom at $(1.6,0.45)$ receives mass from different sources in the covariance-constrained and normalized Wasserstein case.}\label{fig:ex:normalized:Wasserstein}
\end{figure}
Since Theorem~\ref{thm:ex-opt-shape} does not cover empirical measures, we provide in Section~\ref{sec:geo:Exs} further Examples covered by our theory highlighting both observations in a rigorous way.

We also investigate existence of optimizers for the rotational constrained and unconstrained moment problems.
\begin{theorem}[Existence of rotationally constrained moment geodesics]\label{thm:ex-Ropt-moments}
Let $C_0,C_1\in\S_{\succcurlyeq0}^d$, $R\in \SO(d)$ and $m_0,m_1\in\R^d$ such that Assumption~\ref{ass:Wfinite} holds, i.e. $m_1-m_0\in \im C_0=\im C_1$.
Then there exists an optimal pair $(m_t,C_t)_{t\in[0,1]}$ achieving the infimum for $\MOP_R$ in \eqref{eq:mean-cov-min:R}, and is given by $C_t = A_t A_t^\tran$ with $(A_t)_{t\in [0,1]}$  such that $\dot A_t A_t^\tran \in \S^d$ for $t\in[0,1]$ satisfying $A_0 = C_0^{\frac{1}{2}}$, $A_1=C_1^{\frac{1}{2}}R$ and solving the following optimality conditions: there exist $\alpha\in \R^d$ and a skew-symmetric matrix $Q$ such that  
\begin{subequations}\label{eq:EL2R}
\begin{align}
	A_t^{-\tran} A_t^{-1} \dot m_t &= \alpha \;, \label{eq:EL-mean2R} \\
	\frac{\dd}{\dd t}\big(A_t^{-1}\dot A_t\big) &=  [A_t^{-1}\dot A_t, Q] -(A_t^\tran\alpha)^{\otimes 2} \;.  \label{eq:EL-cov2R}
\end{align}
\end{subequations}
\end{theorem}
The result is a consequence of Proposition~\ref{prop:opt-mC-ex} and Proposition~\ref{prop:geo:AR}.
We establish the result using sub-Riemannian geometry by embedding the optimization problem~\eqref{eq:mean-cov-min:R} into a structure
where the constraint implied through $R[C]_1=R$ in~\eqref{e:def:MCR} can be understood as a symmetry condition (see Section~\ref{sec:optimal-moments} for details). In this way, we are able to apply results about the existence of curves and geodesics in sub-Riemannien geometry~\cite{Rifford2014}.

For spherical symmetric normalized marginals~\eqref{eq:def:spherical-normalization} (see Corollary~\ref{cor:split-Wcov}), we can also study the existence of optimizers for the covariance-constrained optimal transport \eqref{eq:OT-cov-constraint-main} and the moment optimization \eqref{eq:mean-cov-min}, separately without the need of taking the role of the rotation into account. 

Showing existence of geodesics for the unconstrained moment optimization problem \eqref{eq:mean-cov-min} is easier and even explicit solutions are available in specific cases. 
\begin{theorem}[Existence of unconstrained moment geodesics]\label{thm:ex-opt-moments}
Let $C_0,C_1\in\S_{\succcurlyeq0}^d$, and $m_0,m_1\in\R^d$ such that Assumption~\ref{ass:Wfinite} holds, i.e. $m_1-m_0\in \im C_0=\im C_1$.
Then there exists an optimal pair $(m_t,C_t)_{t\in[0,1]}$ achieving the infimum for $\MOP$ in \eqref{eq:mean-cov-min}, and satisfying for some $\alpha\in \R^d$ the optimality conditions
\begin{subequations}\label{eq:EL2}
	\begin{align}\label{eq:EL-mean2}
		C^{-1}_t\dot m_t &= \alpha\;,\\ \label{eq:EL-cov2}
		\ddot C_t &= \dot C_tC^{-1}_t \dot C_t -2 C_t(\alpha\otimes \alpha)C_t\;.
	\end{align}
\end{subequations}
Moreover, the optimizers are explicit in the following cases:
\begin{itemize}
    \item If $m_0=m_1$, then $m_t=m_0$ and
    \begin{equation}\label{eq:alpha0:Ct-intro}
        C_t=C_0^{1/2}\bra[\big]{C_0^{-1/2}C_1C_0^{-1/2}}^tC_0^{1/2} \qquad \text{for $t\in [0,1]$.}
    \end{equation}
	It follows that
    \begin{align*}
        \MOP\bra[\big]{(m_0,C_0),(m_1,C_1)}^2=
       \frac18\norm[\big]{\log\bra[\big]{C_0^{-1/2}C_1C_0^{-1/2}}}_{\HS}^2 .
    \end{align*}
    \item If $C_k=\sum_i \lambda_i(k) e^i\otimes e^i$ for $k\in \set{0,1}$, then $C(t)=\sum_{i} \lambda_i(t) e^i \otimes e^i$ and $m_i(t)=r_i\tanh(\beta_i t + \tau_i)$ for $i\in\{1,...,d\}$ and $t\in [0,1]$ are explicit solutions to~\eqref{eq:EL2} with $r_i, \beta_i, \tau_i$ being explicit constants depending on $(m_0,C_0), (m_1,C_1)$. In particular, by writing $\delta:=m_1-m_0\in \R^d$
    \begin{align}\label{eq:ex-opt-moments:orth}
        \MOP\bra[\big]{(m_0,C_0),(m_1,C_1)}^2
        \leq \sum_{i=1}^d\pra*{\mathrm{arcosh}\bra[\Big]{\frac{\delta_i}{2\lambda_i}+1}}^2\,.
    \end{align}
\end{itemize}
\end{theorem}
This result is a direct consequence of Proposition~\ref{prop:opt-mC-ex}, Proposition~\ref{prop:optimal}, Corollary~\ref{cor:alpha0} and Corollary~\ref{cor:mC:diagonal}. 
The geodesic equations~\eqref{eq:EL2} for $\alpha = 0$ provide a Riemannian distance on $\S_{\succ0}^d$, already studied in~\cite{Skovgaard1984,OharaSudaAmari1996,Moakher2005,Bhatia2006,bhatia2009positive}. We are not aware of an extension to the case $\alpha\neq 0$ including the mean, which gives rise to new effects. 
For this reason, we expect the bound~\eqref{eq:ex-opt-moments:orth} to be only optimal under suitable assumptions on the isotropy of $C_0, C_1$ and smallness of $m_1 - m_0$. 

Comparing $\cW$ with the classical $W_2$ for two identical but shifted Gaussians~$\NN(m_0,C)$ and~$\NN(m_1,C)$, we obtain from Corollary~\ref{cor:mC:diagonal} together with Corollary~\ref{cor:ass-mean} that for $\|\delta\|=\|m_1-m_0\| \gg 1$,
\begin{equation*}
    \cW(\NN(m_0,C),\NN(m_1,C))\lesssim \log\|\delta\| \,, \text{ whereas }\, W_2(\NN(m_0,C),\NN(m_1,C))=\|\delta\|\,.
\end{equation*}
This illustrates one of the fundamental differences between the covariance-modulated optimal transport distance and the classical Wasserstein distance.

\begin{remark}
	By defining $C_t = A_t A_t^\tran$ with $(m,A)$ solving~\eqref{eq:EL2R}, we obtain a geodesic equation in terms of $C_t$ for $\MOP_R$. However, it not a closed equation in $C_t$ alone, but still needs the variable $A_t$ (which can be obtained from $C_t$ by solving~\eqref{eq:Adot}).
	Indeed, we get by a calculation using the symmetry $A^{-1} \dot A = \dot A^{\tran} A^{-\tran}$ and differentiating $\dot C = A A^\tran$ the equation
	\begin{align*}
		\ddot C = \dot C C^{-1} \dot C - \dot C \frac{A Q A^{-1} + (AQA^{-1})^{\tran}}{2} \dot C - 2 (C \alpha)^{\otimes 2} . 
	\end{align*}
	Hence, the unconstrained moment geodesics~\eqref{eq:EL2} are induced by special solutions of~\eqref{eq:EL2R} for $Q=0$.	
\end{remark}

\subsubsection{Gradient flows and convergence rates to equilibrium}\label{sec:GFintro}

Endowing the space of probability measures with bounded second moment, $\mathcal{P}_2(\R^d)$, with the  covariance-modulated optimal transport distance, we can consider infinite-dimensional gradient flow structures with respect to it. 
More precisely, given some energy functional $\cF:\mathcal{P}_2(\R^d)\to \R$, we introduce the covariance-modulated evolution equation
\begin{equation}
	\label{eq:GF}
	\partial_t\rho_t =  \nabla\cdot\bra*{\rho_t \C(\rho_t) \nabla \cF'(\rho_t)}\;, \qquad\text{for } \rho_0\in \cP_2(\R^d), 
\end{equation}
where $\cF'$ denotes the first variation of $\cF$, if it exists. This formulation provides powerful tools for analysis \cite{AGS}. For example, it follows immediately from the semi-definiteness of $\C(\rho_t)$, that the energy $\cF$ decays along solutions to \eqref{eq:GF},
\begin{align*}
	\frac{d}{dt}\cF(\rho_t)
	= -\int \langle\nabla \cF'(\rho_t), \C(\rho_t) \nabla \cF'(\rho_t)\rangle \, \dx \rho_t \le 0\,.
\end{align*}
Further, we expect (local) minimizers of $\cF$ to correspond to the asymptotic profiles of equation~\eqref{eq:GF}. 
Here, and in what follows, we focus in particular on the relative entropy $\cF(\cdot)=\cE(\cdot\,|\,\pi)$ with respect to a reference density $\pi$ that is proportional to $e^{-H}$ for the family of quadratic potentials $H:\R^d\to\R$ of the form
\begin{align}\label{eq:potential}
	H(x)= \tfrac12 \abs*{x-x_0}_B^2 \qquad\text{with mean } x_0\in\R^d \text{ and covariance }  B\in\S_{\succ0}^d .
\end{align}
Then the driving free energy becomes
\begin{equation}\label{eq:def:Erel}
    \cF(\rho)=\cE(\rho\,|\,\pi) = \int \log\bra[\Big]{\frac\rho{\pi}}\dd\rho
    = \cE(\rho) +\int H \dd \rho \;,
\end{equation}
where $\cE(\rho)=\int\rho\log\rho$ denotes the Boltzmann entropy.
The gradient flow evolution \eqref{eq:GF} becomes a covariance-modulated Fokker-Planck equation
\begin{equation}
	\label{eq:GF-FP-cov-quad-intro}
	\partial_t\rho_t = \nabla\cdot\bra*{\C(\rho_t)\bra*{\nabla \rho_t + \rho_t B^{-1}(x-x_0)}}\;.
\end{equation}
A remarkable property is that if $\rho_t$ solves \eqref{eq:GF-FP-cov-quad-intro},
then the normalized solution $\eta_t=(T_{m_t,A_t})_\#\rho_t$ with $m_t=\mean(\rho_t)$ and $A_t$ solving \eqref{eq:Adot} satisfies the classical Fokker-Planck equation with potential $h(x)=\frac12|x|^2$, also called Ornstein-Uhlenbeck semigroup, 
\begin{equation}\label{eq:FP-normalized-quad-intro}
	\partial_t \eta_t = \Delta \eta_t + \nabla \cdot (x\eta_t )\,.
\end{equation}
This fundamental property of Gaussian targets is shown in Section~\ref{sec:FP} (see Lemma~\ref{lem:FP-normalized}) and is the basis for quantified sharp estimates on the longtime behavior of solutions.
Let us denote with~$\normal_{m,C}$ a Gaussian with mean $m$ and covariance $C$,
\begin{align}
	\label{eq:defnormal}
	\normal_{m,C}(x)=\frac{1}{(2\pi)^{d/2}\bra{\det C}^{1/2}} \exp\bra[\Big]{-\frac12 |x-m|_C^2}\,,
\end{align}
For solutions $\rho_t$ to the non-linear Fokker-Planck equation \eqref{eq:GF-FP-cov-quad-intro}, we consider Gaussian approximations $\normal_{m_t,C_t}$ of $\rho_t$. 
The first crucial observation is that the moments $(m_t, C_t)=(\mean(\rho_t),\C(\rho_t))$ themselves satisfy 
a closed system of equations (see~\eqref{eq:moments-GF}).
In particular, any Gaussian $\normal_{m_t,C_t}$ solves~\eqref{eq:GF-FP-cov-quad-intro} if and only if $(m_t,C_t)$ solves this closed system as already shown in \cite{GHLS}.
Further, we denote the dissipation of the relative entropy $\cE(\rho\,|\,\pi)$ defined in \eqref{eq:def:Erel} by
\begin{align}
	\label{eq:def:Icov}
	\covI(\rho\,|\,\pi) = \int \abs[\Big]{\C(\rho)^{1/2}\nabla\log\bra[\Big]{\frac\rho{\pi}}}^2\dd \rho.
\end{align}
Note that the definition of $\cE(\cdot\,|\,\pi)$ is the one for the standard Fokker-Planck equation, while $\covI$ is a modification of the usual Fisher information
\begin{align}
	\label{eq:defI}
	\cI(\rho\,|\,\pi) = \int \abs[\Big]{\nabla\log\bra[\Big]{\frac\rho{\pi}}}^2\dd \rho.
\end{align}
The modified information $\covI$ has been introduced in \cite{GHLS}.
In the spirit of splitting shapes and moments, we prove in Lemma~\ref{lem:ent-split} the following splitting of entropy and Fisher information
\begin{subequations}\label{eq:split:EntI}
	\begin{align}
	\cE(\rho\,|\,\normal_{x_0,B}) &= \cE(\widebar\rho\,|\,\normal_{0,\Id}) + \cE(\normal_{\mean(\rho),\C(\rho)}\,|\,\normal_{x_0,B})\:, \\
	\covI(\rho\,|\,\normal_{x_0,B}) &= \cI(\widebar\rho\,|\,\normal_{0,\Id}) + \covI(\normal_{\mean(\rho),\C(\rho)}\,|\,\normal_{x_0,B})\:,
\end{align}
\end{subequations}
with $\widebar\rho$ being a normalization of $\rho$ according to Definition \ref{def:normalization}.
Similarly, for the Wasserstein distance, we obtain in Lemma~\ref{lem:W2:split} a splitting estimate of the form
\begin{equation}\label{eq:decomp:W2}
	W_2(\rho, \normal_{x_0,B}) \leq \norm{\C(\rho)}_2^{1/2} W_2\bra*{\widebar\rho,\normal_{0,\Id}} + W_2\bra*{ \normal_{\mean(\rho),\C(\rho)}, \normal_{x_0,B}}.
\end{equation}
Those splitting results are the basis to prove convergence results in entropy, Fisher information, and in the classical Wasserstein transport distance $W_2$ in Section~\ref{sec:cv}, which we summarize below.
\begin{theorem}[Entropy decay]\label{thm:decay}
Define \emph{relative condition number} by
\begin{equation}\label{eq:def:CondEnt}
    \lambda(B,C_0):=\max\set[\big]{1, \|B^{\frac{1}{2}}C_0^{-1} B^{\frac{1}{2}}\|_2}\max\set[\big]{1, \|B^{-\frac{1}{2}}C_0 B^{-\frac{1}{2}}\|_2}\,.
\end{equation}
	Solutions $\set*{\rho_t}_{t\geq 0}$ to \eqref{eq:GF-FP-cov-quad-intro} satisfy
	\begin{align}
		\label{eq:decay:ent}
		\cE(\rho_t\,|\,\normal_{x_0,B})
		&\le  \lambda(B,C_0) e^{-2t}\cE(\rho_0\,|\,\normal_{x_0,B})\, , \\
		\label{eq:decay:Fish}
		\covI(\rho_t\,|\,\normal_{x_0,B})
		&\le \lambda(B,C_0)^2 e^{-2t}\covI(\rho_0\,|\,\normal_{x_0,B})\,.
	\end{align}
	If $\mean(\rho_0)=x_0$ and $C_0=\C(\rho_0)\succcurlyeq B$, then
	\begin{equation*}
		\cE(\rho_t\,|\,\normal_{x_0,B})
		\le e^{-2t}\cE(\rho_0\,|\,\normal_{x_0,B})
		\qtq{and}
		\covI(\rho_t\,|\,\normal_{x_0,B})
		\le e^{-2t}\covI(\rho_0\,|\,\normal_{x_0,B})
		\,.
	\end{equation*}
\end{theorem}
\begin{remark}
	The decay estimate for the shape-term (corresponding to normalized solutions) follows from classical results on the asymptotic behavior of the Fokker-Planck equation. The decay estimates in Theorem~\ref{thm:decay} then follow by combining this classical estimate with a relaxation result for the moment-term (corresponding to a Gaussian approximation of the solution) that we derive explicitly in Lemma~\ref{lem:Gauss-decay}, see Section~\ref{sec:cv}. 
	In particular, we have individual decay estimates for every term in the splitting~\eqref{eq:split:EntI} and the prefactor $\lambda(B,C_0)$ only enters through the estimate of the moment-part.
	
	Theorem~\ref{thm:decay} also provides decay to equilibrium in $L^1$ thanks to the Csisz\'ar-Kullback-Pinsker inequality~\cite{CK-ineq}.
\end{remark}
The final result obtained in Section~\ref{sec:W2:cv} is a quantitative convergence to equilibrium in the classical Wasserstein distance, where we again make crucial use of the splitting into shape and moments.
\begin{theorem}[Wasserstein decay]\label{thm:W2:convergence}
	For a solution $\set*{\rho_t}_{t\geq 0}$ to \eqref{eq:GF-FP-cov-quad-intro} starting from $\rho_0\in \cP_2(\R^d)$ with $m_0=\mean(\rho_0)$, $C_0=\C(\rho_0)$, holds the decay estimate
	\begin{align*}
	\MoveEqLeft	W_2(\rho_t, \normal_{x_0,B})\\
		&\leq e^{-t} \kappa(B, C_0) \pra[\Big]{ \inf_{R\in \SO(d)} \!\!\! W_2\bra*{R_\# \widebar\rho_0,\normal_{0,\Id}}^2 \! 
			+\! \abs*{m_0-x_0}_{C_0}^2\!\! + \norm[\big]{\Id - \bra[\big]{B^{\frac{1}{2}}C_0^{-1}B^{\frac{1}{2}}}^{\frac{1}{2}}}_\HS^2 }^{\!\frac{1}{2}}
	\end{align*}
	where $\kappa(B, C_0) := \norm{B}_2 \max\set[\big]{ 1, \norm[\big]{B^{-\frac{1}{2}}C_0 B^{-\frac{1}{2}} }_2}$ and 
	$\widebar\rho_0$ is the normalization of $\rho_0$ from Definition~\ref{def:normalization}.
\end{theorem}
\begin{remark}\label{rmk:carrillo-vaes}
	Theorem~\ref{thm:W2:convergence} recovers and improves the convergence result obtained in \cite{CarrilloVaes} for the covariance-weighted Fokker-Planck equation~\eqref{eq:GF-FP-cov-quad-intro}.
	It is already shown in~\cite{CarrilloVaes} that this exponential rate of convergence is optimal. However, \cite{CarrilloVaes} showed this estimate with a complicated prefactor on the right-hand side, which can become quite large as it depends in a non-trivial way on the initial condition $\rho_0$ and the parameters $x_0,B$ of the target measure.
	
	In our case, the normalization technique allows for a decomposition of the classical Wasserstein distance into a shape and moment part~\eqref{eq:decomp:W2} (see also Lemma~\ref{lem:W2:split}), for which convergence estimates can be obtained individually. This makes the constant in the final estimate of Theorem~\ref{thm:W2:convergence} much more transparent consisting of: (i) a multiplicative relative condition number between the covariance of the target and the initial condition; and (ii) additive errors measuring the mismatch in shape and moments (mean, covariance), respectively. 
	
	Actually, we show in Remark~\ref{rem:W2:WC} that the error for the mean and covariance is the Wasserstein distance $W_{2,C_0}$ with respect to the weighted norm $\abs*{\cdot}_{C_0}$. 
	In this way, we can concisely rewrite the main estimate of Theorem~\ref{thm:W2:convergence} as
	\begin{equation*}
		W_2(\rho_t, \normal_{x_0,B}) \leq e^{-t} \kappa(B, C_0) \pra[\Big]{ \inf_{R\in \SO(d)} W_2\bra*{R_\# \widebar\rho_0,\normal_{0,\Id}}^2  + W_{2,C_0}\bra*{\normal_{m_0,C_0},\normal_{x_0,B}}^2}^{\frac{1}{2}} ,
	\end{equation*}
	highlighting the splitting structure. Finally, by inspecting the proof, we also have the bound
	\begin{equation*}
		W_{2,C_0}(\rho_t, \normal_{x_0,B}) \leq e^{-t} \lambda(B, C_0) \pra[\Big]{ \inf_{R\in \SO(d)} W_2\bra*{R_\# \widebar\rho_0,\normal_{0,\Id}}^2  + W_{2,C_0}\bra*{\normal_{m_0,C_0},\normal_{x_0,B}}^2}^{\frac{1}{2}} ,
	\end{equation*}
	with $\lambda(B,C_0)$ as in~\eqref{eq:def:CondEnt}, showing that the \emph{relative condition number} $\lambda(B,C_0)$ is universal in the estimates for entropy, Fisher information and Wasserstein distance.
\end{remark}

\subsubsection{Duality, displacement convexity and functional inequalities}\label{sss:intro:convexity}

The covariance-con\-strained optimal transport problem~\eqref{eq:OT-cov-constraint-main} has a duality structure, which differs by an additional Lagrange multiplier from the one of the Wasserstein distance. The Lagrange multiplier gives rise to a global coupling of the geodesics manifesting the induced interaction due to the covariance-constraint.

Since our main results do not make use of the duality formula so far, we only formally derive the duality statement in Section~\ref{sec:optimal-constraint} and leave its rigorous justification for future research.
\begin{ftheorem}[Dual formulation of the constrained problem]\label{thm:infsup}
	The constrained optimal transport distance $\cW_{0,\Id}$ given in \eqref{eq:OT-cov-constraint-main} can be expressed as
	\begin{align}\label{eq:dualW}
		\cW_{0,\Id}(\mu_0,\mu_1)^2 &= \inf_{\mu,V} \int_0^1\int \frac12 |V_t|^2\dd\mu_t\dd t 
		= 
		\sup_{\psi,\alpha,\Lambda}
		\int \psi_1\dd\mu_1 - \int\psi_0\dd\mu_0 + \int_0^1\int \tr\Lambda_t \dd t\,,
	\end{align}
	where the infimum on the left is taken over $(\mu,V)\in \CE_{0,\Id}(\mu_0,\mu_1)$, while the supremum on the right is taken over functions $\psi:[0,1]\times\R^d\to\R$ and $\Lambda:[0,1]\to\S^{d}$, $\alpha:[0,1]\to\R^d$ subject to the modified Hamilton-Jacobi subsolution constraint
	\begin{equation}\label{eq:HJ-mod}
		\partial_t\psi +\frac12|\nabla\psi|^2 + \tr\bigl[\Lambda(x\otimes x)\bigr]-\ip{\alpha}{x} \leq 0\;.
	\end{equation}
	In particular, the optimality conditions for a $\cW_{0,\Id}$-geodesic are given by
	\begin{equation}\label{eq:geodesic-eqs}
		\begin{split}
			\partial_t\mu +\nabla\cdot(\mu\nabla\psi) & =0\;,\\
			\partial_t\psi +\frac12|\nabla\psi|^2+\tr\big[\Lambda(x\otimes x)\big] &=  0\;,\\
		\end{split}
	\qquad\text{with}\qquad
	\begin{split}
		\alpha&=0\;, \\
			\Lambda &= \frac{1}{2}\int\left(\nabla\psi\otimes \nabla \psi\right) \dd\mu\;.
	\end{split}
	\end{equation}
	In particular, we have that $\tr[\Lambda_t]=\frac12\int |\nabla\psi_t|^2\dd\mu_t = \cW_{0,\Id}(\mu_0,\mu_1)^2$ is constant in time.
\end{ftheorem}
The functions $\alpha$ and $\Lambda$ act as Lagrange multipliers for the constraint in mean and covariance, respectively.  The fact that the mean constraint is not active at the optimum is consistent with the fact that Wasserstein geodesics have mean zero at all times if the marginals have mean zero. 

Another striking effect of the covariance-constraint or -modulated transport geometry is that it improves the convexity properties of internal energy functionals along optimal interpolations in the space of probability measures. Namely, for a probability measure $\mu\in \cP_2(\R^d)$ let us define the Boltzmann-Shannon entropy by
\begin{equation*}
\cE(\mu) = \int \rho(x)\log \rho(x) \dd x\;,
\end{equation*}
if $\mu$ is absolutely continuous w.r.t.~Lebesgue measure with density $\rho$ and let us set $\cE(\mu)=\infty$ else. We show that $\cE$ is 1-convex along geodesics of the covariance-constrained transport distance $\cW_{0,\Id}$, and satisfies a slightly weaker strict convexity property along geodesics of the covariance-modulated transport distance $\cW$. Recall that along geodesics in the Wasserstein distance $W_2$, the entropy $\cE$ is merely convex and not $\lambda$-convex for any $\lambda >0$. 
\begin{theorem}[Geodesic convexity]\label{thm:main-entroBoltzmann entpy-cov}
For any constant speed $\cW_{0,\Id}$-geodesic $(\mu_s)_{s\in[0,1]}$ we have
\[
\cE(\mu_s)\leq (1-s) \cE(\mu_0) + s\cE(\mu_1) -\frac12 s(1-s) \cW_{0,\Id}(\mu_0,\mu_1)^2\;.
\]
For any constant speed $\cW$-geodesic $(\mu_s)_{s\in[0,1]}$ we have that 
\begin{equation}\label{eq:E-convexW}
 \cE(\mu_s)\leq (1-s)\cE(\mu_0) + s\cE(\mu_1) -\frac12 s(1-s) \cW_{0,\Id}(R_{\#}\widebar\mu_0,\widebar\mu_1)^2\;,
\end{equation}
where $R_{\#}\widebar\mu_0$ and $\widebar\mu_1$ are the normalisations of $\mu_0,\mu_1$ appearing in the splitting result in Theorem~\ref{thm:main-cov}.
\end{theorem}
This result can be obtained formally from the optimality conditions for constrained geodesics as we will explain in Section~\ref{sec:convexity}, where we also consider more general entropies, incorporating non-linear diffusion. We will derive it rigorously as a consequence of the following Evolution Variational Inequality (EVI) in Section~\ref{sec:FuncIneq}. We denote by $(P_t)_{t\geq 0}$ the Ornstein-Uhlenbeck semigroup i.e.~$\eta_t := P_t\eta_0$ for $t>0$ is the solution to $\partial_t\eta_t=\Delta\eta_t +\nabla\cdot(\eta_t\nabla V)$ starting from $\eta_0\in \cP_{0,\Id}(\R^d)$ with $V(x)=\frac12|x|^2$. Its stationary solution is	 the standard Gaussian $\gamma={\sf N}_{0,\Id}$.
\begin{theorem}[Evolution Variational Inequality]\label{thm:func-ineq}
Any $\eta,\nu\in \cP_{0,\Id}(\R^d)$ satisfy the Evolution Variational Inequality (EVI)
  \begin{align}\label{eq:EVI-constraint-intro}
    \frac{\dd^+}{\dd t} \cW_{0,\Id}(P_t \eta,\nu)^2 + \cW_{0,\Id}(P_t \eta,\nu)^2\leq \cE(\nu)-\cE(P_t \eta)\;.
\end{align}
\end{theorem}
As a direct consequence of the previous EVI we obtain a stability result for the Fokker--Planck equation in the covariance-constrained distance.
\begin{corollary}[Stability]\label{cor:stability-intro}
	The Ornstein-Uhlenbeck semigroup $(P_t)_{t\geq 0}$ is exponentially contractive on $\cP_{0,\Id}(\R^d)$, that is any $\eta^1_0,\eta^2_0\in \cP_{0,\Id}(\R^d)$ satisfy the contraction estimate
	\begin{equation*}
		\cW_{0,\Id}(P_t\eta^1_0,P_t\eta^2_0)\leq e^{-t}\cW_{0,\Id}(\eta^1_0,\eta^2_0)\;.
	\end{equation*}
\end{corollary}

Under more restrictive assumptions, we also obtain at least formally a stability result for the covariance-modulated gradient flow: For any two solutions $\mu_t^1,\mu_t^2$ of \eqref{eq:GF-FP-cov-quad-intro} such that $\mean(\mu_0^1)=\mean(\mu_0^2)=x_0$ and $\C(\mu_0^1),\C(\mu_0^2)\succcurlyeq \frac12B$ we have
\[\cW(\mu_t^1,\mu_t^2)\leq e^{-t}\cW(\mu_0^1,\mu_0^2)\quad \forall t\geq 0\;.\]
See the discussion in Section \ref{sec:FuncIneq}, in particular Remark \ref{rem:Wstability}.

As a consequence of displacement convexity, we derive a constrained version of the HWI inequality relating the entropy, transport distance and the Fisher information~\eqref{eq:defI}.
\begin{proposition}[HWI Inequality]\label{prop:HWI-intro}
For any $\eta_0,\eta_1\in \mathcal{P}_{0,\Id}(\R^d)$ connected by a $\cW_{0,\Id}$-geodesic, one has
 \begin{align}\label{eq:HWI_1_intro}
 \cE(\eta_0) &\le 
\cE(\eta_1) + 2\sqrt{\cI(\mu_0)}\cW_{0,\Id}(\eta_0,\eta_1)
    - \frac12\cW_{0,\Id}(\eta_0,\eta_1)^2\,.
\end{align}
\end{proposition}

\subsection{Scalar modularity: Variance-modulated optimal transport}\label{sec:var-intro}

\subsubsection{Definition}

In the one-dimensional setting, the distance \eqref{eq:OT-cov-main} corresponds to variance-modulated optimal transport. 
One could instead also consider an optimal transport problem in any dimension with modulation given by $\var(\mu_t) = \tr\C(\mu_t)$, which we refer to as scalar modularity or variance-modulated transport. 
We present the analysis for this setting, highlighting in which way the anisotropy induced by the covariance in problem \eqref{eq:OT-cov-main} differs from the scalar modularity in higher dimensions.
As we are overall mainly concerned with the matrix case, our motivation here is to draw attention to important similarities and differences between the matrix and the scalar modularity. 
We summarize the results for the variance-modulated optimal transport distance in this subsection, and postpone the proofs to Appendix~\ref{sec:scalar}.
\begin{definition}[Variance-Modulated Optimal Transport]
Given $\mu_0,\mu_1\in \cP_2(\R^d)$ set
\begin{equation}
  \label{eq:OT-var}
  \cWv(\mu_0,\mu_1)^2 = \inf \set*{
    \int_0^1\frac{1}{2\var(\mu_t)}\int |V_t|^2\dd \mu_t\dd t~:~(\mu,V)\in\CE(\mu_0,\mu_1)}\;.
\end{equation}
\end{definition}
\begin{remark}
In general, $\cWv(\mu_0,\mu_1)=0$ if and only if $\mu_0=\mu_1$. Further, we have that if $\mu_0\neq \mu_1$ and
  either $\var(\mu_0)=0$ or $\var(\mu_1)=0$, then
  $\cWv(\mu_0,\mu_1)=\infty$ (for details, Lemma~\ref{lem:apriori}). In particular, the distance to any Dirac distribution is infinite.
\end{remark}
Similar to the matrix case, the problem \eqref{eq:OT-var} can be equivalently written as a minimization problem for the evolution of mean and variance plus an independent constrained transport problem where the mean and variance are fixed to~$0$ and~$1$, respectively.
\begin{definition}[Constraint Optimal Transport]
Given $\mu_0,\mu_1\in \cP_2(\R^d)$, set
\begin{equation}\label{eq:OT-var-constraint-main}
   \cWv_{0,1}(\mu_0,\mu_1)^2 = \inf\set*{\int_0^1\int\frac12|V_t|^2\dd\mu_t\dd t~:~(\mu,V)\in \CE^{\var}_{0,1}(\mu_0,\mu_1)}\;,  
\end{equation}
where $\CE^{\var}_{0,1}(\mu_0,\mu_1)$ is the set of pairs $(\mu,V)\in\CE(\mu_0,\mu_1)$ such that
 $\mean(\mu_t)=0$ and $\var(\mu_t)=1$ for all $t\in[0,1]$.
 \end{definition}
\begin{definition}[Variance-normalization]
	\label{def:normalization-var}
	Given $m\in\R^d$, $\sigma>0$, define $T_{m,\sigma}:\R^d\to\R^d$ by
	\begin{align*}
		T_{m,\sigma}(x)=\frac{x-m}{\sigma}\;. %
	\end{align*}
	If $\mu\in\cP_2(\R^d)$ with $\mean(\mu)=m$ and $\var(\mu)=\sigma^2>0$, then
	$\widebar\mu:=(T_{m,\sigma})_\#\mu$ is its \emph{variance-normalization} or just \emph{normalization} if the constext is clear, satisfying
	$\mean(\widebar\mu)=0$, $\var(\widebar\mu)=1$. 
\end{definition}
 
 \begin{remark}\label{rem:varGeos:NormalW2}
  The constrained minimization problem defining $\cWv_{0,1}$ has been
  studied in detail by Carlen and Gangbo in \cite{CarlenGangbo2003}. In
  particular, it is shown that the optimal curve $\widebar\mu_t$ is obtained as the normalization of $\tilde\mu_{\tau(t)}$,
  where $(\tilde\mu_t)$ is the Wasserstein geodesic connecting
  $\widebar\mu_0,\widebar\mu_1$ and $\tau:[0,1]\to[0,1]$ is a time
  reparametrization ensuring that the curve obtained after
  normalization has constant Wasserstein action. 
  From \cite[Lemma 3.2]{CarlenGangbo2003}, it follows that the mean
  and variance of the Wasserstein geodesic $(\tilde\mu_t)_{t\in[0,1]}$ evolve as
  \begin{align}\label{e:mean-var:W2}
    m(\tilde\mu_t)=0\;,\qquad \var(\tilde\mu_t) = 1-t(1-t)W_2(\widebar\mu_0,\widebar\mu_1)\;.
  \end{align}
\end{remark}
We also introduce a minimization problem for mean and variance.
\begin{definition}[Moment Optimization Problem]
Given $\mu_0,\mu_1\in \cP_2(\R^d)$, define
\begin{equation}
  \label{eq:mean-var-min}
  \vMOP(\mu_0,\mu_1)^2 = \inf\set*{\int_0^1\frac{|\dot m_t|^2+|\dot\sigma_t|^2}{2\sigma_t^2}\dd t~:~(m,\sigma)\in\MV(\mu_0,\mu_1)}\;.
\end{equation}
Here $\MV(\mu_0,\mu_1)$ denotes the set of all absolutely continuous
functions $m:[0,1]\to\R^d$ and $\sigma:[0,1]\to[0,\infty)$ such that
$m_i=\mean(\mu_i)$ and $\sigma_i^2=\var(\mu_i)$ for $i=0,1$. 
\end{definition}

\subsubsection{Splitting in shape and moments}

We have the following equivalent description of the variance-modulated optimal transport problem.
\begin{theorem}[Splitting the distance]\label{thm:main-var}
  Let $\mu_0,\mu_1\in \mathcal P_2(\R^d)$ with $\var(\mu_i)>0$ for $i=1,2$, then
  $\cWv(\mu_0,\mu_1)<\infty$ and we have
  \begin{align}\label{eq:rewrite-var-main}
    \cWv(\mu_0,\mu_1)^2 = \vMOP(\mu_0,\mu_1)^2 + \cWv_{0,1}(\widebar\mu_0,\widebar\mu_1)^2\;,
  \end{align}
  where $\widebar\mu_0$ and $\widebar\mu_1$ are the normalizations of the
  marginals.
\end{theorem}
Further, the optimizers of all three minimization problems are given explicitly. 
\begin{theorem}[Optimal curve]\label{thm:optimalcurve-var}
The optimal curve in \eqref{eq:OT-var} exists and is obtained by
  shifting and scaling the optimal curve for
  $\cWv_{0,1}(\widebar\mu_0,\widebar\mu_1)$ to the optimal mean and
  covariance. More precisely, it is given by
  \begin{align*}
    \mu_t = (T_{m_t,\sigma_t}^{-1})_\#\widebar\mu_t\;,
  \end{align*}
  where $\widebar\mu_t$ is the optimal curve for the constrained problem
  $\cWv_{0,1}(\widebar\mu_0,\widebar\mu_1)$ and $(m_t,\sigma_t)$ is the
  optimizer of \eqref{eq:mean-var-min}. 
\end{theorem}
Moreover, in Section~\ref{sec:sol-m-var}, we provide an explicit solution formula for the optimization problem of the mean and variance~\eqref{eq:mean-var-min}, with optimality conditions given by
\begin{subequations}\label{eq:var:moments}
\begin{align}
	\pderiv{}{t}\bra*{ \frac{\dot m}{\sigma^2}} &= 0 \label{eq:mean-var:mean}\\
	\frac{\ddot\sigma}{\sigma} - \frac{\bra*{\dot\sigma}^2}{\sigma^2} &= - \ \frac{\bra*{\dot m}^2}{\sigma^2} \label{eq:mean-var:var}
\end{align}
\end{subequations}
The solution to the system~\eqref{eq:var:moments} is summarized as follows.
\begin{theorem}[Solving the mean-variance optimization problem]
\label{thm:moments-sol}
Let $m_i=\mean(\mu_i)$ and $\sigma_i^2=\var(\mu_i)$ for $i=0,1$.
By setting $n=\abs*{m_1-m_0}$, the moment distance is given by
\begin{equation}\label{e:mean-var:opt}
	\vMOP(\mu_0,\mu_1)^2 
	= \frac12\abs*{ \log \bra*{\frac{n^2 + \sigma_0^2 +\sigma_1^2 
			 - \sqrt{\bra*{n^2 + \sigma_0^2 +\sigma_1^2}^2 
			 	- 4 \sigma_0^2 \sigma_1^2}}{2\sigma_0 \sigma_1}}}^2\:,
\end{equation}
Further, in the case $n>0$,
the optimal curves for \eqref{eq:mean-var-min} are given by
\begin{subequations}\label{eq:moments-var-explicit}
\begin{align}\label{eq:mean-var:mean-explicit}
	m(t)
	&=m_0+(m_1-m_0) \frac{\tanh\bra*{\beta t+t_0}-\tanh\bra*{t_0}}{\tanh\bra*{\beta+t_0}-\tanh\bra*{t_0}}\,, \\
	\label{eq:mean-var:sigma-explicit}
	\sigma(t) &= \frac{n}{\tanh\bra*{\beta+t_0}-\tanh\bra*{t_0}} \;\cdot\;
	\frac{1}{\cosh(\beta t + t_0)} \,.
\end{align}
\end{subequations}
with $\beta= \sqrt{2}\,\vMOP(\mu_0,\mu_1)\geq 0 $ and 
\begin{equation}\label{eq:mean-var:t0-explicit}
	t_0 = \log \bra*{\frac{\sigma_0^2 - \sigma_1^2 -n^2 +
		\sqrt{\bra*{n^2 + \sigma_0^2 +\sigma_1^2}^2 
			- 4 \sigma_0^2 \sigma_1^2}}{2 n \sigma_0}} \geq 0 .
\end{equation}
For $m_0=m_1=m$ (i.e. $n=0$), the curves are given by $m(t)=m$ and $\sigma(t)= \sigma_0^{1-t} \sigma_1^{t}$.
\end{theorem}
By direct inspection, equation \eqref{e:mean-var:opt} leads to the following asymptotic expressions for small and large $n$ respectively.
\begin{corollary}[Asymptotics for the means]\label{cor:ass-mean}
In particular, for $n\ll 1$
\begin{align*}
	\vMOP(\mu_0,\mu_1)^2 &=\frac12 \abs*{\log \sigma_0 - \log \sigma_1}^2 +\frac{\log \sigma_0^2 - \log \sigma_1^2}{\sigma_0^2 - \sigma_1^2} \frac{n^2}{2} + O\bra*{n^4}\,,
\end{align*}
For $n\gg 1$ on the other hand, the asymptotics are given by
\begin{align*}
      	\vMOP(\mu_0,\mu_1)^2 =\frac12 \abs*{\log \bra*{ \frac{\sigma_0\sigma_1}{n^2}  +  O\bra*{n^{-4}}} }^2\,.
\end{align*}
\end{corollary}
Moreover, if $\sigma_0=\sigma_1=\sigma>0$, the expression~\eqref{e:mean-var:opt} simplifies to
\[
  \vMOP(\mu_0,\mu_1)^2 =\frac12 \abs*{\operatorname{arcosh}\bra*{\frac{n^2}{2\sigma^2}+1}}^2\,.
\]
Note that the expression for $\vMOP$  is not a convex function in $n$, since it is quadratic for $n\ll 1$ and behaves logarithmic for $n\gg 1$.

\subsubsection{Gradient flows}\label{sec:GF-var}

Endowing the space of probability measures with bounded second moment $\mathcal{P}_2(\R^d)$ with the $\cWv$-distance, and considering the gradient flow in this topology for a given energy functional $\cF:\mathcal{P}_2(\R^d)\to \R$, we obtain the evolution
\begin{equation}
	\label{eq:var:GF}
	\partial_t\rho = \var(\rho_t) \nabla\cdot\bra*{\rho_t \nabla \cF'(\rho_t)}\;,
\end{equation}
and again we observe that the energy $\cF$ decays along solutions to \eqref{eq:var:GF},
\begin{align*}
    \frac{d}{dt}\cF(\rho_t)
    = -\var(\rho_t)\int \left|\nabla \cF'(\rho_t)\right|^2 \dx\rho_t \,.
\end{align*}
In particular, we consider here the gradient flow for the relative entropy $\cE(\rho\,|\,\rho_\infty)$ with the target given by $\rho_\infty=\normal_{x_0,B}$:
\begin{equation}
	\label{eq:var:GF-KL}
	\partial_t\rho_t = \var(\rho_t) \nabla\cdot\bra*{ \bra*{\nabla \rho_t + \rho B^{-1}(x-x_0)}}\,.
\end{equation}
If $B^{-1}$ is not a multiple of the identity, the mean and variance will in general not satisfy a closed system of ODEs. Hence, we consider the mean and covariance $C_t= \C(\rho)$ of $\rho$, which satisfy the system
\begin{subequations}\label{eq:moments-var}
\begin{align}
	\dot m_t &= - \var(\rho_t) B^{-1} (m_t - x_0) \\
	\dot C_t &= 2 \var(\rho_t) (\Id - B^{-1} C_t) ,
\end{align}
\end{subequations}
which is again a closed system for $(m_t, C_t)$ by noting that $\var(\rho_t) = \tr C_t$.

\begin{proposition}\label{prop:ent-decay-var}
Let 
\begin{equation}\label{eq:def:lambda}
\lambda(B,C_0):=\min\set*{\frac{d}{\kappa(B)}, \frac{d}{\|C_0^{-1}\|_2  \|B\|_2}}\,,
\end{equation}
where $\kappa(B)=\|B^{-1}\|_2 \|B\|_2$ denotes the matrix condition number of $B\in\S_{\succ0}^d$.
Then $\rho_\infty$ satisfies a logarithmic Sobolev inequality (LSI),
\begin{equation}\label{eq:LSI-var}
	\cE(\rho | \rho_\infty) \leq\frac12 \|B\|_2 \int \abs[\Big]{ \nabla \log\bra[\Big]{ \frac{\rho}{\rho_\infty}}}^2 \dx{\rho}\,,
\end{equation}
and the entropy decays exponentially,
\begin{equation*}
	\cE(\rho_t | \rho_\infty) \leq \exp\bra*{- 2 t\lambda} \cE(\rho_0 | \rho_\infty)\,.
\end{equation*}
\end{proposition}
The above result tells us that, for a rate independent of the potential, we need the initial covariance $C_0$ larger than the one of the potential $H$ and need to consider a class of potentials, i.e. matrices $B$, which are uniformly isotropic, measured by the matrix condition number~$\kappa(B)$.
The proof is based on the scalar nature of the variance, which allows us to arrive at a time-homogeneous problem after introducing the new time
\begin{equation}\label{eq:def:tau-scale}
	\dx \tau := \frac{\dx t}{\var(\rho_t)} \,
\end{equation}
and to apply the classic entropy method (see the proof of Proposition~\ref{prop:ent-decay-var} in Appendix~\ref{appendix:Proof:prop:ent-decay-var}).

The following two remarks make the comparison with the entropy decay results of the relative entropy $\cE(\rho\,|\,\rho_\infty)$ for the covariance-modulated Fokker-Planck equation and the classical Wasserstein distance, respectively.
\begin{remark}[Comparison with the covariance case]
 Notice the difference between this entropy decay estimate, and the corresponding estimate for solutions to the covariance-modulated gradient flow as stated in Theorem~\ref{thm:decay}.
 Here, for scalar modularity, the comparison between the initial covariance $C_0$ and the target covariance $B$ happens in the exponential rate, whereas for the covariance-modulated distance, this comparison appears in the multiplicative constant as a pre-factor in the estimate~\eqref{eq:decay:ent},
 which additionally is only present if $m_0\ne x_0$. 
 Consequently, we observe two crucial differences between the variance-modulated (Proposition~\ref{prop:ent-decay-var}) and the covariance-modulated case (Theorem~\ref{thm:decay}) in that
 \begin{enumerate}
 	\item the rate of convergence for the latter is always independent of $B$, $C_0$;
    \item the choice of $C_0$ only matters if $m_0\neq x_0$ and only alters the pre-factor in front of the universal optimal exponential rate in the case when $C_0$ is small compared to $B$.
\end{enumerate}
 Also note that the dependency on the initial covariance $C_0$ of $\lambda$ in~\eqref{eq:def:lambda} is always present, even for $m_0=x_0$.
\end{remark}

\begin{remark}[Comparison with the classical case]\label{rem:varComparison}
Transforming the moment equations~\eqref{eq:moments-var} to the time-scale $\tau$ from~\eqref{eq:def:tau-scale} with $\tilde\rho_\tau=\rho_{t(\tau)}$, we obtain the moment equations associated to the standard Ornstein-Uhlenbeck process, corresponding to the gradient flow of $\cE(\rho\,|\,\rho_\infty)$ for the classical $W_2$-distance. Then, it follows from the LSI \eqref{eq:LSI-var} that
\begin{equation*}
	\cE(\tilde \rho_\tau | \rho_\infty) \leq \exp\bra*{ - \frac{2\tau}{\|B\|_2}} \cE(\rho_0 | \rho_\infty)\, .
\end{equation*}
\end{remark}

\subsection{Connection to inverse problems: The Ensemble Kalman Sampler}\label{sec:IP}

One important application of the covariance-modulated optimal transport problem is the appearance of the corresponding gradient flow \eqref{eq:GF} in the context of analyzing ensemble Kalman methods for solving inverse problems, in particular in the framework of the Bayesian approach. Our analysis of the covariance-modulated distance is strongly motivated by one such method for sampling from the likelihood or the Bayesian posterior distribution, the \emph{Ensemble Kalman Sampler} (EKS) proposed in \cite{GHLS}. There, the distance $\cW$ defined in \eqref{eq:OT-cov-main} has been introduced under the name \emph{Kalman-Wasserstein metric}. A rigorous analysis of this metric and the properties of the corresponding geometry were lacking. Below, we indicate the relation of the gradient flow~\eqref{eq:GF} to this class of inverse problems.

Consider the forward problem~\cite{kaipio2006statistical}
\begin{equation}
    \label{eq:inverse_problem}
    y = G(x) + \xi\,,
\end{equation}
where
the point $x \in \R^d$ is the \emph{unknown parameter},
the map $G:\R^d\to\R^K$ defines the \emph{forward model},
the random vector $\xi$ introduces \emph{observational noise},
and finally $y \in \R^K$ is the (noisy) \emph{observation}.
For the inverse problem, an observation $\bar y\in\R^K$ is given, and the task is to find the posterior distribution $\pi$ that quantifies the probability of a parameter $x$  giving rise to the data we observed. That is, one assumes a priori distributions for $x$ and a given noise distribution for $\xi$, then asks for the conditional distribution $\pi$ of $x$ given $y=\bar y$ in \eqref{eq:inverse_problem}.

One of the standard assumption in the literature
is that both $x$ and $\xi$ are independently normally distributed, with zero mean and respective covariance matrices $\varSigma \in \S^d_{\succ0}$ and $\varGamma\in \S^K_{\succ0}$. In this case, one obtains the following explicit formula for the posterior distribution:
\begin{equation*}%
    \pi(x) = \frac{\exp\bigl(-f(x) \bigr)}{\int_{\R^{d}}\exp\bigl(-f(x) \bigr) \dd x}\,
    \quad\text{with}\quad 
    f(x)=\tfrac12|y-G(x)|_{\varGamma}^2+\tfrac12 |x|_\varSigma^2\,.
\end{equation*}
In many applications, the forward model $G$ may be a complicated non-linear function, or may not have a closed analytical form and should be thought of as a “black box” for which evaluations may be very costly to obtain in some settings. Further, derivatives of $G$ may not be available or prohibitively expensive to compute.
Therefore, instead of working with the explicit formula above directly, one often has to resort to finding samples from $\pi$ that allow for downstream tasks such as approximating moments, different integrals with respect to $\pi$, and other quantities of interest.
A popular approach to generate such an ensemble of (at least approximate) samples $X=\{x^{(j)}\}_{j=1}^J$ is via interacting particle systems. There are manifold possibilities to define such dynamical systems; a particular requirement in the situation at hand, however, is that the dynamics is \emph{derivative free}, which means that the SDEs might involve evaluations of $G$ but not of its Jacobian $DG$. One such derivative free method is the Ensemble Kalman Sampler (EKS) as proposed in \cite{GHLS}.
The stochastic dynamics of the EKS are given by the following system of SDEs driven by Brownian motions $\set*{W^{(j)}}_{j=1,\dots, J}$,
\begin{align} \label{eq:EKS}
    \dot{x}^{(j)} = - \frac{1}{J}\sum_{k = 1}^J  \skp[\big]{ G(x^{(k)}) - \widebar{G}, G(x^{(j)}) - \bar y }_\varGamma  x^{(k)}  - 
    \C(X) \varSigma^{-1}x^{(j)} +\sqrt{2\C(X)}  \dot{W}^{(j)},
\end{align}
where $\widebar G$ is $G$'s empirical average, and $\C(X)$ is the covariance matrix of $X$'s empirical distribution,
\begin{align*}
    \widebar{G} := \frac{1}{J} \sum_{j = 1}^J  G(x^{(j)}),
    \quad
    \bar x := \frac1J \sum_{k=1}^J x^{(j)},
    \quad
    \C(X) := \frac1J \sum_{j=1}^J\big(x^{(j)}-\bar x\big)\otimes\big(x^{(j)}-\bar x\big)\,.
\end{align*}
The first two terms on the right-hand side of \eqref{eq:EKS} are intended to drive particles towards a (local) minimum of $f(x)$. More precisely, the sum of these two terms is an approximation of $-C(X)\nabla f(x)$. The relation of the second term to the gradient of $\frac12|x|_\varSigma^2$ is obvious, and the first term is built from an approximation to the gradient of $\frac12|G(x)-\bar y|_\varGamma^2$ by difference quotients, see the derivation of \eqref{eq:EKS:particle} below for more details. 
The third term was introduced in \cite{GHLS} as a way to prevent particle collapse in the Ensemble Kalman Inversion (EKI) algorithm~\cite{carrassi2018data}, turning an optimization method into a sampling method. Notably, in \eqref{eq:EKS} the noise is acting directly on the particles themselves, whereas in the noisy EKI it arises from the observation~$y$ being perturbed. The benefit of introducing noise on the particles, rather than the data, was demonstrated in \cite{KovachkiStuart2018_ensemble} in the context of optimization. 

We shall now indicate the relation of the EKS algorithm~\eqref{eq:EKS} to the covariance-modulated gradient flow \eqref{eq:GF}. 
With that aim, we perform a linear approximation and a mean-field limit. 
We thus work under the implicit hypothesis either that the particles are all close together so that $|x_k-x_j|$ is small for any $j,k$ and therefore $G(x_j)\approx G(x_k)$, or that the ensemble $X$ is concentrated in a region of $\R^d$ on which the Jacobian of $G$ is approximately constant. As a first step, we substitute the linearization 
\begin{align*}
    \bigl(G(x^{(j)}) - \widebar{G}\bigr) \approx A(x^{(j)}-\widebar{x})\,,\qquad
    A:=DG(\widebar{x})
\end{align*}
into \eqref{eq:EKS} to obtain, using the identity $\frac{1}{J} \sum_{k=1}^J \bra*{G(x^{(k)}) - \widebar G}=0$, and by approximation, the following system
\begin{align}
    \dot{x}^{(j)}
     &= - \frac{1}{J}\sum_{k = 1}^J  \langle G(x^{(k)}) - \widebar{G}, G(x^{(j)}) - \bar y \rangle_\varGamma  \bra[\big]{x^{(k)}-\widebar{x}}  - 
    \C(X) \varSigma^{-1}x^{(j)} +\sqrt{2\C(X)}  \dot{W}^{(j)}  \label{eq:EKS:particle}\\
    &\approx - \frac{1}{J}\sum_{k = 1}^J \skp[\big]{ A(x^{(k)}-\widebar x), Ax^{(j)} - \bar y }_\varGamma \bra[\big]{ x^{(k)}-\widebar x}  - 
    \C(X) \varSigma^{-1}x^{(j)} +\sqrt{2\C(X)}  \dot{W}^{(j)}  \nonumber\\
    &=- \frac{1}{J}\sum_{k = 1}^J \pra[\Big]{ \bra[\big]{x^{(k)}-\widebar x}\otimes \bra[\big]{x^{(k)}-\widebar x}} A^\tran\varGamma^{-1}\bra[\big]{Ax^{(j)} - \bar y}   - 
    \C(X) \varSigma^{-1}x^{(j)} +\sqrt{2\C(X)}  \dot{W}^{(j)} \nonumber\\
    &=-\C(X)\nabla H(x^{(j)}) +\sqrt{2\C(X)}\dot{W}^{(j)}, \nonumber
\end{align}
where we define the quadratic potential $H$ as
\begin{align*}
    H(x) = \frac12 |x-x_0|_B^2 \quad \text{for}\qquad B^{-1}=A^\tran\varGamma^{-1}A+\varSigma^{-1} \text{ and } x_0= BA^\tran\varGamma^{-1}\bar y\,.
\end{align*}
Provided that the initial average $\bar x$ is close to the minimum $x_0$ of the considered local quadratic approximation $H$ of $f$, it is reasonable to assume that particles remain in the region where this approximation is valid as time advances, even though $A=DG(\bar x)$ will not be exactly true anymore at later times. Note that, up to constants, $f(x)=H(x)$ if $G(x)=Ax$ is linear, and then the approximation above is exact. If $G$ is non-linear but differentiable, we can still consider the preconditioned gradient descent derived in~\eqref{eq:EKS:particle} with $H$ replaced by $f$. Investigating how close this evolution is to the particle ensemble obtained from EKS in the setting when $G$ is nearly linear is the subject of ongoing research.

The gradient's pre-factor $C(X)$ in \eqref{eq:EKS:particle} --- which is the origin of the covariance-modulation considered in the work at hand --- is not only convenient for writing out the derivative-free approximation of $f$'s gradient, but actually has significant consequences on the particle dynamics. The system~\eqref{eq:EKS:particle} represents 
a dynamically pre-conditioned Langevin MCMC method, i.e., a time-continuous version of the stochastic Newton method for approximation of~$\pi$. The optimal pre-conditioning for that method is the inverse Hessian of $f$, in the sense that the method's convergence rate is essentially universal, i.e., independent of the specific shape of $f$. See e.g.~\cite{MWBG} and references therein. The idea behind is that under the pre-conditioned dynamics, the forces excerted on the particles are always that of a normalized quadratic potential. If the Hessian of $f$ is not accessible, a surrogate is needed for pre-conditioning. It has been observed, see e.g. \cite{chada2019tikhonov} and references therein, that the empirical covariance matrix of the particles is suitable. The intuitive reason is that as the particle distribution adapts to $\pi$ over time, it becomes approximately Gaussian near $f$'s minimum, with covariance matrix given approximately by the inverse of $f$'s Hessian near the minimum point.

In the second step, we perform the infinite particle limit $J\to\infty$ of system~\eqref{eq:EKS:particle}, assuming that the empirical distribution converges in a sufficiently strong manner to a probability density $\rho$. In particular, we assume convergence of the covariance matrix,
\begin{align*}
    \C(X) \to \C(\rho)=\int \bigl(x-\mean(\rho)\bigr) \otimes\bigl(x-\mean(\rho)\bigr)\,\rho(x)\dd x \qquad \text{ as } J\to \infty\,. 
\end{align*}
The SDE~\eqref{eq:EKS:particle} then becomes $\dot x =-\C(\rho) \nabla H(x) + \sqrt{2\C(\rho)} \,\dot{W}$, with corresponding Fokker-Planck equation
\begin{align}\label{eq:GF-EKS}
    \partial_t \rho 
    &= \nabla \cdot\bigl(\rho \, \C(\rho) \, \nabla H\bigr) + D^2: \bra*{\C(\rho)\rho}
    = \nabla \cdot\bigl(\rho \, \C(\rho) \, \nabla( H + \log\rho)\bigr).
\end{align}
This mean-field limit has recently been rigorously analyzed in \cite{DingLi} together with explicit convergence rates in terms of the number of particles $J$. Note that equation~\eqref{eq:GF-EKS} is precisely our $\cW$-gradient flow as defined in \eqref{eq:GF} for the choice of energy
\begin{equation}\label{eq:energy-EKS}
    \cE(\rho):=\int\log\rho \dx\rho + \int H\,\rho\dd x\,.
\end{equation}
The results in Theorems~\ref{thm:decay} and~\ref{thm:W2:convergence} about the long-time asymptotics of \eqref{eq:GF-EKS} are fully coherent with the aforementioned observation in the literature that pre-conditioning with the covariance matrix leads to a universal convergence rate in the stochastic Newton method for approximating $\pi$. Specifically, we provide quantitative estimates on the speed of convergence of~$\rho$ to its long-time limit $\rho_\infty\approx\pi$ in an adapted metric, and that speed is universal and independent of $H$'s Hessian matrix $B$. Consequently, provided that there are sufficiently many particles to justify the mean-field approximation, and provided those particles are concentrated in a spatial region of $\R^d$ where the quadratic approximation $H$ of the potentially fully nonlinear $f$ is valid, the EKS method is expected to converge with a universal rate.

In the past, there has been significant activity devoted to the gradient flow structure associated with the Kalman
filter itself~\cite{kalman1960new,KB}, which motivated the wider family of algorithms known as \emph{Ensemble Kalman Methods}, including EKI and EKS.  A well-known result is that for a constant state process, Kalman filtering is the gradient flow with respect to the Fisher-Rao metric~\cite{LMMR,HalderGeorgiou2018,Ollivier2017_online}. It is worth noting that the Fisher-Rao metric connects to the covariance matrix, see details in~\cite{IG2}.
Furthermore, the papers \cite{schillings2017analysis,schillings2018convergence} study continuous time limits of EKI algorithms and
exhibit a gradient flow structure for the standard least squares loss function, preconditioned by the empirical covariance of the particles~\cite{Weissmann2022}; a related structure was highlighted in \cite{bergemann2010localization}. 
Recent works \cite{herty2018kinetic,CarrilloVaes} also study the corresponding mean-field perspective for EKI and EKS. In particular, in \cite{CarrilloVaes}, the authors showed exponential convergence to equilibrium for solutions to \eqref{eq:GF-EKS}
with rate~1 in Wasserstein distance, in the case of a linear forward model $G(x)=Ax$. This result was shown to be optimal for the rate of convergence, and corresponds to the rate we obtain in the same setting in the covariance-modulated distance $\cW$ (Corollary~\ref{cor:stability-intro}); choosing the distance adapted to the geometry of the gradient flow allows us to derive an improved multiplicative constant, see Remark~\ref{rmk:carrillo-vaes}.

\section{Shape vs Moments}\label{sec:shape-moments}

\subsection{Basic properties and notation}

Before turning to the covariance-modulated problem, let us recall some facts about the evolution of the mean and covariance along Wasserstein geodesics, that is optimizers of
\begin{equation}\label{def:W2:BB}
	W_2(\mu_0,\mu_1)^2 := \inf\set*{\int_0^1 \int \abs{V_t}^2 \dx{\mu_t}(x) \dx{t}~:~\partial_t\mu_t+\nabla\cdot(\mu_t V_t)=0}\,.
\end{equation}
In this section, we consider a general integer dimension $k\in \set*{1,\dots, d}$.
\begin{proposition}[Covariance matrix along $W_2$-geodesic]\label{prop:cov_along_W2}
	For $\mu_0,\mu_1\in \cP_{2,+}(\R^k)$ let $\set*{\mu_t}_{t\in [0,t]}$ be the optimal $W_2$-geodesic and $\gamma\in \Pi(\mu_0,\mu_1)$ an optimal coupling, then $\mean(\mu_t)=(1-t) \mean(\mu_0) + t \mean(\mu_1)$ and
	\begin{equation}\label{e:cov_along_W2}
		\C(\mu_t) = (1-t)^2 \C(\mu_0)+ t^2\C(\mu_1) + 2 t(1-t)\Cov(\gamma) \qquad\text{for all } t\in [0,1],	
	\end{equation}
	where
	\begin{equation}\label{e:def:cov}
		\Cov(\gamma) := \frac{1}{2} \iint \pra[\big]{ (x-\mean(\mu_0))\otimes (y-\mean(\mu_1)) + (y-\mean(\mu_1))\otimes (x-\mean(\mu_0))} \dx\gamma(x,y) 
	\end{equation}
	satisfies
	\begin{equation}\label{e:Cov:est}
		0 \preccurlyeq \Cov(\gamma) \preccurlyeq  \tfrac{1}{2}\bra[\big]{\C(\mu_0)+ \C(\mu_1)} 
	\end{equation}
	and hence in particular
	\begin{equation}\label{e:est:cov_along_W2}
		(1-t)^2 \C(\mu_0) + t^2 \C(\mu_1) \preccurlyeq\C(\mu_t) \preccurlyeq (1-t) \C(\mu_0) + t \C(\mu_1) \qquad\text{for all } t\in [0,1].
	\end{equation}
	Moreover, the covariance satisfies the identity and bound
	\begin{align}\label{eq:cov-est-ref}
	  \C(\mu_t)
	  = & (1-t)\C(\mu_0) +t\C(\mu_1)-t(1-t)\int \Big[ y-\mean(\mu_1) - x +\mean(\mu_0)\Big]^{\otimes 2}\dd\gamma(x,y)\\
	  \succcurlyeq & (1-t)\C(\mu_0) +t\C(\mu_1)-t(1-t)\big(W_2(\mu_0,\mu_1)^2-|\mean(\mu_0)-\mean(\mu_1)|^2\big)\Id\;.
	\end{align}
	Similarly, the variance satisfies the identity
	\begin{equation}\label{e:var_along_W2}
		\var(\mu_t) = (1-t) \var(\mu_0) + t \var(\mu_1) - t(1-t)\bra*{ W_2(\mu_0,\mu_1)^2 - 2\abs{\mean(\mu_0)-\mean(\mu_1)}^2}  .
	\end{equation}
\end{proposition}
\begin{proof}
	Since all measures in $\cP_{2,+}(\R^k)$ are absolutely continuous, we have the existence of a transport map and accordingly dual Kantorovich potential $\psi: \R^k\to \R$ such that $(\nabla\psi)_\# \mu_0 = \mu_1$ (see e.g.~\cite{Villani2003}). 
	Hence, we have that $\mu_t= \bra*{(1-t)\Id + t \nabla \psi}_\# \mu_0$, which allows us to calculate
	\begin{align*}
		\mean(\mu_t) = (1-t)\int x \dx\mu_0(x) + t \int \nabla \psi(x) \dx\mu_0(x) = (1-t)\mean(\mu_0)+ t\mean(\mu_1).
	\end{align*}
	Hence, by using the notation $x^{\otimes 2} = x \otimes x$, we get
	\begin{align*}
		\C(\mu_t) &= \int \bra*{x-\mean(\mu_t)}^{\otimes 2} \dx\mu_t(x) \\
		&= \int \bra[\Big]{(1-t)x+ t \nabla\psi(x)-(1-t)\mean(\mu_0)-t\mean(\mu_1)}^{\otimes 2} \dx\mu_0(x )\\
		&= (1-t)^2\int \bra*{x-\mean(\mu_0)}^{\otimes 2} \dx\mu_0 + t^2 \int \bra*{\nabla\psi(x)-\mean(\mu_1)}^{\otimes 2} \dx\mu_0(x) \\
		&\quad + t(1-t) \int (x-\mean(\mu_0))\otimes (\nabla\psi(x)-\mean(\mu_1)) \dx\mu_0(x)\\
		&\quad + t(1-t) \int (\nabla\psi(x)-\mean(\mu_1))\otimes (x-\mean(\mu_0))  \dx\mu_0(x) \\
		&= (1-t)^2 \C(\mu_0) + t^2 \C(\mu_1) + 2 t(1-t) \Cov(\gamma), 
	\end{align*}
	for the optimal coupling $\gamma = (\Id,\nabla\psi)_\# \mu_0$.
	By using the identity
	\begin{equation*}
	((1-t)a+tb)^{\otimes 2}=(1-t)a^{\otimes 2}+t b^{\otimes 2}-t(1-t)(a-b)^{\otimes 2}
	\end{equation*}
	with $a=x-\mean(\mu_0)$ and $b=\nabla\psi(x)-\mean(\mu_1)$ and the estimate $(a-b)^{\otimes 2}\preccurlyeq |a-b|^2\Id$, we obtain from the second line above alternatively
	\begin{align}
	    \C(\mu_t) =&(1-t)\C(\mu_0) +t\C(\mu_1)-t(1-t)\int \big[\nabla\psi(x)-x -(\mean(\mu_1)-\mean(\mu_0))\big]^{\otimes 2}\dd\mu_0(x)\nonumber\\ \label{eq:cov-est-ref2}
	    \succcurlyeq &(1-t)\C(\mu_0) +t\C(\mu_1)-t(1-t)\int |\nabla\psi(x)-x -(\mean(\mu_1)-\mean(\mu_0))|^2\dd\mu_0(x)\Id\\\nonumber
	  = &(1-t)\C(\mu_0) +t\C(\mu_1)-t(1-t)\big(W_2(\mu_0,\mu_1)^2-|\mean(\mu_0)-\mean(\mu_1)|^2\big)\Id\;,
	\end{align}
	where in the last step we use that $\int \skp{\nabla\psi(x)-x,\mean(\mu_1)-\mean(\mu_0)} \dx\mu_0(x) = |\mean(\mu_0)-\mean(\mu_1)|^2$.
	To prove that $\Cov(\gamma)\succcurlyeq 0$, we let $m_0=\mean(\mu_0)$ and $m_1=\mean(\mu_1)$, then we have
	\begin{align*}
		2\Cov(\gamma) &= \int \pra*{ \bra*{x- m_0} \otimes \bra*{\nabla\psi(x)-m_1}+\bra*{\nabla\psi(x)-m_1}\otimes \bra*{x- m_0}} \dx\mu_0(x)\\	  
		&=\int \pra*{ \bra*{x- m_0} \otimes \bra*{\nabla\psi(x)-\nabla\psi(m_0)}+\bra*{\nabla\psi(x)-\nabla\psi(m_0)}\otimes \bra*{x- m_0}} \dx\mu_0(x). 
	\end{align*}
	In this form, the non-negativity is easy to see, since by Aleksandrov's theorem~\cite{Aleksandrov1939} (see~\cite[Theorem 6.9]{EvansGariepy2015} for a modern version), the potential $\psi$ has a gradient and a Hessian almost everywhere and we find
	\[
	(a-b) \otimes \bra*{ \nabla\psi(a)-\nabla\psi(b)} \succcurlyeq 0 \qquad\text{for a.e. } a,b\in \R^k.
	\]
	Indeed, note that by a Taylor expansion it holds for some $s\in[0,1]$
	\[
	(a-b) \otimes \bra*{ \nabla\psi(a)-\nabla\psi(b)} = (a-b) \otimes \bra*{\nabla^2\psi((1-s)a+sb) (a-b)} .
	\]
	Now for $A\in \S_{\succcurlyeq0}^k$ and any $x\in \R^k$ there holds $x\otimes Ax\succcurlyeq \lambda_{\min}(A) \bra*{x\otimes x} \succcurlyeq 0$. We obtain
	\[
	\Cov(\gamma)\succcurlyeq \lambda_{\min}\C(\mu_0)\succcurlyeq 0
	\]
	as $\lambda_{\min}=\lambda_{\min}(\nabla^2\psi)(x) \ge 0$  for all $x\in\R^k$ thanks to convexity of $\psi$.
	The upper bound in~\eqref{e:Cov:est} follows by the tensoric Cauchy-Schwarz inequality $x\otimes y + y \otimes x \preccurlyeq x^{\otimes 2} + y^{\otimes 2}$ for $x,y\in \R^k$.
	
	For the identity~\eqref{e:var_along_W2}, we rewrite the Wasserstein distance like in~\cite[Proof of Lemma 3.2]{CarlenGangbo2003}, to obtain the identity
	\begin{align}\label{e:W2:Cov}
		W_2(\mu_0,\mu_1)^2 = \iint \abs{x-y}^2 \dx\gamma(x,y)&=\var(\mu_0) + \var(\mu_1) + 2\abs{\mean(\mu_0)-\mean(\mu_1)}^2 \\
		&\quad - 2\iint (x-\mean(\mu_0))\cdot (y-\mean(\mu_1)) \dx\gamma(x,y).\nonumber
	\end{align}
	We obtain~\eqref{e:var_along_W2} by identifying the last term as $-2\tr\Cov(\gamma)$ and taking the trace in~\eqref{e:cov_along_W2}.
\end{proof}
Next, we observe that a curve $(\mu,V)\in \CE(\mu_0,\mu_1)$ with finite action functional
\begin{equation}\label{eq:def:action}
    \mathcal A (\mu,V)=\frac12\int_0^1\int_{\R^d}|V_t|^2_{\C(\mu_t)}\dd\mu_t\dd t\; < \infty
\end{equation}
has uniformly bounded covariance along its evolution.
\begin{lemma}\label{lem:apriori-cov}
	Let $(\mu,V)\in\CE(\mu_0,\mu_1)$ for $\mu_0,\mu_1\in \cP(\R^d)$
	be of finite action, i.e.
	\begin{align}\label{eq:CovStability:finite:action}
		A:= \mathcal A (\mu,V)< \infty\;.
	\end{align}
	Then, the curves $t\mapsto m_t:=\mean(\mu_t)$ and $t\mapsto C_t:=\C(\mu_t)$ are absolutely continuous and $C_t$ satisfies the bound
	\begin{align*}
		C_0 e^{-2\sqrt{k_0 A}}\preccurlyeq C_t\preccurlyeq C_0 e^{2\sqrt{k_0 A}} \qquad \forall t\in [0,1] \;,
	\end{align*}
	where $k_0$ denotes the rank of $C_0$. In particular, 
	\[
	\rank(C_t)=k_0 \quad \text{ and } \quad 
	\im C_t=\im C_0 \quad \text{ for all } t\in[0,1]\,,
	\]
	and
	\[
	m_1 - m_0 \in \im C_0 . 
	\]
\end{lemma}
In particular, if $\mu_0,\mu_1\in\cP_2(\R^d)$ are such that there exists a curve of finite action between them, then they satisfy Assumption~\ref{ass:Wfinite}.
\begin{proof}
	Note that the assumption of finite action implies that for a.e.~$t$ we have $V_t\in \im C_t$ a.e.~w.r.t.~$\mu_t$. We claim that moreover, for a.e.~$t$ and $\mu_t$-a.e.~$x$ we have $x-m_t\in \im C_t$. Indeed, for $\xi\in \ker C_t$ we have
	\begin{align*}
		\int|\ip{\xi}{x-m_t}|^2\dd\mu_t(x) = \xi^\tran C_t\xi =0\;,
	\end{align*}
	which implies that $x-m_t\in (\ker C_t)^\perp=\im C_t$ for $\mu_t$-a.e.~$x$.
	Denote by $C_t^\dagger$ the pseudo-inverse of $C_t$, and 
	let $k(t)=\rank(C_t)$.
	The estimate on the covariance matrix is obtained by defining for or $\xi \in \R^d$ with $\abs{\xi}=1$ the scalar function 
	\begin{equation}\label{eq:def:hxi}
		h_\xi(t) := \skp{\xi , C_t \xi} \,,
	\end{equation}
	 The map $t\mapsto h_\xi(t)$ is absolutely continuous along a curve of finite action by using a suitable truncation with $\varphi^R(x)\to x$ as $R\to \infty$ and $\norm{\nabla \varphi^R}_\infty \leq 1$ in the definition of $C(\mu)$. Hence, we can estimate its time-derivative for a.e.~$t\in [0,T]$ by the Cauchy-Schwarz inequality
	\begin{align*}
		\abs*{\pderiv{h_\xi(t)}{t}} &= 2 \abs*{\int \xi \cdot \bra*{x-m_t} \ \xi \cdot V_t \dx{\mu_t} }		
		= 2 \abs*{\int C_t^\frac12 \xi \cdot \bra[\big]{C_t^{\frac12}}^\dagger\bra*{x-m_t}\ C_t^\frac12 \xi \cdot \bra[\big]{C_t^{\frac12}}^\dagger V_t \dx{\mu_t}}\\
		&\leq 2h_\xi(t) \bra*{ \int \abs*{ \bra[\big]{C_t^{\frac12}}^\dagger\bra*{x-m_t}}^2 \dx\mu_t \int \abs*{ V_t}_{C_t}^2 \dx\mu_t}^\frac12 .
	\end{align*}
	By symmetry and using an orthonormal eigenbasis $\set{u_i(t)}_{i=1,\dots,d}$ for $C_t$ with associated eigenvalues $\set{\lambda_i(t)}_{i=1,\dots,d}$, we obtain
	\begin{align*}
		\int \abs*{ \bra[\big]{C_t^{\frac12}}^\dagger\bra*{x-m_t}}^2 \dx\mu_t 
		&= \sum_{i:\lambda_i(t)>0} \lambda_i(t)^{-1} \left(\int \langle u_i(t),x-m_t\rangle\dx\mu_t\right)^2
		\\
		&= \sum_{i:\lambda_i(t)>0} \lambda_i(t)^{-1}\langle u_i(t), C_t u_i(t)\rangle 
		= \abs*{\set*{i : \lambda_i(t)>0}} =:  k(t) .
	\end{align*}
	Hence, by using the finite action bound~\eqref{eq:CovStability:finite:action}, we conclude by setting $k^* := \sup_{t\in [0,1]} k(t) \leq d$, that for any $\xi\in\R^d$,
	\begin{align*}
		h_\xi(0) \exp\bra[\big]{-2\sqrt{k^* A}}
		\leq h_\xi(t)
		\leq h_\xi(0) \exp\bra[\big]{2 \sqrt{k^* A}} \,.
	\end{align*}
	This means $h_u(t)=0$ for all $t\in [0,1]$ if $u$ is in the kernel of~$C_0$. Similarly, $h_u(t)>0$ for all $t\in[0,1]$ for any $u$ in $\im(C_0)$. We conclude that $k(t)=k_0=k^*$ as well as $\im C_t=\im C_0$ for all times, and so the statement holds. 
\end{proof}
In general, $\cW(\mu_0,\mu_1)=0$ if and only if $\mu_0=\mu_1$.
Indeed, as a consequence of Lemma~\ref{lem:apriori-cov}, we have $\cW(\mu_0,\mu_1)=\infty$ if $\im C_0\neq \im C_1$ or $m_1-m_0 \not\in \im C_0$. In particular, $\cW(\mu,\delta_x)=\infty$ for all $\mu\in \mathcal{P}_2(\R^d)$ such that $\mu\neq \delta_x$ and $x\in \R^d$, since $\C(\delta_x)=0$ and hence $\ker \C(\delta_x)=\R^d$.
 
We summarize this observation in the following theorem.
\begin{theorem}[Metric structure of covariance-modulated transport]\label{thm:W:metric}
	For $k\in \set{1,\dots,d}$, let $V \subseteq \R^d$ be linear $k$-dimensional subspace, $m\in \R^d$ and denote by $m+V=\set{x \in \R^d: x-m \in V}$ the according affine subspace. Set
	\[
	  \cP_{2,+}(m+V)=\set*{ \mu \in \cP_2(m+V): \skp{\xi , \C(\mu) \xi} > 0 , \forall \xi \in V\setminus{0}} .
	\]
	Then $\bigl(\cP_{2,+}(m+V),\cW\bigr)$ is a metric space. In particular, setting $V=\R^d$, $\bigl(\cP_{2,+}(\R^d),\cW\bigr)$ is a metric space. Moreover, any two $\mu_0,\mu_1\in \cP_{2,+}(m+V)$ satisfy
	\begin{equation}\label{e:W:comp:W2}
		\frac{1}{2\lambda_{\max}^{0,1}} e^{-2\sqrt{k/\lambda_{\min}^{0,1}} W_2(\mu_0,\mu_1)}  W_2(\mu_0,\mu_1)^2 \leq \cW(\mu_0,\mu_1)^2 \leq \frac{1}{\lambda_{\min}^{0,1}} W_2(\mu_0,\mu_1)^2 , 
	\end{equation}
	with $\lambda_{\min}^{0,1}:=\min\set*{\lambda_{\min,V}(\C(\mu_0)),\lambda_{\min,V}(\C(\mu_1))}$ and $\lambda_{\min,V}(C) := \min\set{ \skp{\xi , C\xi}: \xi \in V, \norm{\xi}=1}$ and similarly for $\lambda_{\max}^{0,1}$ with $\min$ replaced by $\max$ in the previous two formulas.
\end{theorem}
\begin{proof}
	We can assume without loss of generality that $m=0$.
	Hence, we can view $\cP_{2,+}(V)$ after a suitable choice of coordinates as $\cP_{2,+}(\R^k)$ for $k=\dim V$, and we consider instead $\mu_0,\mu_1\in \cP_{2,+}(\R^k)$.
	The set over which the $\inf$ in Definition~\ref{def:OT-cov-main} of the covariance-modulated transport is taken is non-empty.
	Indeed, we can consider the Wasserstein geodesic $(\mu_t,V_t)_{t\in [0,1]}$ between $\mu_0$ and $\mu_1$, which thanks to Proposition~\ref{prop:cov_along_W2} has covariance $\C(\mu_t)$ bounded by
	\[
	    \tfrac{1}{2} \lambda_{\min}^{0,1} \preccurlyeq \C(\mu_t) \preccurlyeq \lambda_{\max}^{0,1} \qquad\text{for all } t\in [0,1].
	\]
	With this, we obtain the upper bound
	\[
	  \cW(\mu_0,\mu_1)^2 \leq 
	  \frac{1}{\lambda_{\min}^{0,1}} W_2(\mu_0,\mu_1)^2 =: C_{\cW}
	\]
	Now, we can consider a sequence $(\mu^n,V^n)\in \CE(\mu_0,\mu_1)$ such that $\sup_{n} \int_0^1 \mathcal{A}(\mu^n_t,V^n_t) \dx{t}\leq C_{\cW} < \infty$ thanks to the previous bounds. By Lemma~\ref{lem:apriori-cov}, we obtain for $C_t^n := \C(\mu_t^n)$ the uniform a priori estimate 
	\[
	  e^{-2\sqrt{k C_{\cW}}} C_0 \preccurlyeq C_t^n\preccurlyeq  e^{2\sqrt{k C_{\cW}}} C_0
	\]
	By classical arguments~\cite{Villani2003}, it follows, along another suitable subsequence, that $\mu^n_t\rightharpoonup \mu_t$ weakly  for a.e. $t\in [0,1]$ and $V^n\mu^n\rightharpoonup V\mu$ in duality with $C_c([0,1]\times \R^d)$ for a pair $(\mu,V)\in \CE(\mu_0,\mu_1)$ along which 
	\begin{align*}
	  e^{-2\sqrt{k C_{\cW}}}  \int_0^1\int |V_t|^2_{C_0} \dd\mu_t\dd t 
	&\leq
	e^{-2\sqrt{k C_{\cW}}} \liminf_{n\to \infty} \int_0^1\int |V_t^n|^2_{C_0} \dd\mu_t^n\dd t\\
	&\leq \liminf_{n\to \infty} \int_0^1\int |V^n_t|^2_{C_t^n}\dd\mu^n_t\dd t 
	 \;.  
	\end{align*}  	Hence, we obtain a lower bound for the action function for $\cW$ in terms of the one for the Wasserstein distance in~\eqref{e:W:comp:W2}.
	This allows to conclude the definiteness of $\cW$ on $\cP_{2,+}(\R^k)$. The symmetry is obvious from its definition by considering the time-reversed solution to the continuity equation. For the triangle inequality, we conclude by gluing two solutions of the continuity equation, see for instance~\cite{Dolbeault-Nazaret-Savare}.
\end{proof}

\subsection{Splitting in shape and moments up to rotation}\label{sec:splitting}
In this section, we show how to arrive at the fundamental result Theorem~\ref{thm:main-cov} for splitting (up to rotations) the distance \eqref{eq:OT-cov-main} in two separate problems on the evolution of the shape, given by the covariance-constrained optimal transport problem \eqref{eq:OT-cov-constraint-main}, and the evolution of the moments, given by \eqref{eq:mean-cov-min}. In fact, it is Lemma~\ref{lem:apriori-cov} that allows us to separate the optimization over the
evolution of mean and covariance. The starting point is a two step minimization by first dynamically constraining the mean and covariance
\begin{equation}
  \label{eq:doubleinf-cov}
  \cW(\mu_0,\mu_1)^2=\inf\set*{\cW_{m,C}(\mu_0,\mu_1)^2~:~(m,C)\in\MC(\mu_0,\mu_1)}\;.
\end{equation}
Here $\MC(\mu_0,\mu_1)$, as defined in~\eqref{e:def:MC}, denotes the set of all absolutely continuous
functions $m:[0,1]\to\R^d$ and $C:[0,1]\to \S_{\succcurlyeq0}^d$ such that
$m_i=\mean(\mu_i)$ and $C_i=\C(\mu_i)$ for $i=0,1$. For given
functions $(m,C)\in\MC(\mu_0,\mu_1)$, the term $\cW_{m,C}$ is defined via
the constrained optimal transport problem
\begin{equation}\label{eq:OT-cov-constraint}
   \cW_{m,C}(\mu_0,\mu_1)^2 = \inf \Big\{\int_0^1\int\frac12|V_t|_{C_t}^2\dd\mu_t\dd t~:~(\mu,V)\in \CE_{m,C}(\mu_0,\mu_1)\Big\}\;,  
\end{equation}
where $\CE_{m,C}(\mu_0,\mu_1)$ is the set of pairs $(\mu,V)\in\CE(\mu_0,\mu_1)$ such that
 $\mean(\mu_t)=m_t$ and $\C(\mu_t)=C_t$ for all $t\in[0,1]$.

We show that problem \eqref{eq:doubleinf-cov} can be equivalently
rewritten as a minimization problem for the evolution of mean and
covariance \eqref{eq:mean-cov-min} plus a constrained optimal transport problem 
where the mean and covariance are fixed to~$0$ and~$\Id$, respectively, as stated in Theorem~\ref{thm:main-cov}.

The stated finiteness of $\cW(\mu_0,\mu_1)$ in Theorem~\ref{thm:main-cov} for $\mu_0,\mu_1\in\cP_2(\R^d)$ with $\C(\mu_0),\C(\mu_1)\in\S_{\succ0}^d$ and $\im\C(\mu_0)=\im\C(\mu_1)$ according to Assumption~\ref{ass:Wfinite} is shown in Theorem~\ref{thm:W:metric} by using a suitable possibly lower-dimensional Wasserstein-geodesic.

The result in Theorem~\ref{thm:main-cov} is based on a perfect splitting of the action in the constrained optimal transport~\eqref{eq:OT-cov-constraint}, which we state as a separate result.
\begin{proposition}\label{prop:split:curves}
	Let $(m,C)\in\MC(\mu_0,\mu_1)$ with $C_t\in\S_{\succ0}^d$
	for all $t\in[0,1]$ and $(\mu,V)\in\CE_{m,C}(\mu_0,\mu_1)$
	with 
	\[
	\int_0^1\int |V_t|^2_{C_t}\dd\mu_t\dd t <\infty .
	\]
	Let $(A_t)_{t\in [0,1]}$ solve \eqref{eq:Adot}, 
	consider the normalization $\hat\mu_t:=(T_t)_\#\mu_t$ with $T_t:=T_{m_t,A_t}=A_t^{-1}(\cdot-m_t)$
	and define the normalized vector field
	\begin{align}\label{eq:def:Vnorm}
		\hat V_t(x)&:=  A_t^{-1}\big[V_t(T_t^{-1} x)-\dot m_t - \dot A_tx\big]\qquad\text{for } t\in [0,1] \;.
	\end{align}
	Then $(\hat\mu,\hat V)\in \CE_{0,\Id}(\hat\mu_0,\hat\mu_1)$ and for a.e. $t\in [0,1]$ the following splitting formula holds
	\begin{align}\label{eq:action-trans-cov}
		\int|V_t|_{C_t}^2\dd\mu_t = \frac14\tr\bra*{\dot C_t C_t^{-1} \dot C_t C_t^{-1}}   +\ip{\dot m_t}{C_t^{-1}\dot m_t}+\int|\hat V_t|^2\dd\hat\mu_t \;.
	\end{align}
\end{proposition} 
\begin{proof}
Note that any solution  $A\in \AC([0,T],\S_{\succcurlyeq0}^d)$ to~\eqref{eq:Adot} satisfies $A_tA_t^\tran =C_t=A_t^\tran A_t$ for all $t$, since
$\dd(A_tA_t^\tran )/{\dd t}=\dot C_t$, and hence $A_t$ provides a normalization in the sense of Definition~\ref{def:normalization}.

To prove the identity~\eqref{eq:def:Vnorm}, we show that $(\hat \mu_t)_{t\in[0,1]}$ is a weakly continuous curve of probability measures in $\cP_2(\R^d)$ connecting $\hat \mu_0$ and $\hat \mu_1$ satisfying $\mean(\hat\mu_t)=0$, $\C(\hat\mu_t)=\Id$ and that $(\hat V_t)_{t\in[0,1]}$ is a Borel family of vector fields such that the continuity equation \eqref{eq:CE-cov} holds in the distributional sense. To see this, consider a test function $\psi\in C^\infty_c(\R^d)$, and compute explicitly
\begin{align*}
  \frac{\dd}{\dd t}\int\psi\dx\hat\mu_t &=   \frac{\dd}{\dd t}\int\psi\circ T_t\dd\mu_t
                                       = \int \nabla \psi\big(T_tx\big)\cdot \Big[DT_t(x)V_t(x) + \partial_tT_tx\Big]\dd\mu_t(x)\\
                                     &=\int \nabla \psi\big(T_tx\big)\cdot \Big[A_t^{-1}V_t(x) - A_t^{-1}\dot m_t -A_t^{-1}\dot A_tA_t^{-1}(x-m_t)\Big]\dd\mu_t(x)\\
  &=\int \nabla \psi(x)\cdot A_t^{-1}\Big[V_t(T_t^{-1} x)-\dot m_t-\dot A_tx\Big]\dd (T_t)_\#\mu_t(x) = \int \nabla\psi\cdot \hat V_t \dd\hat\mu_t\;.
\end{align*}
This yields the conclusion  $(\hat\mu,\hat V)\in \CE_{0,\Id}(\hat\mu_0,\hat\mu_1)$.

It remains to prove the decomposition of the action functionals~\eqref{eq:action-trans-cov}, for which we set 
$r_t(x)=A_t^{-1}\big[\dot m_t+\dot A_tT_tx\big]$ and we obtain the splitting
\begin{align*}
  \int |\hat V_t|^2\dd\hat\mu_t = \int |A_t^{-1}V_t(x) - r_t(x)|^2\dd\mu_t =  \int|V_t|^2_{C_t}\dd\mu_t - \I - \II , 
\end{align*}
where 
\begin{align*}
   \I = \int |r_t|^2\dd\mu_t , \qquad\text{and}\qquad \II = 2\int \skp*{r_t,A_t^{-1}V_t-r_t}\dd\mu_t \:.
\end{align*}
We compute, dropping again $t$ from the notation and using $A^\tran A=AA^\tran =C$,
\begin{align*}
  |r|^2 &= \ip{\dot m}{C^{-1}\dot m} +2\ip{\dot m}{C^{-1}\dot A A^{-1}(x-m)} + \ip{x-m}{A^{-\tran}\dot A^\tran A^{-\tran}A^{-1}\dot AA^{-1}(x-m)}\\
  &=\ip{\dot m}{C^{-1}\dot m} +\ip{\dot m}{C^{-1}\dot C C^{-1}(x-m)} + \frac14\ip{x-m}{C^{-1}\dot C C^{-1}\dot C C^{-1}(x-m)}\,.
\end{align*}
Hence,
\begin{align*}
  \I =\ip{\dot m}{C^{-1}\dot m}+ \frac14\tr\big(\dot C C^{-1} \dot C C^{-1}\big)\,.
\end{align*}
Finally, we claim that 
\begin{align*}
  \int_0^1 \II \dd t= \int_0^12\int\ip{\hat V_t}{A^{-1}_t\dot m+ A^{-1}_t\dot A x}\dd\hat\mu_t(x)\dd t=0\;.
\end{align*}
To see this, we define the functions $\eta_t(x):=\frac{1}{2}\langle x$, $B_t x\rangle + \langle x, \alpha_t\rangle$ for $\alpha:[0,1]\to\R^d$ and $B:[0,1]\to\R^{d \times d}$ given by
\begin{align*}
    \alpha_t:=A_t^{-1}\dot m_t\,,\quad\text{and}\quad
    B_t:=A_t^{-1} \dot A_t\,.
\end{align*}
By recalling~\eqref{eq:Adot}, we note that $B_t= \frac12 A^{-1}_t\dot{C_t}A^{-\tran}_t$ and hence it is symmetric. 
Moreover, by using that $\mean(\hat\mu_t)=0$ and $\C(\hat\mu_t)=\Id$, we have
\begin{equation}\label{eq:eta:total-deriv}
	\begin{split}
    \int \partial_t\eta_t\dd \hat\mu_t &=  \int \pra*{ \langle \dot\alpha_t,x\rangle +\frac12\langle x, \dot B_tx\rangle}\dd\hat\mu_t
    =\frac12\tr\pra*{\frac{\dd}{\dd t}(A_t^{-1}\dot A_t)}\\
    &=\frac{\dd}{\dd t}\pra*{\frac12 \tr B_t}
    =\frac{\dd}{\dd t}\int  \eta_t\dd\hat\mu_t.
\end{split}
\end{equation}
Now, by using $\eta$ as test function in the weak formulation of the continuity equation, we obtain
\begin{align*}
\MoveEqLeft\int_0^1 \int\ip{\hat V_t}{A^{-1}_t\dot m+ A^{-1}_t\dot A x}\dd\hat\mu_t(x)\dd t
=\int_0^1\int\ip{\nabla\eta}{\hat V_t}\dd\hat\mu_t\dd t\\
&=\int\eta_1\dd\hat\mu_1-\int \eta_0\dd\hat\mu_0 -\int_0^1\int \partial_t\eta\dd\hat\mu_t\dd t =0\;,
\end{align*}
by~\eqref{eq:eta:total-deriv}. Hence, we have $\II=0$ and so by combining $\I$ and $\II$, the claim \eqref{eq:action-trans-cov}.
\end{proof}
The splitting in Proposition~\ref{prop:split:curves} is exact, whereas the splitting in Theorem~\ref{thm:main-cov} is up to an optimal choice of rotation. The shape and moment terms cannot be made completely independent in \eqref{eq:rewrite-cov}, as the choice of normalization defined via $A_t$ solving \eqref{eq:Adot} used in the shape term depends on the choice of $C_t$ that also appears in the moment term. Before proving Theorem~\ref{thm:main-cov}, we make a series of remarks to highlight the role of the choice of rotation in relation to the choice of normalization in the shape part of the splitting.

\begin{remark}[Choice of left square root for $C_t$]\label{rem:choiceR}
It is the choice of left square root $A_t$ obtained from \eqref{eq:Adot} that makes the splitting result in Proposition~\ref{prop:split:curves} and Theorem~\ref{thm:main-cov} work. The crucial property used there and guaranteed by equation~\eqref{eq:Adot} is the symmetry of $A_t^{-1}\dot A_t$, which is satisfied for any choice of initial data $A_0$. The choice of $A_0$ is therefore a degree of freedom that remains in the problem, and choosing $A_0$ is equivalent to choosing $A_1$ (for given $C_t$) thanks to uniqueness of solutions to \eqref{eq:Adot}, which in turn is equivalent to fixing $R=R[C]_1=C_1^{-1/2}A_1$ in \eqref{eq:def:RC} (for a given $C_t$). The choice of rotation $R$ in the splitting of $\cW$ is therefore equivalent to the choice of the left square root of $C_0$. To understand this degree of freedom, consider instead a different initial condition $\tilde A_0 = C_0^{\frac{1}{2}} \tilde R_0$ for \eqref{eq:Adot} for some rotation $\tilde R_0\in \Orth(d)$, and define $\tilde R_t:=C_t^{-1/2}\tilde A_t$ where $\tilde A_t$ is the corresponding solution to \eqref{eq:Adot}. Then $\tilde A_t = A_t\tilde R_0$, and hence $\tilde A_t \tilde A_t^\tran=A_tA^\tran_t=C_t$. Therefore, rotating $A_0$ results in an alternative choice of left square root for $C_t$, and $\tilde R_t  = R[C]_t \tilde R_0$ with $R[C]_t$ from~\eqref{eq:def:Revolution}.
\end{remark}

\begin{remark}[Choice of normalization for $\mu_t$]\label{rem:normalizations}
When normalizing $\mu_t$ to $\hat\mu_t$ in Proposition~\ref{prop:split:curves} we used the normalization given by $T_t=A_t^{-1}(\cdot-m_t)$, for $A_t$ solving \eqref{eq:Adot} with initial condition $A_0=C_0^{1/2}$. When projecting to the constrained manifold, we could in principle choose any other normalization, for instance the canonical one given by $\widebar\mu_t=(\widebar T_t)_\# \mu_t$ with $\widebar T_t=C_t^{-1/2}(\cdot-m_t)$ as introduced in Definition~\ref{def:normalization}. It is immediate that the choice of normalization corresponds exactly to the degree of freedom in choosing the left square root of $C_t$ discussed in Remark~\ref{rem:choiceR}. To relate $\hat\mu_t$ to $\widebar\mu_t$, writing $A_t = C_t^\frac{1}{2} R[C]_t$ with $R[C]$ defined in~\eqref{eq:def:RC} one obtains
	\begin{equation*}
		\widebar \mu_t = \bra*{R[C]_t}_\# \hat\mu_t \qquad\text{for } t \in [0,1]. 
	\end{equation*}
In the proof of Theorem~\ref{thm:main-cov} we use the normalization $\hat\mu_t$, and then express the result in terms of the normalization $\widebar\mu_t$, stating the problem in terms of the normalized marginals $(\hat\mu_0,\hat\mu_1)=(\widebar\mu_0,R^\tran_\#\widebar\mu_1)$. Note that it is sufficient to only consider the specific normalization $\hat\mu_t$ for the argument in Theorem~\ref{thm:main-cov}. To see this, consider any other normalization $\tilde\mu_t=(\tilde T_t)_\#\mu_t$ with $\tilde T_t=\tilde A^{-1}_t(\cdot-m_t)$ for $\tilde A_t=A_t\tilde R_0$ obtained by fixing a rotation $\tilde R_0\in \SO(d)$, also see Remark~\ref{rem:choiceR}. Then $\widebar\mu_t = (R[C]_t\tilde R_0)_\#\tilde \mu_t$, and so
\begin{align*}
\cW_{0,\Id}(\tilde\mu_0 , \tilde\mu_1)
=   \cW_{0,\Id}(({\tilde R}_0^\tran)_\#\widebar\mu_0 , (R{\tilde R_0})^\tran_\# \widebar\mu_1)
    =\cW_{0,\Id}(\widebar\mu_0 , R^\tran_\# \widebar\mu_1)
    =\cW_{0,\Id}(\hat\mu_0 , \hat\mu_1)
\end{align*}
since $(R \tilde R_0)^\tran_{\#} \widebar\mu_1 = (\tilde R_0)^\tran_{\#} R^\tran_{\#} \widebar\mu_1$ and since $\cW_{0,\Id}$ is invariant under rotation.
	In particular, this invariance also implies that $\cW_{0,\Id}(R_{\#} \widebar\mu_0 , \widebar\mu_1)=\cW_{0,\Id}(\widebar\mu_0 , R^\tran_{\#} \widebar\mu_1)$.
\end{remark}

\begin{remark}[Evolution of rotation]\label{rem:rotation}
	We obtain a differential equation for $R_t = R[C]_t\in\Orth(d)$ by writing $\varSigma_t = C_t^{\frac{1}{2}}$ for the symmetric square root of $C_t$ and using the relation $\dot C_t = \varSigma_t \dot\varSigma_t + \dot\varSigma_t \varSigma_t$. Therewith, we get by substituting $A_t = \varSigma_t R_t$ in~\eqref{eq:Adot} the equation
	\begin{equation}\label{eq:def:Revolution}
		\dot R_t = \frac{1}{2} \pra*{ \dot\varSigma_t, \varSigma_t^{-1}} R_t , \qquad\text{and}\qquad R_0 = \Id,
	\end{equation}
	with $[A,B]=  AB-BA$ the commutator for two matrices $A,B\in \R^{d\times d}$. The symmetry of $\varSigma_t$ and $\dot\varSigma_t$ implies that $\pra{ \dot\varSigma_t, \varSigma_t^{-1}}^\tran = - \pra{ \dot\varSigma_t, \varSigma_t^{-1}}$ and hence~\eqref{eq:def:Revolution} indeed defines an evolution for an orthogonal matrix, since the tangent space in any $R\in \Orth(d)$ is $T_R\Orth(d)= \set{A\in \R^{d\times d}: R A^\tran =- A R^\tran }$. In this representation, the symmetric matrix $A^{-1}_t \dot A_t$ takes the form
	\begin{equation}\label{eq:AinvdotA:varSigma}
		A_t^{-1}\dot A_t = R_t^\tran \left(\frac{\dot\varSigma_t\varSigma_t^{-1}+\varSigma_t^{-1}\dot\varSigma_t}{2}\right) R_t \,.
	\end{equation}	
	Moreover, the representation of~\eqref{eq:def:Revolution} implies that $t \mapsto R_t \in \Orth(d)$ is absolutely continuous. Since we have chosen $R_0=\Id \in \SO(d)$, we also get $R_t \in \SO(d)$ for all $t\in [0,1]$.
\end{remark}

\begin{proof}[Proof of Theorem~\ref{thm:main-cov}:]
	The proof is based on the splitting identity~\eqref{eq:action-trans-cov} from Proposition~\ref{prop:split:curves}. 
	Noting that for $\C(\mu_0),\C(\mu_1)\in\S_{\succ0}^d$, we see from Lemma~\ref{lem:apriori-cov} that the infimum in \eqref{eq:doubleinf-cov} can be restricted to $(m,C)\in\MC(\mu_0,\mu_1)$ with $C_t\in\S_{\succ0}^d$ for all $t$.
	Given a pair $(\mu,V)\in\CE_{m,C}(\mu_0,\mu_1)$ we have also $(\mu,V+W)\in \CE_{m,C}(\mu_0,\mu_1)$ for any divergence free vector field $W$, that is $\int\ip{\nabla\phi}{W_t}\dd\mu_t=0$ for all test functions $\phi\in C^{\infty}_c(\R^d)$ and a.e.~$t$. By arguments similar to the Wasserstein case~\cite{Villani2003}, one sees that for $\mu$ fixed, the optimal vector field $V$ achieving minimial action is charactized by
	\begin{equation*}
		C_t^{-1}V_t\in \overline{\{\nabla\phi:\phi\in C^{\infty}_c(\R^d)\}}^{L^2(\mu_t)}\quad\text{for a.e.~} t \in [0,1] \;.
	\end{equation*}
	Note that if $V$ is optimal for the curve $\mu$ in this sense, then also the vector field $\hat V$ is optimal for the normalized curve $\hat\mu$. Indeed, if $V=C\nabla \phi$, then
	\begin{equation*}
		\hat V(x)=A^\tran\nabla \phi(Ax+m)-A^{-1}[\dot m+\dot Ax] = \nabla \hat\phi(x)\;,
	\end{equation*}
	with
	\begin{equation*}
		\hat\phi(x)=\phi(Ax+m) -\skp*{A^{-1}\dot m,x} -\frac12\ip{x}{A^{-1}\dot A x}\;.
	\end{equation*}
	Here again it is crucial that $A^{-1}\dot A$ is a symmetric matrix thanks to~\eqref{eq:Adot}.

	Finally, note that the admissible sets of admissible curves $\mu$ and $\hat\mu$ are in  bijection via the transformation of normalization from Proposition~\ref{prop:split:curves}.

	It remains to observe that at time $t=1$ the obtained normalization $\hat \mu_t$ differs from the normalization $\widebar\mu_t$ of Definition~\ref{def:normalization} with the symmetric square root $\C(\mu_1)^{\frac{1}{2}}$ by a rotation $R= R[C]_1\in \SO(d)$, see~Remark~\ref{rem:normalizations}.
	Hence, $\bra*{\hat\mu,\hat V}\in \CE_{0,\Id}(\widebar\mu_0,R^\tran_\# \widebar\mu_1)$. By splitting the infimum in~\eqref{eq:doubleinf-cov} into 
	\[
	   \inf\set*{ \cW_{m,C} : (m,C)\in \MC(\mu_0,\mu_1)}=  \inf_{R\in \SO(d)} \set*{ \inf\set*{ \cW_{m,C} : (m,C)\in \MC_R(\mu_0,\mu_1)}} ,
	\]
	we conclude the result~\eqref{eq:rewrite-cov} from identity \eqref{eq:action-trans-cov}.
\end{proof}

\begin{proof}[Proof of Corollary~\ref{cor:split-Wcov}]
	By the spherical symmetry of one of the normalized marginals, see~\eqref{eq:def:spherical-normalization}, and the observations from Remark~\ref{rem:normalizations}, we have that 
	\[ 
		\cW_{0,\Id}(R_\#\widebar\mu_0, \widebar\mu_1) =\cW_{0,\Id}(\widebar\mu_0,R^\tran_\# \widebar\mu_1) = \cW_{0,\Id}(\widebar\mu_0,\widebar\mu_1) \,,
	\]
	and hence the splitting~\eqref{eq:rewrite-cov} simplifies to
	\begin{align*}
\inf_{R\in \SO(d)} \set*{ \cW_{0,\Id}(R_\#\widebar\mu_0 ,  \widebar\mu_1)^2 + \MOP_R(\mu_0,\mu_1)^2 }&=\\
 	\cW_{0,\Id}(\widebar\mu_0 , \widebar\mu_1)^2 + \inf_{R\in \SO(d)} \MOP_R(\mu_0,\mu_1)^2 
&= \cW_{0,\Id}(\widebar\mu_0 , \widebar\mu_1)^2 + \MOP(\mu_0,\mu_1)^2 . \qedhere
	\end{align*}
\end{proof}
\begin{remark}[Gaussian targets]\label{rem:Gaussian}
 	For any $(m,C)\in\R^d\times\S_{\succ 0}^d$ and $R\in\SO(d)$, observe that $R_\#\normal_{m,C}=\normal_{Rm,RCR^\tran}$, and so $R_\#\widebar\normal_{m,C}=R_\#\normal_{0,\Id}=\normal_{0,\Id}=\widebar\normal_{m,C}$. Therefore, the splitting in Theorem~\ref{thm:main-cov} is exact if one of the end points $\mu_0,\mu_1$ is Gaussian. This is precisely the reason why rotations do not play a role for the gradient flows in $\cW$-distance and corresponding convergence results that we study in Section~\ref{sec:cv}, because there we are restricting our analysis to Gaussian targets.
\end{remark}
\begin{proof}[Proof of Proposition~\ref{prop:comp:W-W2}]
	We write $\mu_0 =\bra[\big]{T_{m,A}^{-1}}_\#R_{\#} \widebar\mu_0$ and $\mu_1 =\bra[\big]{T_{m,A}^{-1}}_\# \widebar\mu_0$ with $T^{-1}_{m,A}x=A x + m$ and $A = C^{\frac{1}{2}}R$ is any square root of $C$.
	Then, we apply the push-forward and obtain
	\begin{equation}
		W_2(\mu_0, \mu_1) = W_2\bra*{\bra[\big]{T_{m,A}^{-1}}_\# R_{\#}\widebar \mu_0, \bra[\big]{T_{m,A}^{-1}}_\# \widebar\mu_1}
	\end{equation}
	We can apply~\cite[Lemma 3.1]{CarrilloVaes}, where we note that in the push-forward the same mean cancels out and that $\norm*{A}_2^2 = \norm*{A A^\tran }_2 = \norm*{C}_2$, since $A$ was a square-root of $C$. Hence, we obtain $W_2(\mu_0, \mu_1)^2\leq \lambda_{\max}(C) W_2(R_{\#}\widebar \mu_0, \widebar \mu_1)^2 \leq 2\lambda_{\max}(C) \cW_{0,\Id}(R_{\#}\widebar\mu_0, \widebar\mu_1)^2$, where we note that the constrained distance has the same dynamical formulation as the Wasserstein transport up to a factor of $2$, however over a more constrained set of solution to the continuity equation, making it larger. 
	The splitting formula~\eqref{eq:rewrite-cov} provides the second estimate in~\eqref{eq:comp:W-W2}.
	The final estimate in~\eqref{eq:comp:W-W2} is a consequence of the estimate~\eqref{e:W:comp:W2} from Theorem~\ref{thm:W:metric}.
	
	For proving the estimate~\eqref{eq:W_2comp}, we recall the Benamou-Brenier formula
		\begin{equation}\label{eq:OT-BB}
			\frac12 W_2(\mu_0,\mu_1)^2 = \inf \biggl\{\int_0^1\int\frac12|V_t|^2\dd\mu_t\dd t~:~(\mu,V)\in \CE(\mu_0,\mu_1)\biggr\}\;.  
		\end{equation}
		The first inequality immediately follows since the minimization in the definition of $\cW_{0,\Id}$ is performed over a restricted set of curves. To obtain the second inequality, let $(\mu_s,V_s)_{s\in[0,1]}\in\CE(\mu_0,\mu_1)$ be a $W_2$-geodesic. Combining \eqref{e:cov_along_W2} and \eqref{eq:cov-est-ref}, we have for any $s\in [0,1]$ the estimate
		\begin{equation}\label{eq:Clb}
			\max\bra[\bigg]{\frac12,1-\frac14W_2(\mu_0,\mu_1)^2}\Id \leq \C(\mu_s)\leq \Id\;.
		\end{equation}
		Recall from \eqref{eq:cov-est-ref} that 
		\begin{equation*}
			\C(\mu_t)
			=  (1-t)\C(\mu_0) +t\C(\mu_1)-t(1-t)\Delta(\gamma)\;,\quad \Delta(\gamma)=\int \Big[ y-\mean(\mu_1) - x +\mean(\mu_0)\Big]^{\otimes 2}\!\dd\gamma(x,y),
		\end{equation*}
		with $\gamma$ an optimal coupling of $\mu_0,\mu_1$. Without loss of generality, we can assume that $\Delta(\gamma)=\diag(\delta_1,\dots,\delta_d)$. Then \eqref{eq:cov-est-ref} gives the bounds $0\leq \delta_i \leq \min\big(W_2(\mu_0,\mu_1)^2,2\big)$ and consequently
		\begin{align*}
			C_s=\diag\big(1-s(1-s)\delta_i\big)\;, \qquad\partial_sC^{1/2}_s = \frac12\diag\big((1-s(1-s)\delta_i)^{-1/2}(2s-1)\delta_i\big)\;.
		\end{align*}
		Let $\widebar\mu_s:=(T_s)_\#\mu_s\in \cP_{0,\Id}(\R^d)$ be the normalization of $\mu_s$ with $T_s=T_{C_s^{1/2},0}$ and $C_s=\C(\mu_s)$. From Proposition~\ref{prop:split:curves}, we have $(\widebar\mu,\widebar V)\in \CE(\mu_0,\mu_1)$ with $\widebar V_s(x)=C_s^{-1/2}[V_s(C_s^{1/2}x)-\partial_sC_s^{1/2}x]$ and we infer that for any $\eps$
		\begin{align}
			\MoveEqLeft
			\int |\widebar V_s|^2\dd\widebar\mu_s = \int \abs[\big]{C_s^{-1/2}\bigl(V_s(x)-\partial_s C_s^{1/2}x\bigr)}^2\dd\mu_s\\\nonumber
			&\leq (1+\eps)\int |V_s|^2_{C_s}\dd\mu_s
			+\bigl(1+\eps^{-1}\bigr)\int\abs[\big]{C_s^{-1/2}\partial_s C_s^{1/2}x}^2 \dd\mu_s(x) \label{eq:barVest}\\
			&\leq (1+\eps)\pra*{\max\set*{\frac12,1-\frac14W_2(\mu_0,\mu_1)^2}}^{-1} W_2(\mu_0,\mu_1)^2 +\bigl(1+\eps^{-1}\bigr) \tr\pra*{C_s^{-1}\bigl(\partial_sC_s^{1/2}\bigr)^2}, \nonumber
		\end{align}
		where we used \eqref{eq:Clb} in the last step. The second term above can be estimated as
		\begin{align*}
			\tr\big[C_s^{-1}\big(\partial_sC_s^{1/2}\big)^2\big] = \frac14 (2s-1)^2\sum_{i=1}^d\frac{\delta_i^2}{(1-s(1-s)\delta_i)^2} 
			\leq d\cdot W_2(\mu_0,\mu_1)^4\;, 
		\end{align*}
		where we used $\delta_i\leq W_2(\mu_0,\mu_1)^2$ in the nominator and $\delta_i\leq 2$ in the denominator. Finally, choosing $\eps = W_2(\mu_0,\mu_1)$ for instance, we get
		\begin{equation*}
			\cW_{0,\Id}(\mu_0,\mu_1)^2\leq \frac12\int_0^1\int|\widebar V_t|^2\dd\widebar\mu_t\dd t \leq \frac12 W_2(\mu_0,\mu_1)^2 + o\big(W_2(\mu_0,\mu_1)^2\big)\;,
		\end{equation*}
		which proves the claim~\eqref{eq:W_2comp}.    
\end{proof}

\subsection{Optimality conditions for the moment part}
\label{sec:optimal-moments}

In this section, we are concerned with the existence of optimizer for the problems $\MOP_R(\mu_0,\mu_1)$ in~\eqref{eq:mean-cov-min:R} for $R\in\SO(d)$ and $\MOP(\mu_0,\mu_1)$ in~\eqref{eq:mean-cov-min}. 
By the identity~\eqref{eq:mean-cov-min}, we are mainly concerned with $\MOP_R(\mu_0,\mu_1)$ in~\eqref{eq:mean-cov-min:R}.
For the discussion of existence, it will be more convenient to use directly the parametrization of $(C_t)_{t\in[0,1]}$ in terms of $(A_t)_{t\in[0,1]}$ as defined in~\eqref{eq:Adot}.
It is beneficial to understand the problem $\MOP_R(\mu_0,\mu_1)$ as an existence statement on geodesics on $\subMfd:= \R^d\times \GL_+(d)$ with a sub-Riemannian metric. 
To start the discussion, we endow $\subMfd$ with the standard metric given as the product of Euclidean and Frobenius  $\skp{(m_0,A_0),(m_1,A_0)} = m_0\cdot m_1 + A_0 : A_1$.
The sub-Riemnnian structure is implied by the equation~\eqref{eq:Adot}. More precisely, along any curve $(A_t)_{t\in[0,1]}$ satisfying~\eqref{eq:Adot}, the matrix $A_t^{-1} \dot A_t$ has to be symmetric. 
Hence, we define the admissible \emph{horizontal} tangential vectors at a point $(m,A)\in \subMfd$ as subset of the full tangent space as
\begin{equation}\label{eq:def:A:horizontal}
	H_{m,A}:= \set*{ (r,X)\in T\subMfd : A^{-1} X \in \S^d } \subset  T \subMfd:= \R^d \times \R^{d\times d}.
\end{equation}
Hereby, we note that $H_{m,A}$ is independent of $m$, but we keep it in the notation to indicate $(m,A)\in \subMfd$ as base point.
Next, we consider for $(m_i,A_i)\in \subMfd$ for $i=0,1$ \emph{horizontal} curves
satisfying the symmetry condition implied by~\eqref{eq:Adot}
\begin{equation}\label{eq:mean-cov-min:AR}
	\overline{\MC}((m_0,A_0),(m_1,A_1)) = \set*{ (m,A) \in \AC\bra*{[0,1],\subMfd} : (\dot m_t,\dot A_t) \in H_{m_t,A_t} \text{ for a.e. } t\in [0,1]  } .
\end{equation}
Note, that our notation also the boundary values $m(i) = m_i, A(i)=A_i$ for $i=0,1$ are implied.
Before turning to geodesics, we check that $\overline{\MC}((m_0,A_0),(m_1,A_1))$ is non-empty for any choice of $(m_i,A_i)\in \subMfd$.
For doing so, we apply the Chow-Rashevsky Theorem from~\cite[Theorem 1.14]{Rifford2014}, which asks us to check the existence of suitable vector fields, which are bracket generating. In the following we denote by
$e^i$  is the $i$-th unit vector in $\R^d$, and 
\[S^{(i,j)} := \begin{cases}
e^i\otimes e^i\;, & i=j\;;\\
\frac{1}{\sqrt{2}}\bra*{ e^i\otimes e^j + e^j \otimes e^i}\;, & i\neq j\: .
\end{cases}\]
for $1\leq i\leq j\leq d$ an orthonormal basis of $\S^d$ w.r.t. the Frobenius inner product.
\begin{lemma}[Two-bracket generating vector fields]\label{lem:SR:brackets}
	The horizontal vector fields 
	\begin{equation}\label{eq:def:vecs:horizontal}
		 X^i(m,A) := (Ae^i, 0) \ \text{ for } i=1,\dots,d \quad\text{ and }\quad X^{\alpha}(m,A) := (0,A S^{\alpha}) \ \text{ for }  \alpha \in \triangle_d,
	\end{equation}
	where $\triangle_d:= \set*{ (i,j): 1\leq i\leq j\leq d}$,  are \emph{two-bracket generating}, that is 
	\begin{align*}
	&\Span\set*{ X^\alpha: \alpha \in \set*{1,\dots,d}\cup \triangle_d} = H_{m,A}\\
	\text{and}\qquad
	&H_{m,A} + \Span\set*{ [X^\alpha,X^\beta]: \alpha,\beta \in \set*{1,\dots,d}\cup \triangle_d } = T\subMfd . 
	\end{align*}
	In particular, for any $(m_i,A_i)\in \subMfd$ for $i=0,1$, the set $\overline{\MC}((m_0,A_0),(m_1,A_1))$ is non-empty.
\end{lemma}
\begin{proof}
	Since $A\in \GL_+(d)$, we have that $\Span\set*{ A e^i: i=1,\dots, d} = \R^d$.
	First, we calculate for any $\alpha,\beta,\gamma\in\triangle_d$ the Lie bracket of the two vector fields $V^\alpha = A S^\alpha, V^\beta= A S^\beta$, where we note that $\partial_{A_\gamma} V^{\alpha}= E_\gamma S^\alpha$ with $E_{\gamma} = e^i\otimes e^j$ for $\gamma=(i,j)$.
	With this, we obtain by explicit straightforward calcultion
	\begin{align*}
		[V^\alpha,V^\beta] = \sum_{\delta\in \set*{1,\dots,d}\cup \triangle_d} \bra*{ A \pra*{S^\alpha, S^\beta} }_	\delta  \partial_{A_\delta} .
	\end{align*}
	Hence, it is sufficient to check that $\Span\set*{ S^\alpha, [S^\beta,S^\delta]: \alpha,\beta,\delta \in\triangle_d} = \R^{d\times d}$. 
	We choose any $1\leq i<j \leq d$ and verify
	\begin{align*}
		[ S^{(i,i)}, S^{(i,j)} ] = \bra*{ e^i\otimes e^j - e^j \otimes e^i} ,
	\end{align*}
	proving the claim. The final statement follows now from the Chow-Rashevsky Theorem, see e.g.~\cite[Theorem 1.14]{Rifford2014}.
\end{proof}
Now, since we ensured that $\overline{\MC}((m_0,A_0),(m_1,A_1))$ is non-empty, we can minimize an action among those curves. 
The identity~\eqref{eq:equivalent:MOPactions} provides for $I(m,C)$ in~\eqref{eq:def:MOP:action} an equivalent action for $(m,A)\in\overline{\MC}((m_0,A_0),(m_1,A_1))$ defined by
\begin{equation}\label{eq:def:IA}
	\overline{I}(m,A)  = \frac{1}{2} \int_0^{1} \bra*{\abs*{A_t^{-1} \dot m_t}^2 + \norm{A_t^{-1} \dot A_t}_{\HS}^2 } \dx{t}  \; .
\end{equation}
In this way, we obtain the moment optimization problem on the space $\subMfd$ as
\begin{equation}\label{eq:def:MOP:AR}
	\overline\MOP((m_0,A_0),(m_1,A_1)) = \inf \set*{ \overline{I}(m,A) : (m,A)\in \overline{\MC}((m_0,A_0),(m_1,A_1)) } . 
\end{equation}
By construction, we have the equivalence for $R\in\SO(d)$, $C_0,C_1\in \S_{\succ0}^d$ and $m_0,m_1\in \R^d$
\begin{equation}\label{eq:MOP:equivs}
	\MOP_R\bra*{(m_0,C_0),(m_1,C_1)} = \overline{\MOP}\bra[\big]{(m_0,C_0^{\frac{1}{2}}),(m_1,RC_1^{\frac{1}{2}})} .
\end{equation}
Since $\subMfd$ is non-compact, we cannot directly apply results from sub-Riemannian geometry ensuring the existence of geodesics.
For doing so, we have to ensure relative compactness, which is provided by a stability estimate for curves $(m,A)\in \overline\MOP((m_0,A_0),(m_1,A_1))$ with $I(m,A)<\infty$. 
\begin{lemma}\label{lem:mA-finiteAction}
	Let $(m,A)\in \overline\MC((m_0,A_0),(m_1,A_1))$ such that $2\overline{I}(m,A)=:I<\infty$.
	Then the map $t\mapsto (m_t,A_t)$ 
	satisfies the bound
	\begin{equation}\label{eq:mc-apriori}
		A_0A_0^{\tran} e^{-2\sqrt{I}}\preccurlyeq A_t A_t^{\tran}\preccurlyeq A_0A_0^{\tran} e^{2\sqrt{I}} \qquad\text{for all } t\in [0,1]\;.
	\end{equation} 
	Moreover, the rank of $t\mapsto A_t$ is constant, that is
	\[
		\rank(A_t)=\rank(A_0) \qquad\text{for all } t\in[0,1]\,.
	\]
	Finally, if $A_0 A_0^\tran$ is non-singular we have
	\begin{equation}\label{eq:mA-H1}
		\int_0^{T}\|\dot m_t\|^2\dd t \leq I e^{2\sqrt{I}}\lambda_{\min}(A_0A_0^\tran)^{-1}\;,\qquad \int_0^{T} \|\dot A_t\|^2_{\HS}\dd t\leq  I e^{2\sqrt{I}}\lambda_{\min}(A_0 A_0^\tran)^{-1}\:.
	\end{equation}
\end{lemma}  
\begin{proof}
Denote by $A_t^\dagger$ the pseudo-inverse of $A_t$, and 
let $k(t)=\rank(A_t)$.
The proof follows along the same argument as the proof of Lemma~\ref{lem:apriori-cov} by defining
for $\xi \in \R^d$ with $\abs{\xi}=1$ the function $h_\xi(t) := \skp{\xi , A_t A_t^\tran \xi}$, see also~\eqref{eq:def:hxi}. 
We note that $(\dot m_t,\dot A_t) \in H_{m_t,A_t}$ also implies the symmetry $\dot A_t A_t^\tran = A_t \dot A_t^\tran$. 
By doing so, we can estimate its time-derivative for a.e.~$t\in [0,1]$ by the Cauchy-Schwarz inequality
	\begin{align*}
			\abs*{\pderiv{h_\xi(t)}{t}} 
			&= \abs*{\skp*{\xi,\bra*{\dot A_t A_t^\tran+ A_t \dot A_t^\tran} \xi}}
			= 2 \abs*{\skp*{A_t^\tran \xi, A^\dagger \dot A_t A_t^\tran\xi} }\\
			&\leq 2 \abs*{A_t^\tran\xi}^2 \norm*{A_t^\dagger \dot A_t}_{\HS} 
			= 2 h_\xi(t)  \norm*{A_t^{-1} \dot A_t}_{\HS}  .
	\end{align*}
	Hence, we conclude by Gronwall that for any $\xi\in\R^d$ and any $t\in [0,1]$,
	\begin{align}\label{eq:h:nondeg}
	      h_\xi(0) \exp\bra[\big]{-2\sqrt{I}}
	     \leq h_\xi(t)
	     \leq h_\xi(0) \exp\bra[\big]{2 \sqrt{I}} .
	  \end{align}
  This means $h_u(t)=0$ for all $t\in [0,1]$ if $u$ is in the kernel of~$A_0 A_0^\tran$. Similarly, $h_u(t)>0$ for all $t\in[0,1]$ for any $u$ in $\im(A_0A_0^\tran)$. We conclude that $\im(A_tA_t^\tran)=\im(A_0 A_0^\tran)$ for all $t\in [0,1]$.
  Hence also $U:=\im(A_tA_t^\tran)$ is a linear subspace independent of $t\in [0,1]$. 
  From finiteness of the action we infer that $\dot m_t\in \im(A_tA_t^\tran) = U$ for all $t\in [0,1]$.
  Hence also $m_1-m_0=\int_0^1\dot m_t\dd t$ belongs to $U$. 
  The bound \eqref{eq:h:nondeg} and the finiteness of the action immediately yield the bounds \eqref{eq:mA-H1}.
\end{proof}

\begin{proposition}\label{prop:opt-mC-ex}
Let $\mu_0,\mu_1\in\cP_{2}(\R^d)$ and $m_i=\mean(\mu_i)$, $C_i=\C(\mu_i)$, $i=0,1$. 
Then for any $R\in \SO(d)$, $\MOP_R(\mu_0,\mu_1) = \overline\MOP\bra[\big]{(m_0,C_0^{\frac{1}{2}}),(m_1,C_1^{\frac{1}{2}}R)}$ is finite if and only if $\im C_0=\im C_1$ and $m_1-m_0\in \im C_0=\im C_1$. If it is finite, there exists an optimal pair $(m_t,A_t)_{t\in[0,1]}\in \overline\MOP\big((m_0,A_0),(m_1,A_1)\big)$ achieving the infimum.

Similarly $\MOP(\mu_0,\mu_1)$ is finite if and only if $\im C_0=\im C_1$ and $m_1-m_0\in \im C_0=\im C_1$. If it is finite, then there exists an optimal pair $(m_t,C_t)_{t\in[0,1]}$ achieving the infimum in $\MOP(\mu_0,\mu_1)$.
\end{proposition}
\begin{proof}
Lemma \ref{lem:mA-finiteAction} shows that $\overline I(m,A)$ is infinite if $\im C_0\neq \im C_1$ or $m_1-m_0\notin \im C_0$ (note that $A_i A_i^\tran = C_i$, for $i=0,1$). 
Let us assume that $\im C_0= \im C_1$ and $m_1-m_0\in \im C_0$, then we can restrict the argument by a suitable orthogonal construction to some $\R^k$ with $k=\dim \im C_0$. Hence, we can assume without loss of generality, that $C_0,C_1\in \S_{\succ0}^d$. For brevity, we set $A_0 = C_0^{\frac{1}{2}}\in \GL_+(d)$ and $A_1 = R C_1^{\frac{1}{2}}\in \GL_+(d)$. 
Then, we obtain by an application of Lemma~\ref{lem:mA-finiteAction} the existence of a curve $(m,A)\in \overline{\MC}((m_0,A_0),(m_1,A_1))$. 
Since $t\mapsto (m_t,A_t)\in \subMfd$ is uniformly continuous on $[0,1]$, we have that $\| A^{-1}_t \dot m_t\| \lesssim \| \dot m_t\|$ and $\norm{A_t^{-1} \dot A_t}_{\HS}\leq \norm{\dot A_t}_{\HS}$. Since, $(m,A)\in H^1([0,1])$, we obtain $\overline I(m,A)<\infty$ and so $\overline\MOP((m_0,A_0),(m_1,A_1)) < \infty$. 

Let $(m^n,A^n)$ be a minimizing sequence of functions in $\overline{\MC}((m_0,A_0),(m_1,A_1))$, that is
\[
  \overline\MOP\bra*{(m_0,A_0),(m_1,A_1)}=\inf_n \overline I(m^n,A^n). 
\]
From the bounds \eqref{eq:mA-H1} we infer that $m^n$ and $A^n$ are uniformly bounded in $H^1([0,1])$. 
Hence there is a function $(m,A) \in H^1([0,1])$ such that $(m^n, A^n) \rightharpoonup (m, A)$ weakly in $H^1$ and uniformly as continuous functions. In the notation of Lemma~\ref{lem:SR:brackets}, we have that there exists $u_{\alpha}^n\in L^2([0,1])$ for $\alpha\in \set{1,\dots,d}\cup\triangle_d$ such that $(\dot m^n_t,\dot A^n_t)  = \sum_{\alpha} u_\alpha^n(t) X^\alpha(m_t^n,A_t^n)$ and 
$\overline I(m^n,A^n)=\sum_\alpha\|u_\alpha^n\|^2_{L^2([0,1])}$ is bounded. Thus, up to a further subsequence also $u_{\alpha}^n$ to $u_\alpha$ weakly in $L^2([0,1])$.
Hence by the uniform convergence of $(m^n,A^n)$ and smoothness (linearity) of~$X^\alpha$, we deduce
\[
  (\dot m_t, \dot A_t) = \sum_{\alpha} u_\alpha(t) X^\alpha(m_t,A_t) . 
\]
In particular $(m,A)\in \overline{\MC}((m_0,A_0),(m_1,A_1))$ is again horizontal.

Then we have
\begin{align*}
	\overline I(m,A) &= \frac{1}{2} \int_0^{1} \bra*{\abs*{A_t^{-1} \dot m_t}^2 + \norm*{A_t^{-1} \dot A_t}_{\HS}^2 } \dx{t} 
	\\
	&=
	\liminf_n \frac{1}{2} \int_0^{T} \bra*{ \abs*{A_t^{-1} \dot m_t^n}^2+ \norm*{A_t^{-1} \dot A_t^n}_{\HS}^2 }\dd t \\
	&=
	\liminf_n \frac{1}{2} \int_0^{T} \bra*{ \abs*{(A_t^n)^{-1} \dot m_t^n}^2+ \norm*{(A_t^n)^{-1} \dot A_t^n}_{\HS}^2 }\dd t 
	\\
	&= \liminf_n \overline I(m^n,A^n) 
	= \overline\MOP\bra*{(m_0,A_0),(m_1,A_1)} \:.
\end{align*} 
Hence the pair $(m,A)$ is a minimizer.

The proof of the statement for $\MOP(\mu_0,\mu_1)$ follows the same argument by noting that Lemma~\ref{lem:mA-finiteAction} also provides a bound on $C_t = A_t A_t^\tran$.
\end{proof}  
To characterize the optimal mean and covariance square root solving the Gaussian part of the covariance-modulated optimal transport distance, \eqref{eq:mean-cov-min:R}, we use the Hamiltonian formalism developed for geodesics in sub-Riemannian context (see e.g.~\cite[Sec.~2.2]{Rifford2014}. The constraint gives rise to a Lagrange multiplier, which might be active (non-zero) or not, leading to normal or abnormal geodesics. 
Thanks to Lemma~\ref{lem:SR:brackets} our sub-Riemannian structure is two-bracket generating, which results in only trivial (constant) abnormal geodesics (see~\cite[Theorem 2.22 and Example 2.1]{Rifford2014}, also~\cite[Sec.~6]{AgrachevLee2009} and \cite[Sec.~4.2]{FigalliRifford2010}).
In this way, we can characterize the geodesic equations in the following proposition.
\begin{proposition}\label{prop:geo:AR}
Let $(m_i,A_i)\in \subMfd$ for $i=0,1$, then any optimizer $(m_t,A_t)_{t\in[0,1]}$ for $\overline\MOP((m_0,A_0),(m_1,A_1))$ satisfies for some $\alpha\in \R^d$ the system~\eqref{eq:EL2R} with boundary values implied.
\end{proposition}
\begin{proof}
	We define the cotangent action for $(p,P)\in T^*\subMfd = \R^d \times \R^{d\times d}$ on $(r,X)\in T\subMfd$ by the pairing
	\[
	  \skp*{ (p,P) , (v,V)}_{T^*\subMfd\times T\subMfd} = p\cdot r + P: X . 
	\]
	The Riemannian inner product on $T\subMfd\times T\subMfd$ inducing the action functional~\eqref{eq:def:IA} at the point $(m,A)\in \subMfd$ is given by
	\begin{equation*}
		\skp[\big]{(v,V),(w,W)}_{(m,A)} = (A^{-1}v) \cdot (A^{-1} w) + (A^{-1} V) : (A^{-1} W) . 
	\end{equation*} 
	In the coordinates from Lemma~\ref{lem:SR:brackets} we define the Hamiltonian
	\begin{align*}
		H\bra[\big]{(m,A),(p,P)} &= \frac{1}{2}  \sum_{\alpha \in \set*{1,\dots,d}\cup \triangle_d} \abs[\Big]{\skp[\big]{ (p,P),  X^\alpha(m,A)}_{T^*\subMfd\times T\subMfd}}^2 \\
		&= \frac{1}{2} \pra[\bigg]{ \sum_{i=1}^d \bra*{p \cdot (Ae^i)}^2 + \sum_{\alpha \in\triangle_d} \bra*{P : (A S^\alpha)}^2}
	\end{align*}%
Since the distribution $H$ is 2-bracket generating, there are no strictly abnormal geodesics \cite[Theorem 2.22 and Example 2.1]{Rifford2014}. Hence every constant speed geodesic $(m_t,A_t)_{t\in[0,1]}$ admits a normal extremal lift, i.e.~it can be lifted to a smooth curve $\big((m_t,A_t), (p_t,P_t)\big)_{t\in[0,1]}$ in $T^*\subMfd$. This curve is an integral curve of the Hamiltonian vector field $(\partial H/\partial(p,P), -\partial H/\partial(m,A)\big)$. Explicitly, it satisfies the ODEs
\begin{align}\label{eq:Hamilton1}
\dot m &= \sum_{i=1}^d (p\cdot X^i)X^i\;,\quad
&\dot A &= \sum_{\alpha\in\Delta}(P: X^\alpha)X^\alpha\;,\\\label{eq:Hamilton2}
\dot p &= -\sum_{i=1}^d (p\cdot X^i)p\cdot D_mX^i\;,\quad
&\dot P &=-\sum_{i=1}^d (p\cdot X^i)p\cdot D_AX^i -\sum_{\alpha\in\Delta}(P: X^\alpha)P:D_AX^\alpha\;.
\end{align}
The  first line \eqref{eq:Hamilton1} tells that $p\cdot X^i$ and $P:X^\alpha$ are the coordinates of $(\dot m,\dot A)$ in the orthonormal basis given by $X^i$, $X^\alpha$. Hence we deduce $p\cdot Ae^i = \ip{\dot m}{ X^i}_A= A^{-1}\dot m \cdot e^i$ for $i=1,\dots,d$ and 
$P:AS^\alpha= \ip{\dot A}{X^\alpha}_A = A^{-1}\dot A:S^\alpha$ for all $\alpha\in \triangle_d$. 
This implies that
\begin{equation}\label{eq:pP} 
p=A^{-\tran}A^{-1} \dot m \;, \qquad\text{and}\qquad P= A^{-\tran} \big(A^{-1} \dot A + Q\big)\;,
\end{equation}
for a family of skew-symmetric matrices $(Q_t)_{t\in[0,1]}$. In particular, for any symmetric $S$ we have $A^\tran P:S=A^{-1}\dot A:S$.

The first equation in \eqref{eq:Hamilton2} simplifies since $X^i$ is independent of $m$ for $i=1,\dots,d$ and hence $\dot p = 0$. By setting $\alpha= p$, we get~\eqref{eq:EL-mean2R}. From the second equation in \eqref{eq:Hamilton2} we calculate for any $Y\in \R^{d\times d}$, by noting that $D_A X^i[Y] = Y e^i$ and $D_A X^\alpha[Y] = Y S^\alpha$ and using again~\eqref{eq:pP}, that
	\begin{align*}
		\dot P : Y &= - \sum_{i=1}^d \bra*{ p \cdot X^i} p\cdot (Y e^i) - \sum_{\alpha\in\triangle_d} \bra*{ P : X^\alpha} \bra*{P : (Y S^\alpha)} \\
		&= - \sum_{i=1}^d \bra*{ (A^{-1} \dot m) \cdot e^i} \pra*{\bra{ A^{-\tran} A^{-1} \dot m  \otimes e^i } : Y } - \sum_{\alpha \in \triangle_d} \bra*{ (A^{-1} \dot A) : S^\alpha}\bra*{ P S^\alpha : Y} \\
		&= - \bra*{A^{-\tran} A^{-1} \dot m \otimes A^{-1} \dot m} : Y - P A^{-1} \dot A : Y ,
	\end{align*}
	where we used the identities $\sum_i (\alpha \cdot e^i) e^i =\alpha$ for any $\alpha \in \R^d$ and $\sum_{\alpha\in \triangle_d} (S : S^\alpha) S^\alpha = S$ for any $S\in \S^d$.
	Hence, we identify $\dot P$ as
	\begin{align}\label{eq:Hamilton3}
		\dot P &= - A^{-\tran} A^{-1} \dot m \otimes A^{-1} \dot m -PA^{-1} \dot A  = -\alpha\otimes A^\tran\alpha -PA^{-1}\dot A\;.
	\end{align}
	To arrive at an equation independent of $P$, we take the time derivative in~\eqref{eq:pP} and obtain
	\begin{align*}
		\dot P &= \pderiv{}{t} \bra*{ A^{-\tran} A^{-1} \dot A  +A^{-\tran} Q}\\
		&= - 2 A^{-\tran} A^{-1} \dot A A^{-1} \dot A + A^{-\tran} A^{-1} \ddot A -A^{-\tran}\dot A^\tran A^{-\tran} Q + A^{-\tran}\dot Q, 
	\end{align*}
	where we used that $A$ is horizontal, i.e. $A^{-1} \dot A = \dot A^\tran A^{-\tran}$. Comparing this expression for $\dot P$ with \eqref{eq:Hamilton3} and multiplying with $A^\tran$ leads to 
	\begin{align*}
	\frac{\dd}{\dd t}\big(A^{-1}\dot A\big) &= A^{-1}\ddot A - A^{-1}\dot AA^{-1}\dot A\\
	&= -(A^\tran\alpha)^{\otimes 2} + \big(A^{-1}\dot AQ-QA^{-1}\dot A\big)-\dot Q\;.
	\end{align*}
The last term is antisymmetric while all other terms are symmetric. Hence, we infer that $Q$ is constant.
Thus we have 
\begin{equation*}
\frac{\dd}{\dd t}\big(A^{-1}\dot A\big) =  -(A^\tran\alpha)^{\otimes 2} + [A^{-1}\dot A, Q]\;,
\end{equation*} 	
which gives \eqref{eq:EL-cov2R}.
\end{proof}

 \begin{lemma}\label{cor:Aexplicit:m0_alt}
 	For given $(m_i,A_i)\in \subMfd$ for $i=0,1$ with $m_0=m_1=m$, any optimizer of $\overline\MOP((m,A_0),(m,A_1))$ is of the form
 	\begin{equation}\label{eq:Aexplicit:m0}
 		m_t = m \qquad\text{and}\qquad  A_t = A_0 e^{t B} e^{- t B^{\mathsf{asym}}} \qquad\text{for } t\in[0,1] ,
 	\end{equation}
    with $B\in \R^{d\times d}$ an optimizer of 
    \[
        \frac{1}{2} \inf_{B\in \R^{d\times d}}\set*{ \norm{B^{\mathsf{sym}}}_{\HS}^2 : A_0^{-1} A_1 = e^{B} e^{-B^{\mathsf{asym}}}} = \overline\MOP((m,A_0),(m,A_1)) ,
    \]
    where $B^{\mathsf{sym}} := \frac{B+B^\tran}{2}$ and $B^{\mathsf{asym}} := \frac{B-B^\tran}{2}$.
\end{lemma}
\begin{proof}
We first note from the form of $\overline I(m,A)$ that any optimal curve must satisfy $m_t=m_0=m_1$ for all $t$. Such an optimal curve is a sub-Riemannian geodesic in the moment manifold $\subMfd$ and by Proposition \ref{prop:geo:AR} it must satisfy \eqref{eq:EL2R}. 
Since here $\alpha=A_t^{-\tran}A_t^{-1}\dot m_t=0$ we obtain from~\eqref{eq:EL-cov2R}
$\frac{\dd}{\dd t}(A_t^{-1}\dot A_t) = [A_t^{-1}\dot A_t,Q]$ for some skew symmetric matrix $Q$, and hence
\begin{equation}\label{eq:AdotA}
A_t^{-1}\dot A_t = e^{-tQ}Ze^{tQ}
\end{equation}
for a suitable symmetric matrix $Z\in \S^d$. Defining $X_t=e^{tQ}A_te^{-tQ}$, we deduce that 
\[\dot X_t = QX_t - X_t Q +X_t Z\;.\]
Note that $X_0=A_0$. The unique solution for $X$ is given by
\[X_t = e^{tQ}X_0e^{t(Z-Q)} \; , \]
implying the representation~\eqref{eq:Aexplicit:m0} by setting $B=Z-Q$.
Finally, using \eqref{eq:AdotA}, the action of the curve is given by
\[\overline{I}(m,A) = \frac{1}{2} \int_0^1 \norm{ A_t^{-1}\dot A_t}_{\HS}^2\dd t = \frac{1}{2}\norm{Z}_{\HS}^2\;. \qedhere\]
\end{proof}
We now turn to the unconstrained moment optimization problem $\MOP(\mu_0,\mu_1)$ in~\eqref{eq:mean-cov-min}.
We complement the Hamiltonian approach for the derivation of geodesics in the sub-Riemannian framework with a more straightforward derivation via the minimization of the energy~\eqref{eq:def:MOP:action} for the Riemannian case. 
\begin{proposition}\label{prop:optimal}
	Let $\mu_0,\mu_1\in\cP(\R^d)$ satisfy Assumption~\ref{ass:Wfinite}, then any optimizer $(m,C)\in \MC(\mu_0,\mu_1)$ of the minimization problem \eqref{eq:mean-cov-min} satisfies~\eqref{eq:EL2} for some suitable $\alpha\in\R^d$. 

	Moreover, for any optimal curve $(m,C)\in \MC(\mu_0,\mu_1)$ of \eqref{eq:mean-cov-min} the map $[0,1]\ni t\mapsto \skp[\big]{\dot m_t,C_t^{-1}\dot m_t} +\frac14\tr\bra[\big]{\dot C_tC^{-1}_t\dot C_tC^{-1}_t}$ is constant.
\end{proposition}
\begin{proof}
We show that the Euler-Lagrange equations for the minimization problem \eqref{eq:mean-cov-min} are given by \eqref{eq:EL-mean2}-\eqref{eq:EL-cov2}. 
Indeed, for the optimizer $(m,C)\in \MC(\mu_0,\mu_1)$ and any variation $n\in \AC([0,1],\R^d)$ with $n(0)=n(1)=0$ and $D\in \AC([0,1],\R^{d\times d}_{\sym})$ with $D(0)=D(1)=0$ of $m$ and
$C$, respectively, we obtain (dropping $t$ from the notation):
\begin{align*}
  0 &= \frac{\dd}{\dd \varepsilon} I(m+\varepsilon n, C+\varepsilon D)\\
    &= \int_0^1\pra*{ \dot n \cdot C^{-1}\dot m  - \frac12 \dot m \cdot C^{-1}DC^{-1}\dot m
      +\frac14\tr\big(\dot D C^{-1}\dot C C^{-1}\big)
      -\frac 14 \tr\big(DC^{-1}\dot C C^{-1}\dot C C^{-1}\big)}\dd t\\
    & = \int_0^1 \Biggl[ -\frac{\dd{}}{\dd t}\big(C^{-1}\dot m\big) \cdot n -\frac12 C^{-1} (\dot m\otimes \dot m) C^{-1}: D  \\
    &\qquad\quad 
      -\frac 14 \frac{\dd}{\dd t}\big[C^{-1}\dot C C^{-1}\big]: D 
	  - \frac 14 C^{-1}\dot C C^{-1}\dot C C^{-1}: D \Biggr]\dd t\\
    &=  - \int_0^1 \frac{\dd}{\dd t}\big(C^{-1}\dot m\big) \cdot n \dd t+\int_0^1\frac 14 \big[C^{-1}\dot C C^{-1}\dot C C^{-1}
      - C^{-1}\ddot C C^{-1}-C^{-1}\dot m\dot m^\tran C^{-1}\big]: D \dd t\;.
       \end{align*}
       This yields the Euler-Lagrange equations~\eqref{eq:EL2}.
Next, we differentiate the action density corresponding to the minimization problem~\eqref{eq:mean-cov-min}.
By direct calculation using~\eqref{eq:EL2}, we obtain for any $t\in(0,1)$
\begin{align*}
\MoveEqLeft\pderiv{}{t} \pra*{ \frac12\ip{\dot m_t}{C_t^{-1}\dot m_t} +\frac18\tr \bigl(\dot C_tC^{-1}_t\dot C_tC^{-1}_t\bigr)}\\
&= \frac12\ip{\alpha}{\dot C_t \alpha} +\frac14\tr \big(\ddot C_tC^{-1}_t\dot C_tC^{-1}_t\big)
-\frac14\tr \big(\dot C_tC^{-1}_t\dot C_tC^{-1}_t\dot C_tC^{-1}_t\big)=0\,,
\end{align*}
and hence the action density is equal to $I(m,C)$ on $[0,1]$.
\end{proof}
In the case where $\mean(\mu_0)=\mean(\mu_1)$, the remaining metric for the Covariance part is an existing Riemannian one on $\S_{\succ 0}$ with explicit geodesics, which we state from~\cite[Theorem 6.1.6]{Bhatia2019}.
       \begin{corollary}\label{cor:alpha0}
       Let $\mu_0,\mu_1\in \cP_{2,+}(\R^d)$ with $\mean(\mu_0)=\mean(\mu_1)=m\in \R$ and $C_0=\C(\mu_0)$, $C_1 = \C(\mu_1)$, then the mean is constant,
       i.e.
       \[m_t=m \quad \text{ for all } \quad t\in[0,1]\,,\]
       and the covariance satisfies
       \begin{equation}\label{eq:alpha0:Ct}
       		C_t = C_0^{\frac{1}{2}} \bra*{ C_0^{-\frac{1}{2}} C_1 C_0^{-\frac{1}{2}}}^t C_0^{\frac{1}{2}} .
       \end{equation}
   		Hereby, the power $t\in (0,1)$ is well-defined by spectral calculus, since the matrix $C_0^{-\frac{1}{2}} C_1 C_0^{\frac{1}{2}}$ is symmetric and positive. 
   		The moment distance is explicitly given by
   		\[
   		  \MOP(\mu_0,\mu_1)^2 = \frac{1}{8} \norm[\big]{\log\bra[\big]{C_0^{-\frac{1}{2}} C_1 C_0^{-\frac{1}{2}}}}_{\HS}^2 .
   		\]
        In particular, if $C_0$, $C_1$ commute, 	formula~\eqref{eq:alpha0:Ct} becomes %
       \[
       C_t=C_0^{1-t}C_1^t \qquad\text{and}\qquad \MOP(\mu_0,\mu_1)^2 = \frac{1}{8} \sum_{i=1}^d \abs[\big]{\log \lambda_i(C_0) - \log \lambda_i(C_1)}^2  \,.
       \]
       \end{corollary}
In the case when the covariance matrix admits an autonomous eigendecomposition, we can show that particular solutions of the optimality conditions can be reduced to those in the variance-modulated optimal transport problem, which are explicitly identified in Theorem~\ref{thm:moments-sol}.
\begin{corollary}\label{cor:mC:diagonal}
    Assume the covariance matrices $C_0, C_1 \in \S_{\succ0}^d$ have the same eigenvectors and that $m_1-m_0$ is an eigenvector, that is
    \begin{equation}
	C_t = \sum_{i} \lambda_i(t) e^i \otimes e^i \quad\text{for } t\in \set{0,1} \quad\text{and}\quad m_1-m_0\in \operatorname{span}\set{e^\ell} \quad\text{for some } \ell \in \set*{1,\dots, d}
\end{equation}
where $\set{e^i}_{i=1}^d$ is a time-independent orthornomal system. 

Then, a solution $(m,C)$ to \eqref{eq:EL-mean2}, \eqref{eq:EL-cov2} is given by letting $m_t = \sum_{i} \hat m_i(t) e^i$  and $C_t=\sum_i \lambda_i(t) e^i\otimes e^i$, with coefficients $(\hat m_i(t),\sqrt{\lambda_i(t)})$ given as solutions to
\eqref{eq:mean-var:mean-explicit}--\eqref{eq:mean-var:sigma-explicit} where $n_i=\abs{m_1-m_0} \delta_{i\ell}$,
with boundary conditions $(0,\sqrt{\lambda_i(0)})$ and $(n_i,\sqrt{\lambda_i(1)})$ given via $(m_0,C_0)$ and $(m_1,C_1)$.
In particular, for $i\ne \ell$, we have $\lambda_i(t) = \lambda_i(0)^{1-t} \lambda_i(1)^t$.
\end{corollary}
\begin{proof}
We observe that~\eqref{eq:EL-cov2} preserves the symmetry of $C_t$ and by Lemma~\ref{lem:apriori-cov}, we also get that if $\cW(\mu_0,\mu_1)<\infty$, then $0<C_t<\infty$ for $t\in [0,1]$. Hence, we make the Ansatz that $C_t = E \Lambda_t E^\tran $ with $E=\sum_{i=1}^d e^i \otimes e^i$
and $\Lambda_t = \diag\bra*{\lambda_1(t),\dots,\lambda_d(t)}$. 
With this choice, the equation~\eqref{eq:EL-cov2} simplifies after multiplying by $E^\tran $ and $E$ from left and right to
\begin{align*}
	\ddot \Lambda_t &= \dot \Lambda_t \Lambda_t^{-1} \dot \Lambda_t - 2 \Lambda_t E^\tran  \alpha\alpha^\tran  E \Lambda_t .
\end{align*}
Now, by assumption, we have that $ E^\tran  \alpha\alpha^\tran  E = \diag\bra*{\delta_{i\ell}}_{i=1}^d$ and hence the system is of diagonal form and we get for $i=1,\dots,d$ the ODEs
\begin{equation}
	\ddot \lambda_i(t) = \frac{(\dot \lambda_i(t))^2}{\lambda_i(t)} - 2 \lambda_\ell(t)^2 \hat \alpha_\ell^2 \delta_{i\ell} .
\end{equation}
Substituting $\sigma_i(t) = \sqrt{\lambda_i(t)}$, we arrive at the system~\eqref{eq:mean-var:mean}-\eqref{eq:mean-var:var} of the variance-modulated optimal transport problem. Therefore, Theorem~\ref{thm:moments-sol} provides explicit solutions for $(\hat m_i,\sigma_i)$ as claimed.
\end{proof}

\section{Existence of geodesics}\label{sec:existence}

\subsection{Existence at small distance}

\textbf{Strategy.}
The proof of Theorem~\ref{thm:existence-smalldist} follows an argument by contradiction, which we split into several steps. To show existence of minimizers for $\cW_{0,\Id}$, we consider a minimizing sequence $(\mu^n, V^n)$ for $\cW_{0,\Id}(\mu_0,\mu_1)$ and show relative compactness of the measure-flux pair $(\mu^n$, $V^n\mu^n)$ in weak topologies. By a standard argument, we can extract and identify a limiting measure-velocity pair $(\mu, V)$ solving the continuity equation.
The problem is then to show that the constraints on mean and covariance are not lost in the limit, namely that $\C(\mu_s)=\Id$ and $\mean(\mu_s)=0$ for all~$s\in[0,1]$. By contradiction, we assume the second moments are not tight and use this to construct a competitor by rerouting mass that leaves a large ball and normalizing the resulting measures. The rerouting will decrease the length of a fraction of the transport curves but potentially decrease the covariance. Hence, the normalization step might increase the action. Both competing effects are carefully estimated. Finally, the assumption that $\cW_{0,\Id}(\mu_0,\mu_1)$ is sufficiently small, allows us to show that the competitor has smaller action. This yields the desired contradiction.

We will then reduce the question of the existence of minimizers for $\cW$ to that of $\cW_{0,\Id}$ using the splitting in shape and moments.
 
The main ingredient for showing tightness of the second moments for minimizing sequences is contained in the following proposition. 
For a pair $(\mu,V)\in\CE_{0,\Id}(\mu_0,\mu_1)$ we use the notation
\[
  \cA(\mu,V)=\frac12\int_0^1\int |V_t|^2\dd\mu_t\dd t\;.
\]
\begin{proposition}\label{prop:minimizing_sequence}
 Let $\mu_0^n,\mu_1^n\in \cP_{0,\Id}(\R^d)$ and let $(\mu^n, V^n)$ be a sequence in $\CE_{0,\Id}(\mu_0^n,\mu_1^n)$ such that
 \[
\lim_n\cA(\mu^n,V^n)<\frac18
\]
exists and $\mu^n_t\to\mu_t$ weakly for all $t\in[0,1]$ for a curve $(\mu_t)_{t\in[0,1]}$ in $\cP_{2}(\R^d)$ with $\mu_0,\mu_1\in \cP_{0,\Id}(\R^d)$. 
If there is $t_0\in(0,1)$ such that $\mu_{t_0}\notin\cP_{0,\Id}(\R^d)$ then there exists a sequence $(\tilde\mu^n,\tilde V^n)\in \CE_{0,\Id}(\mu^n_0,\mu^n_1)$ connecting the same sequence of marginals such that 
\[
\liminf_{n\to\infty} \cA(\tilde\mu^{n},\tilde V^{n}) < \lim_{n\to\infty}\cA(\mu^n,V^n)\;. 
\]
\end{proposition}
\begin{proof}\hfill\\ \noindent
	{\bf Step 1.} 
	Assume that there is $t_0\in[0,1]$ with $\mu_{t_0}\notin\cP_{0,\Id}(\R^d)$. Since $\mu^n_t\in \cP_{0,\Id}(\R^d)$ for all $n$, its second moments are uniformly bounded and we can infer $\mean(\mu_t)= \lim_n\mean(\mu^n_t)= 0$ for all $t$. Hence, we must have $\C(\mu_{t_0})\neq \Id=\lim_n\C(\mu_{t_0}^n)$. If the second moment at time $t_0$ was tight, i.e.~
	\begin{equation}\label{eq:tightness_alt}
		\forall \eps>0 \; \exists R>0 \; \forall n:\qquad \int \boldsymbol{1}_{\{|x|\geq R\}}|x|^2\dd\mu^n_{t_0}(x)\leq \eps\;,
	\end{equation}
	this would imply the convergence $\lim_n\C(\mu^n_{t_0})=\C(\mu_{t_0})$. Hence, we infer on the contrary that there exists $\eps>0$ such that for all $k\in\mathbb N$ there exists $n=n(k)$ with
	\begin{equation}\label{eq:contra_alt}
		\int \boldsymbol{1}_{\{|x|\geq k\}}|x|^2\dd\mu^{n(k)}_{t_0}(x)\geq \eps\;.
	\end{equation}
From now on, we consider the (relabeled) subsequence $(\mu^k,V^k)=(\mu^{n(k)},V^{n(k)})$. Using \eqref{eq:contra_alt} we will construct a new sequence $(\tilde\mu^k,\tilde V^k)\in \CE_{0,\Id}(\mu^k_0,\mu^k_1)$ with $\liminf_k\cA(\tilde\mu^k,\tilde V^k)<\lim_k\cA(\mu^k,V^k)$.

	\smallskip\noindent
	{\bf Step 2.} By the superposition principle for absolutely continuous curves in the Wasserstein space~\cite{AGS} there exist probabilities $\pi^k\in \cP(\Gamma)$ on the space $\Gamma:=C([0,1],\R^d)$ concentrated on solutions to the ODE $\dot\gamma_t= V^k_t(\gamma_t)$ such that 
	\[\int_0^1\int_{\R^d} |V^k_t|^2\dd\mu^k_t\dd t = \int_\Gamma \int_0^1 |\dot\gamma_t|^2\dd t\dd\pi^k(\gamma)\;.\]
	Let us set $\tilde\pi^k:=\pi^k |_{\Gamma^k}$ with $\Gamma^k:=\{\gamma \in \Gamma: |\gamma_{t_0}|\geq k\}$.
	Let $q$ be any coupling of $\tilde\mu_i:=(e_i)_\#\tilde\pi^k$ with $i=0,1$, where $e_t:\gamma\mapsto \gamma_t$ is the evaluation map. Denote for $x,y\in \R^d$ by $\gamma^{x,y}$ the straight line connecting $x,y$ and define $\bar\pi^k:=\int_{\R^d\times \R^d} \gamma^{x,y}\dd q(x,y)$. Set $\hat\pi^{k}:=\bar\pi^k-\tilde\pi^k$ and for $\alpha\in[0,1]$ set
	\[\pi^{k,\alpha} := \pi^k+\alpha \hat\pi^k\;.\] 	
	By evaluation for any $t\in [0,1]$, the measure $\pi^{k,\alpha}$ gives rise to a curve of measures $\nu_t^{k,\alpha}:=(e_t)_\#\pi^{k,\alpha}$, where $e_t:C\big([0,1];\R^d\big)\to\R^d$, $\gamma\mapsto \gamma_t$ is the evaluation map at time $t$. Denote the moments of this curve by
\begin{equation}\label{eq:k-moments}	
m_t^{k,\alpha}: = \mean(\nu^{k,\alpha}_t)=\int_\Gamma \gamma_t\dd\pi^{k,\alpha}(\gamma)\,,\qquad
	C^{k,\alpha}_t:=\C(\nu^{k,\alpha}_t)=\int_{\Gamma} \bra[\big]{\gamma_t-m_t^{k,\alpha}}^{\otimes 2} \dd\pi^{k,\alpha}(\gamma)\;.
	\end{equation}
Define $(A_t^{k,\alpha})_{t\geq 0}$ as solution to $\dot A_t^{k,\alpha}:=\frac12 \dot C^{k,\alpha}_t(A_t^{k,\alpha})^{-\tran}$ starting from $A_0^{k,\alpha}:=(C_0^{k,\alpha})^{1/2}$ as usual, so that $A^{k,\alpha}_t(A_t^{k,\alpha})^\tran=C^{k,\alpha}_t$ and $(A_t^{k,\alpha})^{-1}\dot A_t^{k,\alpha}$ is symmetric. Finally, we normalize the curve by setting
\[
\mu^{k,\alpha}_t = \big((A^{k,\alpha}_t)^{-1}(\cdot-m^{k,\alpha}_t)\big)_\#\nu^{k,\alpha}_t\;.
\]
Similarly, we can normalize $\pi^{k,\alpha}$ by setting $\theta^{k,\alpha}$ as the image of $\pi^{k,\alpha}$ under the map $(\gamma_t)\mapsto \big((A^{k,\alpha}_t)^{-1}(\gamma_t-m^{k,\alpha}_t)\big)$.
The measure $\theta^{k,\alpha}$ gives rise to a pair $(\mu^{k,\alpha},V^{k,\alpha})\in\CE(\mu^k_0,\mu^k_1)$, defined by 
	\[\int_{\R^d} \phi \dd\mu^{k,\alpha}_t = \int_\Gamma \phi(\gamma_t)\dd\theta^{k,\alpha}(\gamma)\;,\qquad \int_{\R^d} \ip{\Phi}{ V_t^{k,\alpha}}\dd\mu^{k,\alpha}_t = \int_\Gamma \ip{\Phi(\gamma_t)}{\dot\gamma_t}\dd\theta^{k,\alpha}(\gamma)\]
	for any test functions $\phi\in C_b(\R^d)$, $\Phi\in C_c(\R^d;\R^d)$. Now $\mu_t^{k,\alpha}\in \cP_{0,\Id}(\R^d)$ for all $t$ and we have
	\[
	\cA(\mu^{k,\alpha},V^{k,\alpha})=\frac12\int_0^1\int_{\R^d} |V^{k,\alpha}_t|^2\dd\mu_t^{k,\alpha}\dd t = \frac12\int_\Gamma \int_0^1 |\dot\gamma_t|^2\dd t\dd\theta^{k,\alpha}(\gamma) \;.
	\]
	Note that we have not changed the marginals at times $t=0,1$ in this construction.
	
	\smallskip\noindent
	{\bf Step 3.} We claim that as $k\to \infty$
	\begin{align}
		\label{eq:nomarginalmoment_alt}
		\int_\Gamma \left(1+|\gamma_0|^2+|\gamma_1|^2\right)\dd\tilde\pi^k \to 0 
		\quad\text{and}\quad
		\sup_{t\in[0,1]}\int_\Gamma |\gamma_t|\dd\tilde\pi^k\to 0\;,\quad 
		\int_\Gamma \int_0^1|\dot\gamma_t|\dd t\dd\tilde\pi^k \to 0\;.
	\end{align}
	Indeed, note first that 
	\[\tilde\pi^k(\Gamma)=\int_\Gamma 1_{\{|\gamma_{t_0}|\geq k\}}\dd\pi^k \leq \frac{1}{k^2} \int_\Gamma |\gamma_{t_0}|^2\dd\pi^k=\frac{d}{k^2}\to 0\quad \text{ as } k\to\infty\,. \]
	We have
	\[\int_\Gamma|\gamma_0|^2\dd\tilde\pi^k \leq  \int_\Gamma 1_{\{|\gamma_0|\geq R\}}|\gamma_0|^2\dd\tilde\pi^k + R^2 \tilde\pi^k(\Gamma) = \int_\Gamma 1_{\{|x|\geq R\}}|x|^2\dd\mu^k_0 + R^2 \tilde\pi^k(\Gamma) \;.\]
	Since $\lim_k\C(\mu_0^k)=\C(\mu_0)$, the second moments at $t_0$ are tight. Hence, this term can be made arbitrarily small by first choosing $R$ sufficiently large and then choosing $k$ sufficiently large. The same argument applies to $\int_\Gamma |\gamma_1|^2\dd\tilde\pi^k$, yielding the first claim in \eqref{eq:nomarginalmoment_alt}. To obtain the second claim we estimate
	\begin{align*}
		\int_\Gamma |\gamma_t|\dd\tilde\pi^k  
		&\leq \int_\Gamma 1_{\{|\gamma_t|\geq R\}}|\gamma_t|\dd\tilde\pi^k  +R \tilde\pi^k(\Gamma)
		\leq \frac{1}{R}\int_\Gamma |\gamma_t|^2\dd\pi^k  +R \tilde\pi^k(\Gamma)\;,\\
		\int_\Gamma \int_0^1|\dot\gamma_t|\dd t\dd\tilde\pi^k(\gamma) 
		&\leq
		\int_\Gamma \int_0^1 1_{\{|\dot \gamma_t|\geq R\}}|\dot \gamma_t|\dd t\dd\tilde\pi^k  +R \tilde\pi^k(\Gamma)
		\leq \frac{1}{R}\int_\Gamma \int_0^1|\dot\gamma_t|^2\dd\pi^k +R \tilde\pi^k(\Gamma)\;,
			\end{align*}
where we used that $\tilde\pi^k\leq \pi^k$ by construction. Since the integrals on the right-hand side are bounded in k (uniformly in t in the first case) we can make these terms arbitrarily small as before.

	As a consequence of \eqref{eq:nomarginalmoment_alt}, we obtain
	\begin{equation}\label{eq:noaction_alt}
		\sup_{t\in[0,1]}\int_\Gamma \pra*{1+ |\gamma_t|^2+ |\dot\gamma_t|^2  }\dd\bar\pi^k \to 0 \quad \text{ as } k\to\infty\;,
	\end{equation}
	since for $\bar\pi^k$ a.e.~$\gamma$ we have $\gamma_t=(1-t)\gamma_0+t\gamma_1$ by construction.

	\smallskip\noindent
	{\bf Step 4.} Let us from now on freely drop the superscripts $k,\alpha$ and the subscipt $t$ or parts of them from the notation if clear from context.
	From \eqref{eq:k-moments} and the fact that $\mu^k_t\in \cP_{0,\Id}(\R^d)$ we deduce
	\[m_t = \int_\Gamma \gamma_t\dd(\pi+\alpha(\bar\pi-\tilde\pi))=\alpha \int_\Gamma \gamma_t \dd(\bar\pi-\tilde\pi) \;,\quad \dot m_t = \alpha \int_\Gamma \dot\gamma_t\dd(\bar\pi-\tilde\pi)\;.\]
		Moreover,
	\begin{align*}
		C_t &= \int_\Gamma \big(\gamma_t-m_t\big)^{\otimes 2}
\dd\pi^\alpha=
           \int_\Gamma \gamma_t^{\otimes 2}\dd\big(\pi+\alpha (\bar\pi-\tilde\pi)\big)
-m_t^{\otimes 2}\\
		&= \Id+\alpha \int_\Gamma \gamma_t^{\otimes 2} \dd (\bar\pi-\tilde\pi) - m_t^{\otimes 2}\;,\\
		\dot C_t &= \int_\Gamma \big((\dot\gamma-\dot m)\otimes(\gamma-m) + (\gamma-m)\otimes(\dot\gamma-\dot m)\big)\dd\pi^\alpha \;.
	\end{align*}
	From \eqref{eq:nomarginalmoment_alt} and \eqref{eq:noaction_alt} we infer that for any $\delta>0$ and $k$ sufficiently large, we have
	\begin{align}\label{eq:Cupperlower1_alt}
		(1-\delta)\Id  -\alpha E_t
		\preccurlyeq C_t
		\preccurlyeq (1+\delta)\Id\;,
	\end{align}
	with $E_t:=\int\gamma_t^{\otimes 2}\dd\tilde\pi$. We use the following bounds on the inverse of a matrix: For $B,D$ symmetric, with $B$ positive definite and $0\leq D\leq \frac12 B$ we have
	\begin{equation}\label{eq:inverse-bound_alt}
		B^{-1} -B^{-1}DB^{-1}\preccurlyeq \big(B+D)^{-1}\preccurlyeq B^{-1}+2B^{-1}DB^{-1}\;.
	\end{equation}
	Hence, using that $0\preccurlyeq E_t\preccurlyeq \Id$ and \eqref{eq:inverse-bound_alt}, we obtain for $\alpha\leq \frac12$ that
	\begin{align}\label{eq:Cupperlower2_alt}
		(1- \delta) \Id \preccurlyeq \big(C_t\big)^{-1} \preccurlyeq (1+2\delta) \Id +2\alpha E_t\;.
	\end{align}

	\smallskip\noindent
	{\bf Step 5.}  Let us set 
	\[a^{k,\alpha}_t := \int_\Gamma |\dot\gamma_t|^2\dd\theta^{k,\alpha}\;,\qquad \cA^{k,\alpha}:=\frac12\int_0^1a^{k,\alpha}_t\dd t\;. \]
	We can assume that $t\mapsto \frac12 a_t^{k,0}=\cA^{k,0}$ is constant after possibly reparametrizing in time. This would only decrease the value of $\frac12\int_0^1|V^k_t|^2\dd\mu^k_t\dd t$. Dropping $k$, $\alpha$ and $t$ mostly from the notation, we calculate
	\begin{align*}
		a^{\alpha}&=\int_\Gamma \abs*{\dot\gamma}^2\dd\theta^\alpha =\int_\Gamma \abs[\Big]{\frac{\dd}{\dd t}\bigl(A^{-1}(\gamma-m)\bigr)}^2\dd\pi^\alpha
		= \int_\Gamma\abs[\big]{A^{-1}(\dot\gamma-\dot m) -A^{-1}\dot AA^{-1}(\gamma-m)}^2\dd\pi^\alpha \\
		&=\int_\Gamma |\dot\gamma|_C^2\dd\pi^\alpha \!-\! |\dot m|_C^2 + \!\int_\Gamma \pra*{ \abs[\big]{A^{-1}\!\dot A A^{-1}(\gamma-m)}^2 \! - 2 \skp[\big]{A^{-1}(\dot\gamma-\dot m),A^{-1}\!\dot A A^{-1}(\gamma-m)}}\!\dd\pi^\alpha\\
		&=\int_\Gamma |\dot\gamma|_C^2\dd\pi^\alpha \!-|\dot m|_C^2 + \tr\pra[\big]{A^{-1}\dot A A^{-1}CA^{-\tran}\dot A^\tran A^{-\tran}}
		-\tr\pra[\big]{A^{-1}\dot A A^{-1}\dot C A^{-\tran}}\\
		&=\int_\Gamma |\dot\gamma|_C^2\dd\pi^\alpha \!-|\dot m|_C^2  -\Vert A^{-1}\dot A \Vert^2_{\HS}\leq \int_\Gamma |\dot\gamma|_C^2\dd\pi^\alpha\;.
	\end{align*}
	Here we have used, $\dot A=\frac12 \dot CA^{-\tran}$ and the symmetry of $A^{-1}\dot A$ in the last two equalities.
	From~\eqref{eq:Cupperlower2_alt} and \eqref{eq:noaction_alt} we obtain for any $\delta>0$ and $k$ sufficiently large that 
	\begin{align}\label{eq:action-est_alt}
	a^\alpha &\leq \int_\Gamma |\dot\gamma|_C^2\dd\big(\pi+\alpha(\bar\pi-\tilde\pi)\big) \leq a^0 +\alpha (2 a^0 \tr E  - F)+\delta\;,
		\end{align}
	with $F:=\int_\Gamma |\dot\gamma|^2\dd\tilde\pi$.

	\smallskip\noindent	
	{\bf Step 6.} We bound the quantities appearing in \eqref{eq:action-est_alt}. We have for $\delta>0$ and $k$ sufficiently large using \eqref{eq:nomarginalmoment_alt}
	\begin{align}\nonumber
		\int_0^1\tr\pra*{E_s}\dd s
		&= \int_0^1\int_\Gamma |\gamma_s|^2\dd\tilde \pi\dd s 
		= \int_0^1\int_\Gamma \abs*{\gamma_0+\int_0^s\dot\gamma_r\dd r}^2\dd\tilde\pi\dd s\\\label{eq:tracebound_alt}
		&\leq \delta + 2 \int_\Gamma\int_0^1|\dot\gamma_s|^2\dd s\dd\tilde\pi=\delta+2\int_0^1F_s\dd s\;.
	\end{align}
	We also note that
	\begin{align}\label{eq:tildespeed_alt}
		\int_0^1F_s\dd s=\int_0^1\int_\Gamma|\dot\gamma_s|^2\dd\tilde\pi\dd s
		\geq \frac{1}{t_0}\int_\Gamma |\gamma_{t_0}-\gamma_{0}|^2\dd\tilde\pi+\frac{1}{1-t_0}\int_\Gamma |\gamma_{1}-\gamma_{t_0}|^2\dd\tilde\pi \geq 4 \eps -\delta\;,
	\end{align}
	where in the last step we have expanded the square and used \eqref{eq:nomarginalmoment_alt} and the assumption \eqref{eq:contra_alt} on $\tilde\pi$, which is equivalent to $\epsilon \le \int_\Gamma|\gamma_{t_0}|^2 \dd\tilde\pi$. 

	\smallskip\noindent
	{\bf Step 7.} We conclude from the previous steps as follows. Collecting \eqref{eq:action-est_alt}, \eqref{eq:tracebound_alt}, and using that $s\mapsto a^{k,0}_s=2\cA^{k,0}$ is constant, we have for $k$ sufficiently large
	\begin{align}\label{eq:Abound_alt}
		\cA^{k,\alpha} \leq & \cA^{k,0} +\alpha\bra*{4\cA^{k,0} - \frac12}\int_0^1F_s\dd s +\delta\;.
	\end{align}
	Recall that by assumption $\lim_k \cA^{k,0}<\frac{1}{8}$. Since $\int_0^1F_s\dd s$ is bounded away from $0$ by \eqref{eq:tildespeed_alt} and since $\delta$ can be chosen arbitrarily small as $k\to\infty$, \eqref{eq:Abound_alt} yields that 
	\[
	\liminf_k \cA(\mu^{k,\alpha},V^{k,\alpha}) = \liminf_k \cA^{k,\alpha} <\lim_k \cA^{k,0} = \lim_k \cA(\mu^k,V^k)\;.\]
	Hence, with $(\mu^{k,\alpha},V^{k,\alpha})$ we have found a sequence with lower asymptotic action as claimed. This finishes the proof.
\end{proof}

\begin{proof}[Proof of Theorem~\ref{thm:existence-smalldist}, part (1)]
\hfill \\ \noindent
 Let $\mu_0,\mu_1\in\cP_{0,\Id}(\R^d)$ with $\cW_{0,\Id}(\mu,\mu_1)<\frac18$ and let $(\mu^n,V^n)\subset \CE(\mu_0,\mu_1)$ be a minimizing sequence for $\cW_{0,\Id}(\mu_0,\mu_1)$, i.e.
	\[\cW_{0,\Id}(\mu_0,\mu_1)^2 = \lim_n\int_0^1\int |V^n_t|^2\dd\mu^n_t\dd t\;.\]
	In particular, $\int_0^1\int|V_n|^2\dd\mu_n\dd t$ is bounded in $n$.
	Following e.g.~ the argument in~\cite{Dolbeault-Nazaret-Savare} (see also the proof of Theorem~\ref{thm:W:metric}),
	one can show that up to a subsequence we have $\mu^n_t\dd t\rightharpoonup \mu_t\dd t$ weakly and $V^n_t\mu^n_t\dd t\rightharpoonup V_t\mu_t\dd t$ in duality with $C_c(\R^d\times [0,1])$ for a pair $(\mu,V)\in \CE(\mu_0,\mu_1)$, as well as $\mu^n_t\rightharpoonup\mu_t$ for all $t\in[0,1]$. From Proposition \ref{prop:minimizing_sequence} and the fact that $(\mu^n,V^n)$ is a minimizing sequence, we infer that $\mu_t\in\cP_{0,\Id}(\R^d)$ for all $t\in[0,1]$. It remains to show that $(\mu,V)$ achieves minimal action. For this recall the joint lower semicontinuity of the Benamou-Brenier functional
	\[(\mu,W)\mapsto \mathcal B(\mu,W):=\int \alpha\Big(\frac{\dd W}{\dd\sigma},\frac{\dd\mu}{\dd\sigma}\Big)\dd\sigma\;,\quad\text{with}\quad \alpha(w,s) = \begin{cases}
		\frac{|w|^2}{s}  & s>0\;,\\
		0 & s=0, w=0\;,\\
		\infty &\text{ else ,}
	\end{cases}\]
	where $\sigma$ is an arbitrary reference measure s.t. $\mu,W\ll \sigma$, see \cite{BenamouBrenier2000}.
	This yields
	\begin{align*}
		\int_0^1 \! \int \! |V_t|^2\dd\mu_t\dd t =  \int_0^1 \! \mathcal B(\mu_t,V_t)\dd t\leq 
	 \liminf_n\int_0^1 \! \mathcal B(\mu^n_t,V^n_t)\dd t
		=\liminf_n\int_0^1\!\int \! |V^n_t|^2\dd\mu^n_t\dd t\;.
	\end{align*}
Thus $(\mu,V)\in\CE_{0,\Id}(\mu_0,\mu_1)$ is a minimizer, i.e.~ $\frac12\int_0^1\int|V_t|^2\dd\mu_t = \cW_{0,\Id}(\mu_0,\mu_1)^2$.
\end{proof}

\begin{proof}[Proof of Theorem~\ref{thm:existence-smalldist}, part (2)]
\hfill \\ \noindent
Let $\mu_0,\mu_1\in \cP_{2,+}(\R^d)$ such that $\cW(\mu_0,\mu_1)^2<\frac18+\MOP(\mu_0,\mu_1)^2$. From the splitting result Theorem~\ref{thm:main-cov} we have that
\begin{equation}\label{eq:part2-1}
\cW(\mu_0,\mu_1)^2 
=
\inf\Big\{ \cA(\mu,V) + I(m,C)\Big\}\;,
\end{equation}
with $\cA(\mu,V):=\frac12\int_0^1 \int |V_t|^2\dd\mu_t\dd t$ and where the infimum is taken over $R\in \SO(d)$, $(\mu,V)\in \CE_{0,\Id}(R_\#\bar\mu_0,\bar\mu_1)$, and $(m,C)\in \MC_R(\mu_0,\mu_1)$.
Let $R_n$, $(\mu^n,V^n)\in\CE_{0,\Id}((R_n)_\#\bar\mu_0,\bar\mu_1)$, and $(m^n,C^n)\in \MC_{R_n}(\mu_0,\mu_1)$ be minimizing sequences, such that
\[ \cW(\mu_0,\mu_1)^2=\lim_n \cA(\mu^n,V^n)+I(m^n,C^n) . \]
By compactness of $\SO(d)$ we have up to taking a subsequence that $R_n\to R$ for some $R\in \SO(d)$. Arguing as in Proposition \ref{prop:opt-mC-ex}, we can assume that up to taking a further subsequence $(m^n,C^n)\to (m,C)$ uniformly and weakly in $H^1([0,1])$ for some $(m,C)\in \MC_R(\mu_0,\mu_1)$. As in the proof of part (1) we can show that up to a further subsequence $\mu^n_t\dd t\rightharpoonup \mu_t\dd t$ weakly and $V^n_t\mu^n_t\dd t\rightharpoonup V_t\mu_t\dd t$ in duality with $C_c(\R^d\times [0,1])$ for a pair $(\mu,V)\in \CE(R_\#\bar\mu_0,\bar\mu_1)$, as well as $\mu^n_t\rightharpoonup\mu_t$ for all $t\in[0,1]$. Moreover, we can assume that for this subsequence $\lim_n\cA(\mu^n,V^n)=\liminf_n\cA(\mu^n,V^n)$. Since $I(m^n,C^n)\geq \MOP(\mu_0,\mu_1)$ for all $n$, we deduce $\lim_n\cA(\mu^n,V^n)<\frac18$.
We infer from Proposition \ref{prop:minimizing_sequence} and the fact that $(\mu^n,V^n)$ is part of a minimizing sequence that $\mu_t\in\cP_{0,\Id}(\R^d)$ for all $t$. Finally, we conclude from the lower semicontinuity of $\cA(\cdot,\cdot)$ as in the proof of part (1) and the lower semicontinuity of $I(\cdot,\cdot)$ 
\[
\cA(\mu,V) + I(m,C) \leq \liminf_n \cA(\mu^n,V^n) + I(m^n,C^n) = \cW(\mu_0,\mu_1)^2\;.
\]
Hence, the tuple $\big(R,(\mu,V),(m,C)\big)$ constitutes a minimizer for 
\eqref{eq:part2-1}. The curve $\tilde\mu_t:=(A_t\cdot +m_t)_\#\mu_t$ with $A$ defined from $C$ by \eqref{eq:Adot} is a $\cW$-geodesic connecting $\mu_0,\mu_1$.
\end{proof}

\subsection{Existence under symmetry: Proof of Theorem \ref{thm:ex-opt-shape}}

This section gives the result on the (conditional) existence of geodesics in the covariance-constrained metric under a symmetry hypothesis, see Theorem~\ref{thm:ex-opt-shape}.
The proof below heavily relies on the interplay between the primal and the dual minimization problem and is thus very different from the proof of Theorem \ref{thm:existence-smalldist} above.

Our starting point is the following Lagrangian representation of geodesics. Let $\Gamma:=H^1([0,1];\R^d)$ be the space of curves $\gamma:[0,1]\to\R^d$ of finite energy (representing the ``mass particle trajectories''). Consider the space $\mathfrak M_+(\Gamma)$ of Borel measures on $\Gamma$. A measure $M\in\mathfrak M_+(\Gamma)$ is admissible for the moment constrained problem with two normalized marginals $\mu_0$ and $\mu_1$ if
\begin{align}
    \label{eq:admissibleM}
    \law_M(\gamma_0) = \mu_0,\quad    
    \law_M(\gamma_1) = \mu_1,\quad
    \xpt_M[\gamma_t]\equiv 0,\quad
    \xpt_M[\gamma_t\otimes \gamma_t] \equiv \Id \  \text{ for a.e. $t\in[0,1]$}.
\end{align}
Above, $\gamma_t$ for $t\in[0,1]$ is the random vector associated with the curves $\gamma\in\Gamma$. The first two conditions are the marginal constraints, the third and fourth are the moment constraints, fixing mean and covariance. The marginal constraints imply that $M$ is a probability measure on $\Gamma$, so $\xpt_M$ is a genuine expectation. Among the admissible measures $M$, we seek to minimize the integrated kinetic energy of the curves:
\begin{align}
    \label{eq:kinetic}
    \xpt_M\left[\int_0^1|\dot\gamma_t|^2\dd t\right]\longrightarrow\min.
\end{align}
In the context of unconstrained optimal transport, this formulation in terms of paths has been made rigorous in \cite{Lisini}, see also \cite[Sec.~8.2]{AGS}. 

We define for $\omega\in[0,\pi/2]^d$ the component-wise linear dilation $G^\omega:\R^d\to\R^d$ by
\begin{align}
    \label{eq:Gomega}
    [G^\omega x]_k = \left(\frac{\omega_k}{\sin\omega_k}\right)^{1/2}x_k,
\end{align}
with the understanding that $[G^\omega x]_k=x_k$ if $\omega_k=0$.
\begin{theorem}\label{thm:existence-symmetry}
    Let $\mu_0,\mu_1\in\cP_{0,\Id}(\R^d)$ be $d$-fold reflection symmetric (see Defintion~\ref{def:reflection-symmetry}), and assume that $\mu_0$ is absolutely continuous. Then there exists a minimizer $\widebar M$ for the constrained geodesic problem \eqref{eq:kinetic} subject to \eqref{eq:admissibleM}. Moreover, the minimizer is geometrically characterized as follows: there are parameters $\omega_1,\ldots,\omega_d\in[0,\pi/2]$, and there is a map $S:\R^d\to\R^d$ such that $\widebar M$-almost every trajectory $(\gamma(s))_{s\in[0,1]}$ emerging from $x:=\gamma(0)\in\R^d$ is given component-wise, for $k=1,\ldots,d$, by
    \begin{align}
        \label{eq:geodesicbydilation}
        \gamma_k(s) = \frac{\sin\omega_k(1-s)}{\sin\omega_k}x_k + \frac{\sin\omega_ks}{\sin\omega_k}S_k(x).
    \end{align}
    Finally, with $G^\omega$ from \eqref{eq:Gomega} with the same $\omega_1,\ldots,\omega_d$ as before, the constrained transport distance amounts to
    \begin{align}
        \label{eq:cW20Id}
        \cW_{0,\Id}(\mu_0,\mu_1)^2 = W_2\big(G^\omega_\#\mu_0,G^\omega_\#\mu_1\big)^2 - 2\sum_{k=1}^d\frac{\omega_k}{\sin\omega_k}\Big(1-\cos\omega_k-\frac12\omega_k\sin\omega_k\Big).
    \end{align}
\end{theorem}
\begin{remark}
    The assumption on the absolute continuity of $\mu_0$ has been made mainly to ensure uniqueness of the transport plan in a related unconstrained optimal transport problem, and this uniqueness is essential for continuity of the fixed point map that is defined at the very end of the proof. 
    With slight changes in the proof, the condition can be  relaxed to the assumption that for any choice of $\omega\in[0,\pi/2]^d$, there is a unique reflection symmetric optimal plan $\pi^\omega$ for the transport from $G^\omega_\#\mu_0$ to $G^\omega_\#\mu_1$.
    
    The symmetry hypothesis is far more essential for this proof, and also for the representation~\eqref{eq:geodesicbydilation} of ``mass transport by component-wise dilation''. Intuitively, from the $d^2$-many covariance-related constraints $\xpt[ \gamma_k(t)\gamma_\ell(t)]\equiv\delta_{k\ell}$, only the $d$ constraints for $k=\ell$ are active, and the other ones are automatically satisfied for symmetry reasons. In general, mass particle trajectories have a much more complicated form than \eqref{eq:geodesicbydilation}, as explained in Remark \ref{rmk:aftergeodesic} after the proof. 
\end{remark}
\begin{proof}
    The idea of the proof is to obtain the geodesic from a relatively explicit construction of a saddle point for a related functional. Specifically, we choose $\mathcal X:=\mathfrak M_+(\Gamma)$, the space of (non-negative) Borel measures on $\Gamma=H^1([0,1];\R^d)$, and $\mathcal Y:=L^\infty([0,1];\R^{d\times d}_\text{sym})\times L^\infty([0,1];\R^d)\times C_b(\R^d)\times C_b(\R^d)$, the set of quadruples $(\Lambda,m,\varphi_0,\varphi_1)$ of Lagrange multipliers. The functional $F:\mathcal X\times\mathcal Y\to\overline{\R}$ is given by
    \begin{align*}
        F(X,Y) &= \int_\Gamma\left[\int_0^1 \big(|\dot\gamma_t|^2 - m_t^\tran\gamma_t - \gamma_t^\tran \Lambda_t\gamma_t\big)\dd t - \varphi_0(\gamma_0) - \varphi_1(\gamma_1)\right]\dd M(\gamma) \\
        &\qquad +\int_0^1\tr[\Lambda_t] \dd t + \int\varphi_0\dd\mu_0 + \int\varphi_1\dd\mu_1 
    \end{align*}
    with $X=M\in \mathcal X$ and $Y=(\Lambda,m,\varphi_0,\varphi_1)\in \mathcal Y$.
    Below, we derive an explicit form of the functionals $I:\mathcal X\to\overline{\R}$ and $J:\mathcal Y\to\overline{\R}$, defined by
    \begin{align}
        I(\widebar Y):=\inf_X F(X,\widebar Y), \qquad J(\widebar X):=\sup_Y F(\widebar X,Y).
    \end{align}
    Then, we will use an ansatz in order to find a saddle point $(\widebar X,\widebar Y)\in\mathcal X\times\mathcal Y$ of $F$, that is 
    \begin{align}
        \label{eq:towardsduality}
        -\infty < J(\widebar X) \le I(\widebar Y) < \infty.    
    \end{align}
    Then $J(\widebar X)=F(\widebar X,\widebar Y)=I(\widebar Y)$. For later reference, we point out an immediate but essential consequence of \eqref{eq:towardsduality}:
    \begin{align}
        \label{eq:isminimizer}
        \text{$\widebar X$ is a global minimizer of $J$}.
    \end{align}
     Indeed, $J(\tilde X)<J(\widebar X)$ for some $\tilde X\in\mathcal X$ would imply in particular $F(\tilde X,\widebar Y)<F(\widebar X,\widebar Y)$, contradicting \eqref{eq:towardsduality}.
     
    The computation of $I(\widebar Y)$ is straight-forward from the definition of $F$. Since $\widebar M$ can give arbitrarily large weight to curves $\gamma$ for which the expression in the square bracket is negative, we obtain that $I(\widebar Y)=-\infty$ unless 
    \begin{align}
        \label{eq:phibelow}
        \widebar\varphi_0(\gamma_0)+\widebar\varphi_1(\gamma_1)\le\int_0^1 \big(|\dot\gamma_t|^2 - \widebar m_t^\tran\gamma_t - \gamma_t^\tran \widebar\Lambda_t\gamma_t\big)\dd t
        \quad\text{for all $\gamma\in\Gamma$},
    \end{align}
    in which case the infimum is attained e.g. at $\widebar M\equiv0$, with value
    \begin{align*}
        I(\widebar Y) = \int_0^1\tr[\widebar\Lambda_t] \dd t + \int\widebar\varphi_0\dd\mu_0 + \int\widebar\varphi_1\dd\mu_1.
    \end{align*}
    To compute $J(\widebar X)$ for a given \emph{probability measure} $\widebar X=\widebar M$ on $\Gamma$, rewrite $F$ in the following form:
    \begin{align*}
        F(\widebar X,Y) 
        &= \xpt_{\widebar M}\left[\int_0^1|\dot\gamma_t|^2\dd t\right] 
        - \int_0^1 m_t^\tran\xpt_{\widebar M}[\gamma_t]\dd t
        + \int_0^1\tr\left[\Lambda_t\big(\Id-\xpt_{\widebar M}[\gamma_t\otimes \gamma_t]\big)\right]\dd t \\
        &\qquad + \left(\int\varphi_0\dd\mu_0-\xpt_{\widebar M}[\varphi_0(\gamma_0)]\right) + \left(\int\varphi_1\dd\mu_1-\xpt_{\widebar M}[\varphi_1(\gamma_1)]\right).
    \end{align*}
    It follows that $J(\widebar X)=\infty$ unless $\widebar M$ satisfies \eqref{eq:admissibleM}, in which case  
    \begin{align*}
        J(\widebar X) = \xpt_{\widebar M}\left[\int_0^1|\dot\gamma_t|^2\dd t\right].
    \end{align*}
    Below, we produce $\widebar X$ and $\widebar Y$ such that the saddle point property \eqref{eq:towardsduality} holds. In view of the general fact \eqref{eq:isminimizer}, the respective $\widebar M=\widebar X$ is then a solution of the minimization problem \eqref{eq:kinetic} under the constraint~\eqref{eq:admissibleM}.
    
    From the above representations of $I$ and $J$, it follows that $(\widebar X,\widebar Y)$ is a saddle point if a probability measure $\widebar M$ and Lagrange multipliers $\widebar\Lambda$, $\widebar m$, $\widebar\varphi_0$, $\widebar\varphi_1$ are such that $\widebar M$ satisfies \eqref{eq:admissibleM}, that $\widebar\Lambda$, $\widebar m$, $\widebar\varphi_0$, $\widebar\varphi_1$ satisfy \eqref{eq:phibelow}, and that
    \begin{align}
        \label{eq:IJ0}
        \int_0^1\tr[\widebar\Lambda_t]\dd t + \int\varphi_0\dd\mu_0 + \int\varphi_1\dd\mu_1 \ge \xpt_{\widebar M}\left[\int_0^1|\dot\gamma_t|^2\dd t\right].
    \end{align}
    Using the constraints, \eqref{eq:IJ0} can equivalently be stated as
    \begin{align}
        \label{eq:IJ1}
        \xpt_{\widebar M}\big[\widebar\varphi_0(\gamma_0) + \widebar\varphi_1(\gamma_1)\big]
        \ge \xpt_{\widebar M}\left[\int_0^1\big(|\dot\gamma_t|^2-\widebar m_t^\tran\gamma_t-\gamma_t^\tran \widebar\Lambda_t\gamma_t\big)\dd t\right].
    \end{align}
    Hence, alternatively, for a saddle point, the following is sufficient:

    \paragraph{Saddle point condition I}
    \emph{The constraints \eqref{eq:admissibleM} and \eqref{eq:phibelow} are satisfied, and equality in \eqref{eq:phibelow} holds $\widebar M$-almost surely.}
    \smallskip
    
    We simplify this condition further by making the ansatz that $\widebar m\equiv0$, and that  $\widebar \Lambda$ is a diagonal matrix, with $t$-independent entries $\omega_1^2$ to $\omega_d^2$, where $\omega_k\in[0,\pi/2]$. For these choices, the right-hand side of \eqref{eq:phibelow} reduces to 
    \begin{align}\label{eq:gamma:min}
        \int_0^1 \big(|\dot\gamma|^2 - \widebar m^\tran\gamma - \gamma^\tran \widebar\Lambda\gamma\big)\dd t
        = \sum_{k=1}^d\int_0^1 \big(\dot\gamma_k(t)^2 - \omega_k^2\gamma_k(t)^2\big)\dd t.
    \end{align}
    The sum is minimized by a curve $\gamma$ if each of its terms is minimized by the respective component~$\gamma_k$. The respective minimizer for given $\gamma_k(0)$ and $\gamma_k(1)$ satisfies the Euler-Lagrange equation~$\ddot\gamma_k+\omega_k^2\gamma_k=0$. Since $\omega_k\in[0,\pi/2]$, the minimizing curve is thus given by $\gamma=\theta(\gamma(0),\gamma(1))$, where the continuous linear map $\theta^\omega:\R^d\times\R^d\to\Gamma$ is defined as \begin{align*}
        \big[\theta^\omega(x,y)\big]_k(t) = \frac{\sin\omega_k(1-t)}{\sin\omega_k}x_k + \frac{\sin\omega_kt}{\sin\omega_k}y_k
        \qquad\text{for $t\in[0,1]$}\,,
    \end{align*}
    for each component $k=1,\ldots,d$; if $\omega_k=0$, then
    \begin{align*}
        \big[\theta^\omega(x,y)\big]_k(t) = (1-t)x_k+ty_k
    \end{align*}
    instead. Hence, integrating by parts,
    \begin{equation}
        \label{eq:otbyparts}
        \begin{split}
            \sum_{k=1}^d\int_0^1 \big(\dot\gamma_k^2 - \omega_k^2\gamma_k^2\big)\dd t
            &=  \sum_{k=1}^d\left(\gamma_k(1)\dot\gamma_k(1) - \gamma_k(0)\dot\gamma_k(0)\right) \\
            &=  \sum_{k=1}^d\frac{\omega_k}{\sin\omega_k}\Big(\cos\omega_k[\gamma_k(0)^2+\gamma_k(1)^2]-2\gamma_k(0)\gamma_k(1)\Big) \\
            &= \big|G^\omega(\gamma(1)-\gamma(0))\big|^2 - \sum_{k=1}^d\frac{\omega_k(1-\cos\omega_k)}{\sin\omega_k}\left(\gamma_k(0)^2+\gamma_k(1)^2\right),
        \end{split}
    \end{equation}
    where $G^\omega:\R^d\to\R^d$ is the linear map defined in \eqref{eq:Gomega}. With that, we obtain from Condition I above another sufficient criterion for a saddle point for our particular choice of $\widebar\Lambda$:
    
    \paragraph{Saddle point condition II}
    \emph{$\widebar M$ satisfies \eqref{eq:admissibleM}, $\widebar M$ is concentrated on curves of the form $\theta^\omega(x,y)$, and}
    \begin{align}
        \label{eq:darkside}
        \widebar\varphi_0(x)+\widebar\varphi_1(y) \le \big|G^\omega(y-x)\big|^2 
        - \sum_{k=1}^d\frac{\omega_k(1-\cos\omega_k)}{\sin\omega_k}(x_k^2+y_k^2)
    \end{align}
    \emph{holds for all $(x,y)$, with equality for $(\gamma(0),\gamma(1))_\#\widebar M$-almost all $(x,y)$}. In \eqref{eq:darkside}, the $k$th quotient is interpreted as zero if $\omega_k=0$.
    \smallskip

    $\widebar M$ and $\widebar\varphi_0$, $\widebar\varphi_1$ satisfying condition II are now obtained by specializing our ansatz further. For an $\omega\in[0,\pi/2]^d$ determined below, consider the ``usual'' unconstrained $W_2$-optimal transport from $G^\omega_\#\mu_0$ to $G^\omega_\#\mu_1$. By the assumed absolute continuity of $\mu_0$ and $\mu_1$, there is an essentially unique optimal plan $\pi^\omega$, and it is of the form $\pi^\omega=(\Id,T^\omega)_\#\mu_0$ with a transport map $T^\omega:\R^d\to\R^d$. Further, there are associated Kantorovich potentials $\psi^\omega_0$ and $\psi^\omega_1$, with $T^\omega=\Id-\nabla\psi_0^\omega$. The Kantorovich potentials have the property that
    \begin{align}
        \label{eq:brightside}
        \psi^\omega_0(\xi)+\psi^\omega_1(\eta) \le |\xi-\eta|^2 \quad \text{for all $\xi,\eta\in\R^d$},
    \end{align}
    with equality for $\pi^\omega$-almost all $(\xi,\eta)$. Define accordingly
    \begin{align*}
        \widebar\varphi_0(x) := \psi_0(G^\omega x) - \sum_{k=1}^d\frac{\omega_k(1-\cos\omega_k)}{\sin\omega_k}x_k^2,
        \quad
        \widebar\varphi_1(y) := \psi_1(G^\omega y) - \sum_{k=1}^d\frac{\omega_k(1-\cos\omega_k)}{\sin\omega_k}y_k^2,
    \end{align*}
    again with the $k$th quotient interpreted as zero if $\omega_k=0$, and let 
    \begin{align}
    \label{eq:theM}
        \widebar M:=\theta^\omega_\#\big((G^\omega)^{-1},(G^\omega)^{-1}\big)_\#\pi^\omega. 
    \end{align}
    It easily follows from these definitions that inequality \eqref{eq:brightside} is equivalent to inequality \eqref{eq:darkside}, and that equality in \eqref{eq:brightside} for $\pi^\omega$-almost every $(\xi,\eta)$ is equivalent to equality in \eqref{eq:darkside} for $(\gamma(0),\gamma(1))_\#\widebar M$-almost every $(x,y)$.
    
    Finally, we need to define $\omega$ such that \eqref{eq:admissibleM} is satisfied. The marginal conditions follow immediately, 
    since by definition of $\widebar M$ in \eqref{eq:theM}, we have for every test function $f\in C_c(\R^d)$:
    \begin{align*}
        \MoveEqLeft\int_\Gamma f(\gamma(0))\dd\widebar M
        = \int_{\R^d\times\R^d} f(x)\dd\big((G^\omega)^{-1},(G^\omega)^{-1}\big)_\#\pi^\omega(x,y) \\
        &= \int_{\R^d\times\R^d}f\big((G^\omega)^{-1}(\xi)\big)\dd\pi^\omega(\xi,\eta) 
        = \int_{\R^d} f\big((G^\omega)^{-1}(\xi)\big)\dd G^\omega_\#\mu_0(\xi)
        = \int_{\R^d} f(x)\dd\mu_0(x),
    \end{align*}
    and similarly for the other marginal.
    For the mean constraint, simply note that 
    \begin{align*}
        \xpt_{\widebar M}[\gamma_t]
        = \int \big[\theta((G^\omega)^{-1}(\xi),(G^\omega)^{-1}(\eta))\big](t)\,\dd \pi^\omega(\xi,\eta)
        = \theta(\mean(\mu_0),\mean(\mu_1))(t)=0
    \end{align*}
    thanks to the linearity of $(x,y)\mapsto[\theta^\omega(x,y)](t)$ for any fixed $t\in[0,t]$, and the assumption that $\mean(\mu_0)=\mean(\mu_1)=0$.
    The covariance-constraint amounts to two conditions, on-diagonal and off-diagonal. The off-diagonal condition is
    \begin{align}
        \label{eq:varconstraint-off}
        0=\xpt_{\widebar M}[\gamma_k(t)\gamma_\ell(t)],
    \end{align}
    for all $k,\ell=1,\ldots,d$ with $k\neq\ell$. We show that this is a consequence of the $d$-fold reflection symmetry of $\mu_0$ and $\mu_1$: the symmetry of $\mu_0$ and $\mu_1$ is inherited by $G^\omega_\#\mu_0$ and $G^\omega_\#\mu_1$, and also by the optimal plan $\pi^\omega$ in the sense that $(\sigma_k,\sigma_k)_\#\pi^\omega=\pi^\omega$. For a proof of the latter, consider $\tilde\pi:=(\sigma_k,\sigma_k)_\#\pi^\omega$ for some $k\in\{1,\ldots,d\}$. By $\sigma_k$-invariance of the marginals, $\tilde\pi$ is a transport from $G^\omega_\#\mu_0$ to $G^\omega_\#\mu_1$ as well, and the associated transport cost is the same as for $\pi^\omega$, since $|\sigma_k(x)-\sigma_k(y)|^2=|x-y|^2$ for arbitrary $x,y\in\R^d$. By uniqueness of the optimal plan, $\tilde\pi=\pi^\omega$. Now the symmetry of $\pi^\omega$ implies (recall that $k\neq\ell$):
    \begin{align*}
        \int \xi_k\eta_\ell\dd\pi^\omega(\xi,\eta) 
        = \int [\sigma_k\xi]_k[\sigma_k\eta]_\ell\dd\pi^\omega(\xi,\eta)
        = \int (-\xi_k)\eta_\ell\dd\pi^\omega(\xi,\eta) ,
    \end{align*}
    and therefore, the integral vanishes, implying \eqref{eq:varconstraint-off}.

    It remains to prove that for an appropriate choice of $\omega\in[0,\pi/2]^d$, the on-diagonal condition 
    \begin{align}
        \label{eq:varconstraint-on}
       1=\xpt_{\widebar M}[\gamma_k(t)^2]
    \end{align}
    is satisfied for each $k=1,\ldots,d$. We start by observing that the aforementioned symmetry of $\pi^\eps$ also implies that
    \begin{align}
        \label{eq:justonecone}
      \int \xi_k\eta_k\dd\pi^\omega(\xi,\eta) = 2^d \int_{\R_{\ge0}^{2d}} \xi_k\eta_k\dd\pi^\omega(\xi,\eta) .
    \end{align}
    Indeed, the optimal plan $\pi^\omega$ is cyclically monotone for the squared euclidean distance, and is in particular monotone in the graphical sense, i.e., if $(\xi,\eta)$ and $(\xi',\eta')$ both are in the support of $\pi^\omega$, then $(\eta'-\eta)^\tran(\xi'-\xi)\ge0$. Choosing $\xi'=\sigma_k(\xi)$, $\eta'=\sigma_k(\eta)$, then, by symmetry, $(\xi',\eta')$ is in $\pi^\omega$'s support if and only if $(\xi,\eta)$ is, and the monotonicity inequality amounts to $\xi_k\eta_k\ge0$ in that case. So $\pi^\omega$ is only supported at points $(\xi,\eta)$ where $\xi_k$ and $\eta_k$ have the same sign, for all $k=1,\ldots,d$. There are precisely $2^d$ such cones in $\R^{2d}$, and by symmetry of $\pi^\omega$, the restriction of $\pi^\omega$ to any of these cones is obtained from $\pi^\omega$'s restriction to $\R_{\ge0}^{2d}$ via push-forward by at most $d$ appropriate reflections of the form $(\sigma_k,\sigma_k):\R^{2d}\to\R^{2d}$. Formula \eqref{eq:justonecone} is a simple consequence of that.

    The next step in our proof of \eqref{eq:varconstraint-on} is to show that
    \begin{align}
        \label{eq:fixedpointbound}
        0\le \xpt_{\widebar M}[\gamma_k(0)\gamma_k(1)] \le 1.
    \end{align}
    By definition of $\widebar M$ from $\pi^\omega$, we have that
    \begin{align}
        \label{eq:gamma01}
        \xpt_{\widebar M}[\gamma_k(0)\gamma_k(1)] 
        = \int \big[(G^\omega)^{-1}\xi\big]_k\big[(G^\omega)^{-1}\eta\big]_k\dd\pi^\omega(\xi,\eta)
        = \frac{\sin\omega_k}{\omega_k}\int \xi_k\eta_k\dd\pi^\omega(\xi,\eta).
    \end{align}
    The right-hand side is clearly non-negative thanks to \eqref{eq:justonecone}. An estimate from above follows by means of
    \begin{align*}
        \left|\int \xi_k\eta_k\dd\pi^\omega(\xi,\eta)\right|
        &\le\left(\int\xi_k^2\dd\pi^\omega(\xi,\eta)\right)^{1/2}
        \left(\int\eta_k^2\dd\pi^\omega(\xi,\eta)\right)^{1/2} \\
        &=\left(\int (G^\omega x)_k^2\dd\mu_0(x)\right)^{1/2}\left(\int (G^\omega y)_k^2\dd\mu_1(y)\right)^{1/2}
        =\frac{\omega_k}{\sin\omega_k},
    \end{align*}    
    showing \eqref{eq:fixedpointbound}.

    Next, we make the essential observation that \eqref{eq:varconstraint-on} is implied by
    \begin{align}
        \label{eq:towardsfixedpoint}
        \xpt_{\widebar M}[\gamma_k(0)\gamma_k(1)]=\cos\omega_k
    \end{align} 
    for $k=1,\ldots,d$. Indeed,
    \begin{align*}
        \xpt_{\widebar M}[\gamma_k(t)^2]
        &= \xpt_{\widebar M}[\theta_k(\gamma(0),\gamma(1))^2](t)
        = \xpt_{\widebar M}\left[\left(\frac{\sin\omega_k(1-t)}{\sin\omega_k}\gamma_k(0)+\frac{\sin\omega_kt}{\sin\omega_k}\gamma_k(1)\right)^2\right] \\
        &= \frac{\sin^2\omega_k(1-t)}{\sin^2\omega_k}\xpt_{\widebar M}[\gamma_k(0)^2] + \frac{\sin^2\omega_kt}{\sin^2\omega_k}\xpt_{\widebar M}[\gamma_k(1)^2] \\
        &\qquad  + \frac{2\sin\omega_k(1-t)\sin\omega_kt}{\sin^2\omega_k}\xpt_{\widebar M}[\gamma_k(0)\gamma_k(1)] \\
        &= \frac1{\sin^2\omega_k}\big[\sin^2\omega_k(1-t)+\sin^2\omega_kt+2\cos\omega_k\sin\omega_k(1-t)\sin\omega_kt\big] = 1,
    \end{align*}
    where the last equality follows by elementary trigonometric identities as follows:
    \begin{align*}
        &\sin^2\omega_k(1-t)+\sin^2\omega_kt+2\cos\omega_k\sin\omega_k(1-t)\sin\omega_kt \\
        &= \bigl(\sin\omega_k\cos\omega_kt\!-\!\cos\omega_k\sin\omega_kt\bigr)^2\!\! +\sin^2\omega_kt + \! 2\cos\omega_k\bigl(\sin\omega_k\cos\omega_kt\!-\!\cos\omega_k\sin\omega_kt\bigr)\sin\omega_kt \\
        &= \sin^2\omega_k\cos^2\omega_kt + \cos^2\omega_k\sin^2\omega_kt + \sin^2\omega_kt - 2\cos^2\omega_k\sin^2\omega_kt \\
        & = (1-\cos^2\omega_k)\sin^2\omega_kt + \sin^2\omega_k\cos^2\omega_kt \\
        & = \sin^2\omega_k(\sin^2\omega_kt+\cos^2\omega_kt)
        = \sin^2\omega_k.
    \end{align*}
    In conclusion, the on-diagonal condition \eqref{eq:varconstraint-on} is satisfied if $\omega$ is chosen such that \eqref{eq:towardsfixedpoint} holds. The existence of an appropriate $\omega$ is now obtained by a fixed point argument: define a fixed point operator $R:[0,\pi/2]^d\to[0,\pi/2]^d$ by
    \begin{align*}
        R_k(\omega) = \arccos\xpt_{\widebar M}[\gamma_k(0)\gamma_k(1)].
    \end{align*}
    Well-definedness follows from \eqref{eq:fixedpointbound}. Concerning continuity: if $\omega^n\to\omega^*$ converges in $[0,\pi/2]^d$ as $n\to\infty$, then $G^{\omega_n}_\#\mu_0$ and $G^{\omega_n}_\#\mu_1$ converge to their respective limits $G^{\omega_*}_\#\mu_0$ and $G^{\omega_*}_\#\mu_1$ in $\cP_2(\R^d)$. By \cite[Proposition 7.1.3]{AGS}, the respective optimal transport plans $\pi^{\omega_n}$ from $G^{\omega_n}_\#\mu_0$ to $G^{\omega_n}_\#\mu_1$ are narrowly compact, and any limit point is an optimal plan for the transport from $G^{\omega_*}_\#\mu_0$ to $G^{\omega_*}_\#\mu_1$. Again thanks to absolute continuity, that limit $\pi^{\omega_*}$ is unique. Moreover, by uniform integrability of the second moments of $G^{\omega_n}_\#\mu_0$ and $G^{\omega_n}_\#\mu_1$, also $\pi^{\omega_n}$'s second moment is uniformly integrable, and so
    \[ \int\xi_k\eta_k\dd\pi^{\omega_n}(\xi,\eta) \to \int\xi_k\eta_k\dd\pi^{\omega_*}(\xi,\eta). \]
    Recalling the definition of $R$ above and the representation \eqref{eq:gamma01} of $\xpt_{\widebar M}[\gamma_k(0)\gamma_k(1)]$, it immediately follows that $R(\omega_n)\to R(\omega_*)$.
    
    In conclusion, $R$ possesses at least one fixed point $\omega$, thanks to Brouwer's fixed point theorem. 
    
    Now, the map $S$ in \eqref{eq:geodesicbydilation} is obtained as follows: $\pi^\omega$ connects $\xi$ to $\eta=T^\omega(\xi)$; according to~\eqref{eq:theM}, the measure $\widebar M$ connects $x=(G^\omega)^{-1}(\xi)$ to $y=(G^\omega)^{-1}(T^\omega(\xi))$, and so, $S = (G^\omega)^{-1}\circ T^\omega\circ G^\omega$.
    
    Finally, to compute
    \begin{align*}
        \cW_{0,\Id}(\mu_0,\mu_1)^2 = \xpt_{\widebar M}\left[\int_\Gamma|\dot\gamma(t)|^2\dd t\right],
    \end{align*}
    we use the representation obtained in \eqref{eq:otbyparts}:
    \begin{align*}
        \xpt_{\widebar M}\left[\int_\Gamma|\dot\gamma(t)|^2\dd t\right]
        &= \int_\Gamma \big|G^\omega(\gamma(1))-G^\omega(\gamma(0))\big|^2\dd\widebar M(\gamma) + \sum_{k=1}^d\omega_k^2\int_0^1\xpt[\gamma_k(t)^2]\dd t \\
        &\qquad - \sum_{k=1}^d\frac{\omega_k(1-\cos\omega_k)}{\sin\omega_k}\big(\xpt_{\widebar M}[\gamma_k(0)^2]+\xpt_{\widebar M}[\gamma_k(1)^2]\big) \\
        &= \int_{\R^d\times\R^d}|y-x|^2\dd\pi^\omega(x,y) + \sum_{k=1}^d\omega_k^2 - 2\sum_{k=1}^d \frac{\omega_k(1-\cos\omega_k)}{\sin\omega_k} \\
        &= W_2(G^\omega_\#\mu_0,G^\omega_\#\mu_1)^2 - 2 \sum_{k=1}^d \frac{\omega_k}{\sin\omega_k}\Big(1-\cos\omega_k-\frac12\omega_k\sin\omega_k\Big),
    \end{align*}
    which is \eqref{eq:cW20Id}.
\end{proof}
\begin{remark}
    \label{rmk:aftergeodesic}
    In principle, a similar approach seems feasible even without symmetry assumptions. For that, however, one would need to consider more general (in particular $t$-dependent) symmetric positive semi-definite matrices $\widebar\Lambda$. Notice that the sufficient \emph{Condition I} for a saddle point in the proof above is general, i.e., does not depend on our specific choice of $\widebar\Lambda$. In order to simplify that condition further, one needs to understand the minimizer of the integral on the right-hand side of \eqref{eq:phibelow} for given $\gamma(0)$ and $\gamma(1)$. Assuming that minimizers exist (for that, $\widebar\Lambda$ needs to be ``sufficiently small'', in analogy to $\omega\in[0,\pi/2]^d$ in~\eqref{eq:gamma:min}), they are given by solutions to the Euler-Lagrange equation 
    \begin{align}
        \label{eq:yetanotherEL}
        \ddot\gamma(t)+\widebar\Lambda_t\gamma(t)=0.
    \end{align}
    Solution curves can be written in the form $\gamma(t)=A_t\gamma_0+B_t\gamma_1$, with time-dependent matrices $A_t$ and $B_t$ such that $A_0=B_1=\Id$, $A_1=B_0=0$. An integration by parts in \eqref{eq:phibelow} yields the equivalent condition
    \begin{align*}
        \widebar\varphi_0(x) + \widebar\varphi_1(y) \le y^\tran \dot B_1 y - x^\tran \dot A_0x + y^\tran (\dot B_0^\tran +\dot A_1)x \quad \text{for all $x,y\in\R^d$}.
    \end{align*}
    The goal is to find appropriate $\widebar\varphi_0$, $\widebar\varphi_1$ and a plan $\tilde\pi$ with marginals $\mu_0$ and $\mu_1$ such that equality holds for $\tilde\pi$-a.e. $(x,y)$. Since the right-hand side above is a quadratic form in $(x,y)$, this problem is again related to an optimal transport with respect to the classical Wasserstein distance.
    
    The main obstacle to carrying out the generalization is the covariance-constraint. A necessary condition that restricts the shape of $\widebar\Lambda$ is easily derived: introducing the (a priori $t$-dependent) matrix $V_t=\xpt_{\widebar M}[\dot\gamma_t\otimes\gamma_t]$, and using \eqref{eq:yetanotherEL} above, it follows by successive differentiation in $t$ (recalling symmetry $\widebar\Lambda^\tran =\widebar\Lambda$) that
    \begin{align*}
        \Id = \xpt_{\widebar M}[\gamma_t\otimes \gamma_t]\ \Rightarrow\ 
        0 = V+V^\tran \ \Rightarrow\ 
        0 = \xpt[\dot\gamma\otimes\dot\gamma] - \widebar\Lambda\ \Rightarrow\ 
        0 = V\widebar\Lambda + \widebar\Lambda V^\tran  + \dot{\widebar\Lambda}.
    \end{align*}
    Therefore, it follows that $\dot V = -\widebar\Lambda + \widebar\Lambda=0$, and so the skew-symmetric matrix $V$ is actually independent of $t$. Further, since $\dot{\widebar\Lambda} = \widebar\Lambda V-V\widebar\Lambda$, it follows that $\widebar\Lambda_t=e^{tV^\tran }\Lambda_0e^{tV}$, with a $t$-independent symmetric positive semi-definite $\Lambda_0$. But this is far from sufficient: the parameters $V$ and $\Lambda_0$ still need to be determined such that the consistency relations $V=\xpt[\dot\gamma(0)\otimes\gamma(0)]$ and $\Lambda_0=\xpt[\dot\gamma(0)\otimes \dot\gamma(0)]$ hold --- this corresponds to the fixed point problem in the proof above. 
    The consistency relation for the general case even within the restrictive class of $\widebar\Lambda$'s is left open, mainly because the map from $(V,\Lambda_0)$ to $(\dot A_0,\dot B_0)$ is still only poorly understood.
\end{remark}

\subsection{Simple examples}\label{sec:geo:Exs}
For a $d$-fold reflection symmetric measure $\mu\in\cP(\R^d)$ that does not give mass to the $d$ coordinate hyperplanes, define its \emph{symmetry generator} as the probability measure $\tilde\mu\in\cP(\R_{>0}^d)$ obtained by restriction of $\mu$ to $\R_{>0}^d$ and normalization by the factor $2^d$. Clearly, $\mu$ and $\tilde\mu$ are in one-to-one correspondence, and we call $\mu$ \emph{symmetry generated} by $\tilde\mu$. Notice that a $d$-fold reflection symmetric~$\mu$ belongs to $\cP_{0,\Id}(\R^d)$ if and only if its generator $\tilde\mu$ satisfies 
\begin{align}
    \label{eq:generatorcond}
    \int_{\R_{>0}^d}x_k^2\dd\tilde\mu(x)=1 \quad \text{for all $k=1,\ldots,d$}.
\end{align}
Moreover, if $\pi$ is an optimal plan for the (unconstrained) transport between two $d$-fold reflection symmetric measures $\mu_0$ and $\mu_1$, then $\pi$'s normalized restriction $\tilde\pi$ to $\R_{>0}^d\times\R_{>0}^d$ is an optimal plan for the transport between the respective generators $\tilde\mu_0$ and $\tilde\mu_1$, and conversely, an optimal $\tilde\pi$ generates an optimal $\pi$ via symmetry.
\begin{example}
    \label{xmp:alltoone}   
    Let $\mu_0,\mu_1\in\cP_{0,\Id}(\R^d)$ be symmetry generated by some absolutely continuous $\tilde\mu\in\cP(\R_{>0}^d)$, and by the Dirac measure $\delta_p$ at $p=(1,1,\ldots,1)\in\R^d$, respectively. Define, for $k=1,\ldots,d$,
    \begin{align}
        \label{eq:xmpomega}
        \rho_k := \int_{\R_{>0}^d} x_k\dd\tilde\mu(x)\in(0,1),
        \quad\text{and}\quad
        \omega_k:= \arccos\rho_k\in(0,\pi/2).
    \end{align}
    Then we obtain
    \begin{align}
        \label{eq:xmpW}
        \cW_{0,\Id}(\mu_0,\mu_1)^2 = |\omega|^2.
    \end{align}
    To see this, it suffices to observe that for an arbitrary $\omega\in[0,\pi/2]^d$, the unique optimal plan from $G^\omega_\#\tilde\mu$ to $G^\omega_\#\delta_p$ is given by $\tilde\pi^\omega=G^\omega_\#\tilde\mu\otimes G^\omega_\#\delta_p$, and so
    \begin{align*}
        \int_{\R^d\times\R^d}\xi_k\eta_k\dd\pi^\omega(\xi,\eta) = \frac{\omega_k}{\sin\omega_k}\int_{\R^d}x_k\dd\tilde\mu(x) = \frac{\omega_k}{\sin\omega_k}\rho_k.
    \end{align*}
    Hence, according to \eqref{eq:gamma01},
    \begin{align*}
        \xpt_{\widebar M}[\gamma_k(1)\gamma_k(0)] = \rho_k,
    \end{align*}
    independently of $\omega$. Thus the solution to the fixed point condition \eqref{eq:towardsfixedpoint} is indeed given by $\omega$ from \eqref{eq:xmpomega}. For the unconstrained Wasserstein distance, we obtain
    \begin{align*}
        W_2(G^\omega_\#\mu_0,G^\omega_\#\mu_1)^2 
        &= \int_{\R_{>0}^d\times\R_{>0}^d} |\xi-\eta|^2\dd\pi^\omega(\xi,\eta)
        = \int_{\R_{>0}^d}|G^\omega(x-p)|^2\dd\tilde\mu(x)\\
        &= \sum_{k=1}^d \frac{\omega_k}{\sin\omega_k}\int_{\R_{>0}^d}|x-p|^2\dd\tilde\mu(x) 
        = 2\sum_{k=1}^d\frac{\omega_k}{\sin\omega_k}(1-\rho_k).
    \end{align*}
    This is the first term in \eqref{eq:cW20Id}. For the second term, observe that
    \begin{align*}
        2\frac{\omega_k}{\sin\omega_k}\Big(1-\cos\omega_k-\frac12\omega_k\sin\omega_k\Big)
        = 2\frac{\omega_k}{\sin\omega_k}(1-\rho_k)-\omega_k^2,
    \end{align*}
    so that the difference in \eqref{eq:cW20Id} amounts to \eqref{eq:xmpW}.
    
    Next, we show in this simple example that the geodesics for constrained and for unconstrained optimal transport have a different shape. According to \eqref{eq:geodesicbydilation}, the mass transport from any point $x\in\R_{>0}^d$ in the support of $\tilde\mu$ to $p$ is along a curve $\gamma$ of the form
    \begin{align*}
        \gamma_k(s) = \frac{\sin\omega_ks+x_k\sin\omega(1-s)}{\sin\omega_k}.
    \end{align*}
    For a point $x$ in general position (specifically, $x_k\neq1$ for $k=1,\ldots, d$), the trace of this curve is a straight line segment if and only if $\omega_1=\omega_2=\cdots=\omega_d$, i.e., if all $\rho_k$ in \eqref{eq:xmpomega} are identical. Note that even in this special case, the motion of the mass is not at uniform speed as it is in the unconstrained transport.
    
    We shall now compare the particle traces above with the traces of particles for a properly re-scaled unconstrained optimal transport. The classical Wasserstein geodesic from $\tilde\mu$ to $\delta_p$ is given by the transport map $T^t:\R_{>0}^d\to\R_{>0}^d$ with $T^t(x)=(1-t)x+tp$. We apply a scaling along the $d$ coordinate directions to ensure that, at any $t\in[0,1]$,
    \begin{align*}
        \int_{\R_{>0}^d} x_k^2 \dd T^t_\#\tilde\mu(x) = 1.
    \end{align*}
    Recalling \eqref{eq:generatorcond} and also the notation from \eqref{eq:xmpomega}, we obtain
    \begin{align*}
        1 = \int_{\R_{>0}^d} T^t_k(x)^2\dd\mu(x) 
        = \int_{\R_{>0}^d}\bigl((1-t)^2x_k^2 + t^2 + 2t(1-t)x_k\bigr)\dd\mu(x)
        = 1 - 2t(1-t)(1-\rho_k).
    \end{align*}
    Accordingly, the rescaled transport map $\hat T^t$ is given by
    \begin{align*}
        \hat T_k^t(x) = \frac{T_k^t(x)}{\sqrt{1-2t(1-t)(1-\rho_k)}} 
        = \frac{(1-t)x_k+t}{\sqrt{1-2t(1-t)(1-\rho_k)}}.
    \end{align*}
    We show that the particle traces of $\gamma$ and $\hat T$ --- for the same initial point $x$ --- do not agree in general. More precisely, we show that the terminal velocities $u$ and $v$ of $\gamma$ and $\hat T$ (upon arrival at $p$), respectively, point into different directions. Indeed, for $k=1,\ldots,d$,
    \begin{align*}
        u_k = \frac{\dd}{\dd t}\bigg|_{t=1} T_k^t(x) &= (1-x_k) - (1-q_k) = \cos\omega_k-x_k, \\
        v_k = \frac{\dd}{\dd s}\bigg|_{s=1} \gamma_k(s) &= \frac{\omega_k}{\sin\omega_k}(\cos\omega_k-x_k),
    \end{align*}
    and so $v_k=\frac{\omega_k}{\sin\omega_k}u_k$. For points $x$ in general position ($x_k\neq1$ for all $k=1,\ldots,d$), the only case in which $u$ and $v$ are parallel is the aforementioned special situation that $\omega_1=\omega_k=\cdots=\omega_d$, when all curves $\gamma$ are actually (non-linearly parametrized) straight line segments. 
\end{example}
\begin{example}
    A special case of the previous example in dimension $d=2$ is the transport from the uniform measure $\tilde\mu$ on an axis-parallel rectangle to the Dirac measure at $p=(1,1)$. The rectangle runs from $m_1-\delta_1$ to $m_1+\delta_1$ horizontally, and from $m_2-\delta_2$ to $m_2+\delta_2$ vertically. The conditions on $m_k$ and $\delta_k$ are most easily formulated in terms of $\omega_1,\omega_2\in[0,\pi/6]$: condition \eqref{eq:generatorcond} is satisfied if
    \begin{align*}
        q_k = m_k = \cos\omega_k,\quad \delta_k = \sqrt3 \sin\omega_k \qquad (k=1,2);
    \end{align*}
    the restriction $\omega_k\le\pi/6$ reflects that the rectangle must lies in the first quadrant. In Figure \ref{fig:rectangle}, we compare different particle trajectories from the corners of the rectangle to $p$: unconstrained optimal transport, re-scaled unconstrained optimal transport, and constrained optimal transport. As expected from the computations at the end of Example \ref{xmp:alltoone} above, the curves for the re-scaled unconstrained and for the constrained optimal transport are extremely close.
    \begin{figure}
		\begin{minipage}[b]{\linewidth}
		 \centering
        \includegraphics[width=0.45\textwidth]{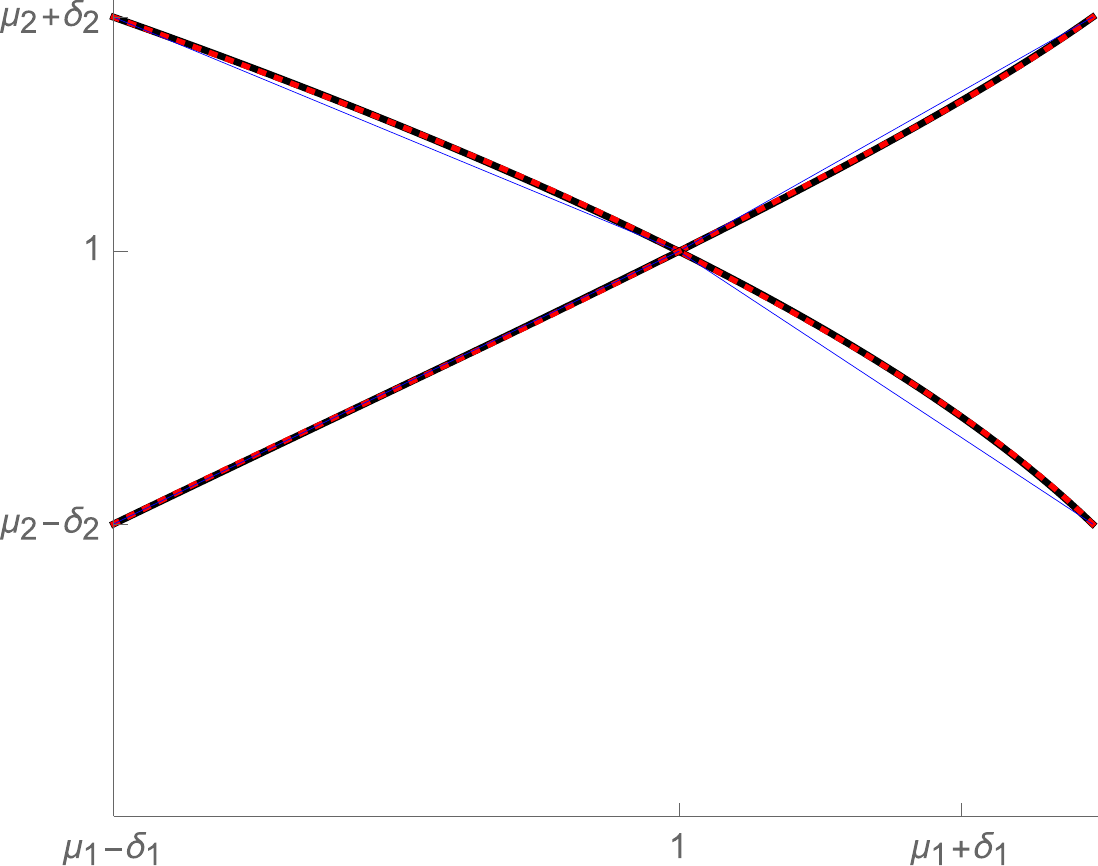}\hfill \includegraphics[width=0.45\textwidth]{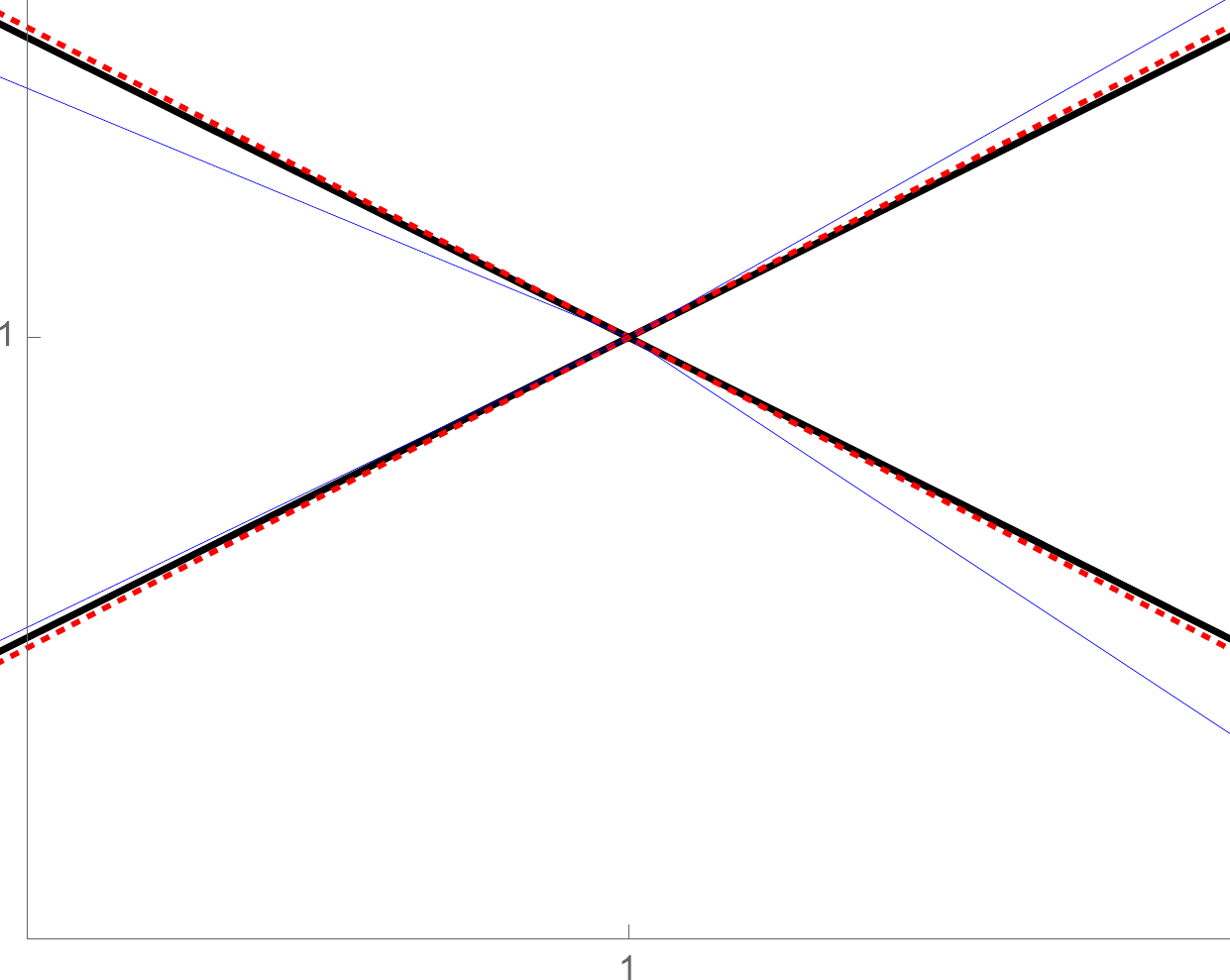}
  		\end{minipage}
        \caption{Comparison of particle trajectories for the optimal transport from the rectangle to $(1,1)$: unconstrained transport (thin blue line), re-scaled unconstrained transport (dotted red line), constrained transport (solid black line). Left: the respective four particle trajectories, emerging from the corners of the rectangle. Right: close-up near the terminal point $(1,1)$, for $x,y\in[0.99,1.01]$.}
        \label{fig:rectangle}
    \end{figure}
\end{example}
\begin{example}
    We consider the constrained transport between two measures $\mu_0$ and $\mu_1$ in the plane $\R^2$ that are symmetry generated by the following probability measures $\tilde\mu_0$ and $\tilde\mu_1$ on $\R_{>0}^2$: $\tilde\mu_0$ is the uniform measure on the disk $D$ with center $c=(m,m)$ and radius $\rho$, and $\tilde\mu_1$ consists of two point measures of mass one half at $p^+$ and $p^-$, respectively. To guarantee \eqref{eq:generatorcond}, the parameters are subject to the following conditions: 
    \begin{align*}
        (p^+_1)^2+(p^-_1)^2=2,\quad (p^+_2)^2+(p^-_2)^2=2,\quad \frac14\rho^2+m^2=1.
    \end{align*}
    Further, $\rho<m$ is obviously needed, which amounts to $\rho<\sqrt{4/5}=0.894\ldots$
    We write the positions of $p^+$ and $p^-$ in the form
    \begin{align*}
        p^\pm = s^\alpha\pm \frac12e^\alpha, \quad e^\alpha = \begin{pmatrix}\cos\alpha \\ \sin\alpha\end{pmatrix},
        \quad s^\alpha_1=\frac12\sqrt{3+\sin^2\alpha},\quad s^\alpha_2=\frac12\sqrt{3+\cos^2\alpha},
    \end{align*}
    i.e., the connecting line from $p^-$ to $p^+$ is of unit length and has an angle $\alpha\in(0,\pi/2)$ with respect to the horizontal axis. The solution to the unconstrained optimal transport problem from $\tilde\mu_0$ to $\tilde\mu_1$ is to cut the circle through its center $c$ along a straight line orthogonal to $e^\alpha$, and then to transport the mass in the ``upper-right half'' and the ``lower-left half'', respectively, to $p^+$ and $p^-$. \emph{We shall see below that the solution to the constrained problem is the same, but with a cut along a line of modified angle $\alpha'$ instead of $\alpha$.} See Figure \ref{fig:anglesketch} below for an illustration.
    
    \begin{figure}
        \centering
        \includegraphics[width=0.4 \linewidth]{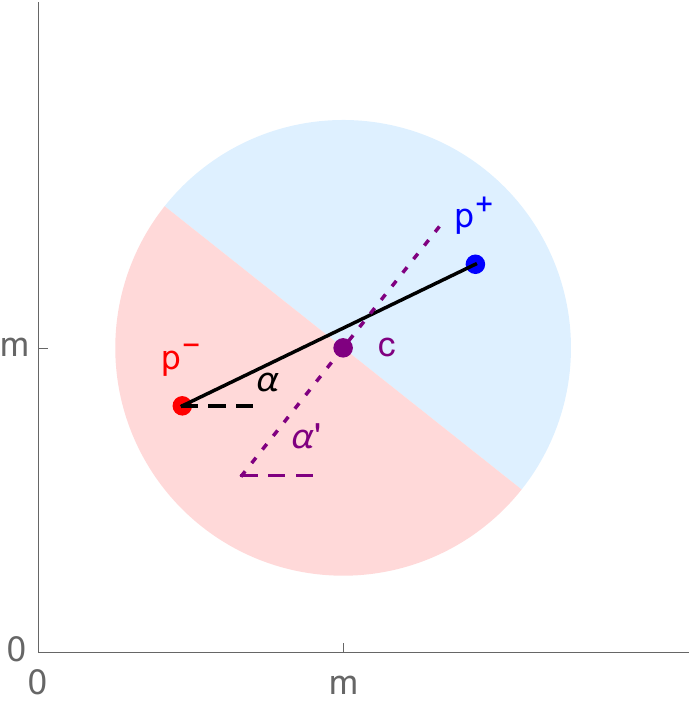}
        \caption{The mass on the upper-right (blue) and lower-left (red) halves of the circle is transported to the point masses at $p^+$ and at $p^-$, respectively. The angle $\alpha'$ of the (purple) line orthogonal to the division line is in general \emph{not} identical to the angle $\alpha$ of the (black) line connecting $p^-$ to $p^+$.}
        \label{fig:anglesketch}
    \end{figure}
    
    For any given pair $(\omega_1,\omega_2)$ with $0\le\omega_k\le\pi/2$, the optimal plan $\tilde\pi^\omega$ for the (unconstrained) transport problem between $G^\omega_\#\tilde\mu_0$ and $G^\omega_\#\tilde\mu_1$ is easily obtained from elementary geometric considerations: observe that $G^\omega_\#\tilde\mu_0$ is the uniform measure on the ellipse $G^\omega(D)$, while $G^\omega_\#\tilde\mu_1$ consists of the two point measures of mass one half each at $G^\omega(p^+)$ and $G^\omega(p^-)$; the connecting line between these points is parallel to $G^\omega e^\alpha$. The two parts of $G^\omega(D)$ that are transported to the points $p^+$ and $p^-$, respectively, are obtained by cutting the ellipse into two halves of equal area along a line \emph{orthogonal to} $G^\omega e^\alpha$; the cut is thus parallel to $JG^\omega e^\alpha$ with the left rotation $J=\begin{pmatrix}0&-1\\1&0\end{pmatrix}$. By symmetry of the ellipse, the cut passes through the center of $G^\omega(D)$. The pulled-back plan $\tilde\sigma^\omega:=((G^\omega)^{-1},(G^\omega)^{-1})_\#\tilde\pi^\omega$ therefore assigns two halves of the disc $D$ to the points $p^+$ and $p^-$, respectively, with the division line parallel to $(G^\omega)^{-1}JG^\omega e^\alpha$ and through $D$'s center $c$. A vector orthogonal to that division line is given by
    \begin{align*}
        v^\omega = J^\tran(G^\omega)^{-1}JG^\omega e^\alpha 
        = \begin{pmatrix}
            (G^\omega_1/G^\omega_2)\cos\alpha \\ (G^\omega_2/G^\omega_1)\sin\alpha
        \end{pmatrix},
    \end{align*}
    and its normalization $V^\omega=v^\omega/|v^\omega|$ of unit length is given by
    \begin{align*}
        V^\omega_1 = \big(1+(G^\omega_2/G^\omega_1)^4\tan^2\alpha\big)^{-1/2}, 
        \quad
        V^\omega_2 = \big(1+(G^\omega_1/G^\omega_2)^4\cot^2\alpha\big)^{-1/2}.
    \end{align*}
    For later reference, note that the slope of $V^\omega$ with respect to the horizontal axis is 
    \begin{align*}
        A^\omega = v^\omega_2/v^\omega_1 = (G^\omega_2/G^\omega_1)^2\tan\alpha,
    \end{align*}
    and that for $\omega^*$ being the fixed point of 
    \begin{align}
        \label{eq:yetanotherfixedpoint}
        \cos\omega_k = \xpt[\gamma_k(0)\gamma_k(1)] \qquad (k=1,2),
    \end{align}
    we have that $\tan\alpha'=A^{\omega^*}$. Below, we shall derive from \eqref{eq:yetanotherfixedpoint} directly a fixed point equation that has $A^{\omega^*}$ as solution.

    The center of mass of the two half discs are located, respectively, at
    \begin{align*}
        c\pm\beta\rho V^\omega \quad\text{with}\quad \beta=\frac4{3\pi}.
    \end{align*}
    We thus obtain --- recalling that $e^\alpha_1=\cos\alpha$ and $e^\alpha_2=\sin\alpha$ ---
    \begin{align*}
        \xpt[\gamma_k(0)\gamma_k(1)]
        &= \int_{\R_{>0}^2\times\R_{>0}^2} x_ky_k\dd\tilde\sigma^\omega(x,y) \\
        &= \frac{p^+_k}4\big(m+\beta\rho V^\omega_k\big)         
        + \frac{p^-_k}4\big(m-\beta\rho V^\omega_k\big)
        = \frac{ms^\alpha_k}2 + \frac{\beta\rho e^\alpha_k}4 V^\omega_k.
    \end{align*}
    Recall that $G^\omega_k=\sqrt{\omega_k/\sin\omega_k}$, and define accordingly the function 
    \begin{align*}
        f(z) = \frac{\arccos z}{\sin\arccos z} = \frac{\arccos z}{\sqrt{1-z^2}}.
    \end{align*}
    Then the fixed point equations \eqref{eq:yetanotherfixedpoint} imply that $A=A^{\omega^*}$ is a solution of
    \begin{align*}
        A = \tan\alpha\,\frac{f\left(\frac{ms^\alpha_2}2 + \frac{\beta\rho e^\alpha_2}4 (1+A^{-2})^{-1/2}\right)}{f\left(\frac{ms^\alpha_1}2 + \frac{\beta\rho e^\alpha_1}4 (1+A^2)^{-1/2}\right)}. 
    \end{align*}
    For any fixed $\alpha\in(0,\pi/2)$, the expression on the right-hand side above is monotonically decreasing with respect to $A>0$, with Lipschitz constant less than one, hence the (unique) fixed point is easily approximated by simple iteration. A numerical evaluation of the difference $\alpha'-\alpha$ for $\alpha\in[0,\pi/2]$ and different radii $\rho$ is given in Figure \ref{fig:modifiedangle}. The expected antisymmetry about $\alpha=\pi/2$ is clearly visible, as well as the coincidence $\alpha'=\alpha$ in the three special positions $\alpha=0$, $\alpha=\pi/2$ and $\alpha=\pi/4$. The left plot indicates a qualitative change in the behaviour of $\alpha'-\alpha$ in dependence of $\rho$; for larger radii $\rho>0.61$, the angle $\alpha'$ lags behind $\alpha$, and for smaller radii $\rho<0.60$, the angle $\alpha'$ is ahead of $\alpha$. The transition behaviour for $0.60<\rho<0.61$ is complicated, as is indicated for $\rho=0.6015$ in the right plot of Figure~\ref{fig:modifiedangle}. 
    \begin{figure}
        \centering
        \includegraphics[width=0.45\linewidth]{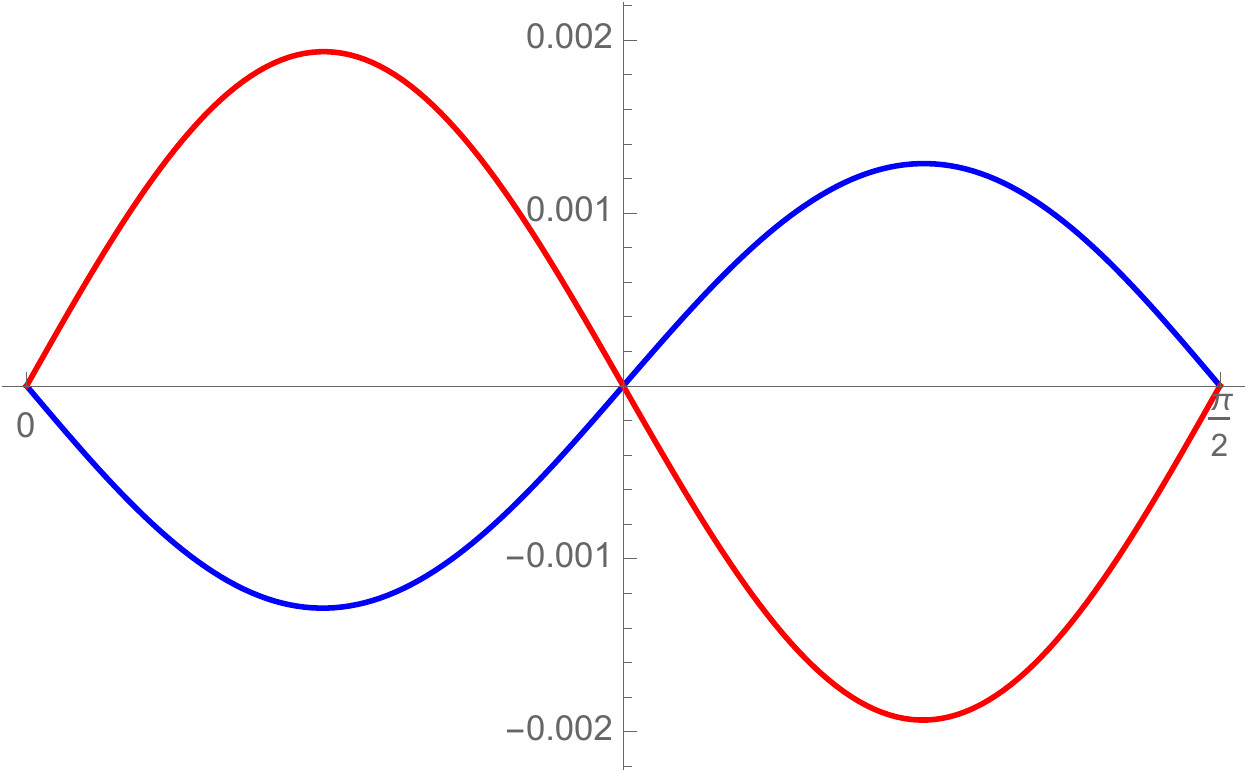}
        \includegraphics[width=0.45\linewidth]{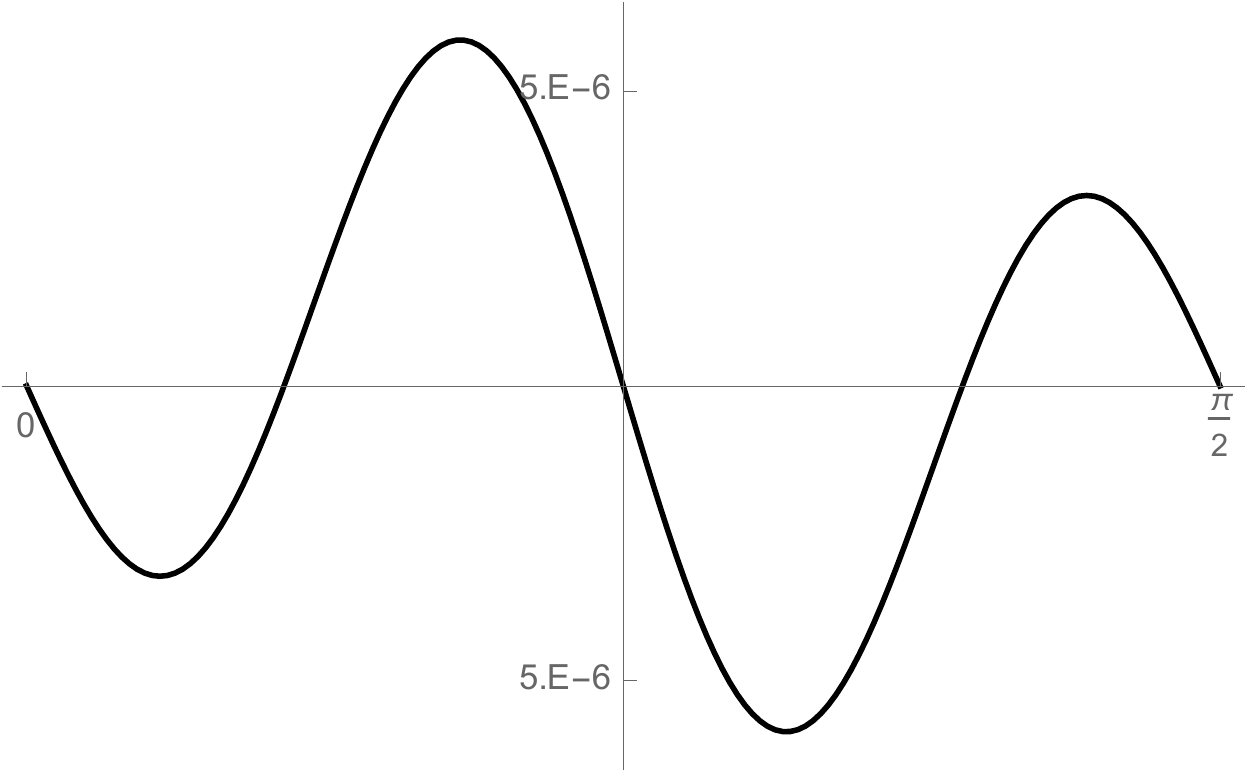}
        \caption{The plots show the difference of the angle $\alpha$ (inclination of the connecting line $p_-$ to $p_+$) to the respective angle $\alpha'$ (of the line orthogonal to the cut of the circle). Left: results for $\rho=0.5$ (blue) and $\rho=0.75$ (red). Right: results for $\rho=0.6015$. Further explanations in the text.}\label{fig:modifiedangle}
    \end{figure}
\end{example}

\section{Gradient flows and convergence to equilibrium}\label{sec:cv}

\subsection{Connections to the Fokker-Planck equation}\label{sec:FP}

In this section, we consider the covariance-modulated Fokker-Planck equation~\eqref{eq:GF-FP-cov-quad-intro}, which we for the sake of convenience repeat here
\begin{equation}
	\label{eq:GF-FP-cov}
	\partial_t\rho_t = \nabla\cdot\bra*{\C(\rho_t)\bra*{\nabla \rho_t + \rho_t \nabla H}}\;,
\end{equation}
with $H(x)= \frac12 |x-x_0|_B^2$ for fixed mean $x_0\in\R^d$ and covariance $B\in\R^{d\times d}$.
Then the mean $m_t=\mean(\rho_t)$ and covariance $C_t=\C(\rho_t)$ evolve along \eqref{eq:GF-FP-cov} according to
\begin{equation}\label{eq:moments-GF}
	\frac{\dd}{\dd t}m_t  = -C_t B^{-1}(m_t-x_0)\;,\qquad\text{and}\qquad 
	\frac{\dd}{\dd t}C_t = 2C_t -2 C_t B^{-1}C_t\;.  
\end{equation}
Since, the equation of $C_t$ is decoupled from $m_t$, we can obtain a solution by integrating first $C_t$ and then $m_t$. However, the system possesses further intrinsic quantities with exponential decay in time.
\begin{lemma}\label{lem:moment-decay}
	Let $A_t$ be an adapted square-root satisfying~\eqref{eq:Adot} for the solution $(C_t)_{t\geq 0}$ of~\eqref{eq:moments-GF}, i.e. $C_t=A_tA_t^\tran$ and $\dot A_t = \tfrac{1}{2} \dot C_t A_t^{-\tran}$, then for all $t\ge0$ the decay estimates hold
	\begin{align}
		\label{eq:m-decay}
		A_t^{-1}(m_t-x_0) &= e^{-t}A_0^{-1}(m_0-x_0), \\
		\label{eq:sol:Ct}
		C_t^{-1} &= ( 1- e^{-2t}) B^{-1}  + e^{-2t} C_0^{-1} .
	\end{align}
\end{lemma}
\begin{proof}
	Using \eqref{eq:moments-GF} and the evolution equation~\eqref{eq:Adot} for $A_t$, we get that
	\begin{align*}
		\pderiv{}{t} (A_t^{-1}) = - \frac{1}{2} A_t^{-1} \dot C_t A_t^{-\tran} A_t^{-1} = -A_t^{-1} \bra*{A_tA_t^\tran - A_tA_t^\tran B A_t A_t^\tran} A_t^{-\tran} A_t^{-1}=  -A_t^{-1} + A_t^\tran  B^{-1}
	\end{align*}
	and so
	\begin{align*}
		\pderiv{}{t} \bra*{ A_t^{-1}(m_t-x_0)} &= -A_t^{-1}(m_t-x_0) + A_t^\tran B^{-1}(m_t-x_0) -A_t^{-1}C_tB^{-1}(m_t-x_0),
	\end{align*}    
	which proves \eqref{eq:m-decay}, after using $C_t=A_tA_t^\tran $ at all times $t\ge0$.
	
	For showing~\eqref{eq:sol:Ct}, we observe using once more~\eqref{eq:moments-GF}
	\begin{align*}
		\pderiv{}{t} \bra*{C^{-1}_t-B^{-1}} = - C_t^{-1} \dot C_t C_t^{-1} = - 2 C_t^{-1} \bra*{ C_t - C_t B^{-1}C_t} C_t^{-1}  = - 2 (C_t^{-1} - B^{-1})
	\end{align*}
	and we immediately obtain the explicit solution \eqref{eq:sol:Ct}.
\end{proof}
We now normalize the flow along the generalized Fokker-Planck equation~\eqref{eq:GF-FP-cov} according to Definition~\ref{def:normalization}, by setting for $t\geq 0$
\begin{align}\label{e:def:eta}
	\eta_t:=(T_{m_t,A_t})_\#\rho_t\;,\quad\text{where}\quad T_{m,A}x= A^{-1}(x-m)\;,
\end{align}
with $m_t=\mean(\rho_t)$ and choosing $A_t$ such that
\begin{align}\label{e:def:A}
	A_tA_t^\tran =C_t\;,\quad\text{and}\quad A^{-1}_t \dot A_t \text{ is symmetric}\;.
\end{align}
This is achieved by picking $A_t$ a solution to \eqref{eq:Adot}.

We now claim that the normalized flow satisfies the Ornstein-Uhlenbeck flow, that is the Fokker-Planck equation with standard quadratic potential.
\begin{lemma}\label{lem:FP-normalized}
	If $\rho_t$ solves the generalized Fokker-Planck equation~\eqref{eq:GF-FP-cov} with quadratic potential $H:\R^d\to\R$ given by \eqref{eq:potential}, namely
	\begin{equation}
		\label{eq:GF-FP-cov-quad}
		\partial_t\rho_t = \nabla\cdot\bra*{\C(\rho_t)\bra*{\nabla \rho_t + \rho B^{-1}(x-x_0)}}\;,
	\end{equation}
	then the normalized solution $\eta_t=(T_{m_t,A_t})_\#\rho_t$ with $m_t=\mean(\rho_t)$ and $A_t$ solving for $C_t=\C(\rho_t)$ 
    \begin{equation}
        	A_t^{-1}\dot A_t = \frac12 A_t^{-1}\dot C_t A_t^{-\tran}= \Id-A_t^\tran B^{-1}A_t \;,\label{eq:AdotB}
    \end{equation}
    satisfies the Ornstein-Uhlenbeck evolution
	\begin{equation}\label{eq:FP-normalized-quad}
		\partial_t \eta_t = \Delta \eta_t + \nabla \cdot (x\eta_t )\,.
	\end{equation}
\end{lemma}
\begin{proof}
	For convenience of notation, we write $\abs{A}=\det A$ and by the definition of the push-forward, the explicit relation
	\begin{equation}\label{e:eta:rho}
		\eta_t(x)= \rho_t(A_tx+m_t)\cdot |A_t|\;.
	\end{equation}
	Thus we get, writing $\dot\eta, \dot m, \dot A, \dot C$, etc.~for the derivatives w.r.t.~$t$ and neglecting the explicit time-dependence
	\begin{align*}
		\dot\eta(x) &= \abs{A}\, \pra*{\dot\rho(Ax+m)+\ip{\nabla\rho(Ax+m)}{(\dot Ax+\dot m)}+\rho(Ax+m)\tfrac{\dd}{\dd t}\log| A|}\;.
	\end{align*}
	We further note the identities
	\begin{align*}
		\nabla \eta(x)&=A^\tran \nabla \rho(Ax+m)\, \abs{A}\;,\\
		\Delta\eta(x) &= (AA^\tran )_{ij}\partial_{ij}\rho(Ax+m)\, \abs{A}= (\nabla\cdot C\nabla\rho)(Ax+m)\, \abs{A}\;,\\
		\nabla\cdot\pra*{\eta(x)\nabla \bra*{\tfrac{1}{2}\abs{x}^2}}&=\ip{\nabla \eta(x)}{x} +d \eta(x)\;. \\
	\end{align*}
	Moreover, the evolution~\eqref{eq:GF-FP-cov-quad} becomes
	\begin{align*}
		\abs{A}\; \dot \rho (Ax+m)&=\abs{A}\,(\nabla\cdot C\nabla \rho)(Ax+m) + \abs{A}\,\big(\nabla\cdot(\rho C \nabla H\big)(Ax+m)\\
		&= \abs{A}\,(\nabla\cdot C\nabla \rho)(Ax+m) +\abs{A}\,\ip{\nabla\rho(Ax+m)}{ CB^{-1}(Ax+m-x_0)}\\
		&\quad+ \abs{A}\,\rho(Ax+m)\tr[CB^{-1}]\\
		&=\Delta\eta(x)+ \skp*{\nabla \eta(x),A^\tran B^{-1}(Ax+m-x_0)} + \eta(x)\tr[CB^{-1}]\;.
	\end{align*}
	By \eqref{eq:moments-GF} and \eqref{eq:AdotB} we obtain
	\begin{align*}A
		A^{-1}\dot m & = -A^\tran B^{-1}(m-x_0)\;,\\
		\frac{\dd}{\dd t}\log\abs{A}&= \tr[A^{-1}\dot A]=d-\tr[A^\tran B^{-1}A]=d-\tr[CB^{-1}]\;.
	\end{align*}
	The above equations imply
	\begin{align*}
		\abs{A}\skp[\big]{\nabla\rho(Ax+m),\dot Ax+\dot m} &= \skp*{\nabla\eta(x),A^{-1}\dot A x+A^{-1}\dot m}\\
		&= \skp{\nabla\eta(x),x-A^\tran B^{-1}(Ax+m-x_0)}.
	\end{align*}
	Thus we finally obtain
	\begin{equation*}
		\dot\eta(x)=\Delta\eta(x)+\ip{\nabla\eta(x)}{x} +d\eta(x) =\Delta\eta(x) +\nabla\cdot(\eta(x) x) \;. \qedhere
	\end{equation*}
\end{proof}
This result allows us to translate all the well-known properties for the classical Fokker-Planck equation (see for instance~\cite{BakryEmery,MarkowichVillani,AMTU,Bakry2014}) into the framework of covariance-modulated flows, such as exponential
convergence of the relative entropy and Fisher information
\begin{align}
	\label{eq:OU-decay}
	\cE(\eta_t | \eta_\infty)  \leq  e^{-2\lambda t} \cE(\eta_0 | \eta_\infty) 
	\qquad\text{and}\qquad
	\cI(\eta_t | \eta_\infty)  \leq  e^{-2\lambda t} \cI(\eta_0 | \eta_\infty) \,,
\end{align}
where $\eta_\infty= \normal_{0,\Id}$.
But also the evolution variational inequality (EVI) for the \emph{usual}
$L^2$-Wasserstein distance $W_2$~\cite{AGS}, implying exponential contraction in $W_2$ of rate $1$
\begin{equation}\label{eq:OU-W2-decay}
	W_2(\eta_t,\eta_\infty) \leq e^{-t} W_2(\eta_0,\normal_{0,\Id}).
\end{equation}

\subsection{Convergence in entropy}\label{sec:Ent:cv}

In this section we prove the convergence of the evolution~\eqref{eq:GF-FP-cov} to equilibrium in relative entropy and Fisher information. For this we will apply the observation that the normalized density $\eta$ evolves along a standard Ornstein-Uhlenbeck evolution as shown in Lemma~\ref{lem:FP-normalized} and make use of the fact that the mean and covariance are given by the ODE system \eqref{eq:moments-GF}.
Recall the definition \eqref{eq:defnormal} of $\normal_{m,C}$ a Gaussian with mean $m$ and covariance $C$.
With this, we decompose solutions $(\rho_t)_{t\geq}$ to the generalized Fokker-Planck equation \eqref{eq:GF-FP-cov} into a Gaussian approximations $\normal_{m_t,C_t}$ of $\rho_t$ where the moments $(m_t, C_t)$ satisfy \eqref{eq:moments-GF} and the remainder $(\eta_t)_{t\geq 0}$.
Based on this decomposition, we have the following splitting for the relative entropy and Fisher information.
\begin{lemma}[Entropic decomposition]\label{lem:ent-split}
	Given $\rho\in\cP_2(\R^d)$, let $\normal_{m,C}$ be the Gaussian with the same mean $m:=\mean(\rho)$ and the same covariance $C=\C(\rho)$ as $\rho$, and let $\eta=(T_{m,A})_\#\rho$ be the normalization of $\rho$ according to Definition \ref{def:normalization}.
	Then, for any $x_0\in\R^d$ and $B\in\S_{\succ 0}^d$, the splitting formula holds
	\begin{align}
		\label{eq:splitE}
		\cE(\rho\,|\,\normal_{x_0,B}) &= \cE(\eta\,|\,\normal_{0,\Id}) + \cE(\normal_{m,C}\,|\,\normal_{x_0,B})\:, \\
		\label{eq:splitI}
		\covI(\rho\,|\,\normal_{x_0,B}) &= \cI(\eta\,|\,\normal_{0,\Id}) + \covI(\normal_{m,C}\,|\,\normal_{x_0,B})\:.
	\end{align}
	Moreover, the latter terms in~\eqref{eq:splitE} and~\eqref{eq:splitI}, respectively, have the explicit representations
	\begin{align}
		\label{eq:residualE}
		\cE(\normal_{m,C}\,|\,\normal_{x_0,B}) &= -\frac12\bra[\Big]{\log\det(B^{-1}C)+\tr[\Id-B^{-1}C]-\abs[\big]{B^{-\frac{1}{2}}(m-x_0)}^2} \;, \\
		\label{eq:residualI}
		\covI(\normal_{m,C}\,|\,\normal_{x_0,B}) &= \norm[\big]{\Id-B^{-1}C}_\HS^2+\abs[\big]{C^{\frac{1}{2}}B^{-1}(m-x_0)}^2.
	\end{align}
\end{lemma}
\begin{proof}
	For reference throughout the proof, note that in view of \eqref{eq:defnormal}, that
	\begin{equation}
		\label{eq:log}
		\begin{split}
			\log\bra*{\frac{\normal_{m,C}}{\normal_{x_0,B}}} 
			&= - \frac12 \skp*{x-m,(C^{-1}-B^{-1})(x-m)} + \skp*{x-m, B^{-1}(m-x_0)} \\
			&\qquad + \frac12\skp*{m-x_0, B^{-1}(m-x_0)} -\frac12\log\det (B^{-1}C) .
		\end{split}
	\end{equation}
	We split the expression for the relative entropy as follows
	\begin{align*}
		\cE(\rho\,|\,\normal_{x_0,B}) 
		= \int \log\left(\frac{\rho}{\normal_{x_0,B}}\right)\dd\rho
		= \int \log\left(\frac{\rho}{\normal_{m,C}}\right)\dd\rho
		+ \int \log\left(\frac{\normal_{m,C}}{\normal_{x_0,B}}\right)\dd\rho.
	\end{align*}
	We recall that $T_{m,A}(x)=A^{-1}(x-m)$, where $A\in\R^{d\times d}$ is an invertible matrix such that $ AA^\tran  = C$.
	On the one hand, since $\eta=(T_{m,A})_\#\rho$ and $\normal_{0,\Id}=(T_{m,A})_\#\normal_{m,C}$, we have
	\begin{align}
		\label{eq:etarho}
		\log\left(\frac{\rho}{\normal_{m,C}}\right) = \log\left(\frac{\eta}{\normal_{0,\Id}}\right)\circ T_{m,A}\,,
	\end{align}    
	and thus a change of variables inside the first integral yields
	\begin{align*}
		\int \log\left(\frac{\rho}{\normal_{m,C}}\right)\dd\rho
		= \int \log\left(\left(\frac{\eta}{\normal_{0,\Id}}\right)\circ T_{m,A}\right)\dd\rho
		= \int \log\left(\frac{\eta}{\normal_{0,\Id}}\right)\dd\eta
		= \cE(\eta\,|\,\normal_{0,\Id})\,.
	\end{align*}
	And on the other hand, the expression $\log(\normal_{m,C}/\normal_{x_0,B})$ is a second order polynomial in $x$, see \eqref{eq:log} above, and since $\rho$ and $\normal_{m,C}$ have the same first moment and covariance,
	we conclude that
	\begin{align}
		\label{eq:rhonormal}
		\int \log\left(\frac{\normal_{m,C}}{\normal_{x_0,B}}\right)\dd\rho 
		= \int \log\left(\frac{\normal_{m,C}}{\normal_{x_0,B}}\right)\dd\normal_{m,C}
		= \cE(\normal_{m,C}\,|\,\normal_{x_0,B})\,.
	\end{align}
	This shows \eqref{eq:splitE}.
	For the information functional, we proceed in a similar way
	\begin{align*}
		\covI(\rho\,|\,\normal_{x_0,B})
		&= \int \left|A^\tran \nabla\log\left(\frac{\rho}{\normal_{x_0,B}}\right)\right|^2 \dd\rho \\ 
		&= \int \left|A^\tran \nabla\log\left(\frac{\rho}{\normal_{m,C}}\right)+A^\tran \nabla\log\left(\frac{\normal_{m,C}}{\normal_{x_0,B}}\right)\right|^2 \dd\rho \\ 
		& = \int \left|A^\tran \nabla\log\left(\frac{\rho}{\normal_{m,C}}\right)\right|^2 \dd\rho
		+ \int \left|A^\tran \nabla\log\left(\frac{\normal_{m,C}}{\normal_{x_0,B}}\right)\right|^2 \dd\rho\\
		&\qquad + 2\int \skp*{ \nabla\log\left(\frac{\rho}{\normal_{m,C}}\right)\,,\, C\nabla\log\left(\frac{\normal_{m,C}}{\normal_{x_0,B}}\right)} \,\dd\rho
		=: I_1+I_2+I_3\,.
	\end{align*}
	From \eqref{eq:etarho}, we conclude via differentiation that
	\begin{align*}
		\nabla\log\left(\frac{\eta}{\normal_{0,\Id}}\right) 
		=  \left(A^\tran\,\nabla\log\left(\frac{\rho}{\normal_{m,C}}\right)\right)\circ T_{m,A}^{-1}\,.
	\end{align*}
	Recalling from Definition \ref{def:normalization} that $A^{-1}CA^{-\tran}=\Id$, a change of variables in $I_1$ thus leads to
	\begin{align*}
		I_1 = \int \left|A^\tran \nabla\log\left(\frac{\rho}{\normal_{m,C}}\right)\right|^2\circ T_{m,A}^{-1} \dd\eta
		= \int \left|\nabla\log\left(\frac{\eta}{\normal_{0,\Id}}\right)\right|^2 \dd\eta
		= \covI(\eta\,|\,\normal_{0,\Id})\,,
	\end{align*}
	where we have used that $\C(\normal_{0,\Id})=\Id$.
	Next, we observe that $|A^\tran \nabla\log(\normal_{m,C}/\normal_{x_0,B})|^2$ is a quadratic polynomial in $x$, see \eqref{eq:log}. Since $\rho$ and $\normal_{m,C}$ have identical mean and covariance, we conclude --- similarly as in \eqref{eq:rhonormal} above --- that $I_2=\covI(\normal_{m,C}\,|\,\normal_{x_0,B})$.
	Finally, we split $I_3$ as follows
	\begin{align*}
		I_3 = 2\int \skp{\nabla\rho, v} \dd x - 2 \int \skp*{\nabla\log\normal_{m,C}, v}\dd\rho\,, 
	\end{align*}
	where we introduced --- see \eqref{eq:log} above ---
	\begin{align*}
		v := C\,\nabla\log\left(\frac{\normal_{m,C}}{\normal_{x_0,B}}\right) 
		=  -\big(\Id-CB^{-1}\big)(x-m) + CB^{-1}(m-x_0)\,.
	\end{align*}
	From integrating by parts, we get
	\begin{align*}
		\int \skp{\nabla\rho , v} \dd x 
		= -\int \skp{\rho \nabla, v}\dd x
		= \tr[\Id-CB^{-1}]\,.
	\end{align*}
	This is matched with
	\begin{align*}
		\int \skp{\nabla\log\normal_{m,C}, v}\dd\rho 
		= -\int \skp{C^{-1}(x-m), v}\dd\rho 
		= \tr[\Id-CB^{-1}]\,,
	\end{align*}
	which yields $I_3=0$.
	In conclusion,
	\begin{align*}
		I_1+I_2+I_3 = \covI(\eta\,|\,\normal_{0,\Id}) + \covI(\normal_{m,C}\,|\,\normal_{x_0,B}) + 0,
	\end{align*}
	which is \eqref{eq:splitI}.
	Finally, for the proof of \eqref{eq:residualE}\&\eqref{eq:residualI}, we first integrate the expression in \eqref{eq:log} against $\normal_{m,C}$ --- using that $\mean(\normal_{m,C})=m$ and $\C(\normal_{m,C})=C$ --- which yields \eqref{eq:residualE}. Next, we differentiate the expression in \eqref{eq:log} to obtain
	\begin{align*}
		\left|C^{1/2}\nabla\log\left(\frac{\normal_{m,C}}{\normal_{x_0,B}}\right)\right|^2 
		&= |C^{-1/2}(\Id-B^{-1}C)(x-m)|^2 + |C^{1/2}B^{-1}(m-x_0)|^2 \\ 
		& \qquad -\skp*{2B^{-1}(m-x_0),(\Id-C B^{-1})(x-m)}\,.
	\end{align*}
	As before, integration against $\normal_{m,C}$ provides the desired expression in \eqref{eq:residualI}.
\end{proof}

In other words, by writing for short $\normal_t:=\normal_{m_t,C_t}$ and $\normal_*:=\rho_\infty =\normal_{x_0,B}$, 
the solutions~$\rho_t$ to~\eqref{eq:GF-FP-cov} satisfy
\begin{align*}
	\cE(\rho_t \,|\, \rho_\infty) &= \cE(\eta_t \,|\, \normal_{0,\Id} ) + \cE(\normal_t \,|\, \normal_*)\, ,\\
	\covI(\rho_t \,|\, \rho_\infty) &= \cI(\eta_t \,|\, \normal_{0,\Id} ) + \covI(\normal_t \,|\, \normal_*)\, ,
\end{align*}	
with $\eta_t$ given in~\eqref{e:def:eta}. 
The respective first terms $\cE(\eta_t\,|\,\normal_{0,\Id})$ and $\cI(\eta_t\,|\,\normal_{0,\Id})$ in this decomposition 
can simply be controlled by the entropy--dissipation bounds \eqref{eq:OU-decay} for the Ornstein-Uhlenbeck semigroup. Therefore, it remains to obtain a rate for the relaxation of $\cE(\normal_t \,|\, \normal_*)$ and $\covI(\normal_t \,|\, \normal_*)$.
We recall, that $\|Q\|_2$ denotes the largest singular value of $Q$, i.e., the largest eigenvalue of~$(Q^\tran Q)^{\frac{1}{2}}$.
\begin{lemma}[Relaxation for Gaussians]\label{lem:Gauss-decay}
Any $(m_t,C_t)$ solving the moment equations~\eqref{eq:moments-GF} satisfies
\begin{align}
	\label{eq:ent-decay-cov}
	\cE(\normal_t \,|\, \normal_*) 
	&\leq \max\set*{ 1,  \|B^{\frac{1}{2}}C_0^{-1} B^{\frac{1}{2}}\|_2}\max\set*{ 1,\|B^{-\frac12}C_0B^{-\frac12}\|_2}  e^{-2t}  \cE(\normal_0 | \normal_*) \, ,\\
	\label{eq:fish-decay-cov}
	\covI(\normal_t \,|\, \normal_*) 
	&\leq \max\set*{ 1,  \|B^{\frac{1}{2}}C_0^{-1} B^{\frac{1}{2}}\|_2^2}\max\set*{ 1,\|B^{-\frac12}C_0B^{-\frac12}\|_2^2} e^{-2t}  \covI(\normal_0 \,|\, \normal_*) \, ,
\end{align}
If $m_0=x_0$, then	
\begin{align}
	\label{eq:ent-decay-fixedmean}
	\cE(\normal_t \,|\, \normal_*) 
	&\leq  \max\set*{ 1,  \|B^{\frac{1}{2}}C_0^{-1} B^{\frac{1}{2}}\|_2} e^{-2t}  \cE(\normal_0 \,|\, \normal_*) \, , \\
	\label{eq:fish-decay-fixedmean}
	\covI(\normal_t \,|\, \normal_*) 
	&\leq  \max\set*{ 1,  \|B^{\frac{1}{2}}C_0^{-1} B^{\frac{1}{2}}\|_2^2} e^{-4t}  \covI(\normal_0 \,|\, \normal_*) \,.
\end{align}
\end{lemma}
\begin{remark}
There is seemingly a discrepancy between the exponential rates of decay of two and four in \eqref{eq:ent-decay-fixedmean} and in \eqref{eq:fish-decay-fixedmean}, respectively: since $\covI$ is the dissipation of $\cE$, one would expect the rates to agree. Indeed, a more detailed analysis of $\cE$'s asymptotics --- see formulas \eqref{eq:sigma-sol}\&\eqref{eq:losing4} in the proof below --- reveals exponential decay at rate four eventually (i.e., once $\cE$ is sufficiently small, thanks to the quadratic behaviour of $s\mapsto s-1-\log s$ near $s=1$), but only decay at rate two initially (i.e., for large values of $\cE$, due to the linear growth of $s\mapsto s-1-\log s$ for large $s$). That is, the exponential rate in \eqref{eq:ent-decay-fixedmean} could be improved from two to four at the price of enlarging the constant on the right-hand side.
\end{remark}
\begin{proof}
We prove the estimates for the entropy first. From \eqref{eq:residualE}, we obtain
\begin{equation}\label{eq:estimate0}
	\cE(\normal_t \,|\, \normal_*) = S(t) + \frac{1}{2} |m_t-x_0|_B^2 \ \ {with}\ \  S(t):=\frac{1}{2} \pra*{ \tr \bra[\big]{B^{-\frac12} C_t B^{-\frac12}} \! - d - \log \det\bra[\big]{B^{-\frac12} C_tB^{-\frac12}}}  
\end{equation} 
Concerning the norm of $m_t-x_0$, note that thanks to \eqref{eq:m-decay} above, we have $\abs*{m_t - x_0}^2_{C_t} \leq e^{-2t} \abs*{m_0 - x_0}^2_{C_0}$. Using further that, for any vector $v$,
\begin{align*}
	|v|_B^2 = \skp[\big]{v , B^{-1}v} = \skp[\big]{C_t^{-\frac{1}{2}}v, (C_t^{\frac{1}{2}}B^{-1}C_t^{\frac{1}{2}})(C_t^{-\frac{1}{2}}v} \le \norm[\big]{C_t^{\frac{1}{2}}B^{-1}C_t^{\frac{1}{2}}}_2 \; |v|_{C_t}^2, 
\end{align*}
and similarly that $|v|_{C_t}^2\le\|B^{\frac{1}{2}}C_t^{-1}B^{\frac{1}{2}}\|_2|v|_B^2$,
we conclude that
\begin{align*}
	\abs*{m_t - x_0}^2_{B} 
	&\leq \norm[\big]{C_t^{\frac{1}{2}}B^{-1}C_t^{\frac{1}{2}}}_2 \abs*{m_t - x_0}^2_{C_t} \leq  \norm[\big]{C_t^\frac{1}{2}B^{-1}C_t^\frac{1}{2}}_2 e^{-2t} \abs*{m_0 - x_0}^2_{C_0}\\
	&\leq \norm[\big]{C_t^\frac{1}{2}B^{-1}C_t^\frac{1}{2}}_2 \norm[\big]{B^{\frac{1}{2}}C_0^{-1} B^{\frac{1}{2}}}_2 e^{-2t} \abs*{m_0 - x_0}^2_{B}.
\end{align*}
Next, we note that a similarity transformation with $B^{\frac{1}{2}} C_t^{-\frac{1}{2}}$ shows that the eigenvalues of $C_t^\frac{1}{2}B^{-1}C_t^\frac{1}{2}$ and $B^{-\frac{1}{2}}C_{t}B^{-\frac12}$ agree, and by positivity and symmetry also the singular values. In combination with the solution formula~\eqref{eq:sol:Ct}, it follows that
\begin{align*}
	\norm[\big]{C_t^\frac{1}{2}B^{-1}C_t^\frac{1}{2}}_2
	&=\norm[\big]{B^{-\frac{1}{2}}C_{t}B^{-\frac12}}_2
	=\norm[\big]{\big(B^\frac{1}{2}C_t^{-1}B^\frac{1}{2}\big)^{-1}}_2 \\
	&= \norm[\big]{\bra[\big]{(1-e^{-2t})\Id+ e^{-2t}  B^{\frac12}C_0^{-1}B^{\frac12}}^{-1}}_2
	\leq \max\set[\big]{ 1, \norm[\big]{B^{-\frac12}C_0B^{-\frac12}}_2}.
\end{align*} 
In combination, this concludes the estimate for the last term in \eqref{eq:estimate0}
\begin{align}\label{eq:estimate1}
	|m_t-x_0|_B^2  
	&\leq  \max\set[\big]{ 1, \norm[\big]{B^{-\frac12}C_0B^{-\frac12}}_2} \norm[\big]{B^{\frac{1}{2}}C_0^{-1} B^{\frac{1}{2}}}_2 e^{-2t} |m_0-x_0|_B^2 .
\end{align}
Next, to estimate $S(t)$ from~\eqref{eq:estimate0}, we write it in terms of the eigenvalues $\sigma_i(t) = \sigma_i(B^{-\frac12}C_tB^{-\frac12})$, $i=1,\ldots,d$. The latter are real and positive (and agree with their singular values), and are given by
\begin{equation}\label{eq:sigma-sol}
	\sigma_i(t) = \frac{1}{1-e^{-2t} + e^{-2t} \sigma_i(0)^{-1} }  
\end{equation}
thanks to the explicit solution representation~\eqref{eq:sol:Ct}.
We thus have
\begin{equation}
	\label{eq:losing4}
	S(t) = \sum_{i=1}^d \bra*{ \sigma_i(t) - 1 - \log \sigma_i(t)}.
\end{equation}
The goal is to bound $S(t)$ by a multiple of $e^{-2t}S(0)$ from above, uniformly in $t\ge0$. To this end, we control the terms for the eigenvalues separately. From the solution formula \eqref{eq:sigma-sol}, it follows immediately that each $\sigma_i(t)$ converges to $1$ monotonically. Next, notice that $s\mapsto s-1-\log s$ is strictly convex with minimum zero at $s=1$. Thus the ``below-secant-formula'' for convex functions implies
\begin{align*}
	\sigma_i(t)-1-\log \sigma_i(t) &\le \frac{\sigma_i(t)-1}{\sigma_i(0)-1} (\sigma_i(0)-1-\log\sigma_i(0)) 
	= \frac{e^{-2t}(\sigma_i(0)-1-\log\sigma_i(0))}{\sigma_i(0)(1-e^{-2t})+e^{-2t}},
\end{align*}
where the last equality directly follows from \eqref{eq:sigma-sol}.
Estimating the pre-factor
\begin{align*}
	\frac1{\sigma_i(0)(1-e^{-2t})+e^{-2t}} \le \max\big(1,\sigma_i(0)^{-1}\big),
\end{align*}
and recalling that $\max(\sigma_1(0)^{-1},\ldots,\sigma_d(0)^{-1})=\|B^\frac12C_t^{-1}B^\frac12\|_2$, we conclude in the same spirit as above that 
\begin{align*}
	S(t) \le \max\set*{1,\norm[\big]{B^\frac12C_0^{-1}B^\frac12}_2} S_0 e^{-2t}.        
\end{align*}
This directly yields \eqref{eq:ent-decay-fixedmean}, and in combination with \eqref{eq:estimate1} also \eqref{eq:ent-decay-cov}.

We now turn to the modified Fisher information. By \eqref{eq:residualI}, we have
\begin{align}
	\covI(\normal_t \,|\, \normal_*) = \norm[\big]{\Id-B^{-\frac12}C_tB^{-\frac12}}^2_\HS + \abs[\big]{C_t^\frac12B^{-1}(m_t-x_0)}^2. 
\end{align}
The estimates are carried out in analogy to the ones above. On the one hand, for the difference $m_t-x_0$, we obtain
\begin{align}
		|C_t^\frac12B^{-1}(m_t-x_0)|^2 \nonumber
		&\le\norm[\big]{C_t^\frac12B^{-1}C_tB^{-1}C_t^\frac12}_2|m_t-x_0|_{C_t}^2 \nonumber\\
		&\le\norm[\big]{(B^{-\frac12}C_tB^{-\frac12})^2}_2\; e^{-2t}|m_0-x_0|_{C_0}^2 \nonumber\\
		&\le\norm[\big]{B^{-\frac12}C_tB^{-\frac12}}_2^2 \; \norm[\big]{C_0^{-\frac12}BC_0^{-1}BC_0^{-\frac12}}_2 \; e^{-2t} \abs[\big]{C_0^\frac12B^{-1}(m_0-x_0)}^2 \nonumber\\
		&\le\max\set[\big]{1,\norm[\big]{B^{-\frac12}C_0B^{-\frac12}}_2^2}\norm[\big]{B^\frac12C_0^{-1}B^\frac12}_2^2\; e^{-2t}\abs[\big]{C_0^\frac12B^{-1}(m_0-x_0)}^2. \label{eq:estimate2}
\end{align}
And on the other hand, for the Hilbert-Schmidt norm, we have, again with $\sigma_j(t)$ being the real and positive (singular and) eigenvalues of $B^{-\frac12}C_tB^{-\frac12}$:
\begin{align*}
	\norm[\big]{\Id-B^{-\frac12}C_tB^{-\frac12}}^2 
	&= \sum_j(1-\sigma_j(t))^2   
	= \sum_j\left(\frac{e^{-2t}(\sigma_j(0)-1)}{(1-e^{-2t})\sigma_j(0)+e^{-2t}}\right)^2 \notag \\
	&\le e^{-4t}\max_j\max\left(1,\frac1{\sigma_j(0)^2}\right)\sum_j(1-\sigma_j(0))^2 \notag \\
	&\le \max\set[\big]{1,\norm[\big]{B^{-\frac12}C_0B^{-\frac12}}_2^2} \; e^{-4t} \norm[\big]{\Id-B^{-\frac12}C_0B^{-\frac12}}^2. %
\end{align*}
This directly yields \eqref{eq:fish-decay-fixedmean}, and in combination with \eqref{eq:estimate2} also \eqref{eq:fish-decay-cov}.
\end{proof}
\begin{proof}[Proof of Theorem~\ref{thm:decay}]
Combine the resulting decay estimates for the shape $\eta_t$ from~\eqref{eq:OU-decay} with the decay estimates for moments from Lemma~\ref{lem:Gauss-decay} above, the result follows immediately from the decomposition of $\cE(\rho_t\,|\,\rho_*)$ and $\covI(\rho_t\,|\,\rho_*)$ given in Lemma~\ref{lem:ent-split}.
\end{proof}

\subsection{Convergence in Wasserstein distance}\label{sec:W2:cv}

\begin{lemma}[Splitting estimate for $W_2$]\label{lem:W2:split}
Let $\rho_t$ be a solution to~\eqref{eq:GF-FP-cov-quad} starting from $\rho_0$. Let $m_t = \mean(\rho_t)$ and $A_t$ solve~\eqref{eq:Adot} given $C_t=\C(\rho_t)$.
Then, the Wasserstein distance satisfies the splitting estimate
\begin{equation}\label{eq:W2:split}
	W_2(\rho_t, \normal_{x_0,B}) \leq \norm{\C(\rho_t)}_2^{\frac{1}{2}} W_2\bra*{\eta_t,\normal_{0,\Id}} + W_2\bra*{ \normal_{m_t,C_t}, \normal_{x_0,B}},
\end{equation}
with the normalization $\eta_t = (T_{m_t,A_t})_\# \rho_t$.
\end{lemma}
\begin{proof}
We apply the triangle inequality
\begin{equation}
	W_2(\rho_t, \normal_{x_0,B}) \leq  W_2\bra*{\bra[\big]{T_{m_t,A_t}^{-1}}_\# \eta_t, \bra[\big]{T_{m_t,A_t}^{-1}}_\# \normal_{0,\Id}}
	+ W_2\bra*{\bra[\big]{T_{m_t,A_t}^{-1}}_\# \normal_{0,\Id}, \normal_{x_0,B} }
\end{equation}
To the first term, we can apply~\cite[Lemma 3.1]{CarrilloVaes}, where we note that in the push-forward the same mean cancels out and that $\norm*{A_t^i}_2^2 = \norm*{A_t^i (A_t^i)^\tran }_2 = \norm*{C_t^i}_2$ by construction of $A_t^i$ in~\eqref{eq:Adot}.
As for the second term, we observe that the coupling measure is $\bra[\big]{T_{m_t,A_t}^{-1}}_\# \normal_{0,\Id} = \normal_{m_t,C_t}$, which follows from the fact that $\abs{A_t^{-1} ( x- m_t)}^2 = \abs{ x- m_t }_{C_t}^2$ by construction of $A_t$ as square-root of $C_t$ in~\eqref{eq:Adot}. This proves the claim~\eqref{eq:W2:split}.
\end{proof}
It remains to bound the term involving $\C(\rho_t)$ and $W_2\bra*{ \normal_{m_t,C_t}, \normal_{x_0,B}}^2$, which we do in the next two Lemmas.
\begin{lemma}\label{lem:Cov:moment:est}
In the setting of Lemma~\ref{lem:W2:split}, for all $t\geq 0$ the covariance matrix satisfies
\begin{align}\label{eq:Cov:moment:est}
	\norm*{ \C(\rho_t)}_2 \leq \kappa(B, C_0) := \norm{B}_2 \max\set[\big]{ 1, \norm[\big]{B^{-\frac{1}{2}}C_0 B^{-\frac{1}{2}} }_2} .
\end{align}
\end{lemma}
\begin{proof}
The prefactor is estimated by the explicit representation of the solution in~\eqref{eq:sol:Ct} obtained in Lemma~\ref{lem:moment-decay}. Indeed, the submultiplicativity of the norm implies
\begin{equation}
	\norm*{ \C(\rho_t)}_2 \leq \norm{B}_2 \norm[\big]{ B^{-\frac{1}{2}} \C(\rho_t)B^{-\frac{1}{2}}}_2
\end{equation}
By setting $D_t=  B^{-\frac{1}{2}} \C(\rho_t)B^{-\frac{1}{2}}$, we proceed similarly as in the proof of Lemma~\ref{lem:Gauss-decay}, we introduce the eigenvalues $\sigma_i(t) = \sigma_i(D_t) = \sigma_i\bra[\big]{B^{-\frac{1}{2}}C_tB^{-\frac{1}{2}}}$, which are real and positive, and are given by
\[
\sigma_i(t) = \frac{1}{1-e^{-2t} + e^{-2t} \sigma_i(0)^{-1}} .
\]
Hence, $\norm{D_t}_2 = \max_{i=1,\dots,d} \set*{ \sigma_i(t)} \leq \max_{i=1,\dots,d}\set*{\max\set*{1,\sigma_i(0)}} = \max\set*{1,\norm{D_0}_2}$.
\end{proof}
\begin{lemma}\label{lem:W2:est:moments}
Let $(m_t,C_t)$ be a solution to~\eqref{eq:moments-GF}, then
\begin{equation}\label{eq:W2:est:moments}
	W_2\bra*{ \normal_{m_t,C_t}, \normal_{x_0,B}}^2\leq e^{-2t} \kappa(B, C_0) \bra[\Big]{ \abs*{m_0-x_0}_{C_0}^2 + \norm[\big]{\Id - \bra[\big]{B^{\frac{1}{2}}C_0^{-1}B^{\frac{1}{2}}}^{\frac{1}{2}}}_\HS^2 }
\end{equation}
with $\kappa(B, C_0)$ given by~\eqref{eq:Cov:moment:est}.
\end{lemma}
\begin{proof}
By~\cite{GivensShortt1984}, see also~\cite[Lemma 3.3]{CarrilloVaes}, we have
\begin{equation}\label{eq:W2:Gauss}
	W_2\bra*{ \normal_{m_t,C_t}, \normal_{x_0,B}}^2 = \abs*{m_t-x_0}^2 + \tr \pra[\Big]{ C_t + B - 2 \bra[\big]{ B^{\frac{1}{2}}  C_t  B^{\frac{1}{2}}}^{\frac{1}{2}}} . 
\end{equation}
To the first term, we apply~\eqref{eq:m-decay} and get
\begin{align*}
	\abs*{m_t-x_0}^2 \leq \norm{C_t}_2^2 \abs*{ A_t^{-1}(m_t-x_0)}^2 \leq  \kappa(B, C_0) e^{-2t} \abs{m_0 - x_0}_{C_0}^2 ,
\end{align*}
where we also used~\eqref{eq:Cov:moment:est}. 

By the explicit representation~\eqref{eq:sol:Ct}, we can write
\[
C_t^{-1} = B^{-\frac{1}{2}} \bra[\big]{ (1-e^{-2t}) \Id + e^{-2t} B^{\frac{1}{2}} C_0^{-1}  B^{\frac{1}{2}}} B^{-\frac{1}{2}} =: B^{-\frac{1}{2}} D_t^{-1} B^{-\frac{1}{2}}.
\]
Hence, we can write and estimate via the Araki-Lieb-Thirring inequality~\cite{Araki1990,LiebThierring1991}
\[
\tr\pra[\Big]{\bra[\big]{ B^{\frac{1}{2}}  C_t  B^{\frac{1}{2}}}^{\frac{1}{2}}} = \tr \pra[\Big]{\bra[\big]{ B D_t B}^{\frac{1}{2}}} \geq \tr\pra[\Big]{ B^{\frac{1}{2}} D_t^{\frac{1}{2}} B ^{\frac{1}{2}}} .
\]
With this, we arrive for the second term in~\eqref{eq:W2:Gauss} at the bound
\begin{align*}
	\MoveEqLeft
	\tr \pra[\Big]{ C_t + B - 2 \bra[\big]{ B^{\frac{1}{2}}  C_t  B^{\frac{1}{2}}}^{\frac{1}{2}}}
	\leq \tr \pra[\Big]{  B^{\frac{1}{2}} \bra[\big]{  D_t + \Id - 2 D_t^{\frac{1}{2}}}B^{\frac{1}{2}}} \\
	&= \tr \pra[\big]{ B \bra{ D_t^{\frac{1}{2}} - \Id}^2 } 
	= \norm[\big]{ B^{\frac{1}{2}} \bra[\big]{D_t^{\frac{1}{2}} - \Id}}_\HS^2
	\leq \norm{B}_2 \norm[\big]{D_t^{\frac{1}{2}} - \Id}_\HS^2
\end{align*}
We proceed similarly as in the proof of Lemma~\ref{lem:Gauss-decay}, we introduce the eigenvalues $\sigma_i(t) = \sigma_i(D_t) = \sigma_i\bra[\big]{B^{-\frac{1}{2}}C_tB^{-\frac{1}{2}}}$, %
which are given by
$\sigma_i(t) = \bra*{1-e^{-2t} + e^{-2t} \sigma_i(0)^{-1}}^{-1}$.
Hence, it is left to estimate
\begin{align*}
	\norm[\big]{D_t^{\frac{1}{2}} - \Id}_\HS^2 = \sum_{i=1}^d \abs[\bigg]{ \frac{1}{\sqrt{1-e^{-2t} + e^{-2t} \sigma_i(0)^{-1}}} - 1}^2 
	=\sum_{i=1}^d \frac{\abs[\big]{\sqrt{1-e^{-2t} + e^{-2t} \sigma_i(0)^{-1}} - 1}^2}{1 + e^{-2t}\bra*{\sigma_i(0)^{-1}-1}} .
\end{align*}
The denominator is bounded by
$\bra{1 + e^{-2t}\bra*{\sigma_i(0)^{-1}-1}}^{-1} \leq \max\set{ 1, \sigma_i(0)}$.
For the nominator we note, that after multiplying with $e^{2t}$ that the function 
\[
t \mapsto e^{2t} \abs[\big]{\sqrt{1-e^{-2t} + e^{-2t} \sigma_i(0)^{-1}} - 1}^2
\]
is non-negative and monotone decreasing (a longish elementary calculation) and hence bounded by its values for $t=0$ given by $\abs{\sqrt{\sigma_i(0)^{-1}}-1}^2$. Hence, by recalling that $\sigma_i(0)= \sigma_i(B^{-\frac{1}{2}}C_0 B^{-\frac{1}{2}})$, we obtain the bound
\begin{align*}
	\norm[\big]{D_t^{\frac{1}{2}} - \Id}_\HS^2 &\leq e^{-2t} \sum_{i=1}^d \max\set{1,\sigma_i(0)} \abs[\big]{\sqrt{\sigma_i(0)^{-1}}-1}^2 \\
	&\leq e^{-2t} \max\set[\big]{ 1, \norm[\big]{B^{-\frac{1}{2}}C_0B^{-\frac{1}{2}}}_2} \norm[\big]{\bra[\big]{B^{\frac{1}{2}}C_0^{-1}B^{\frac{1}{2}}}^{\frac{1}{2}}-\Id}_\HS^2 . \qedhere
\end{align*}
\end{proof}
\begin{remark}\label{rem:W2:WC}
	The term in brackets on the right-hand side of~\eqref{eq:W2:est:moments} can be identified as the Wasserstein distance with respect to the weighted norm $\abs*{\cdot}_{C_0}$, that is
	\begin{equation*}
		W_{2,C_0}\bra*{\normal_{m_0,C_0}, \normal_{x_0,B}}^2 = \abs*{m_0-x_0}_{C_0}^2 + \tr\pra[\Big]{ \Id + C_0^{-\frac{1}{2}} B C_0^{-\frac{1}{2}}  - 2 \bra[\big]{ C_0^{-\frac{1}{2}} B C_0^{-\frac{1}{2}}}^{\frac{1}{2}}} ,
	\end{equation*}
	where for $C\in \S_{\succ0}^d$ 
	\begin{equation}\label{eq:def:W2:C}
		W_{2,C}(\mu_0,\mu_1)^2 =  \inf\set*{\int_0^1 \int \abs{V_t}^2_C \dx{\mu_t}(x) \dx{t}~:~\partial_t\mu_t+\nabla\cdot(\mu_t V_t)=0}\,.
	\end{equation}
	Indeed, by inspection of the proof in~\cite[Lemma 3.3]{CarrilloVaes} and using the effecitve covariances $\tilde\varSigma_i = C^{-\frac{1}{2}} \varSigma_i C^{-\frac{1}{2}}$ for $i=0,1$, one verifies for any $C,\varSigma_0,\varSigma_1$ and $m_0,m_1\in \R^d$ the general identity
	\begin{equation*}
		W_{2,C}\bra*{\normal_{m_0,\varSigma_0}, \normal_{m_1,\varSigma_1}}^2 = \abs*{m_0-m_1}_C^2 + \tr\pra[\Big]{ \tilde\varSigma_0+ \tilde\varSigma_1 - 2 \bra[\big]{ \tilde\varSigma_0^{\frac{1}{2}} \tilde\varSigma_1 \tilde\varSigma_0^{\frac{1}{2}}}^{\frac{1}{2}}} .
	\end{equation*}
\end{remark}
\begin{proof}[Proof of Theorem~\ref{thm:W2:convergence}]
	We use the splitting estimate~\eqref{eq:W2:split} from Lemma~\ref{lem:W2:split} together with the uniform bound~\eqref{eq:Cov:moment:est} from Lemma~\ref{lem:Cov:moment:est}. 
	The Wasserstein distance $W_2(\eta_t,\normal_{0,\Id})$ converges exponentially with rate $1$ thanks to~\eqref{eq:OU-W2-decay}, since $\eta_t$ solves the Ornstein-Uhlenbeck equation~\eqref{eq:FP-normalized-quad}.
	We estimate the Wasserstein distance between the Gaussians $W_2\bra*{ \normal_{m_t,C_t}, \normal_{x_0,B}}$ using~\eqref{eq:W2:est:moments} from Lemma~\ref{lem:W2:est:moments}.
	The conclusion follows by noting that $\eta_0 = (T_{m_0,A_0})_\# \rho_0$ and the choice of the square root for $C_0=\C(\rho_0)$ was arbitrary. Hence, we can choose $A_0= \C(\rho_0)^{\frac{1}{2}} R$ for any $R\in \SO(d)$ and arrive at the bound in the statement of Theorem~\ref{thm:W2:convergence}.
\end{proof}

\section{Geodesic convexity and functional inequalities}

\subsection{Formal duality for the constraint transport problem}
\label{sec:optimal-constraint}

In this section, we formally derive the geodesic equations for the constrained distance $\cW_{0,\Id}$ as the optimality conditions of a saddle point problem. 
\begin{proof}[Proof of Formal Theorem~\ref{thm:infsup}]
	Introducing Langrange multipliers $\psi:[0,1]\times\R^d\to\R$ for the continuity equation constraint, $\alpha:[0,1]\to\R^d$ for the mean constraint, and $\Lambda:[0,1]\to\R^d$ for the covariance-constraint, we can rewrite $\cW_{0,\Id}(\mu_0,\mu_1)$ as an unconstrained saddle point problem
	\begin{align*}
		\cW_{0,\Id}(\mu_0,\mu_1)^2 &=
		\inf_{\rho,V}\sup_{\psi,\alpha,\Lambda}
		\biggl\{\int_0^1\int \frac12|V_t|^2\dd\mu_t\dd t +  \int_0^1\int\psi_t[\partial_t\mu+\nabla\cdot(\rho V_t)]\dd t\\
		&\qquad\qquad\quad+\int_0^1\int \ip{\alpha}{ x}\dd\mu_t\dd t +\int_0^1\int \tr[\Lambda(\Id-x\otimes x)] \dd\mu_t\dd t \biggr\}\\
		&= \sup_{\psi,\alpha,\Lambda}\inf_{\rho,V}
		\biggl\{\int_0^1\int \frac12|V_t|^2\dd\mu_t\dd t -  \int_0^1\int[\partial_t\psi+\nabla\psi_t\cdot V_t]\dd\mu_t\dd t\\
		&\qquad\qquad\quad+\int \psi_1\dd\mu_1-\psi_0\dd\mu_0\\
		&\qquad\qquad\quad+\int_0^1\int \ip{\alpha}{ x}\dd\mu_t\dd t +
		\int_0^1\int \tr[\Lambda(\Id-x\otimes x)] \dd\mu_t\dd t \biggr\}\;,\\
		&= \sup_{\psi,\alpha,\Lambda}\inf_{\rho,V}
		\biggl\{\int_0^1\int \frac12|V_t-\nabla\psi_t|^2\dd\mu_t\dd t +\int \psi_1\dd\mu_1-\psi_0\dd\mu_0 +\int_0^1\tr[\Lambda_t]\dd t\\
		&\qquad\qquad\quad -  \int_0^1\int\big[\partial_t\psi+\frac12|\nabla\psi_t|^2+\tr[\Lambda_t(x\otimes x)]-\ip{\alpha}{x}\big]\dd\mu_t\dd t\biggr\} \:,
	\end{align*}
where we have integrated by parts and interchanged $\inf$ and $\sup$ in the second step. The infimum over $V$ is attained at $V=\nabla \psi$. The infimum over $\mu$ yields $-\infty$ unless
\[
\partial_t\psi+\frac12|\nabla\psi_t|^2+\tr[\Lambda_t(x\otimes x)]-\ip{\alpha}{x} \leq 0\;.
\]
Thus we arrive at \eqref{eq:dualW}. We further obtain formally the following optimality conditions in the saddle point problem above by considering variations in $V, \psi, \rho, \alpha, \Lambda$ respectively:
	\begin{align*}
		V\mu &=\nabla \psi \mu\;,\quad
		\partial_t\mu +\nabla\cdot(\mu\nabla\psi)=0\;,\\
		\partial_t\psi & + \frac12|\nabla\psi|^2 + \tr[\Lambda (x\otimes x)] -\ip{\alpha}{x}=0\;,\\
		\int x_i\dd\mu&=0\;,\quad  \int x_ix_j\dd\mu = \delta_{ij}\quad \text{ for all } i,j\in\{1,\dots, d\}\;.
	\end{align*}
To obtain some information on the multipliers, we differentiate the constraints. First, the mean constraint gives
	\begin{align*}
		0 &= \frac{\dd}{\dd t} \int x_i\dd\mu_t= \int \nabla(x_i)\cdot\nabla\psi \dd\mu_t = \int\partial_i\psi\dd\mu_t\\
		0 &= \frac{\dd^2}{\dd t^2} \int x_i\dd\mu_t = \int\nabla \partial_i\psi\cdot\nabla\psi \dd\mu_t + \int \partial_i\partial_t\psi\dd\mu_t\\
		&= \int \pra*{\frac12\partial_i|\nabla\psi|^2 - \frac12\partial_i|\nabla\psi|^2+\partial_i\Big[\ip{\alpha}{x} - \tr[\Lambda(x\otimes x)]\Big]}\dd\mu_t\\
		&= \alpha_i - \int \sum_l(\Lambda_{il}+\Lambda_{li})x_l\dd\mu_t = \alpha_i\;,
	\end{align*}
	where we used the mean constraint in the last step. Hence $\alpha=0$. 
	
	The covariance-constraint gives
	\begin{align*}
		0 &= \frac{\dd}{\dd t} \int x_ix_j\dd\mu_t= \int \nabla(x_ix_j)\cdot\nabla\psi \dd\mu_t\\
		0 &= \frac{\dd^2}{\dd t^2} \int x_ix_j\dd\mu_t = \int\nabla\big(\nabla(x_ix_j)\cdot\nabla\psi\big)\cdot\nabla\psi\dd\mu_t 
		+\int \nabla(x_ix_j)\cdot\nabla\partial_t\psi \dd\mu_t\\
		&= \int\nabla\big(\nabla(x_ix_j)\cdot\nabla\psi)\big)\cdot\nabla\psi\dd\mu_t
		-\int \nabla(x_ix_j)\cdot\frac12\nabla|\nabla\psi|^2\dd\mu_t\\
		&\quad +\sum_{kl}\int \nabla(x_ix_j)\cdot\big(\alpha_k\nabla(x_k)-\Lambda_{kl}\nabla(x_k x_l)\big)\dd\mu_t\\
		&= \int \nabla\psi\cdot D^2(x_ix_j)\cdot\nabla \psi\dd\mu_t
		-\sum_{kl}\Lambda_{kl}\int \nabla(x_ix_j)\cdot\nabla(x_k x_l)\dd\mu_t\\
		&= 2 \int\partial_i\psi\partial_j\psi\dd\mu_t -  4 \Lambda_{ij}\;,
	\end{align*}
	where we have used $\alpha=0$, and in the last line the covariance-constraint together with the fact that $\Lambda^\tran =\Lambda$. In particular, we see  $\tr[\Lambda_t]=\frac12\int |\nabla\psi_t|^2\dd\mu_t$. Note that the latter expression is constant in time, equal to $\cW_{0,\Id}(\mu_0,\mu_1)^2$, since the minimiser of $\int_0^1\int |V_t|^2\dd\mu_t\dd t$ is necessarily parametrised in such a way that  $\int |V_t|^2\dd\mu_t$ is constant, equal to the infimum value $\cW_{0,\Id}(\mu_0,\mu_1)^2$. This concludes the proof.
\end{proof}

\subsection{Formal geodesic convexity}\label{sec:convexity}

In this section, we formally investigate the convexity properties of the following free energy, including an internal energy $\cU$ and a potential energy $\cH$:
\begin{align}\label{eq:energy}
    \cF[\mu] &= \cU[\mu]+\cH[\mu]
    = \int U(\rho) + \int H\rho 
\end{align}
Here, for an absolutely continuous density $\mu$, we write $\mu(\dd x)=\rho(x)\dd x$. This energy is intrinsically related to the partial differential equation
\begin{align}\label{eq:GF-UHK}
    \partial_t \rho = \nabla \cdot \left(\rho \C(\rho) \nabla\left[ U'(\rho)+H\right]\right)\,.
\end{align}
Indeed, equation~\eqref{eq:GF-UHK} is the gradient flow of $\cF$ w.r.t.~the distance $\cW$ as discussed in Section~\ref{sec:GFintro}, see equation~\eqref{eq:GF}.

We are interested in geodesic convexity both w.r.t.~the distance $\cW_{0,\Id}$ and w.r.t.~$\cW$ under suitable conditions on the functions $U$ and $H$. The former case will be investigated by calculating the Hessians for the two contributions of $\cF$ by looking at the second derivative along geodesics. Then we will see how geodesic convexity transfers from the covariance-constraint to the modulated situation.
We make the by now classical assumptions on the function $U$ for the internal energy~\cite{CarrilloMcCannVillani}.

\begin{assumption}[Diffusion]\label{ass:U}
Consider a density of internal energy $U:\R_{>0}\to \R$ satisfying
\begin{enumerate}
\item[(a)] (Convexity) $U(s)=0$ (no diffusion), or $U(s)=\sigma s\log s$ for some $\sigma>0$ (linear diffusion), or $U$ is strictly convex for $s>0$.
\item[(b)] (Dilation condition) $\lambda\mapsto \lambda^d U (\lambda^{-d})$ is convex and non-increasing on $\R_{>0}$
\end{enumerate}
\end{assumption}
For non-linear diffusion, the PDE~\eqref{eq:GF-UHK} can also be written as
\begin{align*}
    \partial_t \rho = \nabla \cdot \left( \C(\rho) \nabla P(\rho)\right) + \nabla\cdot \left(\rho \C(\rho) \nabla H\right) 
\end{align*}
where the pressure $P:\R_+\to\R$ is non-negative and given by
\begin{align}\label{eq:pressure}
    P(s):=\int_0^s\tau U''(\tau)\,\dd \tau =
s U'(s)-U(s)\,.
\end{align}
\begin{remark}\label{rmk:pressure}
Note that strict convexity of $U$ as stated in Assumption~\ref{ass:U}(a) corresponds to the statement that the pressure $P$ is increasing since $P'(s)=sU''(s)$. Further, the dilation condition Assumption~\ref{ass:U}(b), which was first introduced by McCann in \cite{McCann97}, corresponds to the statement that 
\[
s\mapsto \frac{P(s)}{s^{1-1/d}} \quad \text{ is non-negative and non-decreasing; }
\]
in other words, $\rho P'(\rho) \ge (1-1/d)P(\rho)\ge 0$. Also see \cite[Chapter 9]{AGS}, \cite[p.26]{CJMTU}, \cite[Theorem 1.3]{Sturm05}, \cite[Chapter 17]{Villani09}.
\end{remark}

\begin{remark}
 The functional $\cU$ can be extended to the full set of Borel probability measures on $\R^d$ by setting $\cU(\mu)=\infty$ for measures $\mu\in\mathcal{P}(\R^d)$ that are not absolutely continuous with respect to the Lebesgue measure.
\end{remark}

\begin{example}\label{ex:nonlindiff}
A typical choice for the diffusion term satisfying Assumption~\ref{ass:U} is
\[
U(s)= \frac{s(s^{m-1}-1)}{m-1}
\]
for $m>0$. Then $P(s)=s^m$, hence condition (a) is automatically satisfied, and condition (b) corresponds to requiring $m\ge 1-1/d$. In the limit  $m\to 1$ one recovers the Boltzmann-Shannon entropy corresponding to the choice $U(s)=s\log s$. The Boltzmann-Shannon entropy will be denoted by $\mathcal E$.
\end{example}

\begin{ftheorem}\label{prop:convex-U}
Under Assumption~\ref{ass:U}, the internal energy $\cU$ satisfies along any $\cW_{0,\Id}$-geodesics $(\mu_t)_{t\in[0,1]}$ with densities $\mu_t(\dd x)=\rho_t(x)\dd x$ the estimate
    \begin{align*}
    \frac{\dd^2}{\dd t^2} \cU(\mu_t)
    \ge  \cW_{0,\Id}(\mu_0,\mu_1)^2 \int P(\rho_t)\,\dd x\;.
    \end{align*}
	 In particular, the Boltzmann-Shannon entropy $\mathcal E$ is geodesically $1$-convex, i.e.
    \begin{align*}
    \cE(\mu_t)\leq (1-t)\cE(\mu_0) + t\cE(\mu_1) -\frac12 t(1-t) \cW_{0,\Id}(\mu_0,\mu_1)^2\;.
    \end{align*}
\end{ftheorem}
For the Boltzmann-Shannon entropy we will rigorously prove this statement in the next section.

\begin{proof}
    For the internal energy $\cU$, we have, using the optimality conditions \eqref{eq:geodesic-eqs},
    \begin{align*}
    \frac{\dd}{\dd t} \cU(\mu) &= 
    - \int U'(\rho) \nabla\cdot(\rho\nabla\psi)=
    \int \nabla U'(\rho)\cdot \nabla\psi \rho\;,\\
    \frac{\dd^2}{\dd t^2} \cU(\mu) &=
    \int U''(\rho) \left|\nabla\cdot(\rho\nabla\psi)\right|^2
    + \int \nabla U'(\rho)\cdot \nabla \partial_t\psi\, \rho + \int \nabla U'(\rho)\cdot \nabla \psi\, \partial_t\rho\\
    &=
    \int U''(\rho) \left|\nabla\cdot(\rho\nabla\psi)\right|^2
    -\frac12 \int \nabla U'(\rho)\cdot \nabla |\nabla\psi|^2\, \rho 
    -\sum_{k,l} \Lambda_{kl} \int  \nabla U'(\rho)\cdot \nabla (x_k x_l)\rho\\
    &\quad+\int  \nabla\left(\nabla U'(\rho)\cdot \nabla\psi\right)\cdot \nabla \psi\,\rho\\
    &= 
    \int U''(\rho) \left|\nabla\cdot(\rho\nabla\psi)\right|^2
    +\int  \langle \nabla\psi, D^2 U'(\rho) \nabla\psi\rangle \,\rho
    -\sum_{k,l} \Lambda_{kl} \int \partial_kU'(\rho) x_l\,\rho\;.
\end{align*}
After a simple, but long and tedious calculation, this expression can be written as 
    \begin{align*}
    \frac{\dd^2}{\dd t^2} \cU(\mu) &=
    \int \left[P'(\rho)\rho-P(\rho)\right]|\Delta \psi|^2 + \int P(\rho) \|D^2\psi\|_{\HS}^2 +   \cW_{0,\Id}(\mu_0,\mu_1)^2\int P(\rho)\;,
\end{align*}
where $P$ is given in~\eqref{eq:pressure}. Under Assumption~\ref{ass:U}, see also Remark~\ref{rmk:pressure}, we are able to make use of the inequality $\|D^2\psi\|_{\HS}\ge d^{-1/2} |\Delta \psi|$ to conclude
    \begin{align*}
    \frac{\dd^2}{\dd t^2} \cU(\rho) &\ge
   \int P(\rho) \left[ -\frac{1}{d}|\Delta \psi|^2+ \|D^2\psi\|_{\HS}^2\right] +  \cW_{0,\Id}(\mu_0,\mu_1)^2\int P(\rho) \\
   &\ge   \cW_{0,\Id}(\mu_0,\mu_1)^2 \int P(\rho)\;. \qedhere
\end{align*}
\end{proof}

\begin{remark}
 If $H:\R^d\to\R$ is a quadratic form $H(x)=\ip{Ax}{x}+\ip{b}{x}+c$, then the potential energy $\cH(\mu)=\int H\dd\mu$ is constant on $\cP_{0,\Id}(\R^d)$, in particular along $\cW_{0,\Id}$-geodesics. Indeed, due to the mean and variance constraint, we have for any $\mu\in\cP_{0,\Id}(\R^d)$  that $\cH(\mu) = \tr[A] + c$.
\end{remark}

Next, we will focus on the Boltzmann-Shannon $\mathcal E$ corresponding to the choice $U(s)=s\log s$ and investigate its convexity properties along geodesics of the covariance-modulate transport distance $\cW$.

\begin{ftheorem}\label{fthm:convex-W}
For any $\cW$-geodesic $(\mu_t)_{t\in[0,1]}$ we have that 
\begin{equation*}%
 \cE(\mu_t)\leq (1-t)\cE(\mu_0) + t\cE(\mu_1) -\frac12 t(1-t) \cW_{0,\Id}(R_{\#}\widebar\mu_0,\widebar\mu_1)^2\;,
\end{equation*}
where $R_{\#}\widebar\mu_0$ and $\widebar\mu_1$ are the normalisations of $\mu_0,\mu_1$ appearing in the splitting result in Theorem~\ref{thm:main-cov}.
\end{ftheorem}
\begin{proof}
Let $(\mu_t)_{t\in[-1/2,1/2]}$ be a $\cW$-geodesic, i.e.~an
optimal curve for \eqref{eq:OT-cov-main}, with marginals of finite entropy and let $(m_t,C_t)$ be the mean and covariance of $\mu_t$. According to Theorem \ref{thm:main-cov}, we can write
\begin{align*}
  \mu_t &= (T_{m_{t},A_{t}}^{-1})_\#\tilde\mu_{t}\;,
\end{align*}
where $\tilde\mu_t$ a solution to the covariance-constraint transport problem between $R_{\#}\widebar\mu_0$ and $\widebar\mu_1$ and $A_t$ is given by the solution of the constrained moment problem. Denoting $\tilde\mu_t(\dd y)=\tilde\rho_t(y)\dd y$ and $\mu_t(\dd y)=\rho_t(y)\dd y$, we deduce $\rho_t(T_{m_t,A_t}^{-1}(x))=\tilde\rho_t(x)/\det A_t$ and readily compute that
\begin{align}\label{eq:entropy-split}
  \mathcal E(\mu_{t}) = \int \log(\rho_t(T_{m_t,A_t}^{-1}(x)))\tilde\mu_t(\dd x)
  = \mathcal E(\tilde\mu_t) -\log \det A_{t}\;. 
\end{align}
Setting $l(t) :=-\log \det A_t$,  we obtain the derivatives of $l(t)$ from \eqref{eq:EL-cov2R} in Theorem~\ref{thm:ex-Ropt-moments} for some $\alpha \in \R^d$ as
\begin{align*}
  \dot l(t) &= -\tr\big(A^{-1}\dot A\big)\;,\\
  \ddot l(t) &= - \tr\big(A^{-1}\ddot A\big)+\tr\big(A^{-1}\dot A A^{-1} \dot A\big)
            = \tr\pra*{A^\tran(\alpha\otimes\alpha)A} = \ip{\alpha}{C\alpha}\;.
\end{align*}
Hence, $\ddot l(t)\geq 0$. Now the claim follows directly from \eqref{eq:entropy-split} and the convexity of $\cE$ along geodesics for distance $\cW_{0,\Id}$. 
\end{proof}

Finally, let us investigate convexity along geodesics of the modulated transport problem of potential energies $\cH$ with quadratic potential of the form 
\[
H(x)= \frac12 \ip{x-x_0}{B^{-1}(x-x_0)}\;.
\]
In this case, we readily compute that 
\[
\cH(\mu) = \frac12\tr\big[CB^{-1}\big] + \frac12\ip{m-x_0}{B^{-1}(m-x_0)}\;,
\]
with $m,C$ the mean and covariance of $\mu$. Now, consider a $\cW$-geodesic $(\mu_t)_{t\in[0,1]}$ with mean $\mean(\mu_t)=m_t$ and covariance $\C(\mu_t)=C_t$. From the splitting result in Theorem \ref{thm:main-cov} and the optimality conditions for the moment part \eqref{eq:EL-cov2R} we compute with $A_tA_t^\tran=C_t$, $\dot A_t = \frac12\dot C_tA^{-\tran}$ that
\begin{align*}
\dot m_t = C_t \alpha\;,\qquad
\ddot C_t = \dot C C^{-1}\dot C - C_t(\alpha \otimes \alpha) C_t\;. 
\end{align*}
Hence, we obtain
\begin{align}\nonumber
\frac{\dd}{\dd t}\cH(\mu_t) &= \frac12 \tr\big[\dot C_tB^{-1}\big]+\ip{C\alpha}{B^{-1}(m-x_0)}\;,\\\label{eq:d2-H}
\frac{\dd^2}{\dd t^2}\cH(\mu_t) &= \frac12 \tr\big[\dot C_t C_t^{-1}\dot C_t B^{-1}\big] + \ip{\dot C\alpha}{B^{-1}(m_t-x_0)}\;.
\end{align}

In general, this expression for the second derivative of the potential energy is hard to control. However, we have the following result for the relative entropy w.r.t.~$\normal_{x_0,B}$ defined for $\mu=\rho\normal_{x_0,B}$ by
\[\cE(\mu|\normal_{x_0,B})= \int\log\rho \dd\mu = \cE(\mu) + \cH(\mu)\;.\]

\begin{proposition}\label{prop:moment-convex-special}
Let $(\mu_t)_{t\in[0,1]}$ be a $\cW$-geodesic such that $\mean(\mu_t)=x_0$ and $\C(\mu_t)\succcurlyeq \frac12B$ for all $t\in[0,1]$. Then the relative entropy $\cE(\cdot|\normal_{x_0,B})$ is $1$-convex along $(\mu_t)$, i.e.
\[\cE(\mu_t|\normal_{x_0,B})\leq (1-t) \cE(\mu_0|\normal_{x_0,B}) + t \cE(\mu_1|\normal_{x_0,B})- \frac12 t(1-t) \cW(\mu_0,\mu_1)^2\;.\]
\end{proposition}
\begin{proof}
Using Theorem \ref{thm:main-cov}, we write again
  $\mu_t = (T_{m_{t},A_{t}}^{-1})_\#\tilde\mu_{t}$, where $\tilde\mu_t$ a solution to the covariance-constrained transport problem between $R_{\#}\widebar\mu_0$ and $\widebar\mu_1$ and $A_t$ is given by the solution of the constrained moment problem. Under the assumtions on $\mu_t$ that $\mean(\mu_t)=x_0$ and $\C(\mu_t)\succcurlyeq \frac12 B$, we have from \eqref{eq:d2-H}
\[
\frac{\dd^2}{\dd t^2}\cH(\mu_t) \geq \frac14\tr\big[\dot C_tC_t^{-1}\dot C_tC_t^{-1}\big]=\tr\big[\dot A_tA_t^{-1}\dot A_tA_t^{-1}\big] = \mathcal D_R(\mu_0,\mu_1)^2\;.
\]
Hence,
\[
\cH(\mu_t)\leq (1-t)\cH(\mu_0) + t\cH(\mu_1) -\frac12 t(1-t) \mathcal D_R(\mu_0,\mu_1)^2\;.
\]
Combining this with \eqref{eq:E-convexW} and the fact that $\cW(\mu_0,\mu_1)^2= \cW_{0,\Id}(\mu_0,\mu_1)^2 + \mathcal D_R(\mu_0,\mu_1)^2$, we obtain the claim.
\end{proof}

\begin{remark}
The set 
 \begin{equation}\label{eq:setG}
 \mathcal G := \Bigl\{\mu\in\cP_2(\R^d) ~:~ \mean(\mu)=x_0\;,~ \C(\mu)\succcurlyeq \tfrac12 B\Bigr\}
 \end{equation}
 is probably not geodesically convex w.r.t.~$\cW$. Note however that by Theorem \ref{thm:main-cov} and the form of $\mathcal D_R$, any $\cW$-geodesic $(\mu_t)_t$ with $\mean(\mu_0)=\mean(\mu_1)=x_0$ satisfies $\mean(\mu_t)=x_0$ for all $t\in[0,1]$. Moreover,  the bound \eqref{eq:mc-apriori} from Lemma \ref{lem:mA-finiteAction} readily implies that for $\C(\mu_0),\C(\mu_1)\succcurlyeq \frac12 B+\eps\Id$ and $\cW(\mu_0,\mu_1)$ sufficiently small, we also have $\C(\mu_t)\succcurlyeq \frac12 B$. Hence, the interior of $\mathcal G$ is locally geodesically convex, in the sense that it can be covered by $\cW$-balls such that any $\cW$ geodesic connecting points in the same ball stays inside $\mathcal G$.
 \end{remark}

\subsection{Functional Inequalities}
\label{sec:FuncIneq}

In this section, we will provide the proofs for the results stated in Section \ref{sss:intro:convexity}. We will first prove the Evolution Variational Inequality Theorem~\ref{thm:func-ineq}. As corollaries we obtain rigorously the geodesic convexity of the Boltzmann-Shannon entropy stated in Theorem~\ref{thm:main-entroBoltzmann entpy-cov} and contractivity for the gradient flow in the constrained distance from Corollary~\ref{cor:stability-intro}. Finally, we prove the HWI inequality in Proposition~\ref{prop:HWI-intro}. 

Recall that the Fokker-Planck equation~\eqref{eq:GF-FP-cov-quad-intro} can be rewritten as gradient  flow~\eqref{eq:GF} with respect to the covariance-modulated transport distance $\cW$ of the relative entropy $\cE(\mu | \mu_\infty)$ for $\mu_\infty=\normal_{x_0,B}$. Let us write $\eta_\infty=\normal_{0,\Id}$. For convenience, we recall the statements of the results below.

\begin{theorem}[EVI for Shape]\label{thm:EVI}
 Let $\eta,\nu\in \cP_{0,\Id}(\R^d)$ and let $\eta_t=P_t\eta$ with $(P_t)_{t\geq 0}$ being the Ornstein-Uhlenbeck semigroup. Then we have the following Evolution Variational Inequality (EVI):
  \begin{align}\label{eq:EVI-constraint}
    \frac{\dd^+}{\dd t} \cW_{0,\Id}(\eta_t,\nu)^2 + \cW_{0,\Id}(\eta_t,\nu)^2\leq \cE(\nu|\eta_\infty)-\cE(\eta_t|\eta_\infty)\;,
  \end{align}
  The same estimate holds with $\cE(\cdot|\eta_\infty)$ replaced by $\cE(\cdot)$.
\end{theorem}
 As a corollary of the above EVI we obtain the statement of Theorem \ref{thm:main-entroBoltzmann entpy-cov} on convexity $\cE(\cdot|\eta_\infty)$ and $\cE$ along $\cW_{0,\Id}$-geodesics. More generally, the entropy is almost $1$-convex along almost shortest curves.
 
\begin{corollary}\label{cor:approx_cvx}
Let $(\mu_s)_{s\in[0,1]}$ be a Lipschitz curve in $\big(\cP_{0,\Id}(\R^d), \cW_{0,\Id}\big)$ satisfying 
\begin{equation*}
\cW_{0,\Id}(\mu_s,\mu_r)\leq L|r-s|\;,\quad L^2\leq \cW_{0,\Id}(\mu_0,\mu_1)^2 +\eps^2 \qquad \forall s,r\in [0,1]\;, 
\end{equation*}
for some $\eps>0$. Then for every $t>0$ and $s\in [0,1]$
\begin{equation*}
    \cE(P_t\mu_s) \leq (1-s)\cE(\mu_0) + s\cE(\mu_1) -\frac12 s(1-s)\Big(\cW_{0,\Id}(\mu_0,\mu_1)^2 +\frac{\eps^2}{e^{2t}-1}\Big)\;.
\end{equation*}
In particular, if $(\mu_s)_s$ is a geodesic, we have for all $s\in [0,1]$
\begin{equation*}
    \cE(\mu_s) \leq (1-s)\cE(\mu_0) + s\cE(\mu_1) -\frac12 s(1-s)\cW_{0,\Id}(\mu_0,\mu_1)^2\;.
\end{equation*}
The same estimates hold for $\cE(\cdot|\eta_\infty)$ instead of $\cE(\cdot)$. 
Moreover, for any $\cW$-geodesic $(\mu_s)_{s\in[0,1]}$ we have that 
\begin{equation*}
 \cE(\mu_s)\leq (1-s)\cE(\mu_0) + t\cE(\mu_1) -\frac12 s(1-s) \cW_{0,\Id}(R_{\#}\widebar\mu_0,\widebar\mu_1)^2\;,
\end{equation*}
where $R_{\#}\widebar\mu_0$ and $\widebar\mu_1$ are the normalisations of $\mu_0,\mu_1$ appearing in the splitting result in Theorem~\ref{thm:main-cov}.
\end{corollary}
 
\begin{proof}
The statements concerning (almost) $\cW_{0,\Id}$ geodesics follow from the EVI by the general result \cite[Thm.~3.2]{DaneriSavare2008}. The last statement follows from the $\cW_{0,\Id}$-geodesic convexity of $\cE$ as shown in the proof of Formal Theorem \ref{fthm:convex-W}. In fact this argument was already rigorous conditional on the $\cW_{0,\Id}$-geodesic convexity of $\cE$.
\end{proof}

\begin{proof}[Proof of Theorem~\ref{thm:EVI}]
This statement is well known when the distance $\cW_{0,\Id}$ is replaced by (half of) the $L^2$-Wasserstein distance $W_2$, see \cite{DaneriSavare2008}. Since $\frac12 W_2$ and $\cW_{0,\Id}$  are defined through the same action functional, one can repeat the argument of Daneri and Savar\'e \cite[Thm.~5.1]{DaneriSavare2008} to obtain \eqref{eq:EVI-constraint}, as we shall briefly indicate.
Recall that $(\eta_t)_{t\geq0}$ solves the classical Fokker-Planck equation~\eqref{eq:FP-normalized-quad}, i.e. $\partial_t \eta_t=L^*\eta_t$.
It is sufficient to establish the claim for a dense set of measures $\nu,\eta$. %
So assume $\eta,\nu$ are smooth and let $(\mu_s,\nabla \phi_s)_{s\in[0,1]}\in CE_{0,\Id}(
\nu,\eta)$ be a smooth curve with $\frac12\int_0^1|\nabla\phi_s|^2\dd\mu_s\dd s \leq W_{0,\Id}(\nu, \eta)^2+\eps^2$ and $W_2(\mu_s,\mu_r)\leq L|r-s|$ with $L^2=W_{0,\Id}(\nu,\eta)^2+\eps^2$ (the latter can be achieved by reparametrisation). 
We set $\mu^t_s:=P_{st}\mu_s$ with $(P_{r})_r$ the
Ornstein-Uhlenbeck semigroup. Note that $\mean(P_r\mu)=0$ and $\C(P_r\mu)=\Id$ for all $r>0$ provided $\mean(\mu)=0$ and $\C(\mu)=\Id$. Hence $\mu^t_s\in \cP_{0,\Id}(\R^d)$ and it solves
\begin{equation}\label{eq:L*}
 \partial_t \mu^t_s = sL^*\mu^t_s\,.   
\end{equation}
Note that $\nu=\eta^t_0$ and $\eta_t=\mu^t_1$. We can find smooth functions $\phi_s^t$  satisfying
\begin{equation}\label{eq:interp}
  \partial_s\mu^t_s +\nabla\cdot(\mu^t_s\nabla \phi^t_s)=0\;.
\end{equation}
Differentiating the relative entropy along the interpolation, one obtains
\begin{align}\label{eq:ent-est}
\partial_s   \cE(\mu^t_s|\eta_\infty)
= - \int L\phi^t_s \mu^t_s\,.
\end{align}
Following the calculations in \cite{DaneriSavare2008}, we compute the derivative of the action along the semigroup.
\begin{align*}
\partial_t \mathcal A(\mu_s^t,\phi^t_s)
&= \int\nabla\partial_t\phi^t_s\cdot \nabla\phi^t_s\,\mu^t_s
+ \frac12\int|\nabla\phi^t_s|^2\,\partial_t\mu^t_s\\
&= \int\nabla\partial_t\phi^t_s\cdot \nabla\phi^t_s\,\mu^t_s
+ \frac{s}{2}\int L|\nabla\phi^t_s|^2\,\mu^t_s\,.
\end{align*}
In order to compute the first term on the right-hand side, we compare
\begin{align*}
\partial_t\partial_s\mu^t_s 
&= - \partial_t \nabla\cdot(\mu^t_s\nabla \phi^t_s)
= -\nabla\cdot( \partial_t \mu^t_s\nabla \phi^t_s) - \nabla\cdot(\mu^t_s\nabla  \partial_t\phi^t_s)\\
&= -s\nabla\cdot( L^*\mu^t_s\nabla \phi^t_s) - \nabla\cdot(\mu^t_s\nabla  \partial_t\phi^t_s)
\end{align*}
with
\begin{align*}
\partial_s\partial_t\mu^t_s    
&=\partial_s(sL^*\mu^t_s)
= L^*\mu^t_s -sL^*(\nabla\cdot(\mu^t_s\nabla\phi^t_s))
\end{align*}
to conclude
\begin{align*}
   & \int\nabla\partial_t\phi^t_s\cdot \nabla\phi^t_s\,\mu^t_s
    = - \int \phi^t_s\nabla\cdot\left(\mu^t_s\nabla\partial_t\phi^t_s\right)\\
    &\qquad=s\int\phi^t_s\nabla\cdot( L^*\mu^t_s\nabla \phi^t_s)
    + \int \phi^t_s L^*\mu^t_s -s \int \phi^t_sL^*(\nabla\cdot(\mu^t_s\nabla\phi^t_s))\\
 &\qquad=-s\int L|\nabla\phi^t_s|^2\mu^t_s
    + \int L\phi^t_s \mu^t_s +s \int \nabla L \phi^t_s\cdot\nabla\phi^t_s \mu^t_s\,.
\end{align*}
Together with \eqref{eq:ent-est}, we obtain the estimate
\begin{align}
  \partial_t \mathcal A(\mu_s^t,\phi^t_s) +\partial_s\mathcal \cE(\mu^t_s|\eta_\infty) &= -s\mathcal B(\mu^t_s,\phi^t_s)\label{eq:preEVI}
  \leq -2s\mathcal A(\mu_s^t,\phi^t_s)\;.
\end{align}
where
\begin{align*}
  \mathcal B(\mu,\phi):= \int \biggl[\frac12 L|\nabla\phi|^2-\ip{\nabla\phi}{\nabla L\phi}\biggr] \dx\mu \geq \int |\nabla\phi|^2 \dx\mu\;.
\end{align*}
From here on one completes the proof as in \cite[Thm.~5.1]{DaneriSavare2008} roughly by integrating in $s$ and optimizing over the curve $(\mu_s)_s$.
The analoguous claim, for the entropy $\cE(\cdot)$ follows immediately, since $\cE(\mu)=\cE(\mu|\eta_\infty)-\frac12\int|x|^2\dd\mu(x)+{\sf const}$ and $\eta_t, \nu$ have the same second moment.
\end{proof}

As another consequence of the above EVI, we obtain the stability estimates for the normalized gradient flow from the general result \cite[Prop.~3.1]{DaneriSavare2008}.

\begin{corollary}[Stability]\label{cor:stability}
For any two solutions $\eta_t^1,\eta^2_t$
  of \eqref{eq:FP-normalized-quad},
  \begin{equation*}
    \cW_{0,\Id}(\eta^1_t,\eta^2_t)\leq e^{-t}\cW_{0,\Id}(\eta^1_0,\eta^2_0)\quad \forall t\geq 0\;.
  \end{equation*}
 \end{corollary}

\begin{remark}\label{rem:Wstability}
Under more restrictive assumptions, we also obtain at least formally a stability result for the covariance-modulated gradient flow:

For any two solutions $\mu_t^1,\mu_t^2$ of \eqref{eq:GF-FP-cov-quad-intro} such that $\mean(\mu_0^1)=\mean(\mu_0^2)=x_0$ and $\C(\mu_0^1),\C(\mu_0^2)\succcurlyeq \frac12B$ we have
\[\cW(\mu_t^1,\mu_t^2)\leq e^{-t}\cW(\mu_0^1,\mu_0^2)\quad \forall t\geq 0\;.\]
Indeed, from the explicit evolution of mean and covariance along the gradient flow \eqref{eq:moments-GF} we infer that the set $\mathcal G$ from \eqref{eq:setG} is invariant under the flow. Now the desired stability estimate is formally equivalent to $1$-convexity of $\mathcal E(\cdot|\normal_{x_0,B})$, see e.g.~the discussion in \cite{OW05}. The latter is granted (locally) on $\mathcal G$ by Proposition \ref{prop:moment-convex-special}. With some more work, this result could be made rigorous by proving an EVI inside $\mathcal G$ along the lines of Theorem \ref{thm:EVI}.
\end{remark}

 Finally, we discuss the HWI inequality as a consequence of the strict convexity of the entropy, along $W_{0,\Id}$-geodesics.

\begin{proposition}[HWI Inequality]\label{prop:HWI-new}
Assume $\eta_0,\eta_1\in \mathcal{P}_{0,\Id}(\R^d)$ are connected by a $\cW_{0,\Id}$-geodesic. Then we have
 \begin{align}\label{eq:HWI_1}
 \cE(\eta_0) &\le 
\cE(\eta_1) + 2\sqrt{\cI(\eta_0)}\cW_{0,\Id}(\eta_0,\eta_1)
    - \frac12\cW_{0,\Id}(\mu_0,\eta_1)^2\,,\\\label{eq:HWI_2}
 \cE(\eta_0|\eta_\infty) &\le 
\cE(\eta_1|\eta_\infty)+2\sqrt{\cI(\eta_0|\eta_\infty)}\cW_{0,\Id}(\eta_0,\eta_1)
    - \frac12\cW_{0,\Id}(\eta_0,\eta_1)^2\,.
\end{align}
\end{proposition}

\begin{proof}
Let $\eta_0,\eta_1\in \mathcal{P}_{0,\Id}(\R^d)$ and assume without restriction that $\mathcal{I}(\eta_0),\cE(\eta_0),\cE(\eta_1)<\infty$. Hence $\eta_0=\sigma\mathcal{L}^d$ for a suitable density $\sigma$. We will use the fact that $\nabla\log\sigma$ is in the subdifferential of the relative entropy. More precisely, by \cite[Thm.~10.4.6]{AGS} we have $\sigma\in W^{1,1}_{\sf loc}$ with $\nabla\sigma=\sigma w$ for some $w\in L^2(\eta_0;\R^d)$ and $\mathcal I(\eta_0)=\int |w|^2\dd \eta$. Moreover, $w$ belongs to the subdifferential of $\mathcal{E}$ at $\eta_0$, i.e. taking into account that $\mathcal{E}$ is convex along Wasserstein geodesics and \cite[Sec.~10.1.1 B]{AGS} we have 
\begin{equation}\label{eq:subdiff}
\mathcal{E}(\nu)-\mathcal{E}(\eta_0) \geq \int \ip{w}{T_{\eta_0}^\nu-\Id}\dd\eta_0 \qquad\forall \nu\in \cP_2(\R^d)\;,
\end{equation}
where $T_{\eta_0}^\nu$ is the optimal transport map from $\eta_0$ to $\nu$.

Now, let $(\eta_s)_{s\in[0,1]}$ be a $\cW_{0,\Id}$-geodesic connecting $\eta_0,\eta_1$.
Corollary \ref{cor:approx_cvx} yields after rearranging and dividing by $s$
\begin{equation*}
     \frac{\mathcal{E}(\eta_s)-\mathcal{E}(\eta_0)}{s} \leq \mathcal{E}(\eta_1)-\mathcal{E}(\eta_0) -\frac12(1-s) \cW_{0,\Id}(\eta_0,\eta_1)^2
     \;.
\end{equation*}
Using \eqref{eq:subdiff} and Cauchy-Schwartz inequality yields
\begin{align*}
 \frac{\mathcal{E}(\eta_s)-\mathcal{E}(\eta_0)}{s}
 \geq -\frac{1}{s}\|T_{\eta_0}^{\eta_s}-\Id\|_{L^2(\eta_0)}\|w\|_{L^2(\eta_0)} = -\frac{1}{s}W_2(\eta_0,\eta_s)\sqrt{\cI(\eta_0)}\;.
\end{align*}
Finally, the bound~\eqref{eq:W_2comp} gives that 
\begin{equation*}
    \lim_{s\to 0} \frac{1}{s}W_2(\eta_0,\eta_s)= \lim_{s\to 0} \frac{2}{s}\cW_{0,\Id}(\eta_0,\eta_s) = 2\cW_{0,\Id}(\eta_0,\eta_1)\;.
\end{equation*}
Combining the last three observations yields after letting $s\to0$ that 
\begin{equation*}
    -2\cW_{0,\Id}(\eta_0,\eta_1)\sqrt{\cI(\eta_0)} \leq \cE(\eta_1)-\cE(\eta_0) -\frac12 \cW_{0,\Id}(\eta_0,\eta_1)^2\;.
\end{equation*}
The second claim \eqref{eq:HWI_2} follows by the same argument using that also $\cE(\cdot|\gamma)$ is $1$-convex along $\cW_{0,\Id}$-geodesics.
\end{proof}

\appendix

\section{Scalar modularity: variance-modulated optimal transport}\label{sec:scalar}

Similarly to Lemma~\ref{lem:apriori-cov}, we begin by demonstrating that non-degeneracy is preserved along finite action curves.

\begin{lemma}\label{lem:apriori}
  Let $(\mu,V)\in\CE(\mu_0,\mu_1)$ with $\mu_0,\mu_1\in \cP_2(\R^d)$ and $\var \mu_0>0$ be of finite action, i.e. 
  \begin{align*}
    A:= \int_0^1\frac{1}{2\var(\mu_t)}\int|V_t|^2\dd\mu_t\dd t < \infty\;.
  \end{align*}
  Then the curves $t\mapsto m_t:=m(\mu_t)$ and $\sigma_t:=\sqrt{\var(\mu_t)}$ are absolutely continuous and 
  \begin{align}\label{e:var:apriori}
	\sigma_0\, e^{-\sqrt{2A}} \leq \sigma_t\leq \sigma_0\, e^{\sqrt{2A}}\;.
  \end{align}
\end{lemma}
\begin{proof}
  By a suitable truncation argument, we can use the function $x\mapsto \abs{x-m_t}^2$ as a test function in the weak formulation of the continuity equation. Let us consider the case $\sigma_0>0$.
  Hence, we can differentiate and get by an application of the Cauchy-Schwarz inequality
  \begin{equation*}
  	\abs*{\pderiv{\sigma_t}{t}} = \frac{1}{\sigma_t} \abs*{\int \bra*{x-m_t}\cdot V_t \dx{\mu_t}  } \leq \sqrt{2} \sigma_t   \sqrt{\frac{1}{2\sigma_t^2}\int \abs{V_t}^2\dx{\mu_t}} .\qedhere
  \end{equation*}
\end{proof}

\subsection{Separation of Optimization Problems: Proof of Theorem~\ref{thm:main-var}}

Lemma~\ref{lem:apriori} allows us to separate the optimization over the evolution of mean and variance. We have  that
\begin{equation}
  \label{eq:doubleinf}
  \cWv(\mu_0,\mu_1)^2=\inf\set*{\cWv_{m,\sigma}(\mu_0,\mu_1)^2~:~(m,\sigma)\in\MV^{\var}(\mu_0,\mu_1)}\;.
\end{equation}
Here $\MV^{\var}(\mu_0,\mu_1)$ denotes the set of all absolutely continuous
functions $m:[0,1]\to\R^d$ and $\sigma:[0,1]\to[0,\infty)$ such that
$m_i=m(\mu_i)$ and $\sigma_i^2=\var(\mu_i)$ for $i=0,1$. For given
functions $(m,\sigma)\in\MV^{\var}$, the term $\cWv_{m,\sigma}$ is defined via
the variance-constrained optimal transport problem
\begin{equation}\label{eq:OT-var-constraint}
   \cWv_{m,\sigma}(\mu_0,\mu_1)^2 = \inf\set*{\int_0^1\frac{1}{2\sigma_t^2}\int|V_t|^2\dd\mu_t\dd t~:~(\mu,V)\in \CE^{\var}_{m,\sigma}(\mu_0,\mu_1)}\;,  
\end{equation}
where $\CE^{\var}_{m,\sigma}(\mu_0,\mu_1)$ is the set of pairs $(\mu,V)\in\CE(\mu_0,\mu_1)$ such that
 $m(\mu_t)=m_t$ and $\var(\mu_t)=\sigma_t^2$ for all $t\in[0,1]$.
 
We will now show that problem \eqref{eq:doubleinf} can be equivalently
written as a minimization problem for the evolution of mean and
variance \eqref{eq:mean-var-min} plus an independent variance-constrained transport problem where the mean and variance are fixed to~$0$ and~$1$, respectively, given by \eqref{eq:OT-var-constraint-main}.
\begin{proof}[Proof of Theorem~\ref{thm:main-var}]~\\
  \textbf{Step 1.} Assume that $\var(\mu_0),\var(\mu_1)>0$. The
  Wasserstein geodesic connecting $\mu_0$ and $\mu_1$ is thanks to a priori bound on the variance~\eqref{e:var_along_W2} a feasible
  candidate in the optimization problem~\eqref{eq:OT-var} and the
  variance is by Lemma~\ref{lem:apriori} bounded away from zero along this curve. This shows that
  $\cWv(\mu_0,\mu_1)<\infty$. 
  
  \smallskip\noindent
  \textbf{Step 2.} Fix $(m,\sigma)\in\MV^{\var}(\mu_0,\mu_1)$ with $\sigma_t>0$
  for all $t\in[0,1]$ and let $(\mu,V)\in\CE^{\var}_{m,\sigma}(\mu_0,\mu_1)$
  with
  \begin{equation*}
    \int_0^1\frac{1}{2\sigma_t^2}\int |V_t|^2\dd\mu_t\dd t <\infty\;.
  \end{equation*}
  Consider the normalizations $\widebar\mu_t=(T_t)_\#\mu_t$ with
  $T_t=T_{m_t,\sigma_t}$. Then, we have that
  $(\widebar\mu,\widebar V)\in \CE_{0,1}(\widebar\mu_0,\widebar\mu_1)$ with
  \begin{align*}
    \widebar V_t(x)&=  \frac1{\sigma_t}V_t(T_t^{-1}x)-\nabla \phi_{m,\sigma}(t,T_t^{-1}x)\;,\\
    \phi_{m,\sigma}(t,x)&=\frac{\dot m_t\cdot x}{\sigma_t} +\frac{\dot\sigma_t}{2\sigma_t^2}|x-m_t|^2\;.
  \end{align*}
  Moreover, we have
  \begin{align}\label{eq:action-trans-var}
    \frac{1}{2\sigma^2_t}\int|V_t|^2\dd\mu_t = \frac{|\dot m_t|^2 +\dot \sigma^2_t}{2\sigma^2_t} + \frac{1}{2}\int|\widebar V_t|^2\dd\widebar\mu_t\;.
  \end{align}
Indeed, for a test function $\psi\in C^\infty_c(\R^d)$, we have
\begin{align*}
  \frac{\dd}{\dd t}\int\psi\dd\widebar\mu_t &=   \frac{\dd}{\dd t}\int\psi\circ T_t\dd\mu_t
                                       = \int \nabla \psi\bigl(T_t(x)\bigr)\cdot \Big[DT_t(x)V_t(x) + \partial_tT_t(x)\Big]\dd\mu_t(x)\\
                                     &=\int \nabla \psi\bigl(T_t(x)\bigr)\cdot \Big[\frac1{\sigma_t}V_t(x) - \nabla\phi_{m,\sigma}(t,x)\Big]\dd\mu_t(x)\\
  &=\int \! \nabla \psi(x)\cdot \Big[\frac1{\sigma_t}V_t\bigl(T_t^{-1}(x)\bigr) - \nabla\phi_{m,\sigma}\big(t,T_t^{-1}(x)\big)\Big]\dd (T_t)_\#\mu_t(x) = \int \! \nabla\psi\cdot \widebar V_t \dd\widebar\mu_t .
\end{align*}
This yields the first claim. For the action we obtain
\begin{align*}
  \frac{1}{2}\int |\widebar V_t|^2\dd\widebar\mu_t &= \frac{1}{2}\int \abs*{\frac{1}{\sigma_t}V_t - \nabla\phi_{m,\sigma}(t,\cdot)}^2\dd\mu_t\\
                                     & = \frac1{2\sigma_t^2}\int|V^2_t|\dd\mu_t + \frac{1}{2}\int |\nabla\phi_{m,\sigma}(t,\cdot)|^2\dd\mu_t -\frac{1}{\sigma_t}\int V_t\cdot\nabla\phi_{m,\sigma}(t,\cdot)\dd\mu_t\\
  &=\I_1+\I_2+\I_3\;.
\end{align*}
We easily compute
$\I_2 = \bra*{|\dot m_t|^2+\dot\sigma_t^2}/(2\sigma_t^2)$.
To compute $\I_3$, note the following. Fix $\alpha\in \R^d$ and $\beta\in(0,\infty)$ and let $\eta(x):=\alpha\cdot x + \frac{\beta}{2}|x|^2$. Then, we have
\begin{align*}
  \int V_t\cdot\nabla\eta\dd\mu_t = \frac{\dd}{\dd t}\int \eta\dd\mu_t =  \frac{\dd}{\dd t} \Big(\alpha\cdot m_t+\frac{\beta}{2}\big(\sigma^2_t+|m_t|^2\big)\Big)=  \alpha\cdot\dot m_t + \beta\big(\sigma_t\dot\sigma_t+m_t\cdot\dot m_t\big)\;.
\end{align*}
Putting $\alpha=\dot m_t/\sigma_t-m_t\dot\sigma_t/\sigma_t^2$ and $\beta=\dot\sigma_t/\sigma_t^2$, we have $\nabla\phi_{m,\sigma}(t,\cdot)=\nabla\eta$ and hence
\begin{align*}
  \I_3= -\frac{|\dot m_t|^2+\sigma_t^2}{\sigma_t^2} = - 2 \I_2\;.
\end{align*}
Combining $\I_1, \I_2,\ I_3$, we obtain \eqref{eq:action-trans-var}.

\smallskip\noindent
\textbf{Step 3.} If $\var(\mu_0),\var(\mu_1)>0$, we see from Lemma
\ref{lem:apriori} that the infimum in \eqref{eq:doubleinf} can be
restricted to $(m,\sigma)\in\MV^{\var}(\mu_0,\mu_1)$ with $\sigma_t>0$ for
all $t$. Then, the sets of admissible curves $\mu$ and $\widebar \mu$ are
in bijection via the transformation of normalization. Moreover, for fixed~$\mu$ the vector field $V$ is optimal, i.e.~achieving minimal action, if it is of gradient form. Thus, if $V$ is optimal then so is $\widebar V$. From the
previous step we conclude for fixed such $(m,\sigma)$ that
\begin{equation}\label{eq:constraint-rewrite}
  \cWv_{m,\sigma}(\mu_0,\mu_1)^2 = \cWv_{0,1}(\widebar\mu_0,\widebar\mu_1)^2
    + \int_0^1\frac{|\dot m_t|^2 +\dot \sigma_t^2}{2\sigma_t^2}\dd t\;.
  \end{equation}
   Moreover, the optimal curve for
  $\cWv_{m,\sigma}$ is $\mu_t=(T_t^{-1})_\#\widebar\mu_t$, where
  $(\widebar\mu_t)$ is the optimal curve for $\cWv_{0,1}$.
  Taking the infimum over $(m,\sigma)\in\MV^{\var}(\mu_0,\mu_1)$ yields
  \eqref{eq:rewrite-var-main}.
\end{proof}

\subsection{The Mean-Variance Optimization Problem: Proof of Theorem~\ref{thm:moments-sol}}\label{sec:sol-m-var}

\begin{proof}[Proof of Theorem~\ref{thm:moments-sol}]
We can integrate the first equation \eqref{eq:mean-var:mean} of the Euler-Lagrange conditions and get that $\dot m(t) = \alpha \sigma^2(t)$ for some $\alpha \in \R^d$ and all $t\in [0,1]$, which satisfies 
\begin{equation}\label{eq:var:alpha}
	\alpha = \frac{m_1- m_0}{\int_0^1 \sigma(t)^2 \dx{t}}\,.
\end{equation}
In particular, we arrive at
\begin{equation}\label{eq:sigma}
	\frac{\ddot\sigma}{\sigma} - \frac{\bra*{\dot\sigma}^2}{\sigma^2} = - \abs{\alpha}^2 \sigma^2\,.
\end{equation}
For $m_0=m_1$, we have $\alpha=0$ and hence we get in this case the solution
\begin{equation}\label{e:sol:m0EQm1}
	\sigma(t) = \sigma_0^{1-t} \sigma_1^{t}\, . 
\end{equation}
The general solution for $m_0\ne m_1$ and hence $\alpha\ne 0$ to the equation~\eqref{eq:sigma} is for some $\beta\in \R$ and $t_0\geq 0$ 
\begin{equation}\label{eq:sol-alphanot0}
	\sigma(t) = \frac{\beta}{\abs{\alpha} \cosh(\beta t+t_0)}\, . 
\end{equation}
Before resolving the boundary values in terms of $\beta$ and $t_0$, we first look for the value of $\alpha$ in terms of the solution in the form~\eqref{eq:sol-alphanot0}, where we note that $\int \frac{\dx{t}}{\cosh(t)^2}= \tanh(t)$ and hence for any $t\in [0,1]$ 
\begin{equation*}
    \int_0^t \sigma(\tau)^2 \dx{\tau} = \frac{\beta}{\alpha^2} \int_{t_0}^{\beta t+t_0} \frac{\dx{s}}{\cosh(s)^2} = \frac{\beta}{\abs*{\alpha}^2}\bra*{\tanh\bra*{\beta t+t_0}-\tanh\bra*{t_0}} \,. 
\end{equation*}
Hence, we obtain from~\eqref{eq:mean-var:mean} and~\eqref{eq:var:alpha} provided that $\beta\ne 0$ and using $n=\abs{m_1-m_0}>0$, that
\begin{equation}\label{eq:mean-var:alpha:exp}
	\abs{\alpha} = \frac{\beta}{n} \bra*{\tanh\bra*{\beta+t_0}-\tanh\bra*{t_0}}\, .
\end{equation}
Therewith, we get for the mean from~\eqref{eq:var:alpha} the explicit expression~\eqref{eq:mean-var:mean-explicit}
and from~\eqref{eq:sol-alphanot0} also~\eqref{eq:mean-var:sigma-explicit}.
Next, we aim to evaluate the optimal cost depending on the parameters $\beta$ and $t_0$. 
Recalling that $\dot m(t) = \alpha \sigma(t)^2$, and noting that $\int \tanh(t)^2 \dx{t} = t - \tanh(t)$, we have by~\eqref{eq:mean-var:alpha:exp} the identity
\begin{align}
\MoveEqLeft{\int_0^1 \frac{\abs*{ \dot m(t)}^2 + \abs*{\dot \sigma(t)}^2}{\sigma(t)^2}  \dx{t}= \abs*{\alpha}^2 \int_0^1 \sigma(t)^2 \dx{t} + \beta^2 \int_0^1 \tanh\bra*{\beta t+t_0}^2 \dx{t} }\nonumber\\
	&= \beta  \bra*{\tanh\bra*{\beta+t_0}-\tanh\bra*{t_0}} + \beta\bra*{\beta + \bra*{\tanh\bra*{t_0}-\tanh\bra*{\beta+t_0}}} = \beta^2\, .  \label{e:meanvar:dist:beta}
\end{align}
Hence, we have to solve for the boundary values $\sigma(0)=\sigma_0$ and $\sigma(1)=\sigma_1$ in terms of $\beta$ and $t_0$. For this we recall the addition theorem for the hyperbolic trigonometric functions and can write
the system $\sigma(0)=\sigma_0$, $\sigma(1)=\sigma_1$ as
\begin{align*}
	\frac{\sigma_0}{n} &= \frac{\cosh(\beta+t_0)}{\sinh(\beta+t_0)\cosh(t_0)-\sinh(t_0)\cosh(\beta+t_0)} = \frac{\cosh(\beta+t_0)}{\sinh(\beta)}\,, \\
	\frac{\sigma_1}{n} &= \frac{\cosh(t_0)}{\sinh(\beta)}\,.
\end{align*}
We set $\eta_0 = \frac{\sigma_0}{n}>0$ and $\eta_1=\frac{\sigma_1}{n}>0$. Moreover, we make the substitutions 
\begin{equation}\label{eq:mean-var:gamma-delta-subst}
\beta=\log \delta \quad\text{for some $\delta >1$}\qquad\text{ and }\qquad t_0=\log \gamma \quad\text{for some $\gamma>0$.}
\end{equation}
By noting that $2\cosh(\log r) = r + \frac{1}{r}$ and $2\sinh(\log r)= r - \frac{1}{r}$ for $r>0$, we arrive at the simplified system
\begin{align*}
	\eta_0 &= \frac{\delta \gamma + \frac{1}{\delta \gamma}}{\delta - \frac{1}{\delta}} = \frac{\delta^2 \gamma^2 + 1}{\gamma(\delta^2  - 1)} \,,\qquad
	\eta_1 = \frac{\gamma + \frac{1}{\gamma}}{\delta - \frac{1}{\delta}} = \frac{\delta(\gamma^2 +1)}{\gamma(\delta^2-1)}\,.
\end{align*}
We solve the first equation for $\delta$ leading to 
\begin{equation}\label{e:meanvar:delta}
	\delta = \frac{\sqrt{(1+\eta_0 \gamma)}}{\sqrt{(\eta_0 - \gamma)\gamma}}\,.
\end{equation}
Plugging this into the second equation, we obtain
\begin{equation*}
	\eta_1 \gamma =  \sqrt{(\eta_0 - \gamma)\gamma(1+\eta_0 \gamma)}\,.
\end{equation*}
Since $\gamma>0$, we get another quadratic equation after dividing by $\sqrt{\gamma}$ and squaring. Its positive solution is given by
\begin{equation}\label{eq:mean-var:gamma}
	\gamma = \frac{1}{2\eta_0} \bra*{ \eta_0^2 - \eta_1^2 -1 + \sqrt{4\eta_0^2 + \bra*{\eta_0^2 - \eta_1^2 -1}^2}}\, ,
\end{equation}
which immediately gives $\gamma \geq 1$. Similarly, we can evaluate $\delta$ from~\eqref{e:meanvar:delta} and using the identity
\[
  4\eta_0^2 + \bra*{\eta_0^2 - \eta_1^2 -1}^2 = \bra*{\eta_0^2 + \eta_1^2 + 1}^2 - 4 \eta_0^2 \eta_1^2\,,
\]
we arrive at
\begin{equation*}
	\delta = \frac{1}{2\eta_0 \eta_1} \bra*{  \eta_0^2 + \eta_1^2  + 1 - \sqrt{\bra*{\eta_0^2 + \eta_1^2 + 1}^2 - 4 \eta_0^2 \eta_1^2}}\, ,
\end{equation*}
which again entails that $\delta \geq 1$.
Using the relation $\beta=\log \delta$ and~\eqref{e:meanvar:dist:beta}, we obtain the right-hand side of~\eqref{e:mean-var:opt}.
Similarly, using $t_0 = \log \gamma$, we obtain~\eqref{eq:mean-var:t0-explicit} from~\eqref{eq:mean-var:gamma} and also the non-negativity of $\beta$ and $t_0$.
\end{proof}

\subsection{Convergence rates for gradient flows: Proof of Proposition~\ref{prop:ent-decay-var}}\label{appendix:Proof:prop:ent-decay-var}
		
\begin{proof}[Proof of Proposition~\ref{prop:ent-decay-var}]
	The LSI follows from classical arguments (e.g.~\cite{Gross1975,BakryEmery,Bakry2014}), noting that the optimal constant is given by $C_{\textup{LSI}} = \frac{1}{2\lambda_{\min}(\operatorname{Hess} H) }=\|B\|_2/2$ with the confining potential $H$ as given in \eqref{eq:potential}. 
	
	The scalar nature of the variance allows us to arrive at the time-homogeneous problem after introducing the new time scale~\eqref{eq:def:tau-scale}, that is $\pderiv{t}{\tau} = \var(\rho_t)$. 
	Therewith, time-rescaled solution $\tilde C_\tau  = C_{t(\tau)}$ then satisfies $ \dot{\tilde  C}_t = 2 (\Id - B^{-1} \tilde C_t)$ and is explicitly given by 
	\begin{equation*}
		\tilde C_\tau = (\Id-  e^{-2B^{-1}\tau}) B + e^{-2B^{-1} \tau} C_0\, .
	\end{equation*}
	In particular, we get a uniform lower and upper bound for all $\tau >0$ by
	\begin{equation}\label{e:varbound}
		d \min\set*{ \|B^{-1}\|_2^{-1},  \|C_0^{-1}\|_2^{-1}}  \leq  \var(\rho_\tau)  =\tr \tilde C_\tau \leq d \max\set*{\|B\|_2,  \|C_0\|_2} \,.
	\end{equation}
	To conclude, we combine this bound with the usual relative entropy method
	\begin{align*}
		\pderiv{}{t} \cE(\rho_t|\rho_\infty ) 
		&= - \var(\rho_t) \int \abs*{ \nabla \log \rho_t + B^{-1}(x-x_0)}^2 \dx{\rho_t} \\
		&\leq - \frac{2d  \min\set*{ \|B^{-1}\|_2^{-1},  \|C_0^{-1}\|_2^{-1}}  }{\|B\|_2}  \cE(\rho_t| \rho_\infty) \,,
	\end{align*}
	where we used the variance bound~\eqref{e:varbound} and the LSI \eqref{eq:LSI-var}.
\end{proof}

\bibliographystyle{abbrv}
\bibliography{bib}

\end{document}